\documentclass[11pt]{article}
\usepackage{amsmath,latexsym,amssymb,amsfonts,amsbsy, amsthm}
\usepackage{graphicx}
\usepackage{color,xcolor}
\usepackage{ulem}
\def\inte#1{
\displaystyle\mathop{#1\kern0pt}^\circ }

\let\grad\nabla

\def\virgp{\raise 2pt\hbox{,}}
\def\cdotpv{\raise 2pt\hbox{;}}

\def\eqdefa{\buildrel\hbox{\footnotesize def}\over =}

\def\C{\mathop{\bf C\kern 0pt}\nolimits}
\def\DD{\mathop{\bf D\kern 0pt}\nolimits}
\def\K{\mathop{\bf K\kern 0pt}\nolimits}
\def\N{\mathop{\bf N\kern 0pt}\nolimits}
\def\Q{\mathop{\bf Q\kern 0pt}\nolimits}
\def\R{\mathop{\bf R\kern 0pt}\nolimits}
\def\SS{\mathop{\bf S\kern 0pt}\nolimits}
\def\ZZ{\mathop{\bf Z\kern 0pt}\nolimits}
\def\TT{\mathop{\bf T\kern 0pt}\nolimits}

\def\dive{\mathop{\rm div}\nolimits}

\newcommand{\beq}{\begin{equation}}
\newcommand{\eeq}{\end{equation}}
\newcommand{\ben}{\begin{eqnarray}}
\newcommand{\een}{\end{eqnarray}}
\newcommand{\beno}{\begin{eqnarray*}}
\newcommand{\eeno}{\end{eqnarray*}}

\newtheorem{thm}{Theorem}[section]
\newtheorem{lem}{Lemma}[section]
\newtheorem{rmk}{Remark}[section]

\newtheorem{prop}{Proposition}[section]
\renewcommand{\theequation}{\thesection.\arabic{equation}}

\begin{document}

\title{Lagrangian approach to global well-posedness of the viscous surface wave equations without surface tension}

\author{ Guilong Gui \footnote{Center for Nonlinear Studies, School of Mathematics, Northwest University, Xi'an 710069, China. Email: {\tt glgui@amss.ac.cn}.}
}

\date{}
\maketitle

\begin{abstract}

In this paper, we revisit the global well-posedness of the classical viscous surface waves in the absence of surface tension effect with the reference domain being the horizontal infinite slab, for which the first complete proof was given in Guo-Tice (2013, Analysis and PDE) via a hybrid of Eulerian and Lagrangian schemes. The fluid dynamics are governed by the gravity-driven incompressible Navier-Stokes equations. Even though Lagrangian
formulation is most natural to study free boundary value problems for incompressible flows, few mathematical works for global existence are based on such an approach in the absence of surface tension effect, due to breakdown of Beale's transformation. We develop a mathematical approach to establish global well-posedness based on the Lagrangian framework by analyzing suitable "good unknowns" associated with the problem, which requires no nonlinear compatibility conditions on the initial data.

\end{abstract}


\vskip 0.2cm


\renewcommand{\theequation}{\thesection.\arabic{equation}}
\setcounter{equation}{0}
\section{Introduction}
\subsection{Formulation in Eulerian Coordinates}
We consider in this paper the global existence of time-dependent flows of an
viscous incompressible fluid in a moving domain $\Omega(t)$
with an upper free surface $\Sigma_{F}(t)$ and a fixed bottom $\Sigma_B$
\begin{equation}\label{VFS-eqns-1}
\begin{cases}
  &\partial_t u + (u\cdot \grad) u + \grad p- \nu\Delta\,u=-g\,e_1 \quad\mbox{in} \quad \Omega(t),\\
  &\grad \cdot\, u=0 \quad \mbox{in} \quad \Omega(t),\\
 & (p\,\mathbb{I}-\nu\mathbb{D}(u))n(t)=p_{\mbox{\tiny atm}}n(t)\quad \mbox{on} \quad \Sigma_{F}(t),\\
 &  \mathcal{V}(\Sigma_{F}(t))=u\cdot n(t) \quad \mbox{on} \quad \Sigma_{F}(t),\\
 &u|_{\Sigma_B}=0,
\end{cases}
\end{equation}
where we denote $n(t)$ the outward-pointing unit normal on $\Sigma_F(t)$, $I$ the $3\times3$ identity matrix, $(\mathbb{D}u)_{ij}=\partial_iu_j+\partial_ju_i$ the symmetric gradient of $u$, the constant $g > 0$ the strength of gravity (usually taking $g=1$ without loss of generality), and $\nu > 0$ the coefficient of viscosity.
The tensor $(pI -\nu\mathbb{D}(u))$ is known as the viscous stress tensor. Equation $\eqref{VFS-eqns-1}_1$ is the conservation of momentum, where gravity is the only external force; the second equation in \eqref{VFS-eqns-1} means the fluid is incompressible; Equation $\eqref{VFS-eqns-1}_3$ means the fluid satisfies the kinetic boundary condition on the free boundary $\Sigma_F(t)$, where $p_{atm}$ stands for the atmospheric pressure, assumed to be constant. the kinematic boundary condition $\eqref{VFS-eqns-1}_4$ states that the free boundary
$\Sigma_F(t)$ is moving with speed equal to the normal component of the fluid velocity; $\eqref{VFS-eqns-1}_5$ implies that the fluid is no-slip, no-penetrated on the fixed bottom boundary. Here the effect of surface tension is neglected on the free surface.

The problem can be equivalently stated as follows. Given an initial domain $\Omega_0 \subset \mathbb{R}^3$ bounded by a bottom surface $\Sigma_{B}$, and a top surface $\Sigma_F(0)$, as well as an initial velocity field $u_0$, where the upper boundary does not touch the bottom, we wish to find for each $t\in [0, T]$ a
domain $\Omega(t)$, a velocity field $u(t, \cdot)$ and pressure $p(t, \cdot)$ on $\Omega(t)$, and a transformation
$\bar{\eta}(t, \cdot):\,\Omega_0\rightarrow \mathbb{R}^3$ so that
\begin{equation}\label{VFS-eqns-12}
\begin{cases}
&\Omega(t)=\bar{\eta}(t, \Omega_0), \quad\bar{\eta}(t, \Sigma_B)=\Sigma_B,\\
&\partial_t\bar{\eta}=u\circ\bar{\eta},\\
  &\partial_t u + (u\cdot \grad) u + \grad p- \nu\Delta\,u=-g\,e_1 \quad \text{in} \quad \Omega(t),\\
  &\grad \cdot\, u=0 \quad \mbox{in} \quad \Omega(t),\\
 & (p\,\mathbb{I}-\nu\mathbb{D}(u))n(t)=p_{\mbox{\tiny atm}}n(t)\quad \mbox{on} \quad \Sigma_F(t),\\
 &u|_{\Sigma_B}=0,\\
 &u|_{t=0}=u_0,\quad \bar{\eta}|_{t=0}=x
\end{cases}
\end{equation}

For convenience, it is natural to
subtract the hydrostatic pressure from $p$ in the usual way by adjusting the actual pressure $p$ according to $\widetilde{p}=p+g\,x_1-p_{atm}$, and still denote the new pressure $\widetilde{p}$ by $p$ for simplicity, so that after substitution the gravity term in $\eqref{VFS-eqns-12}_3$ and the atmospheric pressure term in $\eqref{VFS-eqns-12}_5$  are eliminated. A gravity term appears in $\eqref{VFS-eqns-12}_5$.

The conditions on the initial domain $\Omega_0$ are as follows: Let the equilibrium domain $\Omega \subset \mathbb{R}^3$ be the horizontal infinite slab
\begin{equation}\label{def-domain-1}
\begin{split}
&\Omega=\{x=(x_1, x_h)|-\underline{b}<x_1<0,\quad x_h\in\mathbb{R}^2\}
\end{split}
\end{equation}
with the bottom $\Sigma_b=\{x_1=-\underline{b}\}$ and the top surface  $\Sigma_0=\{x_1=0\}$, where
$\underline{b}$ will be the depth of the fluid at infinity. We assume that $\Omega_0$ is the image of $\Omega$
under a diffeomorphism $\overline{\sigma}: \Omega \rightarrow  \Omega_0$, where $\overline{\sigma}(\Sigma_b)=\Sigma_B$, $\overline{\sigma}(\Sigma_0)=\Sigma_F(0)$, $\overline{\sigma}$ is of the form $\overline{\sigma}(x)=x+\xi_0(x)$, $\xi_0 \in \mathring{\mathcal{H}}_{\text{tan}, N-1}^{2, \Sigma_0}(\Omega)$ (to be defined later) with $N \geq 3$.

\subsection{Known results}

There are two methods used to solve this viscous surface problem: the first one is Lagrangian coordinates transformation, and the other is the flattening transformation introduced by Beale \cite{Beale-1984}.

In Lagrangian coordinates, the geometry of the domain is encoded in the flow map $\eta: (0, T)\times\Omega_0 \rightarrow\Omega(t)$ satisfying $\partial_t\bar{\eta}(t, x) = v(t, x)$ which gives the trajectory of a particle
located at $x\in \Omega_0=\Omega(0)$ at $t = 0$, where $v = u \circ \bar{\eta}$ is the Lagrangian velocity field in the fixed domain $\Omega_0$ with $u$ being the Eulerian velocity field in the
moving domain $\Omega(t)$. The Lagrangian pressure is $q = p \circ \bar{\eta}$ for $p$ the pressure in Eulerian coordinates.

By using Lagrangian coordinates transformation, Solonnikov \cite{Solonnikov-1977} proved the local well-posedness of this viscous surface problem in H\"{o}lder spaces for the fluid motion in a bounded domain whose entire boundary is a free surface. For the fluid motion in a horizontal infinite domain in this paper, Beale \cite{Beale-1981} obtained the local well-posedness in $L^2$-based space-time Sobolev spaces under the assumption with necessary compatibility condition, which was extended to $L^p$-based  space-time Sobolev spaces by Abels \cite{Abels-2005}. The local well-posedness in \cite{Beale-1981} showed that, given $v_0=u_0\in H^{r-1}(\Omega_0)$ for $r\in (3, 7/2)$, there exists a unique solution on a time interval $(0, T)$, so that $u \in K^r((0,T)\times\Omega_0)$, where
\begin{equation}
\label{1-1-8}
K^r((0,T)\times\Omega_0):=H^0((0,T);H^r(\Omega_0))\cap H^{r/2}((0,T); H^0(\Omega_0)).
\end{equation}
It also showed \cite{Beale-1981} that, if the initial domain $\Omega_0$ is the image of the equilibrium domain $\Omega$ given by \eqref{def-domain-1} under a diffeomorphism $\overline{\sigma}: \Omega \rightarrow  \Omega_0$ with $\overline{\sigma}(x)=x+\xi_0(x)$, then, for any fixed $0 < T < \infty$, there
exists a collection of sufficiently small data so that a unique solution exists on $(0, T)$ and so
that the solutions depend analytically on the data, where the flow map is modified by $\eta=Id+\xi\eqdefa \bar{\eta} \circ \bar{\sigma}$. This result suggests that solutions should exist globally in time for small data. If global solutions do exist, it is natural to expect that the
solutions should approach equilibrium as $t\rightarrow +\infty$, and therefore $\xi(t)$ should have
a limiting value such that
\begin{equation}\label{decay-eta-cond-1}
  \begin{split}
   \xi^1(\infty)=0 \quad \text{on} \quad \Sigma_0.
  \end{split}
\end{equation}
 However, Beale gave the negative answer to the global well-posedness in \cite{Beale-1981}. As a matter of fact, he proved a non-decay theorem which showed that a {\it reasonable} extension to small-data global well-posedness with decay of the free surface fails. More precisely, for $r\in (3, 7/2)$, let $\overline{K}^r(\mathbb{R}^+, \Omega)$
the space of $(v,q)$ such that $v\in K^r(\mathbb{R}^+, \Omega)$, $\xi\in K^{r-\frac{3}{2}}(\mathbb{R}^+, \Sigma_0)$, $\nabla\,q\in K^{r-2}(\mathbb{R}^+, \Omega)$,
and also $v\in L^1(\mathbb{R}^+; H^r(\Omega))$, there exists certain $\theta\in H_0^r({\Omega})$ so that there cannot exist a curve $(v(\epsilon),q (\epsilon) )\in\overline{K}^r(\mathbb{R}^+, \Omega)$ , defined for $\epsilon$ near $0$, such that the system \eqref{VFS-eqns-1} hold with $\xi_0(\epsilon) = \epsilon \theta,\,\, v_0 = 0$, $\xi(t=+\infty)=0$ on $\Sigma_0$ holds for each $\epsilon$, and $v(\epsilon)$ is of the form
$v(\epsilon)= \epsilon v^{(1)} + \epsilon^2 v^{(2)}+O(\epsilon^3)$. Let us mention that the extra condition $v\in L^1(\mathbb{R}^+; H^r(\Omega))$ in $(v,q)\in\overline{K}^r(\mathbb{R}^+, \Omega)$
 is included so that $\xi(\infty)$ is meaningful.
In the discussion of this result, Beale pointed out that it does not imply the non-existence of
global-in-time solutions, in particular by using Lagrangian approach, but rather that establishing global-in-time results requires weaker or different hypotheses than those imposed in the non-decay theorem.

In order to get the global well-posedness of the viscous surface problem, Beale \cite{Beale-1984} turned to introduce another method---the flattening transformation to successful solve the viscous surface problem with surface tension in the space-time Sobolev spaces. Such a viscous surface problem (without surface tension) was studied by Sylvester \cite{Sylvester-1990} and Tani-Tanaka \cite{Tani-Tanaka-1995} in the Beale-Solonnikov functional framework, where higher regularity and more compatibility conditions were required.

Recently, by using the flattening transformation, Guo and Tice \cite{Guo-Tice-1, Guo-Tice-2} employed the geometric structure in the Eulerian coordinates to study the local well-posedness of the system \eqref{VFS-eqns-1} in Sobolev spaces for small initial data, and then introduced the two-tiered energy method to get the algebraic decay rate of the solutions, which leads to the construction of global-in-time solutions to the surface wave problem \eqref{VFS-eqns-1}. We mention that, in this new framework, the two-tiered energy method couples the decay of low-order energy and the boundedness of high-order energy, where higher regularity and more compatibility conditions, as well as low horizontal regularity of the solutions were needed, which also avoids to contradict Beale's non-decay theorem. Wang, Tice, and Kim \cite{WTK-2014} also considered the global well-posedness and decay for two layers compressible fluid by using this two-tiered energy method. Wu \cite{Wu-2014} extended their local well-posedness result from small data to general data. More recently, Ren, Xiang and Zhang \cite{Ren-X-Zhang-2019} proved the local well-posedness of the viscous surface wave equation in low regularity Sobolev spaces, where only the first-order time-derivative of the velocity is involved in the compatibility condition, but the global existence of the solution is left open.

Beale and Nishida studied the decay properties of solutions to viscous surface waves with surface tension in \cite{Beale-Nishida-1985}, and then Hataya \cite{Hataya-2011} gave a complete proof to their decay estimates: if $\xi_0 \in L^1(\Omega)$ and initial data is small, the global smooth solution to the viscous surface waves with surface tension has the optimal (upper) decay rate $\|u(t)\|_{H^2} \lesssim (1+t)^{-1}$, which means the decay of $\|u(t)\|_{H^2}$ is not fast enough to guarantee that $u\in L^{1}([0,\infty);H^{2}(\Omega))$. This result also suggests that, for the solution $u$ to the viscous surface waves without surface tension, its norm $\|u\|_{L^{1}([0,\infty); H^{2}(\Omega))}$ may still  blow-up. In effect, with the additional low horizontal regularity assumption on the small initial data, Guo and Tice \cite{Guo-Tice-2} got the global smooth solution with the decay rate $\|u\|_{C^2(\Omega)} \lesssim (1+t)^{-1-\epsilon_0}$ for some positive constant $\epsilon_0$. From this, the norm $\|u\|_{L^{1}([0,\infty); C^2(\Omega))}$ does not blow up, but the quantity $\|u\|_{L^{1}([0,\infty); H^{2}(\Omega))}$ may.

\subsection{Our goal and ideas}

Motivated by Beale's remark on the non-decay theorem in \cite{Beale-1981}, it is of interest to know whether the Lagrangian approach works on the study of the global well-posedness of the viscous surface waves without surface tension. So far as I know, even though Lagrangian formulation is quite direct approach to study free boundary value problems for incompressible flows, there are few mathematical works for global existence based on such an approach.
On the other hand, in previous works, the global well-posedness was established for the initial data which has
high normal regularity and some compatibility conditions in terms of the time-derivatives of the velocity on the initial data are needed. Usually, it is difficult to verify the valid of the compatibility condition in terms of the time-derivatives of the velocity, which is essentially some type of the nonlinear compatibility conditions from the momentum equation in \eqref{VFS-eqns-1}. Suppose we considered the Navier-Stokes equations in the fixed domain
$\Omega$ with the condition $u = 0$ on $\partial\Omega$, for small initial velocity $u_0$, Heywood \cite{Heywood-1980} investigated the unique, global solvability to this system without any compatibility conditions in terms of $\partial_t u$ on the initial data. In the present case, a natural and important question is whether a corresponding well-posedness result can be obtained with low normal regularity and without any compatibility conditions in terms of the time-derivatives of the velocity (nonlinear compatibility conditions) on the initial data.

The aim of this paper is to employ the Lagrangian approach to study the global well-posedness of the viscous surface waves under the assumptions of low normal regularity and no any compatibility conditions in terms of the time-derivatives of the velocity on the initial data.

Let us explain our main ideas to get the global well-posedness of the viscous surface waves without surface tension.

As mentioned above, in the proof of Beale's non-decay theorem \cite{Beale-1981}, to make $\eta(\infty)|_{\Sigma_0}=0$ meaningful, Beale introduced the space-time Sobolev space $\overline{K}^r(\mathbb{R}^+, \Omega)$ in which the additional condition $v \in L^1(\mathbb{R}^+; H^2(\Omega))$ is required, in particular, the quantity $\|v(t)\|_{L^2(\Omega)}$ should be globally integral with respect to the time $t\in \mathbb{R}^+$. While inspired by the optimal decay rate $\|v(t)\|_{H^2} \lesssim (1+t)^{-1}$ for the solution $u$ to the the viscous surface waves with surface tension under the assumption $\eta_0 \in L^1(\Omega)$, it is natural to expect that the solution $v$ to the the viscous surface waves without surface tension also satisfies the optimal decay rate $\|v(t)\|_{L^2} \lesssim (1+t)^{-1}$ if $\eta_0 \in L^1(\Omega)$, which means that the condition $v \in L^1(\mathbb{R}^+; L^2(\Omega))$ will be not satisfied in any case. Therefore, the condition $v \in L^1(\mathbb{R}^+; L^2(\Omega))$  is not essential to get Beale's non-decay theorem. In effect, we will give a modified non-decay theorem (Theorem \ref{thm-nondecay} below) where the conditions $\dot{\Lambda}_h^{\sigma_0}v \in L^1([0, +\infty); \mathcal{H}^2_{\text{tan}, 2}(\Omega)) $ and $\dot{\Lambda}_h v\in L^2([0, +\infty); \mathcal{H}^1_{\text{tan}, 3}(\Omega))$ for some positive $\sigma_0 \in (0, 1)$ are required, which makes $\eta_1(\infty)|_{\Sigma_0}=0$ meaningful in \eqref{decay-eta-cond-1}, while $\eta(\infty)=0$ may not. Due to the Sobolev embedding theorem, the condition $\dot{\Lambda}_h^{\sigma_0}v \in L^1([0, +\infty); \mathcal{H}^2_{\text{tan}, 2}(\Omega)) $  ensures that $\|\nabla\,v(t)\|_{L^\infty(\Omega)}$ belongs to $L^1(\mathbb{R}^+_t)$, from this, it may guarantee the propagation of horizontal regularities of the velocity as well as its smallness, which plays a key role in the study of the global well-posedness of the the viscous surface waves. Proceed analogously to the proof of the optimal decay $\|v(t)\|_{L^2} \lesssim (1+t)^{-1}$ whence $\eta_0 \in L^1(\Omega)$ in \cite{Beale-Nishida-1985, Hataya-2011}, we may expect to get that $\|\dot{\Lambda}_h^{\sigma_0}v \|_{L^2(\Omega))} \lesssim (1+t)^{-(\sigma_0+\lambda)}$ if  $\dot{\Lambda}_h^{-\lambda}\eta_0 \in L^2(\Omega)$ with $\lambda \in (0, 1)$. As a consequence, the requirement on the low horizontal regularity $\dot{\Lambda}_h^{-\lambda}\eta_0 \in L^2(\Omega)$ with $\lambda>1-\sigma_0\in (0, 1)$ is necessary to guarantee the condition $\dot{\Lambda}_h^{\sigma_0}v \in L^1([0, +\infty); \mathcal{H}^2_{\text{tan}, 2}(\Omega))$, which also implies that the initial condition $\dot{\Lambda}_h^{-\lambda}\eta_0 \in L^2(\Omega)$ contradicts the choice of the initial data in Beale's non-decay theorem. With low horizontal regularity assumptions on the initial data, we will employ Guo-Tice's two-tiered energy method to study the decay estimates of the solution. Without any compatibility conditions in terms of the time-derivatives of the velocity on the initial data, we should avoid using the energy in terms of $\partial_tv$ and its derivatives. In this case, to close the nonlinear energy estimates, we will first introduce four good unknowns replacing $\nabla v$ and the pressure $q$, which plays a crucial role in the proof of our main theorem about the global well-posedness of the system \eqref{VFS-eqns-1}.

\subsection{Plan of the paper}

The rest of the paper is organized as follows. Section \ref{sect-form} gives some equivalent formulations of the viscous surface waves \eqref{VFS-eqns-1} in Lagrangian coordinates, and introduces four new variables according to the kinetic boundary condition $\eqref{VFS-eqns-1}_3$ in \eqref{VFS-eqns-1}. In Section \ref{sect-main}, we define the energy and the dissipation, and state the main result of our paper. We first recall, in Section \ref{sect-tool}, some necessary estimates, and then derive the basic estimates in terms of the flow map. Then, in Section \ref{sect-energy}, we apply the four new good unknowns introduced first in the paper to get energy estimates of the horizontal derivatives of the velocity. Second normal derivatives of the velocity as well as their more horizontal regularities are showed by using the Stokes estimates in Section \ref{sect-stokes}. As a consequence, the total energy estimates are stated in Section \ref{sect-total}. Finally, in Section \ref{sect-proof-mainthm}, we complete the proof of our main theorem about the global well-posedness of the system \eqref{VFS-eqns-12}. In Appendix, we give the proof of the modified non-decay theorem.

\subsection{Notations}

\medbreak  Let us end this introduction by some notations that will be used in all that follows.

For operators $A,B,$ we denote $[A;B]=AB-BA$ to be the  commutator of $A$ and $B.$ We denote $1+t$ by $\langle t\rangle$. For~$a\lesssim b$, we mean that there is a uniform constant $C,$ which may be different on different lines, such that $a\leq Cb$.  The notation $a\thicksim b$ means both $a\lesssim b$ and $b\lesssim a$. Throughout the paper, the subscript notation for vectors and tensors as well as the Einstein summation convention has been adopted unless otherwise specified, Einstein¡¯s summation convention means that repeated Latin indices $i,\, j,\, k$, etc., are summed from $1$ to $3$, and repeated Greek indices $\alpha, \,\beta, \,\gamma$, etc., are summed from $2$ to $3$.

We introduce the operator $\mathcal{P}(\partial_h)$ to denote the horizontal derivatives of some functions, here $\mathcal{P}(\partial_h)$ is a pseudo-differential operator with the symbol  $\mathcal{P}(\varsigma_h)$ depending only on the horizontal frequency $ \varsigma_h=(\varsigma_2,\, \varsigma_3)^T$, that is,
 \begin{equation}\label{def-pdo-horizontal-1}
\begin{split}
 \mathcal{P}(\partial_h)f(x_1,x_h):=\mathcal{F}^{-1}_{\varsigma_h\rightarrow x_h}\bigg(\mathcal{P}(\varsigma_h)\mathcal{F}_{x_h\rightarrow \varsigma_h}(f(x_1,x_h))\bigg).
\end{split}
\end{equation}
In particular, we denote $\dot{\Lambda}_h^s$ (or $\Lambda_h^s$) the homogeneous (or nonhomogeneous) horizontal differential operator with the symbol $|\varsigma_h|^s$ (or $\langle\varsigma_h\rangle^s$) respectively, where $s \in \mathbb{R}$, $\langle\varsigma_h\rangle\eqdefa (1+|\varsigma_h|^2)^{1/2}$.

\renewcommand{\theequation}{\thesection.\arabic{equation}}
\setcounter{equation}{0}
\section{Formulation in Lagrangian Coordinates}\label{sect-form}

\subsection{Derivation of the system in fluid in Lagrangian coordinates}\label{susect-free-surface-1}

The quite direct approach used to solve the system \eqref{VFS-eqns-1} is Lagrangian. Let us now introduce the Lagrangian coordinates in which the free boundary becomes fixed.

{\bf 1. The flow map}:

Let $\eta\eqdefa \bar{\eta} \circ \bar{\sigma}$ be a position of the fluid particle $x$ in the equilibrium domain $\Omega$ at time $t$ so that
\begin{equation}\label{def-flowmap-1}
\begin{cases}
&\frac{d}{dt}\eta(t, x)=u(t, \eta(t, x)), \quad t>0, \, x\in \Omega,\\
&\eta|_{t=0}=x+\xi_0(x), \quad x\in \Omega,
\end{cases}
\end{equation}
then the displacement $\xi(t, x)\eqdefa \eta(t, x)-x$ satisfies
\begin{equation}\label{def-flowmap-2}
\begin{cases}
&\frac{d}{dt}\xi(t, x)=u(t, x+\xi(t, x)),\\
&\xi|_{t=0}=\xi_0.
\end{cases}
\end{equation}
We define Lagrangian quantities the velocity $v$ and the pressure $q$ in fluid as (where $x=(x_1, x_2, x_3)^T\in \Omega$):
\begin{equation*}
\begin{split}
&v(t, x)\eqdefa u(t, \eta(t, x)),\quad q(t, x)\eqdefa p(t, \eta(t, x)).
\end{split}
\end{equation*}
Denote the Jacobian of the flow map $\eta$ by
$J \eqdefa \mbox{det}(D\eta),
$
where
\begin{equation*}\label{pert-momentum-1-1}
\begin{split}
 &D\eta=\left(
\begin{array}{lll}
 1+\partial_1\xi^1 &  \partial_2\xi^1& \partial_3\xi^1\\
 \partial_1\xi^2&1+\partial_2\xi^2  &\partial_3 \xi^2\\
  \partial_1\xi^3 &\partial_2\xi^3& 1+\partial_3\xi^3 \\
\end{array}\right).
\end{split}
\end{equation*}
Define $\mathcal{A} \eqdefa  (D\eta)^{-T}$, then according to definitions of the flow map $\eta$ and the displacement $\xi$, we may get the identities:
\begin{equation}\label{flow-map-identity-1}
\mathcal{A}_{i}^k \partial_{k} \eta^j=\mathcal{A}_{k}^j \partial_{i} \eta^k=\delta_i^j, \quad \partial_k(J\mathcal{A}_{i}^k)=0,\quad\partial_{i} \eta^j=\delta_i^j+\partial_{i} \xi^j, \quad \mathcal{A}_{i}^j=\delta_i^j-\mathcal{A}_{i}^k \partial_k\xi^j.
\end{equation}
Set $a_{ij}\eqdefa\,J\,\mathcal{A}_i^j$,
the explicit form of the entries of the matrix $J\,\mathcal{A}$ can be found from the definition of $\mathcal{A}$ that
\begin{equation}\label{expre-a-1}
  \begin{split}
  & a_{11}=(1+\partial_2\xi^2)(1+\partial_3\xi^3)-\partial_2\xi^3  \partial_3\xi^2,\, a_{12}=-(\partial_1\xi^2+\partial_1\xi^2\partial_3\xi^3-\partial_1\xi^3  \partial_3\xi^2),\\
&a_{13}=-(\partial_1\xi^3+\partial_1\xi^3\partial_2\xi^2-\partial_1\xi^2  \partial_2\xi^3),\, a_{21}=-(\partial_2\xi^1+\partial_2\xi^1\partial_3\xi^3-\partial_2\xi^3  \partial_3\xi^1),\\
&a_{22}=(1+\partial_1\xi^1)(1+\partial_3\xi^3)-\partial_1\xi^3  \partial_3\xi^1,\,a_{23}=-(\partial_2\xi^3+\partial_2\xi^3\partial_1\xi^1-\partial_1\xi^3  \partial_2\xi^1),\\
&a_{31}=-(\partial_3\xi^1+\partial_3\xi^1\partial_2\xi^2-\partial_2\xi^1  \partial_3\xi^2),\,a_{32}=-(\partial_3\xi^2+\partial_3\xi^2\partial_1\xi^1-\partial_1\xi^2  \partial_3\xi^1),\\
&a_{33}=(1+\partial_1\xi^1)(1+\partial_2\xi^2)-\partial_1\xi^2  \partial_2\xi^1.
  \end{split}
\end{equation}
If the displacement $\xi$ is sufficiently small in an appropriate Sobolev space, then the flow mapping $\eta$ is a
diffeomorphism from $\Omega$ to $\Omega(t)$, which makes us to switch back and forth from Lagrangian
to Eulerian coordinates.

Simple computation implies that
\begin{equation*}\label{J-expres-1}
  \begin{split}
  & J=1+\grad \cdot\xi+\mathcal{B}_{00}+\mathcal{B}_{000}
  \end{split}
\end{equation*}
with
\begin{equation*}\label{J-expression-2a}
  \begin{split}
&\mathcal{B}_{00}:=\partial_1\xi^1 \nabla_h\cdot \xi^h-\partial_1\xi^h\cdot\nabla_h\xi^1+\nabla_h^{\perp}\xi^2\cdot\nabla_h\xi^3,\\
&\mathcal{B}_{000}:=\partial_1\xi^1 \nabla_h^{\perp}\,\xi^2\cdot \nabla_h\,\xi^3+\partial_1\xi^2\nabla_h^{\perp}\,\xi^3\cdot \nabla_h\,\xi^1+\partial_1\xi^3 \nabla_h^{\perp}\,\xi^1\cdot \nabla_h\,\xi^2
  \end{split}
\end{equation*}
with
$\nabla_h\eqdefa (\partial_{x_2},\, \partial_{x_3})^T,\quad \nabla_h^{\perp}\eqdefa (-\partial_{x_3},\, \partial_{x_2})^T$
and then
$\nabla_h^{\perp}\,f \cdot \nabla_h\,g=-\nabla_h^{\perp}\,g \cdot \nabla_h\,f$.

{\bf 2. Derivatives of $J$ and $\mathcal{A}$ in Lagrangian coordinates}

Next, we give some useful equations which we often use in what follows.

Since $\mathcal{A}(D\eta)^T=I$, differentiating it with respect to $t$ and $x$ once yields
\begin{equation}\label{identity-Lagrangian-1}
\begin{split}
&\partial_t \mathcal{A}_{i}^j=-\mathcal{A}_{k}^j\mathcal{A}_{i}^{m} \partial_{m}v^k,\,\partial_{s} \mathcal{A}_{i}^j=-\mathcal{A}_{k}^j\mathcal{A}_{i}^{m} \partial_{m}\partial_{s}\xi^k,
\end{split}
\end{equation}
where we used the fact $\partial_t\eta=v$ in the first equation in \eqref{identity-Lagrangian-1}.
Whence differentiate the Jacobian determinant $J$, we get
\begin{equation}\label{identity-deri-J}
\begin{split}
&\partial_t J =J \mathcal{A}_{i}^j \partial_j v^i, \quad \partial_{k} J =J \mathcal{A}_{i}^j \partial_j\partial_{k} \xi^i.
\end{split}
\end{equation}
Moreover, it is easy to verify the following Piola identity:
\begin{equation}\label{identity-Piola}
\begin{split}
&\partial_j (J \mathcal{A}_{i}^j) =0 \quad \forall \,i = 1, 2, 3.
\end{split}
\end{equation}
Here and in what follows, the subscript notation for vectors and tensors as well as the Einstein summation convention has been adopted unless otherwise specified.

{\bf 3. Navier-Stokes equations in Lagrangian coordinates}

Under Lagrangian coordinates, we may introduce the differential operators with
their actions given by $(\nabla_{\mathcal{A}}f)_i=\mathcal{A}_i^j  \partial_jf$, $ \mathbb{D}_{\mathcal{A}} (v)=\nabla_{\mathcal{A}} v+(\nabla_{\mathcal{A}} v)^T$, $\Delta_{\mathcal{A}} f=\nabla_{\mathcal{A}}\cdot \nabla_{\mathcal{A}} f$, so the Lagrangian version of the system \eqref{VFS-eqns-1} can be written on the fixed reference domain $\Omega$ as
\begin{equation}\label{eqns-pert-1}
\begin{cases}
 & \partial_t \xi=v\quad \mbox{in}\quad  \Omega, \\
 & \partial_t v + \nabla_{\mathcal{A}}\,q-\nu\grad_{\mathcal{A}} \cdot \mathbb{D}_{\mathcal{A}}(v)=0\quad \mbox{in}\quad  \Omega,\\
  & \grad_{\mathcal{A}} \cdot v=0\quad \mbox{in}\quad  \Omega,\\
 & (q-\xi^1)\, \mathcal{N}- \nu\mathbb{D}_{\mathcal{A}}(v) \mathcal{N}=0 \quad \mbox{on} \quad \Sigma_{0},\\
  & v|_{\Sigma_{b}}=0,\\
 &\xi|_{t=0}=\xi_0, \,   v|_{t=0}=v_0,
\end{cases}
\end{equation}
where $n_0=(1, 0, 0)^T$ is the outward-pointing unit normal vector on the interface $\Sigma_{0}$, $\mathcal{N}:= J\mathcal{A}\,n_0$ stands for the outward-pointing normal vector on the moving interface $\Sigma_F(t)$.


\subsection{The incompressibility condition}

In the fluid, under the assumption $J \neq 0$, the incompressibility condition $\nabla_{\mathcal{A}}\cdot v=0$ is equivalent to the equation $\nabla_{J\mathcal{A}}\cdot v=0$, which implies that
 the incompressibility condition $\grad_{J\mathcal{A}}\cdot v=0$ implies
\begin{equation*}\label{incomp-cond-fluid-1}
  \begin{split}
&a_{11}\partial_1v^1
+a_{21}\partial_1v^2+a_{31}\partial_1v^3=-\nabla_h\cdot\,v^h+\mathcal{B}_{4, i}^{\alpha}\partial_{\alpha}v^{i}
  \end{split}
\end{equation*}
with
\begin{equation}\label{incomp-cond-fluid-1a}
  \begin{split}
&\mathcal{B}_{4, i}^{\alpha}\partial_{\alpha}v^{i}\eqdefa -a_{1\beta}\partial_{\beta}v^1-(a_{22}-1)\partial_2v^2
-a_{23}\partial_3v^2-a_{32}\partial_2v^3-(a_{33}-1)\partial_3v^3.
  \end{split}
\end{equation}
Hence, the incompressibility condition $\nabla_{\mathcal{A}}\cdot v=0$ is equivalent to the linearized form of the divergence-free condition that
\begin{equation}\label{incomp-cond-fluid-3}
  \begin{split}
&\nabla\cdot v= -\widetilde{a_{\alpha\,1}}\partial_1v^\alpha+\mathcal{B}_{5, i}^{\alpha}\partial_{\alpha}v^{i}
  \end{split}
\end{equation}
with
\begin{equation*}\label{incomp-cond-fluid-4}
  \begin{split}
&\widetilde{a_{21}}:=a_{11}^{-1}a_{21},\quad \widetilde{a_{31}}:=a_{11}^{-1}a_{31}, \\
 &\mathcal{B}_{5, i}^{\alpha}\partial_{\alpha}v^{i}\eqdefa-a_{11}^{-1}a_{1\beta}\partial_{\beta}v^1-a_{11}^{-1}(a_{22}-a_{11})\partial_2v^2\\
 &\qquad\qquad\qquad
-a_{11}^{-1}a_{23}\partial_3v^2-a_{11}^{-1}a_{32}\partial_2v^3-a_{11}^{-1}(a_{33}-a_{11})\partial_3v^3.
  \end{split}
\end{equation*}

\subsection{Four good unknowns}
In this section, we want to introduce four new good unknowns related to the vertical derivatives $\partial_1v$ and the pressure $q$, which plays a significant role for the proof of the main result of the paper.

Notice that $\mathcal{N}=J\mathcal{A}\,n_0$ is the outer normal vector field on the interface $\Sigma_0$,
\begin{equation*}
  \begin{split}
&\mathcal{N}= J\mathcal{A}\,n_0=a_{j1}e_j=a_{11}e_1+a_{21}e_2+a_{31}e_3.
\end{split}
\end{equation*}
Let $\tau_2$ and $\tau_3$ be two independent tangential vector fields on $\Sigma_0$
\begin{equation*}
  \begin{split}
&\tau_2=-a_{21}e_1+a_{11}e_2,\quad \tau_3=-a_{31}e_1+a_{11}e_3.
\end{split}
\end{equation*}
It is easy to see that $\mathcal{N}\cdot \tau_{\alpha}=0$ with $\alpha=2, 3$.

Since $\bigl(\mathbb{D}_{J\mathcal{A}}(v)\mathcal{N}\bigr)^i
=(\mathbb{D}_{J\mathcal{A}}(v))_{ij}\mathcal{N}^j
=(a_{ik}\partial_k v^j+a_{jk}\partial_k v^i)a_{j1}
$ with $i, j, k=1, 2, 3$, we separate all the components of $\mathbb{D}_{J\mathcal{A}}(v)\mathcal{N}$ into two types of the terms: the one is of the vertical derivative of $v$, and the other is of the horizontal derivative of $v$
\begin{equation}\label{deform-bdry-1}
  \begin{split}
&\bigl(\mathbb{D}_{J\mathcal{A}}(v)\mathcal{N}\bigr)^1=(a_{11}^2+|\overrightarrow{\rm{a}_1}|^2)\partial_1 v^1+a_{11}a_{\alpha\,1}\partial_1 v^\alpha+\mathcal{B}_{1, i}^{\alpha}\partial_{\alpha}v^{i}, \\
&\bigl(\mathbb{D}_{J\mathcal{A}}(v)\mathcal{N}\bigr)^2=a_{21}a_{11}\,\partial_1 v^1+(|\overrightarrow{\rm{a}_1}|^2+a_{21}^2)\partial_1 v^2+a_{21}a_{31}\,\partial_1 v^3+\partial_2 v^1+\mathcal{B}_{2, i}^{\alpha}\partial_{\alpha}v^{i}, \\
&\bigl(\mathbb{D}_{J\mathcal{A}}(v)\mathcal{N}\bigr)^3=a_{31}a_{11}\partial_1 v^1+a_{31}a_{21}\partial_1 v^2+(|\overrightarrow{\rm{a}_1}|^2+a_{31}^2)\partial_1 v^3+\partial_3 v^1+\mathcal{B}_{3, i}^{\alpha}\partial_{\alpha}v^{i},\\
  \end{split}
\end{equation}
where $\overrightarrow{\rm{a}_1}:=(a_{11}, a_{21}, a_{31})^T,\quad |\overrightarrow{\rm{a}_1}| :=(\sum_{i=1}^3a_{i1}^2)^{1/2}$, and
\begin{equation*}\label{deform-bdry-2}
  \begin{split}
&\mathcal{B}_{1, i}^{\alpha}\partial_{\alpha}v^{i}\eqdefa a_{i\beta}a_{i1}\partial_{\beta} v^1+a_{1\beta}a_{i\,1}\,\partial_\beta v^i,\\
&\mathcal{B}_{2, i}^{\alpha}\partial_{\alpha}v^{i}\eqdefa (a_{22}a_{11}-1)\,\partial_2 v^1+a_{23}a_{11}\,\partial_3 v^1+ a_{i\beta}a_{i1}\partial_{\beta} v^2+a_{2\beta}a_{\gamma\,1}\,\partial_{\beta} v^\gamma,\\
&\mathcal{B}_{3, i}^{\alpha}\partial_{\alpha}v^{i}\eqdefa a_{32}a_{11}\partial_2 v^1+a_{3\beta}a_{\gamma\,1}\partial_{\beta} v^{\gamma}+ a_{i\beta}a_{i1} \partial_{\beta} v^3+ (a_{33}a_{11}-1) \,\partial_3 v^1.
\end{split}
\end{equation*}

Hence, multiplying the boundary equation on $\Sigma_0$ in \eqref{eqns-pert-1} by  the tangential vectors $\tau_{\alpha}$ (with $\alpha=2, 3$) implies
\begin{equation*}
  \begin{split}
&0=\tau_2\cdot(\mathbb{D}_{J\mathcal{A}}(v)\mathcal{N})
=-a_{21} |\overrightarrow{\rm{a}_1}|^2 \partial_1 v^1+a_{11} |\overrightarrow{\rm{a}_1}|^2\partial_1 v^2\\
&\qquad\qquad\qquad\qquad\qquad+a_{11}\partial_2 v^1+(-a_{21}\mathcal{B}_{1, i}^{\alpha}+a_{11}\mathcal{B}_{2, i}^{\alpha})\partial_{\alpha}v^{i},
\end{split}
\end{equation*}
and
\begin{equation*}
  \begin{split}
&0=\tau_3\cdot(\mathbb{D}_{J\mathcal{A}}(v)\mathcal{N})
=-a_{31} |\overrightarrow{\rm{a}_1}|^2 \partial_1 v^1+a_{11} |\overrightarrow{\rm{a}_1}|^2\partial_1 v^3\\
&\qquad\qquad\qquad\qquad\qquad+a_{11}\partial_3 v^1+(-a_{31}\mathcal{B}_{1, i}^{\alpha}+a_{11}\mathcal{B}_{3, i}^{\alpha})\partial_{\alpha}v^{i},
\end{split}
\end{equation*}
which along with the incompressibility condition \eqref{incomp-cond-fluid-3} gives rise to
\begin{equation}\label{prtv-interface-1a}
  \begin{cases}
  &-a_{21}\partial_1 v^1+a_{11} \partial_1 v^2=- a_{11}|\overrightarrow{\rm{a}_1}|^{-2}\partial_2 v^1+ (a_{21}\mathcal{B}_{1, i}^{\alpha}-a_{11}\mathcal{B}_{2, i}^{\alpha}) |\overrightarrow{\rm{a}_1}|^{-2}\partial_{\alpha}v^{i},\\
&-a_{31}\partial_1 v^1+a_{11}\partial_1 v^3=- a_{11}|\overrightarrow{\rm{a}_1}|^{-2}\partial_3 v^1+ (a_{31}\mathcal{B}_{1, i}^{\alpha}-a_{11}\mathcal{B}_{3, i}^{\alpha})|\overrightarrow{\rm{a}_1}|^{-2}\partial_{\alpha}v^{i},\\
&a_{11}\partial_1v^1
+a_{21}\partial_1v^2+a_{31}\partial_1v^3=-\nabla_h\cdot\,v^h+\mathcal{B}_{4, i}^{\alpha}\partial_{\alpha}v^{i}.
  \end{cases}
\end{equation}
Solving \eqref{prtv-interface-1a} in terms of the vertical derivative of $v$, we conclude that on the interface $\Sigma_0$
\begin{equation}\label{prtv-interface-1}
  \begin{split}
  &\partial_1v^1=-\nabla_h\cdot\,v^h+\mathcal{B}_{6, i}^{\alpha}\partial_{\alpha}v^{i},\,\partial_1 v^2=-\partial_2 v^1+\mathcal{B}_{7, i}^{\alpha}\partial_{\alpha}v^{i},\, \partial_1 v^3=-\partial_3 v^1+\mathcal{B}_{8, i}^{\alpha}\partial_{\alpha}v^{i},
\end{split}
\end{equation}
where
\begin{equation*}\label{prtv-interface-2}
  \begin{split}
   &\mathcal{B}_{6, i}^{\alpha}\partial_{\alpha}v^{i}\eqdefa -\widetilde{a_{21}}\partial_2 v^1-\widetilde{a_{31}}\partial_3 v^1+(\widetilde{a_{21}}\mathcal{B}_{7, i}^{\alpha}+\widetilde{a_{31}}\mathcal{B}_{8, i}^{\alpha}+\mathcal{B}_{5, i}^{\alpha})\partial_{\alpha}v^{i},\\
  &\mathcal{B}_{7, i}^{\alpha}\partial_{\alpha}v^{i}\eqdefa a_{11}^{-1}|\overrightarrow{\rm{a}_1}|^{-4}\bigg[-a_{11}(a_{11}^2+a_{31}^2-|\overrightarrow{\rm{a}_1}|^{4})\partial_2v^1\\
  &\qquad
  +(a_{11}^2+a_{31}^2)\bigg((a_{21}\mathcal{B}_{1, i}^{\alpha}-a_{11}\mathcal{B}_{2, i}^{\alpha})\partial_{\alpha}v^{i} +|\overrightarrow{\rm{a}_1}|^{2}a_{21}(-\nabla_h\cdot\,v^h+\mathcal{B}_{5, i}^{\alpha}\partial_{\alpha}v^{i})\bigg)\\
  &\qquad-a_{21}a_{31}\bigg(- a_{11} \partial_3 v^1+ (a_{31}\mathcal{B}_{1, i}^{\alpha}-a_{11}\mathcal{B}_{3, i}^{\alpha}) \partial_{\alpha}v^{i}\\
  &\qquad\qquad\qquad\qquad\qquad\qquad+|\overrightarrow{\rm{a}_1}|^{2}a_{31}(-\nabla_h\cdot\,v^h+\mathcal{B}_{5, i}^{\alpha}\partial_{\alpha}v^{i})\bigg)\bigg],\\
&\mathcal{B}_{8, i}^{\alpha}\partial_{\alpha}v^{i}\eqdefa   a_{11}^{-1}|\overrightarrow{\rm{a}_1}|^{-4}\bigg[-a_{11}(a_{11}^2+a_{21}^2-|\overrightarrow{\rm{a}_1}|^{4})\partial_3v^1\\
&\qquad+(a_{11}^2+a_{21}^2)\bigg((a_{31}\mathcal{B}_{1, i}^{\alpha}-a_{11}\mathcal{B}_{3, i}^{\alpha}) \partial_{\alpha}v^{i}+|\overrightarrow{\rm{a}_1}|^{2}a_{31}(-\nabla_h\cdot\,v^h+\mathcal{B}_{5, i}^{\alpha}\partial_{\alpha}v^{i})\bigg)\\
  &\qquad-a_{21}a_{31}\bigg(-a_{11}\partial_2v^1+a_{11}(a_{21}\mathcal{B}_{1, i}^{\alpha}-a_{11}\mathcal{B}_{2, i}^{\alpha})\partial_{\alpha}v^{i} \\
  &\qquad\qquad\qquad\qquad\qquad\qquad+|\overrightarrow{\rm{a}_1}|^{2}a_{21}(-\nabla_h\cdot\,v^h+\mathcal{B}_{5, i}^{\alpha}\partial_{\alpha}v^{i})\bigg)\bigg].
\end{split}
\end{equation*}
On the other hand, taking the dot product between the boundary condition on $\Sigma_0$ in \eqref{eqns-pert-1} and $\mathcal{N}$, it immediately follows
\begin{equation*}
  \begin{split}
&(q-\xi^1)\, \mathcal{N}\cdot\mathcal{N}- \nu\mathcal{N}\cdot(\mathbb{D}_{\mathcal{A}}(v) \mathcal{N})=0,
  \end{split}
\end{equation*}
which, together with $\mathcal{N}\cdot\mathcal{N}=|\overrightarrow{\rm{a}_1}|^{2}$, \eqref{incomp-cond-fluid-1}, \eqref{deform-bdry-1}, \eqref{prtv-interface-1a}, and
\begin{equation*}
  \begin{split}
\mathcal{N}\cdot\bigl(\mathbb{D}_{J\mathcal{A}}(v)\mathcal{N}\bigr)
&=2|\overrightarrow{\rm{a}_1}|^{2}a_{i1}\partial_1 v^i+a_{\beta\,1}\partial_{\beta} v^1 + a_{j1}\mathcal{B}_{j, i}^{\alpha}\partial_{\alpha}v^{i}\\
&=2|\overrightarrow{\rm{a}_1}|^{2}(-\nabla_h\cdot\,v^h+\mathcal{B}_{4, i}^{\alpha}\partial_{\alpha}v^{i})+a_{\beta\,1}\partial_{\beta} v^1 + a_{j1}\mathcal{B}_{j, i}^{\alpha}\partial_{\alpha}v^{i},\\
\end{split}
\end{equation*}
implies
\begin{equation*}\label{bdry-q-1}
\begin{split}
  & J\, (q-\xi^1)\,-\nu (2 (-\nabla_h\cdot\,v^h+\mathcal{B}_{4, i}^{\alpha}\partial_{\alpha}v^{i})+|\overrightarrow{\rm{a}_1}|^{-2}(a_{\beta\,1}\partial_{\beta} v^1 + a_{j1}\mathcal{B}_{j, i}^{\alpha}\partial_{\alpha}v^{i}))=0 \quad \mbox{on} \, \Sigma_{0}.
\end{split}
\end{equation*}
We thus conclude that
\begin{equation}\label{q-interface-1}
\begin{split}
& q=\xi^1-2\nu\,\nabla_h\cdot v^h+\nu\,\mathcal{B}_{9, i}^{\alpha}\partial_{\alpha}v^{i} \quad\mbox{on} \quad \Sigma_0,
\end{split}
\end{equation}
where the nonlinear term
\begin{equation*}
\begin{split}
  \mathcal{B}_{9, i}^{\alpha}\partial_{\alpha}v^{i} \eqdefa &-2 (J^{-1}-1)\nabla_h\cdot\,v^h+2J^{-1}\mathcal{B}_{4, i}^{\alpha}\partial_{\alpha}v^{i}+  J^{-1}|\overrightarrow{\rm{a}_1}|^{-2}(a_{\beta\,1}\partial_{\beta} v^1 + a_{j1}\mathcal{B}_{j,i}^{\alpha}\partial_{\alpha}v^{i}).
\end{split}
\end{equation*}
Combining \eqref{prtv-interface-1} with \eqref{q-interface-1}, the boundary condition on $\Sigma_0$ in \eqref{eqns-pert-1} can be written as the linearized form
\begin{equation}\label{linearf-bdry-1}
\begin{split}
q\, e_1- \nu\mathbb{D}(v)\,e_1=
\left(
  \begin{array}{c}
   \xi^1+\nu( \mathcal{B}_{9, i}^{\alpha}-2 \mathcal{B}_{6, i}^{\alpha})\partial_{\alpha}v^{i}\\
    -\nu\, \mathcal{B}_{7, i}^{\alpha}\partial_{\alpha}v^{i}\\
    -\nu\,\mathcal{B}_{8, i}^{\alpha}\partial_{\alpha}v^{i}
  \end{array}
\right).
\end{split}
\end{equation}

In order to extend the interface boundary forms of $\partial_1v$ and $q$ in \eqref{prtv-interface-1} and \eqref{q-interface-1} to the interior domain of the fluid, let us first introduce $\mathcal{H}(f)$ as the harmonic extension of $f|_{\Sigma_0}$ into $\Omega$:
\begin{equation*}\label{harmonic-ext-xi-1-1}
\begin{cases}
   &\Delta \mathcal{H}(f)=0 \quad\mbox{in} \quad \Omega,\\
   &\mathcal{H}(f)|_{\Sigma_0}=f|_{\Sigma_0},\quad \mathcal{H}(f)|_{\Sigma_b}=0.
\end{cases}
\end{equation*}
We state the classical result on such an operator, which proof is left to the reader.
\begin{lem}\label{est-harmonic-ext-1}
\begin{equation*}\label{harmonic-ext-xi-1-2}
\begin{split}
 &\|\mathcal{H}(f)\|_{H^1(\Omega)} \lesssim \|f\|_{H^{\frac{1}{2}}(\Sigma_0)}, \quad \|\mathcal{H}(f)\|_{H^{r}(\Omega)} \lesssim \|f\|_{H^{r-\frac{1}{2}}(\Sigma_0)} \quad (r\geq 2),\\
     &\|\partial_h^k\mathcal{H}(f)\|_{H^1(\Omega)} \lesssim \|\partial_h^kf\|_{H^{\frac{1}{2}}(\Sigma_0)}, \quad
  \|\partial_h^k\mathcal{H}(f)\|_{H^{2}(\Omega)} \lesssim \|\partial_h^kf\|_{H^{\frac{3}{2}}(\Sigma_0)} \quad (k\in \mathbb{N}).
\end{split}
\end{equation*}
\end{lem}
Based on \eqref{prtv-interface-1} and \eqref{q-interface-1}, we introduce four new good variables in the fluid
\begin{equation}\label{unknown-0}
  \begin{cases}
  &\mathcal{G}^1\eqdefa\partial_1v^1+\nabla_h\cdot\,v^h-\mathcal{H}(\mathcal{B}_{6, i}^{\alpha})\partial_{\alpha}v^{i},\\
  &\mathcal{G}^2\eqdefa\partial_1 v^2+\partial_2 v^1-\mathcal{H}(\mathcal{B}_{7, i}^{\alpha})\partial_{\alpha}v^{i},\\
&\mathcal{G}^3\eqdefa\partial_1 v^3+\partial_3 v^1-\mathcal{H}(\mathcal{B}_{8, i}^{\alpha})\partial_{\alpha}v^{i},\\
&\mathcal{Q}  \eqdefa  q-\mathcal{H}(\xi^1)+2\nu\,\nabla_h\cdot v^h-\nu\,\mathcal{H}(\mathcal{B}_{9, i}^{\alpha})\partial_{\alpha}v^{i} \quad \text{in} \quad \Omega.
\end{cases}
\end{equation}
It is easy to obtain from \eqref{prtv-interface-1} and \eqref{q-interface-1}  that
\begin{equation}\label{unknown-bdry-1}
  \begin{split}
  &\mathcal{G}^1=0,\quad \mathcal{G}^2=0,\quad \mathcal{G}^3=0,\quad \mathcal{Q}=0 \quad\text{on}\quad \Sigma_0,
\end{split}
\end{equation}
and $\partial_1v$ and $q$ satisfy
\begin{equation}\label{partialv-q-1}
  \begin{cases}
  &\partial_1v^1=\mathcal{G}^1-\nabla_h\cdot v^h+\mathcal{H}(\mathcal{B}_{6, i}^{\alpha})\partial_{\alpha}v^{i},\\
  &\partial_1 v^2=\mathcal{G}^2-\partial_2 v^1+\mathcal{H}(\mathcal{B}_{7, i}^{\alpha})\partial_{\alpha}v^{i},\\
&\partial_1 v^3=\mathcal{G}^z-\partial_3 v^1+\mathcal{H}(\mathcal{B}_{8, i}^{\alpha})\partial_{\alpha}v^{i},\\
&q=\mathcal{Q}  +\mathcal{H}(\xi^1)-2\nu\,\nabla_h\cdot v^h+\nu\,\mathcal{H}(\mathcal{B}_{9, i}^{\alpha})\partial_{\alpha}v^{i}\quad \mbox{in} \quad \Omega.
\end{cases}
\end{equation}
For convenience, we briefly write
\begin{equation*}\label{q-Q-1}
\begin{split}
   &q=\mathcal{Q}  +\widetilde{\mathcal{Q}} \quad\mbox{in} \quad \Omega
\end{split}
\end{equation*}
with
\begin{equation*}\label{titde-Q-1}
\begin{split}
   &\widetilde{\mathcal{Q}}\eqdefa  \mathcal{H}(\xi^1)-2\nu\,\nabla_h\cdot v^h+\nu\,\mathcal{H}(\mathcal{B}_{9, i}^{\alpha})\partial_{\alpha}v^{i}.
\end{split}
\end{equation*}

These four new unknowns $\mathcal{G}^i$ with $i=1, 2, 3$ and $\mathcal{Q}$ play an important role in the estimate of the $L^\infty_tL^2$ norms of $\nabla\,v$ as well as its tangential derivatives, as it is possible to close all the estimates by using integration in parts and the relations \eqref{partialv-q-1} with \eqref{unknown-bdry-1} when we  encounter the integrals including the vertical derivative of $\partial_1v$ and $q$. More details will be showed in the proof of Lemma \ref{lem-pseudo-energy-tv-1}. This makes it possible to avoid using the compatibility conditions in terms of the acceleration $\partial_tv$ and its derivatives.

\subsection{Linearized form of the system}

Let's now derive the linearized form of the system \eqref{eqns-pert-1} in the Lagrangian coordinates, which is crucially used to recover the bounds for high order derivatives.

Under the Lagrangian coordinates, we first divide the dissipative term $\grad_{\mathcal{A}} \cdot \mathbb{D}_{\mathcal{A}}(v)$ into linear and nonlinear parts
\begin{equation}\label{dissipative-split-1}
\begin{split}
 &(\grad_{\mathcal{A}} \cdot \mathbb{D}_{\mathcal{A}}(v))^{k}=(\nabla \cdot \mathbb{D}(v))^{k}+g_{kk}
\end{split}
\end{equation}
with $k=1, 2, 3$, where nonlinear parts $g_{kk}$ have the forms
\begin{equation*}
\begin{split}
 &g_{11}:=[f_{0, 11}+(\mathcal{A}_1^1)^2-1]\partial_1^2v^{1}+2(\mathcal{A}_\beta^1 \mathcal{A}_\beta^\alpha+2\mathcal{A}_1^1 \mathcal{A}_1^\alpha)\partial_\alpha\partial_1v^{1}+ (\mathcal{A}_\alpha^1\mathcal{A}_1^1)\partial_1^2v^{\alpha}\\
&+ B_{1,\beta}^{\alpha} \partial_\alpha\partial_1\,v^\beta+ B_{2,i}^{\alpha\beta} \partial_\alpha\partial_\beta\,v^i+
 ({f}_{1, j}+\mathcal{A}_1^i\partial_i \mathcal{A}_1^j)\,\partial_jv^{1}+\mathcal{A}_\alpha^i\partial_i \mathcal{A}_1^1\,\partial_1v^{\alpha}
+\mathcal{A}_\beta^i\partial_i\mathcal{A}_1^\alpha\,\partial_\alpha\,v^{\beta},\\
 &g_{22}:= (\mathcal{A}_1^1\mathcal{A}_2^1)\partial_1^2v^{1}+\mathcal{A}_1^i\partial_i \mathcal{A}_2^1\,\partial_1v^{1}+[f_{0, 11}+(\mathcal{A}_2^1)^2]\partial_1^2v^{2}+ (\mathcal{A}_3^1\mathcal{A}_2^1)\partial_1^2v^{3}+{f}_{1, 1}\partial_1v^{2}\\
 &+\mathcal{A}_\beta^i\partial_i \mathcal{A}_2^1\,\partial_1v^{\beta}+B_{3}^{\alpha}\partial_\alpha\partial_1\,v^1
 +B_{4,i}^{\alpha\beta}\partial_\alpha\partial_\beta\,v^i+B_{5,i}^{\alpha\beta}\partial_\alpha\partial_\beta\,v^i
 +B_{6, \beta}^{\alpha}\partial_\alpha\partial_1\,v^\beta,\\
 &g_{33}:=\mathcal{A}_1^i\partial_i \mathcal{A}_3^1\,\partial_1v^{1}+ (\mathcal{A}_1^1\mathcal{A}_3^1)\partial_1^2v^{1}+B_{7,i}^{\alpha\beta}\partial_\alpha\partial_\beta\,v^i+B_{8, \beta}^{\alpha}\partial_\alpha\partial_1\,v^\beta\\
 &\qquad\qquad\qquad\qquad+B_{9}^{\alpha}\partial_\alpha\partial_1\,v^1+B_{10}^{\alpha}\partial_\alpha\,v^i+B_{11, \alpha}\partial_1^2\,v^\alpha+B_{12, \alpha}\partial_1\,v^\alpha,
\end{split}
\end{equation*}
with
\begin{equation*}
\begin{split}
 & B_{1,\beta}^{\alpha} \partial_\alpha\partial_1\,v^\beta:=(\mathcal{A}_2^1 \mathcal{A}_1^3+\mathcal{A}_2^3\mathcal{A}_1^1) \partial_3\partial_1v^{2}+(\mathcal{A}_3^1 \mathcal{A}_1^2+\mathcal{A}_3^2\mathcal{A}_1^1) \partial_2\partial_1v^{3}\\
 &\qquad\qquad
+(\mathcal{A}_2^1\mathcal{A}_1^2+\mathcal{A}_2^2\mathcal{A}_1^1-1) \partial_2\partial_1v^{2}+( \mathcal{A}_3^1 \mathcal{A}_1^3+\mathcal{A}_3^3\mathcal{A}_1^1 -1) \partial_3\partial_1v^{3},\\
&B_{2,i}^{\alpha\beta}\partial_\alpha\partial_\beta\,v^i:= (\mathcal{A}_i^\beta\mathcal{A}_1^\alpha)\partial_\beta\partial_\alpha\,v^{i}+ 2(\mathcal{A}_i^2\mathcal{A}_i^3)\partial_2\partial_3\,v^{1}+ \sum_{\alpha=2}^3f_{0, \alpha\alpha}\partial_\alpha^2\,v^{1},\\
\end{split}
\end{equation*}
\begin{equation*}
\begin{split}
&B_{3}^{\alpha}\partial_\alpha\partial_1\,v^1:=(\mathcal{A}_1^2\mathcal{A}_2^1+\mathcal{A}_1^1 \mathcal{A}_2^2-1) \partial_2\partial_1v^{1}+ (\mathcal{A}_1^3\mathcal{A}_2^1+\mathcal{A}_1^1 \mathcal{A}_2^3 ) \partial_3\partial_1v^{1},\\
&B_{4,i}^{\alpha\beta}\partial_\alpha\partial_\beta\,v^i:= \sum_{\alpha=2}^3f_{0, \alpha\alpha}\partial_\alpha^2\,v^{2}+ (2\mathcal{A}_i^2\mathcal{A}_i^3+\mathcal{A}_2^3\mathcal{A}_2^2)\partial_2\partial_3\,v^{2}+ (\mathcal{A}_2^2\mathcal{A}_2^2-1)\partial_2^2\,v^{2}\\
&\quad+ (\mathcal{A}_3^3\mathcal{A}_2^2-1)\partial_3\partial_2\,v^{3}
+ (\mathcal{A}_3^2\mathcal{A}_2^2)\partial_2^2\,v^{3}+ (\mathcal{A}_1^\beta\mathcal{A}_2^2)\partial_\beta\partial_2\,v^{1}
+(\mathcal{A}_i^\beta\mathcal{A}_2^3)\partial_\beta\partial_3\,v^{i},\\
\end{split}
\end{equation*}
\begin{equation*}
\begin{split}
&B_{5,i}^{\alpha\beta}\partial_\alpha\partial_\beta\,v^i:=\mathcal{A}_1^k\partial_k\mathcal{A}_2^\alpha\,\partial_\alpha\,v^{1} + f_{1,\alpha}\partial_\alpha\,v^{2}
+\mathcal{A}_\beta^k\partial_k\mathcal{A}_2^\alpha\,\partial_\alpha\,v^{\beta},\\
&B_{6, \beta}^{\alpha}\partial_\alpha\partial_1\,v^\beta:= 2\mathcal{A}_i^1 \mathcal{A}_i^\alpha\,\partial_\alpha\partial_1v^{2}
+(\mathcal{A}_\beta^1 \mathcal{A}_2^\alpha+\mathcal{A}_\beta^\alpha\mathcal{A}_2^1)\partial_\alpha\partial_1v^{\beta},
\end{split}
\end{equation*}
\begin{equation*}
\begin{split}
&B_{7,i}^{\alpha\beta}\partial_\alpha\partial_\beta\,v^i:=\sum_{\alpha=2}^3f_{0, \alpha\alpha} \partial_\alpha^2\,v^{3}+ (\mathcal{A}_1^\beta\mathcal{A}_3^\alpha)\partial_\beta\partial_\alpha\,v^{1}+ (\mathcal{A}_2^2\mathcal{A}_3^2)\partial_2^2\,v^{2}+ (\mathcal{A}_2^3\mathcal{A}_3^3)\partial_3^2\,v^{2}\\
 &+ (\mathcal{A}_2^3\mathcal{A}_3^2+\mathcal{A}_2^2\mathcal{A}_3^3-1)\partial_2\partial_3\,v^{2}+ (\mathcal{A}_3^2\mathcal{A}_3^2)\partial_2^2\,v^{3}+ (\mathcal{A}_3^3\mathcal{A}_3^3-1)\partial_3^2\,v^{3}+ 2(\mathcal{A}_i^2\mathcal{A}_i^3+\mathcal{A}_3^2\mathcal{A}_3^3)\partial_2\partial_3\,v^{3},\\
 &B_{8, \beta}^{\alpha}\partial_\alpha\partial_1\,v^\beta:= 2 (\mathcal{A}_i^1 \mathcal{A}_i^\alpha+\mathcal{A}_3^\alpha\mathcal{A}_3^1) \partial_\alpha\partial_1v^{3}+(\mathcal{A}_2^1 \mathcal{A}_3^\alpha+\mathcal{A}_2^\alpha\mathcal{A}_3^1)\partial_\alpha\partial_1v^{2},\\
 &B_{9}^{\alpha}\partial_\alpha\partial_1\,v^1:=(\mathcal{A}_1^2\mathcal{A}_3^1 + \mathcal{A}_1^1 \mathcal{A}_3^2) \partial_2\partial_1v^{1}+ (\mathcal{A}_1^3\mathcal{A}_3^1 +\mathcal{A}_1^1 \mathcal{A}_3^3-1) \partial_3\partial_1v^{1},\\
 \end{split}
\end{equation*}
\begin{equation*}
\begin{split}
  &B_{10}^{\alpha}\partial_\alpha\,v^i:=f_{1,\alpha}\partial_\alpha\,v^{3}
+\mathcal{A}_1^k\partial_k\mathcal{A}_3^\alpha\,\partial_\alpha\,v^{1}
+\mathcal{A}_\beta^k\partial_k\mathcal{A}_3^\alpha\,\partial_\alpha\,v^{\beta},\\
 &B_{11, \alpha}\partial_1^2\,v^\alpha:= (\mathcal{A}_2^1\mathcal{A}_3^1)\partial_1^2v^{2}+[f_{0, 11}+(\mathcal{A}_3^1)^2]\partial_1^2v^{3},\,B_{12, \alpha}\partial_1\,v^\alpha:={f}_{1, 1}\partial_1v^{3}+\mathcal{A}_\beta^k\partial_k \mathcal{A}_3^1\,\partial_1v^{\beta},
\end{split}
\end{equation*}
and ${f}_{0, 11}:=\sum_{i=1}^3(\mathcal{A}_i^1)^2-1$, $ f_{0,\alpha\alpha}:=\sum_{i=1}^3(\mathcal{A}_i^{\alpha})^2-1$, ${f}_{1, j}:=\mathcal{A}_i^k\partial_k \mathcal{A}_i^j$.

Combining \eqref{incomp-cond-fluid-3}, \eqref{linearf-bdry-1}, \eqref{dissipative-split-1}, the system \eqref{eqns-pert-1} can be written as the linearized form
\begin{equation}\label{eqns-linear-1}
\begin{cases}
&\partial_t\xi=v,\\
 & \partial_t v +\nabla\,q- \nu \nabla\cdot\mathbb{D}(v)=g,\\
 &\nabla\cdot v= -\widetilde{a_{\alpha\,1}}\partial_1v^\alpha+\mathcal{B}_{5, i}^{\alpha}\partial_{\alpha}v^{i} \quad \text{in} \quad \Omega,\\
  &q\, e_1- \nu\mathbb{D}(v)\,e_1=
\left(
  \begin{array}{c}
   \xi^1+\nu( \mathcal{B}_{9, i}^{\alpha}-2 \mathcal{B}_{6, i}^{\alpha})\partial_{\alpha}v^{i}\\
    -\nu\, \mathcal{B}_{7, i}^{\alpha}\partial_{\alpha}v^{i}\\
    -\nu\,\mathcal{B}_{8, i}^{\alpha}\partial_{\alpha}v^{i}
  \end{array}
\right) \quad \text{on} \quad \Sigma_0,\\
&v|_{\Sigma_b}=0.
    \end{cases}
\end{equation}
where $g=(g_1, g_2, g_3)^T$ with
\begin{equation}\label{def-g-linear-1}
\begin{split}
 &g_1:=-\mathcal{A}_1^h\partial_hq -(\mathcal{A}_1^1-1)\partial_1q +\nu\,g_{11},\\
 &g_2:=-\mathcal{A}_2^1\partial_1q
   -(\mathcal{A}_2^2-1)\partial_2q- \mathcal{A}_2^3 \partial_3q+\nu\,g_{22},\\
 &g_3:=-\mathcal{A}_3^1\partial_1q
 -(\mathcal{A}_3^3-1)\partial_3q- \mathcal{A}_3^2 \partial_2q+\nu\,g_{33}.
\end{split}
\end{equation}
Using the incompressibility condition \eqref{incomp-cond-fluid-3},
one has $\partial_{k}\nabla\,\cdot\,v=\widetilde{g}_{kk}$ with $k=1, 2, 3$, and
\begin{equation*}
  \begin{split}
&\widetilde{g}_{kk}=-\widetilde{a_{21}}\partial_{k}\partial_1v^2
-\widetilde{a_{31}}\partial_{k}\partial_1v^3+\mathcal{B}_{5, i}^{\alpha}\partial_{k}\partial_{\alpha}v^{i}-\partial_{k}\widetilde{a_{21}}\partial_1v^2
-\partial_{k}\widetilde{a_{31}}\partial_1v^3+\partial_{k}\mathcal{B}_{5, i}^{\alpha}\partial_{\alpha}v^{i}.
  \end{split}
\end{equation*}
In particular, for $\partial_1^2v^1$, there holds
\begin{equation}\label{incomp-cond-v1-1}
  \begin{split}
&\partial_1^2v^1=-\partial_1\nabla_h\,\cdot\,v^h+\widetilde{g}_{11}.
    \end{split}
\end{equation}
From this, the momentum equations in \eqref{eqns-linear-1} are equivalent to the equations
\begin{equation*}
\begin{cases}
 & \partial_t v^{1} + \partial_1q -\nu\Delta\,v^{1}=g_1+\nu\widetilde{g}_{11},\\
   & \partial_t v^{2} +\partial_2q-\nu\Delta\,v^{2}=g_2+\nu\,\widetilde{g}_{22},\\
 & \partial_t v^{3}+ \partial_3q-\nu\Delta\,v^{3}=g_3+\nu\,\widetilde{g}_{33},\\
     \end{cases}
\end{equation*}
which follows
\begin{equation}\label{prt1-q-vh-1}
\begin{cases}
 &\partial_1q=-\partial_t v^{1}+\nu(\Delta_h\,v^{1}-\partial_1\nabla_h\,\cdot\,v^h)
  +g_1+2\nu\,\widetilde{g}_{11},\\
   &-\nu\partial_1^2\,v^{\beta}+\partial_\beta q=-\partial_t v^{\beta}+\nu\Delta_h\,v^{\beta}+g_\beta+\nu\widetilde{g}_{\beta\beta}\quad \text{for}\quad \beta=2, 3.
     \end{cases}
\end{equation}
Thanks to \eqref{incomp-cond-v1-1} and \eqref{prt1-q-vh-1}, it is easy to get, for any smooth sub-channel domain $\Omega_1\subseteq\Omega$,
\begin{equation}\label{prt1-q-vh-rela-2}
\begin{split}
 &\|\mathcal{P}(\partial_h)\partial_1^2v^1\|_{L^2(\Omega_1)}\lesssim \|\mathcal{P}(\partial_h)\partial_h\nabla\,v\|_{L^2(\Omega_1)}
 +\|\mathcal{P}(\partial_h)\widetilde{g}_{11}\|_{L^2(\Omega_1)},\\
 &\|\mathcal{P}(\partial_h)\partial_1q\|_{L^2(\Omega_1)}\lesssim\|\mathcal{P}(\partial_h)\partial_t v^{1}\|_{L^2(\Omega_1)}
 +\|\mathcal{P}(\partial_h)\partial_h\,\nabla\,v\|_{L^2(\Omega_1)}\\
 &\qquad\qquad\qquad\qquad\qquad\qquad
  +\|\mathcal{P}(\partial_h)g_1\|_{L^2(\Omega_1)}+\|\mathcal{P}(\partial_h)\widetilde{g}_{11}\|_{L^2(\Omega_1)},\\
   &\|\mathcal{P}(\partial_h)(-\nu\partial_1^2\,v^{\beta}+\partial_\beta\,q)\|_{L^2(\Omega_1)}\lesssim \|\mathcal{P}(\partial_h)\partial_t v^{\beta}\|_{L^2(\Omega_1)}+\|\mathcal{P}(\partial_h)\Delta_h\,v^{\beta}\|_{L^2(\Omega_1)}\\
   &\qquad\qquad\qquad\qquad\qquad\qquad+\|\mathcal{P}(\partial_h)g_\beta\|_{L^2(\Omega_1)}+
   \|\mathcal{P}(\partial_h)\widetilde{g}_{\beta\beta}\|_{L^2(\Omega_1)}\quad \text{for}\quad \beta=2, 3.
     \end{split}
\end{equation}

\renewcommand{\theequation}{\thesection.\arabic{equation}}
\setcounter{equation}{0}

\section{Main result}\label{sect-main}

Let us explain that, at the beginning of this section, how to define the energy and dissipation functions. Usually, in order to get the global well-posedness of the nonlinear system, we look at it as a perturbation of its linearized equation, which relies on the global control of the perturbed nonlinear terms. In our case, it depends on t the uniformly global (in time) estimates of the $L^\infty(\Omega)$-norms of $\mathcal{A}-I$ (or equivalently  $a_{ij}-\delta_i^j$ ) as well as its derivatives, and at least, it requires that $\|\mathcal{A}-I\|_{L^\infty(\Omega)} < 1$. Take the $(1, 2)$-entry $\mathcal{A}_1^2$ of $\mathcal{A}-I$ as an example, its mainly linear part is $\partial_1\xi^2$. Since there is no information of $\partial_1\xi^2$ in the energy based on the momentum equations, we turn to use the flow map equation $\partial_t\xi=v$ to get $\partial_1\xi^2(t)=\partial_1\xi^2(0)+\int_0^t\partial_1v^2(\tau)\,d\tau$. Hence, in order to get the control of $\|\partial_1\xi^2(t)\|_{L^\infty(\Omega)}$ uniformly in terms of time $t$, it requires the global $L^1_t$ integrability of $\|\partial_1v^2(t)\|_{L^\infty(\Omega)}$ on $\in [0, +\infty)$. Because of the Sobolev embedding theorem, $\|\partial_1v^2(t)\|_{L^\infty} \leq C \|\dot{\Lambda}_h^{\sigma_0}\Lambda_h\partial_1v^2\|_{H^1}$, we turn to find the global $L^1_t$ integrability of $\|\dot{\Lambda}_h^{\sigma_0}\Lambda_h\partial_1v^2\|_{H^1}$ on $\in [0, +\infty)$, which depends on the decay (enough) of $\|\dot{\Lambda}_h^{\sigma_0}\Lambda_h\partial_1v^2\|_{L^2}$ according to the energy estimates arising from the momentum equations.

\subsection{Energy and dissipation}

We first define the decay instantaneous energy and dissipation $\dot{\mathcal{E}}_{N}=\dot{\mathcal{E}}_{N}(\xi,\,v)$ and $\dot{\mathcal{D}}_{N}=\dot{\mathcal{D}}_{N}(\xi,\,v)$ with $N\geq 3$ as follows.

The decay instantaneous energy $\dot{\mathcal{E}}_{N}$ includes three parts:
\begin{equation*}\label{def-energy-1}
\begin{split}
 \dot{\mathcal{E}}_{N}\eqdefa  \dot{\mathcal{E}}_{\ell, N}+ \dot{\mathcal{E}}_{\ell, \sigma_0} +\dot{\mathcal{E}}_{\ell, 1+\sigma_0} \quad(\text{for} \, N\geq 3),
\end{split}
\end{equation*}
where $\dot{\mathcal{E}}_{\ell, \sigma}\eqdefa \|\dot{\Lambda}_h^{\sigma}v\|_{L^2(\Omega)}^2 +\|\dot{\Lambda}_h^{\sigma}\xi^1\|_{L^2(\Sigma_0)}^2 \quad(\text{for} \, \sigma \in (-1, 1))$,
\begin{equation*}
\begin{split}
 & \dot{\mathcal{E}}_{\ell, 1}\eqdefa\|\partial_h\,v\|_{L^2(\Omega)}^2 +\|\partial_h\xi^1\|_{L^2(\Sigma_0)}^2,\quad \dot{\mathcal{E}}_{\ell, 1+\sigma_0}\eqdefa \|\dot{\Lambda}_h^{\sigma_0}\,\nabla\,v\|_{L^2(\Omega)}^2,\\
&\dot{\mathcal{E}}_{\ell, N}\eqdefa\sum_{i=1}^{N-1}(\|\partial_h^i\,v\|_{H^1(\Omega)}^2 +\|\partial_h^{i}\xi^1\|_{H^1(\Sigma_0)}^2)\,(\text{for} \, N\geq 2),
\end{split}
\end{equation*}
and the decay instantaneous dissipation $\dot{\mathcal{D}}_{N}$ (with $N\geq 2$) is defined as
\begin{equation*}\label{def-dissipa-1}
\begin{split}
&\dot{\mathcal{D}}_{2}\eqdefa \dot{\mathcal{D}}_{\ell, \sigma_0}+ \dot{\mathcal{D}}_{\ell, 1+\sigma_0}+\dot{\mathcal{D}}_{\ell, 2},\quad \dot{\mathcal{D}}_{N}\eqdefa \dot{\mathcal{D}}_{2}+\dot{\mathcal{D}}_{\ell, N}+\ddot{\mathcal{D}}_{\ell, 1+\sigma_0}+\ddot{\mathcal{D}}_{\ell, N}\, (\text{for}\, N\geq 3),
\end{split}
\end{equation*}
where
\begin{equation*}
\begin{split}
&\dot{\mathcal{D}}_{\ell, \sigma_0}:= \|\dot{\Lambda}_h^{\sigma_0}\,\nabla\,v\|_{L^2(\Omega)}^2, \quad\dot{\mathcal{D}}_{\ell, 1+\sigma_0}:= \|\dot{\Lambda}_h^{\sigma_0} \partial_tv\|_{L^2}^2,\\
&\dot{\mathcal{D}}_{\ell, N}\eqdefa \sum_{i=1}^{N}\|\partial_h^i\nabla\,v\|_{L^2(\Omega)}^2+\sum_{i=1}^{N-1}
\|\partial_h^i\partial_tv\|_{L^2}^2 \quad (\text{for}\,\, N\geq 2),\\
&\ddot{\mathcal{D}}_{\ell, 1+\sigma_0}:= \|\dot{\Lambda}_h^{\sigma_0}(\nabla^2v,\,\nabla p)\|_{L^2(\Omega)}^2+\|\dot{\Lambda}_h^{\sigma_0+1}\xi^1\|_{L^2(\Sigma_0)}^2,\\
&\ddot{\mathcal{D}}_{\ell, 2}:=  \|\partial_h (\nabla^2v,\,\nabla p)\|_{L^2(\Omega)}^2+\|\partial_h^2\xi^1\|_{L^2(\Sigma_0)}^2,\,\ddot{\mathcal{D}}_{\ell, 3}:=  \sum_{i=1}^{2}\|\partial_h^i (\nabla^2v,\,\nabla p)\|_{L^2(\Omega)}^2+\|\partial_h^2\xi^1\|_{{H}^{\frac{1}{2}}(\Sigma_0)}^2,\\
&\ddot{\mathcal{D}}_{\ell, N}:= \ddot{\mathcal{D}}_{\ell, 3}+\sum_{i=3}^{N-1}(\|\partial_h^i( \nabla^2v,\,\nabla p)\|_{L^2(\Omega)}^2+\|\partial_h^i \xi^1\|_{{H}^{\frac{1}{2}}(\Sigma_0)}^2) \quad (\text{for}\, N\geq 4).
\end{split}
\end{equation*}
\begin{rmk}\label{rmk-notation-1}
We mention that in the notations of the above energy and dissipation,  the energy $\dot{\mathcal{E}}_{\ell, s}$ and dissipation $\dot{\mathcal{D}}_{\ell, s}$ with $s\in \mathbb{R}$ come from the tangential derivatives estimates of the system \eqref{eqns-linear-1}, and the dissipation $\ddot{\mathcal{D}}_{\ell, s}$ with $s\in \mathbb{R}$ comes from the Stokes estimates of the system \eqref{eqns-linear-1}.
\end{rmk}

\begin{rmk}\label{rmk-dissip-1}It is easy to check that for any $N\geq 4$
\begin{equation*}\label{dissip-equiv-1}
\begin{split}
&\dot{\mathcal{D}}_{N}=\dot{\mathcal{D}}_{N-1}+\mathring{\mathcal{D}}_N,
\end{split}
\end{equation*}
where $\mathring{\mathcal{D}}_N\eqdefa \|\partial_h^{N-1} (\nabla^2v,\,\nabla p, \,\partial_tv)\|_{L^2(\Omega)}^2+\|\partial_h^{N-1} \xi^1\|_{{H}^{\frac{1}{2}}(\Sigma_0)}^2$.
\end{rmk}
On the other hand, we define the bounded instantaneous energy $\mathcal{E}_{N}$  and dissipation $\mathcal{D}_{N}$ with $N\geq 3$ as
\begin{equation*}\label{def-bdd-energy-dissi-1}
\begin{split}
&\mathcal{E}_{N}\eqdefa \|\dot{\Lambda}_h^{-\lambda}\,v\|_{H^1(\Omega)}^2
+\|(\dot{\Lambda}_h^{-\lambda}\,\xi^1,\,\xi^1)\|_{L^2(\Sigma_0)}^2+ \dot{\mathcal{E}}_{N},\\
&\mathcal{D}_{N}\eqdefa\|(\dot{\Lambda}_h^{-\lambda}\nabla\,v,\,\dot{\Lambda}_h^{-\lambda}\partial_tv,\,\nabla\,v, \,\partial_tv)\|_{L^2(\Omega)}^2+\dot{\mathcal{D}}_{N}.
\end{split}
\end{equation*}

\begin{rmk}\label{rmk-bdd-energy-1}
Set
\begin{equation*}\label{def-energy-2a}
\begin{split} \mathring{\mathcal{E}}_{N}\eqdefa\|\partial_h^{N-1}\nabla\,v\|_{L^2(\Omega)}^2 +\|\partial_h^{N}\xi^1\|_{L^2(\Sigma_0)}^2,
\end{split}
\end{equation*}
we have
\begin{equation*}\label{def-energy-2}
\begin{split}
 &\mathcal{E}_{N}= \mathring{\mathcal{E}}_{N}+  \mathcal{E}_{N-1}.
\end{split}
\end{equation*}
\end{rmk}

We also define instantaneous energy quantities $E_N=E_{N}(\xi,\,v)$ (with $N \geq 3$) in terms of the velocity $v$ and the flow map $\eta$:
\begin{equation*}\label{def-energyE-2}
\begin{split}
 E_N\eqdefa E_{\text{h}, N}+\mathcal{E}_{N} \quad(\text{for} \, N\geq 3),
\end{split}
\end{equation*}
with
\begin{equation*}\label{def-energy-2}
\begin{split}
 E_{\text{h}, N}\eqdefa\|\dot{\Lambda}_h^{\sigma_0-1}\partial_1\xi^1\|_{H^1}^2
 +\|\dot{\Lambda}_h^{\sigma_0}\partial_1\xi^h\|_{H^1}^2 +\|\dot{\Lambda}_h^{\sigma_0}\xi\|_{L^2}^2+ \sum_{i=1}^{N-1}\|\partial_h^i\xi\|_{H^2}^2,
\end{split}
\end{equation*}
which is corresponding to the energy space
\begin{equation*}\label{def-energy-2}
\begin{split}
 \mathfrak{F}_N\eqdefa \{(\xi,\,v)|\,{E}_{N}<+\infty\} \quad(\text{for} \, N\geq 3)
\end{split}
\end{equation*}
equipped with the norm
\begin{equation*}\label{def-energy-2}
\begin{split}
 \|(\xi,\,v)\|_{\mathfrak{F}_N}\eqdefa {E}_{N}(\xi, \,v)^{\frac{1}{2}}.
\end{split}
\end{equation*}

If we set
\begin{equation*}\label{def-energy-2}
\begin{split}
 \mathring{E}_N\eqdefa\|\partial_h^{N-1}\nabla^2\,\xi\|_{L^2(\Omega)}^2
 +\|\partial_h^{N-1}\nabla\,v\|_{L^2(\Omega)}^2+\|\partial_h^{N}\xi^1\|_{L^2(\Sigma_0)}^2
 =\|\partial_h^{N-1}\nabla^2\,\xi\|_{L^2(\Omega)}^2+\mathring{\mathcal{E}}_{N},
\end{split}
\end{equation*}
then
\begin{equation*}\label{def-energy-2}
\begin{split}
{E}_N= \mathring{E}_N+ E_{N-1}\thicksim\|\partial_h^{N-1}\xi\|_{H^2}^2+\mathcal{E}_{N}+E_{N-1}.
\end{split}
\end{equation*}

\subsection{Main theorem}
Let
\begin{equation*}
E_{\ell, N}=E_{\ell, N}(\xi,\,v)\eqdefa\sum_{i=0}^{N-1}(\|\partial_h^i\,v\|_{H^1(\Omega)}^2 +\|\partial_h^{i}\xi^1\|_{H^1(\Sigma_0)}^2+\|\partial_h^i\nabla\xi\|_{H^1(\Omega)}^2).
\end{equation*}
We define
\begin{equation*}\label{def-energy-1}
\begin{split}
&\mathcal{H}^s_{\text{tan}, N}(\Omega)\eqdefa\{f \in {H}^{s}(\Omega)|\sum_{i=0}^{N}\|\partial_h^i\,f\|_{H^s(\Omega)}^2 <+\infty\},\\
&\mathcal{H}^s_{0, \text{tan}, N}(\Omega)\eqdefa\{f \in {H}^{s}(\Omega)|\sum_{i=0}^{N}\|\partial_h^i\,f\|_{H^s(\Omega)}^2 <+\infty,\,f|_{\Sigma_b}=0\},\\
&\mathring{\mathcal{H}}_{\text{tan}, N}^{2, \Sigma_0}(\Omega)\eqdefa \{f\in \mathcal{H}^2_{\text{tan}, 0}(\Omega)|\sum_{i=0}^{N}(\|\partial_h^{i}f^1\|_{H^1(\Sigma_0)}^2+\|\partial_h^i\nabla\,f\|_{H^1(\Omega)}^2)<+\infty\} \quad(\text{for} \, N \in \mathbb{N})
\end{split}
\end{equation*}
equipped with the norms respectively
\begin{equation*}\label{def-energy-2}
\begin{split}
&\|f\|_{\mathcal{H}^s_{\text{tan}, N}(\Omega)}\eqdefa (\sum_{i=0}^{N}\|\partial_h^i\,f\|_{H^s(\Omega)}^2)^{1/2},\quad \|f\|_{\mathcal{H}^s_{0, \text{tan}, N}(\Omega)}\eqdefa (\sum_{i=0}^{N}\|\partial_h^i\,f\|_{H^s(\Omega)}^2)^{1/2}\\
 &\|f\|_{\mathring{\mathcal{H}}_{\text{tan}, N}^{2, \Sigma_0}(\Omega)}\eqdefa (\sum_{i=0}^{N}(\|\partial_h^{i}f^1\|_{H^1(\Sigma_0)}^2+\|\partial_h^i\nabla\,f\|_{H^1(\Omega)}^2))^{\frac{1}{2}}.
\end{split}
\end{equation*}
 Inspired by the works in \cite{Gui-WW-2019} and \cite{Guo-Tice-1}, we may get the following local well-posedness of the system \eqref{eqns-pert-1}, which proof is similar to the one in \cite{Gui-WW-2019}, and  we leave it to the reader.
\begin{thm}[Local well-posedness]\label{thm-local}
Let $N \geq 3$ be an integer. Assume $(\xi_0,\,v_0) \in \mathring{\mathcal{H}}_{\text{tan}, N-1}^{2, \Sigma_0}(\Omega)\times \mathcal{H}_{0, \text{tan}, N-1}^{1}(\Omega)$ satisfies $\nabla_{J_0\mathcal{A}_0}\cdot v_0=0$. There exists a positive constant $\epsilon_0$ such that the viscous surface wave problem \eqref{eqns-pert-1} with initial data $(\xi_0, v_0)$ has a unique solution $(\xi,\,v,\,p)$ (depending continuously on the initial data) in
\begin{equation*}
\begin{split}
\mathcal{C}([0, T_0]; \mathring{\mathcal{H}}_{\text{tan}, N-1}^{2, \Sigma_0}(\Omega))&\times (\mathcal{C}([0, T_0]; \mathcal{H}_{0, \text{tan}, N-1}^{1}(\Omega))\cap L^2([0, T_0]; \mathcal{H}^2_{\text{tan}, N-1}(\Omega))\\
&\times L^2([0, T_0]; \mathcal{H}^1_{\text{tan}, N-1}(\Omega)),
\end{split}
\end{equation*}
with $T_0=\min\{1, (\|\xi_0\|_{\mathring{\mathcal{H}}_{\text{tan}, N-1}^{2, \Sigma_0}(\Omega))})^{-2}\}$ provided $\|\xi_0\|_{\mathring{\mathcal{H}}_{\text{tan}, N-2}^{2, \Sigma_0}(\Omega))}+\|v_0\|_{\mathcal{H}_{0, \text{tan}, N-1}^{1}(\Omega)} <\epsilon_0$, and the solution satisfies the estimate
\begin{equation}\label{def-energy-local}
\begin{split}
&\sup_{t \in[0, T_0]}(\|\xi(t)\|_{\mathring{\mathcal{H}}_{\text{tan}, N-1}^{2, \Sigma_0}(\Omega)}^2+\|v(t)\|_{\mathcal{H}_{0, \text{tan}, N-1}^{1}(\Omega)}^2)+\|(\nabla\,v,\,p)\|_{L^2([0, T_0]; \mathcal{H}^1_{\text{tan}, N-1}(\Omega)}^2 \\
&\leq C(\|\xi_0\|_{\mathring{\mathcal{H}}_{\text{tan}, N-2}^{2, \Sigma_0}(\Omega))}^2+\|v_0\|_{\mathcal{H}_{0, \text{tan}, N-1}^{1}(\Omega)}^2+T\,\|\xi_0\|_{\mathring{\mathcal{H}}_{\text{tan}, N-1}^{2, \Sigma_0}(\Omega))}^2),
\end{split}
\end{equation}
And if the maximal existence time $T^{\ast} <+\infty $, then
\begin{equation}\label{loc-blow-crit}
\begin{split}
\lim_{t  \nearrow T^{\ast}}(\|\xi\|_{\mathring{\mathcal{H}}_{\text{tan}, N-1}^{2, \Sigma_0}(\Omega))}+\|v\|_{\mathcal{H}_{0, \text{tan}, N-1}^{1}(\Omega)}+\|J^{-1}\|_{L^\infty})=+\infty.
\end{split}
\end{equation}
Moreover, if in addition $(\xi_0,\,v_0) \in \mathfrak{F}_{N+1}$, then $(\xi,\,v) \in \mathcal{C}([0, T_0]; \mathfrak{F}_{N+1})$.
\end{thm}
\begin{rmk}\label{rmk-euler-1}
According to \eqref{def-energy-local}, we see that, if $\epsilon_0$ in Theorem \ref{thm-local} is small enough, there is a positive $\delta_0 \in (0, \frac{1}{2})$ such that  $1-\delta_0\leq \sup_{(t, x)\in[0,T]\times \Omega}J(t, x)\leq 2$, which implies that the flow-map $\eta(t, x)$ defines a diffeomorphism from the equilibrium domain $\Omega$ to the moving domain $\Omega(t)$ with the boundary $\Sigma_F(t)$. From this, together with the fact that $\eta_0$ is a diffeomorphism from the equilibrium domain $\Omega$ to the initial domain $\Omega(0)$, we deduce a diffeomorphism from the initial domain $\Omega(0)$ to the evolving domain $\Omega(t)$ for any $t \in [0, T]$. Denote the inverse of the flow map $\eta(t, x)$ by $\eta^{-1}(t, y)$ for $t \in [0, T]$ so that if $y = \eta(t, x)$ for $y \in \Omega(t)$ and $t \in [0, T]$, then $x = \eta^{-1}(t, y) \in \Omega$.
\end{rmk}

Along the proof of Beale's non-decay theorem in \cite{Beale-1981}, we give the modified the non-decay theorem as follows, which proof will be given in Appendix.
\begin{thm}[Non-decay theorem]\label{thm-nondecay}
Suppose thet $\sigma_0 \in (\frac{1}{2}, \, 1)$. For certain $\theta\in \mathring{\mathcal{H}}_{\text{tan}, 3}^{2, \Sigma_0}(\Omega)$, there cannot exist a
curve $(v(\varepsilon),\,q(\varepsilon),\,\xi(\varepsilon))$ with
\begin{equation*}\label{space-vqxi-1}
\begin{split}
    v(\varepsilon)\in  & \mathcal{C}_b([0, +\infty); \mathcal{H}_{0, \text{tan}, 3}^{1}(\Omega))\cap L^2_{\text{loc}}([0, +\infty); \mathcal{H}^2_{\text{tan}, 3}(\Omega))\\
      &\cap  L^1([0, +\infty); \dot{\Lambda}_h^{-\sigma_0}\mathcal{H}^2_{\text{tan}, 2}(\Omega)) \cap L^2([0, +\infty); \dot{\Lambda}_h^{-1}\mathcal{H}^1_{\text{tan}, 3}(\Omega)),\\
q(\varepsilon) \in & L^2_{\text{loc}}([0, +\infty); \mathcal{H}^1_{\text{tan}, 3}(\Omega)), \\
\xi(\varepsilon))\in &\mathcal{C}([0, +\infty); \mathring{\mathcal{H}}_{\text{tan}, 3}^{2, \Sigma_0}(\Omega))\cap \mathcal{C}_b([0, +\infty); H^{4}(\Sigma_0))\cap \mathcal{C}_b([0, +\infty); \mathring{\mathcal{H}}_{\text{tan}, 2}^{2, \Sigma_0}(\Omega)),
\end{split}
\end{equation*}
defined for $\varepsilon$ near $0$, such that viscous surface wave problem \eqref{eqns-pert-1}
hold with the initial data $\xi_0=\varepsilon\,\theta$, $v_0=0$, and
\begin{equation}\label{cond-infty-xi-1}
  \begin{split}
  \|\xi^1(t)\|_{H^2(\Sigma_0)} \rightarrow 0 \quad(\text{as} \quad t\rightarrow +\infty)
    \end{split}
\end{equation}
 holds for each $\varepsilon$ , and $v(\varepsilon)$ is of the form
\begin{equation}\label{expansion-v-1}
  \begin{split}
    v(\varepsilon)=\varepsilon\,v^{(1)}+\varepsilon^2\,v^{(2)}+\varepsilon^3 \,v^{(3)}
  \end{split}
\end{equation}
in the space
\begin{equation}\label{expansion-v-2}
  \begin{split}
  \mathbf{E}:=\mathcal{C}_b([0, +\infty); &\mathcal{H}_{0, \text{tan}, 3}^{1}(\Omega))\cap L^2_{\text{loc}}([0, +\infty); \mathcal{H}^2_{\text{tan}, 3}(\Omega))\\
      &\cap  L^1([0, +\infty); \dot{\Lambda}_h^{-\sigma_0}\mathcal{H}^2_{\text{tan}, 2}(\Omega)) \cap L^2([0, +\infty); \dot{\Lambda}_h^{-1}\mathcal{H}^1_{\text{tan}, 3}(\Omega)).
  \end{split}
\end{equation}
\end{thm}
\begin{rmk}\label{rmk-space-neg-1}
We call $v \in L^1([0, +\infty); \dot{\Lambda}_h^{-\sigma_0}\mathcal{H}^2_{\text{tan}, 2}(\Omega)) $ and $ v\in L^2([0, +\infty); \dot{\Lambda}_h^{-1}\mathcal{H}^1_{\text{tan}, 3}(\Omega))$ if and only if
$\dot{\Lambda}_h^{\sigma_0}v \in L^1([0, +\infty); \mathcal{H}^2_{\text{tan}, 2}(\Omega)) $ and $\dot{\Lambda}_h v\in L^2([0, +\infty); \mathcal{H}^1_{\text{tan}, 3}(\Omega))$ respectively.
\end{rmk}

The main result of this paper states as follows, the proof of which will be presented in Section \ref{sect-proof-mainthm}.
\begin{thm}[Global well-posedness]\label{thm-main}
Let $N \geq 3$ be an integer, $(\lambda,\,\sigma_0) \in (0, 1)$ satisfy $1-\lambda< \sigma_0\leq 1-\frac{1}{2}\lambda$. If $(\xi_0,\,v_0) \in \mathfrak{F}_{N+1}$ satisfies $\nabla_{J_0\mathcal{A}_0}\cdot v_0=0$, then there exists a small positive constant $\epsilon_0$, such that, if $E_{N+1}(0)\leq \epsilon_0$, then the viscous surface wave problem \eqref{eqns-pert-1} is globally well-posed in $\mathfrak{F}_{N+1}$. Moreover, for any $t>0$, there hold
\begin{equation*}\label{total-high-thm-1}
\begin{split}
&E_{N}(t) \lesssim  E_{N}(0)+\mathcal{E}_{\ell, N+1}(0),\quad   E_{N+1}(t) \lesssim E_{N+1}(0)+\langle\,t\,\rangle\,\mathcal{E}_{N+1}(0),\\
&\sup_{\tau\in[0, t]}\bigg(\dot{\mathcal{E}}_{N}(\tau)\langle\tau\rangle^{\lambda+\sigma_0}+\mathcal{E}_{N+1}(\tau)\bigg)\\
&\qquad\qquad\qquad\qquad\qquad
+\int_0^t\bigg(\langle\tau\rangle^{(\lambda+\sigma_0+1)/2}\dot{\mathcal{D}}_{N}(\tau)
+\mathcal{D}_{N+1}(\tau)\bigg)\,d\tau \lesssim \mathcal{E}_{N+1}(0).
 \end{split}
\end{equation*}
\end{thm}
\begin{rmk}\label{rmk-mainthm-1}
Notice that the global solution $v$ stated in Theorem \ref{thm-main} belongs to the space $\mathbf{E}$ in \eqref{expansion-v-2} and the free surface $\xi^1$ satisfies \eqref{cond-infty-xi-1} in the non-decay theorem (Theorem \ref{thm-nondecay}), we know that the assumption of the low horizontal regularity on the initial data is the essential element to prevent from occurring the statement in Theorem \ref{thm-nondecay}.
\end{rmk}

\begin{rmk}\label{rmk-mainthm-2}
Compared to the work of Guo-Tice \cite{Guo-Tice-2} about the global well-posedness of the system \eqref{VFS-eqns-1}, there is no requirement about the compatibility conditions in terms of the material derivative of the velocity in Theorem \ref{thm-main}, which plays a significant role in the study of the vacuum free boundary problem of the viscous compressible Navier-Stokes system (see \cite{Gui-WW-2019} for example). On the other hand, the Lagrangian method used in the proof of Theorem \ref{thm-main} is a new approach to study the global well-posedness of the incompressible Navier-Stokes equations with free boundary problem, which can be applied to investigate the well-posedness of many types of free boundary problem of the incompressible or compressible fluid dynamical system. Finally, some anisotropic algebraic decay estimates in terms of the energy and dissipation are uncovered in Theorem \ref{thm-main}, and this kind of result may be helpful to understand the possible blow-up or decay mechanism of the solution.
\end{rmk}

\renewcommand{\theequation}{\thesection.\arabic{equation}}
\setcounter{equation}{0}

\section{Preliminary estimates}\label{sect-tool}

Let us first recall some basic estimates, which will be heavily used  throughout the paper.

\begin{lem}[\cite{Alinhac-1986}, Theorem 2.61 in \cite{BCD}]\label{lem-composition-1}
Let $f$ be a smooth function on $\mathbb{R}$ vanishing at $0$, $s_1$, $s_2$ be two positive
real number, $s_1\in (0, 1)$, $s_2>0$. If $u$ belongs to $\dot{H}^{s_1}(\mathbb{R}^2) \cap \dot{H}^{s_1}(\mathbb{R}^2)\cap L^{\infty}(\mathbb{R}^2)$, then so does $f\circ u$,
and we have
\begin{equation*}
\|f\circ u\|_{\dot{H}^{s_i}} \leq C(s, f', \|u\|_{L^{\infty}})\|u\|_{\dot{H}^{s_i}}\quad \text{for} \quad i=1, 2.
\end{equation*}
\end{lem}

\begin{lem}[Classical product laws in Sobolev spaces \cite{BCD}]\label{lem-product-law-1}
\begin{equation}\label{product-law-1}
\begin{split}
&\|f\,g\|_{\dot{H}^{s_1+s_2-1}(\mathbb{R}^2)} \lesssim  \|f\|_{\dot{H}^{s_1}(\mathbb{R}^2)} \|g\|_{\dot{H}^{s_2}(\mathbb{R}^2)} \quad \text{for} \quad |s_1|,\,|s_2| < 1,\quad s_1+s_2>0,\\
&\|\dot{\Lambda}_h^s(f\,g)\|_{L^2(\mathbb{R}^2_h)} \lesssim
\|\dot{\Lambda}_h^sf\|_{L^2(\mathbb{R}^2_h)}\|g\|_{L^\infty(\mathbb{R}^2_h)}
+\|\dot{\Lambda}_h^sg\|_{L^2(\mathbb{R}^2_h)}\|f\|_{L^\infty(\mathbb{R}^2_h)} \quad \text{for} \quad s>0.
\end{split}
\end{equation}
\end{lem}

\begin{lem}[Embedding inequality \cite{BCD}]\label{lem-embedding-ineq-1}
For any $\sigma \in (0, 1)$, there holds
\begin{equation}\label{embedding-ineq-1}
\begin{split}
&\|f\|_{L^\infty(\mathbb{R}_h^2)} \lesssim  \|\dot{\Lambda}_h^{\sigma}\Lambda_h f\|_{L^2(\mathbb{R}_h^2)}.
\end{split}
\end{equation}
\end{lem}
The following result gives a version of Korn's type inequality for the equilibrium domain
$\Omega$, which proof can be found in Beale's paper \cite{Beale-1981}.
\begin{lem}[Korn's lemma, Lemma 2.7 in \cite{Beale-1981}]\label{lem-korn-2}
Let $\Omega$ be the equilibrium domain given in \eqref{def-domain-1}, then there exists a positive constant $C$, independent of $u$, such that
\begin{equation*}\label{korn-2}
\|u\|_{H^1(\Omega)} \leq C\,\|\mathbb{D}(u)\|_{L^2(\Omega)}
\end{equation*}
for all $u \in H^1(\Omega)$ with $u|_{\Sigma_b}=0$.
\end{lem}

The entries of the matrices $\mathcal{A}-I$,  $J\mathcal{A}-I$, as well as the Jacobian $J$ of the flow map and its inverse $J^{-1}$, are anisotropic when we consider their low horizontal regularities, which is stated as follows.

\begin{lem}\label{lem-est-aij-1}
Let $i, \, j=1, 2, 3$, $2 \leq  k\in \mathbb{N}$, if $E_3(t)$ is uniformly bounded for all existence time $t$, say $E_3(t) \leq 1$, then there hold
\begin{equation}\label{est-aij-J-1}
  \begin{split}
&(1). \,\|\dot{\Lambda}_h^{\sigma}\,(a_{12},\,a_{13},\,\mathcal{A}_1^2,\,\mathcal{A}_1^3)\|_{H^1}+  \|(a_{12},\,a_{13},\,\mathcal{A}_1^2,\,\mathcal{A}_1^3)\|_{L^\infty}\lesssim E_3^{\frac{1}{2}} \quad \forall \, \sigma \in [\sigma_0, 2];\\
&(2). \,\|\dot{\Lambda}_h^{\sigma} (J-1,\,J^{-1}-1)\|_{H^1}+\|J-1\|_{L^\infty}\lesssim E_3^{\frac{1}{2}}   \quad \forall \, \sigma \in [\sigma_0-1, 2]; \\
&(3). \,\|\dot{\Lambda}_h^{\sigma}( a_{ij},\,\mathcal{A}_i^j)\|_{H^1}+\|(a_{ij},\,\mathcal{A}_i^j)\|_{L^\infty}\lesssim E_3^{\frac{1}{2}}\quad \forall \, \sigma \in [\sigma_0-1, 2]\\
&\qquad \qquad \qquad \qquad \qquad \qquad \qquad \text{with}  \quad(i, j)= (2, 1), (2, 3), (3, 1), (3, 2);\\
&(4).  \,\|\dot{\Lambda}_h^{\sigma}( a_{ii}-1,\,\mathcal{A}_i^i-1)\|_{H^1}+\|(a_{ii}-1,\,\mathcal{A}_i^i-1)\|_{L^\infty}\lesssim E_3^{\frac{1}{2}}\quad\forall \, \sigma \in [\sigma_0-1, 2];\\
&(5). \, \|\dot{\Lambda}_h^{k} (J-1,\,J^{-1}-1,\,a_{ij}{\bf 1}_{i\neq j},\,\mathcal{A}_i^j{\bf 1}_{i\neq j},\,a_{jj}-1,\, \mathcal{A}_j^j-1 )\|_{H^1}\lesssim E_{k+1}^{\frac{1}{2}}.
            \end{split}
\end{equation}
\end{lem}
\begin{proof}
Due to the expression of $a_{ij}$ in \eqref{expre-a-1}, we obtain from Lemma \ref{lem-product-law-1} that
\begin{equation*}
  \begin{split}
 &\|\dot{\Lambda}_h^{\sigma_0}a_{12}\|_{L^2}\lesssim \|\dot{\Lambda}_h^{\sigma_0}\partial_1\xi^2\|_{L^2}
  +\|\dot{\Lambda}_h^{\sigma_0}(\partial_1\xi^2\partial_3\xi^3)\|_{L^2}
  +\|\dot{\Lambda}_h^{\sigma_0}(\partial_1\xi^3  \partial_3\xi^2)\|_{L^2}\\
  &\lesssim \|\dot{\Lambda}_h^{\sigma_0}\partial_1\xi^2\|_{L^2}
  +\|\dot{\Lambda}_h^{\sigma_0}\partial_1\xi^h\|_{L^2}
  \|\partial_3\xi^h\|_{L^\infty}+\|\partial_1\xi^h\|_{L^2_{x_1}L^\infty_h}
  \|\dot{\Lambda}_h^{\sigma_0}\partial_3\xi^h\|_{L^\infty_{x_1}L^2_h},
            \end{split}
\end{equation*}
which follows from \eqref{embedding-ineq-1} that
\begin{equation}\label{est-aij-2}
  \begin{split}
   \|\dot{\Lambda}_h^{\sigma_0}a_{12}\|_{L^2}\lesssim\|\dot{\Lambda}_h^{\sigma_0}\partial_1\xi^2\|_{L^2}
  &+\|\dot{\Lambda}_h^{\sigma_0}\partial_1\xi^h\|_{L^2}
  \|\dot{\Lambda}_h^{\sigma_0}\Lambda_h \partial_3\xi^h\|_{H^1}\\
  &+\|\dot{\Lambda}_h^{\sigma_0}\Lambda_h \partial_1\xi^h\|_{L^2}
  \|\dot{\Lambda}_h^{\sigma_0}\partial_3\xi^h\|_{H^1}\lesssim E_3^{\frac{1}{2}}.
            \end{split}
\end{equation}
While from $\nabla\,a_{21}=-(\nabla\partial_2\xi^1+\nabla\partial_2\xi^1\partial_3\xi^3+\partial_2\xi^1\nabla\partial_3\xi^3-\nabla\partial_2\xi^3  \partial_3\xi^1-\partial_2\xi^3  \nabla\partial_3\xi^1)$, applying Lemma \ref{lem-product-law-1} implies
\begin{equation}\label{est-aij-2a}
  \begin{split}
    & \|\dot{\Lambda}_h^{\sigma_0}\nabla\,a_{12}\|_{L^2}\lesssim \|\dot{\Lambda}_h^{\sigma_0}\nabla\,\partial_1\xi^2\|_{L^2}
  +\|\dot{\Lambda}_h^{(1+\sigma_0)/2}\nabla\,\partial_1\xi^h\|_{L^2}
  \|\dot{\Lambda}_h^{(1+\sigma_0)/2}\partial_3\xi^h\|_{L^\infty_{x_1}L^2_h} \\
  &\qquad\qquad\qquad\qquad\qquad\qquad+\|\dot{\Lambda}_h^{(1+\sigma_0)/2}\partial_1\xi^h\|_{L^\infty_{x_1}L^2_h}
  \|\dot{\Lambda}_h^{(1+\sigma_0)/2}\partial_3\nabla\,\xi^h\|_{L^2}\lesssim E_3^{\frac{1}{2}}.
            \end{split}
\end{equation}
Combining \eqref{est-aij-2} with \eqref{est-aij-2a} leads to $\|\dot{\Lambda}_h^{\sigma_0}\,a_{12}\|_{H^1}\lesssim E_3^{\frac{1}{2}}$.

The same conclusion can be drawn for $ \|\dot{\Lambda}_h^{2}\,a_{12}\|_{H^1}$, that is, $\|\dot{\Lambda}_h^{2}\,a_{12}\|_{H^1}\lesssim E_3^{\frac{1}{2}}$.

Combining the above two inequalities, we thus use the interpolation inequality to find
\begin{equation*}\label{a12-H1-est-1}
\|\dot{\Lambda}_h^{\sigma}\,a_{12}\|_{H^1}\lesssim E_3^{\frac{1}{2}} \quad \forall \, \sigma \in [\sigma_0, 2],
\end{equation*}
and then, from \eqref{embedding-ineq-1}, it follows
\begin{equation*}\label{a12-H1-est-2}
\begin{split}
&\|a_{12}\|_{L^\infty} \lesssim  \|\dot{\Lambda}_h^{\sigma_0}\Lambda_h a_{12}\|_{L^\infty_{x_1}L^2_h} \lesssim  \|\dot{\Lambda}_h^{\sigma_0}\Lambda_h a_{12}\|_{H^1}\lesssim E_3^{\frac{1}{2}}.
\end{split}
\end{equation*}
By a similar argument, one can see
\begin{equation*}\label{a13-H1-est-1}
\begin{split}
& \|\dot{\Lambda}_h^{\sigma}\,a_{13}\|_{H^1}+\|a_{13}\|_{L^\infty}\lesssim E_3^{\frac{1}{2}} \quad \forall \, \sigma \in [\sigma_0, 2].
\end{split}
\end{equation*}

For $a_{22}-1$, it immediately shows
\begin{equation}\label{a22-est-1}
  \begin{split}
&\|\dot{\Lambda}_h^{\sigma_0-1}(a_{22}-1)\|_{H^1} \lesssim\|\dot{\Lambda}_h^{\sigma_0-1}(a_{22}-1)\|_{L^2}
+\|\dot{\Lambda}_h^{\sigma_0-1}\nabla\,a_{22}\|_{L^2}\\
&\lesssim \|\dot{\Lambda}_h^{\sigma_0-1}(\partial_1\xi^1,\,\partial_3\xi^3)\|_{H^1}
+\|\dot{\Lambda}_h^{\sigma_0-1}(\partial_1\xi^1 \partial_3\xi^3,\,\partial_1\xi^3\partial_3\xi^1)\|_{L^2}\\
&\qquad
+\|\dot{\Lambda}_h^{\sigma_0-1}(\partial_1\xi^1\nabla\,\partial_3\xi^3,\,
\nabla\,\partial_1\xi^1\partial_3\xi^3,\,\nabla\,\partial_1\xi^3  \partial_3\xi^1,\,\partial_1\xi^3  \nabla\, \partial_3\xi^1)\|_{L^2}.
  \end{split}
\end{equation}
which along with \eqref{lem-product-law-1} yields  $\|\dot{\Lambda}_h^{\sigma_0-1}(\nabla\,\partial_1\xi^3  \partial_3\xi^1)\|_{L^2}\lesssim \|\dot{\Lambda}_h^{\sigma_0}\nabla\,\partial_1\xi^3\|_{L^2}
\|\partial_3\xi^1\|_{L^\infty_{x_1}L^2_h}\lesssim E_3.$
The same estimate holds for other nonlinear terms in the right hand side of \eqref{a22-est-1}, which, combining $\|\dot{\Lambda}_h^{\sigma_0-1}(\partial_1\xi^1,\,\partial_3\xi^3)\|_{H^1} \lesssim E_3^{\frac{1}{2}}$, implies $\|\dot{\Lambda}_h^{\sigma_0-1}(a_{22}-1)\|_{H^1}\lesssim E_3^{\frac{1}{2}}$.

Similarly,  for $k=1, 2, 3$, $(i, j)= (2, 1), (2, 3), (3, 1), (3, 2)$, there hold that $\|\dot{\Lambda}_h^{\sigma_0-1}(a_{kk}-1,\,J-1,\,a_{ij})\|_{H^1} \lesssim E_3^{\frac{1}{2}}$
and $\|\dot{\Lambda}_h^{2}(a_{kk}-1,\,J-1,\, a_{ij})\|_{H^1} \lesssim E_3^{\frac{1}{2}}$,
so, applying the interpolation inequality yields
\begin{equation}\label{akk-est-1}
  \begin{split}
&\|\dot{\Lambda}_h^{\sigma}(a_{kk}-1,\,J-1,\, a_{ij})\|_{H^1}+\|(a_{kk}-1,\,J-1,\, a_{ij})\|_{L^\infty} \lesssim E_3^{\frac{1}{2}} \quad \forall \,\sigma \in [\sigma_0-1, 2].
  \end{split}
\end{equation}

Next, applying Lemma \ref{lem-composition-1}, it produces from \eqref{akk-est-1} that
\begin{equation*}\label{J-inverse-1}
  \begin{split}
&\|\dot{\Lambda}_h^{\sigma}(J^{-1}-1)\|_{H^1}\lesssim E_3^{\frac{1}{2}}\quad \text{for} \quad \sigma \in [\sigma_0, 2],
  \end{split}
\end{equation*}
which ensures that $\mathcal{A}_i^j$ shares the same estimate with $a_{ij}$ with $i, j=1, 2, 3$.

Finally, for the last inequality in \eqref{est-aij-J-1}, we need only prove it for the term $a_{22}-1$, and the same proof remains valid for the others.
Indeed, for any $k\in \mathbb{N}$, it is easy to verify that
\begin{equation}\label{a22-Hk-1}
  \begin{split}
&\|\dot{\Lambda}_h^{k}(a_{22}-1)\|_{H^1} \lesssim\|\dot{\Lambda}_h^{k}(a_{22}-1)\|_{L^2}
+\|\dot{\Lambda}_h^{k}\nabla\,a_{22}\|_{L^2}\\
&\lesssim \|\dot{\Lambda}_h^{k}(\partial_1\xi^1,\,\partial_3\xi^3)\|_{H^1}
+\|\dot{\Lambda}_h^{k}(\partial_1\xi^1 \partial_3\xi^3,\,\partial_1\xi^3\partial_3\xi^1)\|_{L^2}\\
&\qquad
+\|\dot{\Lambda}_h^{k}(\partial_1\xi^1\nabla\,\partial_3\xi^3,\,
\nabla\,\partial_1\xi^1\partial_3\xi^3,\,\nabla\,\partial_1\xi^3  \partial_3\xi^1,\,\partial_1\xi^3  \nabla\, \partial_3\xi^1)\|_{L^2}.
  \end{split}
\end{equation}
Due to \eqref{lem-product-law-1}, we may get
\begin{equation*}
  \begin{split}
\|\dot{\Lambda}_h^{k}(\nabla\,\partial_1\xi^3  \partial_3\xi^1)\|_{L^2}&\lesssim \|\dot{\Lambda}_h^{k}\nabla\,\partial_1\xi^3\|_{L^2}
\|\partial_3\xi^1\|_{L^\infty}+\|\nabla\,\partial_1\xi^3\|_{L^2_{x_1}L^\infty_h}
\|\dot{\Lambda}_h^{k}\partial_3\xi^1\|_{L^\infty_{x_1}L^2_h}\\
&\lesssim E_{k+1}^{\frac{1}{2}} E_{3}^{\frac{1}{2}}.
  \end{split}
\end{equation*}
The same estimate holds for other nonlinear terms in the right hand side of \eqref{a22-Hk-1}, which, together with $\|\dot{\Lambda}_h^{k}(\partial_1\xi^1,\,\partial_3\xi^3)\|_{H^1}\lesssim E_{k+1}^{\frac{1}{2}}$,
follows $\|\dot{\Lambda}_h^{k}(a_{22}-1)\|_{H^1}\lesssim E_{k+1}^{\frac{1}{2}}$.

The proof of the lemma is thus completed.
\end{proof}

With Lemma \ref{lem-est-aij-1} in hand, let us now estimate $\mathcal{B}_{j, i}^{\alpha}$ and $\mathcal{B}_{j, i}^{\alpha}\partial_\alpha\,v^i$ with $j=1, \cdots, 9.$
\begin{lem}\label{lem-est-B-Bv-1}
Let $j=1, \cdots, 9.$, $i=1, 2, 3$, $\alpha=2, 3$, $\sigma \in [\sigma_0-1, 2]$, $2 \leq k \in \mathbb{N}$, if $E_3(t) \leq 1$, then there hold that
\begin{align}
&\|\dot{\Lambda}_h^{\sigma}\mathcal{B}_{j, i}^{\alpha}\|_{H^1} \lesssim E_3^{\frac{1}{2}},\quad \|\dot{\Lambda}_h^{k}\mathcal{B}_{j, i}^{\alpha}\|_{H^1}\lesssim E_{k+1}^{\frac{1}{2}},\quad\|\dot{\Lambda}_h^{\sigma}(\mathcal{B}_{j, i}^{\alpha}\partial_\alpha\,v^i)\|_{H^1}\lesssim E_{3}^{\frac{1}{2}} \|\partial_h\Lambda_h^2\,v\|_{H^1},\label{est-B-1}\\
&\|\dot{\Lambda}_h^{k}(\mathcal{B}_{j, i}^{\alpha}\partial_\alpha\,v^i)\|_{L^2}\lesssim E_{k+1}^{\frac{1}{2}}\|\partial_h\Lambda_h^2\,v\|_{L^2}+E_{3}^{\frac{1}{2}} \|\partial_h\Lambda_h^{k}\,v\|_{L^2},\label{est-Bv-1}\\
&\|\dot{\Lambda}_h^{k}(\mathcal{B}_{j, i}^{\alpha}\partial_\alpha\,v^i)\|_{H^1}\lesssim E_{k+1}^{\frac{1}{2}}\|\partial_h\Lambda_h^2\,v\|_{H^1}+E_{3}^{\frac{1}{2}} \|\partial_h\Lambda_h^{k}\,v\|_{H^1},\label{est-Bv-1a}\\
&\|\dot{\Lambda}_h^{\sigma}(\widetilde{a_{21}}\partial_1v^2,\,\widetilde{a_{31}}\partial_1v^3)\|_{H^1} \lesssim  E_3^{\frac{1}{2}}
(\|\dot{\Lambda}_h^{\sigma+1}\partial_1v^h\|_{H^1}+\|\partial_{h}\Lambda_h^2\nabla\,v\|_{L^2}),
\label{low-incom-fluid-1}\\
&\|\dot{\Lambda}_h^{k}(\widetilde{a_{21}}\partial_1v^2,\,\widetilde{a_{31}}\partial_1v^3)\|_{H^1} \lesssim  E_3^{\frac{1}{2}}
\dot{\mathcal{D}}_{k+1}^{\frac{1}{2}}+E_{k+1}^{\frac{1}{2}}
\dot{\mathcal{D}}_{3}^{\frac{1}{2}}.\label{k-incom-fluid-1}
\end{align}
\end{lem}
\begin{proof}
In view of Lemma  \ref{lem-est-aij-1}, it is easy to get the first two inequalities in \eqref{est-B-1}. Let's now focus on the estimates of $\mathcal{B}_{j, i}^{\alpha}\partial_\alpha\,v^i$ with $j=1, \cdots, 9.$

We first directly use \eqref{lem-product-law-1} to get
\begin{equation*}\label{est-Bv-2}
  \begin{split}
&\|\dot{\Lambda}_h^{\sigma_0-1}(\mathcal{B}_{j, i}^{\alpha}\partial_\alpha\,v^i)\|_{H^1}\lesssim \|\dot{\Lambda}_h^{\sigma_0-1}(\mathcal{B}_{j, i}^{\alpha}\partial_\alpha\,v^i)\|_{L^2}+\|\dot{\Lambda}_h^{\sigma_0-1}(\nabla\,\mathcal{B}_{j, i}^{\alpha}\partial_\alpha\,v^i+\mathcal{B}_{j, i}^{\alpha}\nabla\,\partial_\alpha\,v^i)\|_{L^2}\\
&\lesssim \|\dot{\Lambda}_h^{\sigma_0}\mathcal{B}_{j, i}^{\alpha}\|_{L^2}\|\partial_\alpha\,v^i\|_{L^\infty_{x_1}L^2_h}+\|\dot{\Lambda}_h^{\sigma_0}\nabla\,\mathcal{B}_{j, i}^{\alpha}\|_{L^2}\|\partial_\alpha\,v^i\|_{L^\infty_{x_1}L^2_h}+\|\dot{\Lambda}_h^{\sigma_0}\mathcal{B}_{j, i}^{\alpha}\|_{L^\infty_{x_1}L^2_h}\|\nabla\,\partial_\alpha\,v^i\|_{L^2},
  \end{split}
\end{equation*}
then applying the Sobolev embedding theorem shows
\begin{equation*}\label{est-Bv-2}
  \begin{split}
&\|\dot{\Lambda}_h^{\sigma_0-1}(\mathcal{B}_{j, i}^{\alpha}\partial_\alpha\,v^i)\|_{H^1}\lesssim \|\dot{\Lambda}_h^{\sigma_0}\mathcal{B}_{j, i}^{\alpha}\|_{H^1}\|\partial_\alpha\,v^i\|_{H^1}\lesssim E_{3}^{\frac{1}{2}} \|\partial_\alpha\,v^i\|_{H^1}.
  \end{split}
\end{equation*}
Similar estimates hold for $\|\dot{\Lambda}_h^{\sigma}(\mathcal{B}_{j, i}^{\alpha}\partial_\alpha\,v^i)\|_{H^1}$ with $\sigma \in [\sigma_0-1, 2]$.

On the other hand, thanks to \eqref{lem-product-law-1} again, it can be found that
\begin{equation*}\label{est-Bv-4}
  \begin{split}
\|\dot{\Lambda}_h^{k}(\mathcal{B}_{j, i}^{\alpha}\partial_\alpha\,v^i)\|_{L^2}&\lesssim \|\dot{\Lambda}_h^{k}\mathcal{B}_{j, i}^{\alpha}\|_{L^\infty_{x_1}L^2_h}\|\partial_\alpha\,v^i\|_{L^2_{x_1}L^\infty_h}+\|\mathcal{B}_{j, i}^{\alpha}\|_{L^\infty}\|\dot{\Lambda}_h^{k}\partial_\alpha\,v^i\|_{L^2}\\
&\lesssim E_{k+1}^{\frac{1}{2}}\|\partial_h\Lambda_h^2\,v\|_{L^2}+E_{3}^{\frac{1}{2}} \|\partial_h\Lambda_h^{k}\,v\|_{L^2},
  \end{split}
\end{equation*}
and
\begin{equation*}\label{est-Bv-3}
  \begin{split}
&\|\dot{\Lambda}_h^{k}\nabla(\mathcal{B}_{j, i}^{\alpha}\partial_\alpha v^i)\|_{L^2}\lesssim \|\dot{\Lambda}_h^{k}(\nabla \mathcal{B}_{j, i}^{\alpha}\partial_\alpha v^i+\mathcal{B}_{j, i}^{\alpha}\nabla \partial_\alpha v^i)\|_{L^2}\\
&\lesssim \|\dot{\Lambda}_h^{k}\nabla\,\mathcal{B}_{j, i}^{\alpha}\|_{L^2}\|\partial_\alpha v^i\|_{L^\infty}+\|\dot{\Lambda}_h^{k}\partial_\alpha v^i\|_{L^\infty_{x_1}L^2_h}
\|\nabla\,\mathcal{B}_{j, i}^{\alpha}\|_{L^2_{x_1}L^\infty_h}\\
&\qquad+\|\dot{\Lambda}_h^{k}\mathcal{B}_{j, i}^{\alpha}\|_{L^\infty_{x_1}L^2_h}\|\nabla \partial_\alpha v^i\|_{L^2_{x_1}L^\infty_h}+
\|\mathcal{B}_{j, i}^{\alpha}\|_{L^\infty}\|\dot{\Lambda}_h^{k}\nabla \partial_\alpha v^i\|_{L^2},
  \end{split}
\end{equation*}
so using the Sobolev embedding theorem yields \eqref{est-Bv-1a}.

Finally, for $\eqref{low-incom-fluid-1}$, due to the same argument above, we need only prove it in the case of $\sigma=\sigma_0-1$. Notice that from the product law \eqref{product-law-1}
\begin{equation*}
  \begin{split}
&\|\dot{\Lambda}_h^{\sigma_0-1}(\widetilde{a_{21}}\partial_1v^2,\,\widetilde{a_{31}}\partial_1v^3)\|_{L^2}\lesssim \|(\widetilde{a_{21}},\,\widetilde{a_{31}})\|_{L^\infty_{x_1}L^2_h}
\|\dot{\Lambda}_h^{\sigma_0}(\partial_1v^2,\,\partial_1v^3)\|_{L^2},
  \end{split}
\end{equation*}
and
\begin{equation*}
  \begin{split}
\|\dot{\Lambda}_h^{\sigma_0-1}\nabla(\widetilde{a_{21}}\partial_1v^2,\,\widetilde{a_{31}}\partial_1v^3)\|_{L^2}
&\lesssim  \|(\widetilde{a_{21}},\,\widetilde{a_{31}})\|_{L^\infty_{x_1}L^2_h}
\|\dot{\Lambda}_h^{\sigma_0}\nabla(\partial_1v^2,\,\partial_1v^3)\|_{L^2}\\
&\qquad+\|\nabla(\widetilde{a_{21}},\,\widetilde{a_{31}})\|_{L^2}
\|\dot{\Lambda}_h^{\sigma_0}(\partial_1v^2,\,\partial_1v^3)\|_{L^\infty_{x_1}L^2_h},\\
  \end{split}
\end{equation*}
it follows from Lemma \ref{lem-est-aij-1} that\eqref{low-incom-fluid-1}. The estimate in \eqref{k-incom-fluid-1} can be showed by using the same argument in the proof of \eqref{est-Bv-1}.

 The lemma is therefore completely proved.
\end{proof}

Due to the incompressibility condition \eqref{incomp-cond-fluid-3}, the vertical derivatives of the vertical component $v^1$ of the velocity gains the $-1$-order horizontal regularity in the dissipation.
\begin{lem}\label{lem-EN-bdd-1}
Let $(v, \xi)$ be smooth solution to the system \eqref{eqns-pert-1}, if $E_3(t) \leq 1$ for all existence time $t$, then there holds
\begin{equation*}\label{dissip-equiv-1}
  \begin{split}
\ddot{\mathcal{D}}_{\ell, 1+\sigma_0}+\dot{\mathcal{D}}_{\ell, \sigma_0}+\ddot{\mathcal{D}}_{\ell, 2}\thicksim \|(\dot{\Lambda}_h^{\sigma_0-1}\nabla^2v^1,\,&\dot{\Lambda}_h^{\sigma_0}(\nabla^2v^h,\,\nabla p))\|_{L^2(\Omega)}^2\\
&+\|\dot{\Lambda}_h^{\sigma_0+1}\xi^1\|_{L^2(\Sigma_0)}^2+\ddot{\mathcal{D}}_{\ell, \sigma_0}+\dot{\mathcal{D}}_{\ell, 2},
            \end{split}
\end{equation*}
\end{lem}
\begin{proof}
By the definitions of $\ddot{\mathcal{D}}_{\ell, 1+\sigma_0}$, $\dot{\mathcal{D}}_{\ell, \sigma_0}$ and $\ddot{\mathcal{D}}_{\ell, 2}$, it suffices to show that
\begin{equation}\label{dissip-equiv-1a}
\begin{split}
\|\dot{\Lambda}_h^{\sigma_0-1}\partial_1^2v^1\|_{L^2(\Omega)}^2 \lesssim \ddot{\mathcal{D}}_{\ell, 1+\sigma_0}+\dot{\mathcal{D}}_{\ell, \sigma_0}+\ddot{\mathcal{D}}_{\ell, 2}.
\end{split}
\end{equation}
Indeed, thanks to the incompressibility condition \eqref{incomp-cond-fluid-3}, one can see
\begin{equation*}
\begin{split}
&\dot{\Lambda}_h^{\sigma_0-1}\partial_1^2v^1=-\dot{\Lambda}_h^{\sigma_0-1}\partial_1\nabla_h\cdot v^h -\dot{\Lambda}_h^{\sigma_0-1}\partial_1(\widetilde{a_{\alpha\,1}}\partial_1v^\alpha)
+\dot{\Lambda}_h^{\sigma_0-1}\partial_1(\mathcal{B}_{5, i}^{\alpha}\partial_{\alpha}v^{i}),
\end{split}
\end{equation*}
and then
\begin{equation*}
\begin{split}
&\|\dot{\Lambda}_h^{\sigma_0-1}\partial_1^2v^1\|_{L^2}\lesssim\|\dot{\Lambda}_h^{\sigma_0}\partial_1v^h \|_{L^2}+\|\dot{\Lambda}_h^{\sigma_0-1}\partial_1(\widetilde{a_{\alpha\,1}}\partial_1v^\alpha)\|_{L^2}
+\|\dot{\Lambda}_h^{\sigma_0-1}\partial_1(\mathcal{B}_{5, i}^{\alpha}\partial_{\alpha}v^{i})\|_{L^2}.
\end{split}
\end{equation*}
Hence, by virtue of \eqref{est-B-1} and \eqref{low-incom-fluid-1}, we deduce that
\begin{equation*}\label{dissip-equiv-1b}
\begin{split}
\|\dot{\Lambda}_h^{\sigma_0-1}\partial_1^2v^1\|_{L^2}&\lesssim \|\dot{\Lambda}_h^{\sigma_0}\partial_1v^h \|_{L^2}+E_{3}^{\frac{1}{2}} (\|\partial_h\Lambda_h^2\,v\|_{H^1}+\|\dot{\Lambda}_h^{\sigma_0}\partial_1v^h\|_{H^1}+\|\partial_{h}\Lambda_h^2\nabla\,v\|_{L^2})\\
&\lesssim (\ddot{\mathcal{D}}_{\ell, 1+\sigma_0}+\dot{\mathcal{D}}_{\ell, \sigma_0}+\ddot{\mathcal{D}}_{\ell, 2})^{\frac{1}{2}},
\end{split}
\end{equation*}
which results in \eqref{dissip-equiv-1a}, and ends the proof of Lemma \ref{lem-EN-bdd-1}.
\end{proof}

\renewcommand{\theequation}{\thesection.\arabic{equation}}
\setcounter{equation}{0}
\section{Energy estimates of the tangential derivatives of the velocity}\label{sect-energy}

In this section, we will derive some global energy estimates. In view of the local well-posedness theorem (Theorem \ref{thm-local}), it suffices to get necessary {\it a priori} estimates. Here and in what follows, all the $C$-forms, such as $C_j$, $c_j$, $\mathfrak{C}_j$, $\widetilde{\mathfrak{C}}_j$, $\mathfrak{c}_j$, and $\widetilde{\mathfrak{c}}_j$, are generic positive constants, which may be different on different lines.

\subsection{Energy estimates in $L^2$}
We begin with the energy identity in $L^2$ based on the nonlinear structure of the system \eqref{eqns-pert-1}.
\begin{prop}\label{prop-linear-v-L2-1}
Let $(v, \xi)$ be smooth solution to the system \eqref{eqns-pert-1}, then there holds
\begin{equation}\label{linear-v-L2-1-0}
\begin{split}
 & \frac{d}{dt}\bigg(\int_{\Omega}  |v|^2\,J\, dx+\int_{\Sigma_{0}} a_{11}|\xi^1|^2\,dS_0 \bigg)+\nu\int_{\Omega} |\mathbb{D}_{J\mathcal{A}}(v)|^2 \,J^{-1}\, dx=0.
\end{split}
\end{equation}
\end{prop}

\begin{proof}
We multiply the $i$-th component of the momentum equations of \eqref{eqns-pert-1} by $J\,v^i$ , sum over $i$, and integrate over $\Omega$ to find
\begin{equation}\label{linear-v-L2-1-1}
\begin{split}
 & \frac{1}{2}\frac{d}{dt}\int_{ \Omega}J|v|^2\, dx- \frac{1}{2} \int_{ \Omega}|v|^2\,\partial_tJ\, dx+I+II=0
\end{split}
\end{equation}
with
$I=\int_{ \Omega} \grad_{J\mathcal{A}}\,q\,\cdot v\,dx,\quad II=\int_{ \Omega}(-\nu\,\grad_{J\mathcal{A}} \cdot  \mathbb{D}_{\mathcal{A}}(v)) \cdot v\,dx.$

Due to the incompressibility condition $\grad_{J\mathcal{A}}\cdot v=0$, from \eqref{identity-deri-J}, it provides
\begin{equation}\label{linear-v-L2-1-2a}
\begin{split}
 &\partial_tJ=0.
\end{split}
\end{equation}
Integrating by parts in $I$, $II$, and utilizing the Piola identity \eqref{identity-Piola}, yields
\begin{equation}\label{linear-v-L2-1-2}
\begin{split}
 &I=\int_{\Sigma_{0}}v\cdot (q\, \mathcal{N})\,dS_0-\int_{ \Omega}q\,\grad_{J\mathcal{A}} \cdot v \, dx=\int_{\Sigma_{0}}v\cdot (q\, \mathcal{N})\,dS_0,
\end{split}
\end{equation}
and
\begin{equation}\label{linear-v-L2-1-3}
\begin{split}
 &II=\int_{\Sigma_{0}} (- \nu \mathbb{D}_{\mathcal{A}}(v)\mathcal{N} ) \cdot v \, dS_0+\nu\int_{\Omega} \mathbb{D}_{\mathcal{A}}(v): \grad_{J\mathcal{A}} v  \, dx\\
 &=\int_{\Sigma_{0}} (- \nu\mathbb{D}_{\mathcal{A}}(v)\mathcal{N} ) \cdot v \, dS_0+\int_{\Omega}\frac{\nu}{2}|\mathbb{D}_{J\mathcal{A}}(v)|^2\,J^{-1}\, dx,
 \end{split}
\end{equation}
where the identity $ \mathbb{D}_{\mathcal{A}}(v): \grad_{\mathcal{A}} u=\frac{1}{2}\mathbb{D}_{\mathcal{A}}(v):\mathbb{D}_{\mathcal{A}}(u)$ has been used in the second equality in \eqref{linear-v-L2-1-3}.

Plugging \eqref{linear-v-L2-1-2a}-\eqref{linear-v-L2-1-3} into \eqref{linear-v-L2-1-1} implies
\begin{equation}\label{linear-v-L2-1-6}
\begin{split}
 & \frac{1}{2}\frac{d}{dt}\int_{ \Omega}|v|^2 \,J\,dx+\int_{ \Omega} \frac{\nu}{2}|\mathbb{D}_{J\mathcal{A}}(v)|^2 \,J^{-1}\, dx+\int_{\Sigma_{0}}v\cdot \bigg(q\,\mathcal{N}- \nu\mathbb{D}_{\mathcal{A}}(v)\mathcal{N}\bigg)\,dS_0=0.
\end{split}
\end{equation}
For the boundary integral in \eqref{linear-v-L2-1-6}, making use of the interface boundary condition in \eqref{eqns-pert-1} yields
\begin{equation*}
\begin{split}
&\int_{\Sigma_{0}}v\cdot \bigg(q\,\mathcal{N}- \nu\mathbb{D}_{\mathcal{A}}(v)\mathcal{N}\bigg)\,dS_0=\int_{\Sigma_{0}} \xi^1\,\mathcal{N}\,\cdot {v}\,dS_0=\int_{\Sigma_{0}} \xi^1\,a_{i1}v^i\,dS_0\\
&=\int_{\Sigma_{0}} \frac{1}{2}a_{11}\partial_t |\xi^1|^2\,dS_0 +\int_{\Sigma_{0}} \xi^1\,(a_{21}v^2+a_{31} v^3)\,dS_0,
\end{split}
\end{equation*}
which together with the expression of $a_{ij}$ in \eqref{expre-a-1} follows
\begin{equation}\label{linear-v-L2-1-6a}
\begin{split}
&\int_{\Sigma_{0}}v\cdot \bigg(q\,\mathcal{N}- \nu\mathbb{D}_{\mathcal{A}}(v)\mathcal{N}\bigg)\,dS_0\\
& =\frac{1}{2}\frac{d}{dt}\int_{\Sigma_{0}} a_{11}|\xi^1|^2\,dS_0+\int_{\Sigma_{0}} \bigg(\xi^1\,(a_{21}v^2+a_{31} v^3)-\frac{1}{2}|\xi^1|^2\partial_ta_{11}\bigg)\,dS_0\\
& =\frac{1}{2}\frac{d}{dt}\int_{\Sigma_{0}} a_{11}|\xi^1|^2\,dS_0.
\end{split}
\end{equation}
Inserting  \eqref{linear-v-L2-1-6a} into \eqref{linear-v-L2-1-6} ensures \eqref{linear-v-L2-1-0}, which ends the proof of Lemma \ref{prop-linear-v-L2-1}.
\end{proof}

\subsection{Energy estimates of the horizontal derivatives of the velocity}\label{subsect-est-hori-1}

In general, since the equations of the derivatives of the solution $(v, \xi)$ to the system \eqref{eqns-pert-1} lack the nonlinear symmetric structure, we have no idea to get the high-order energy identity. For this reason, we turn to employ the linearized form \eqref{eqns-linear-1} of the system \eqref{eqns-pert-1} to study the high-order energy estimates.

For the general horizontal derivatives of the velocity, we first derive
\begin{lem}\label{lem-tan-pseudo-energy-1}
Let $(v, \xi)$ be smooth solution to the system \eqref{eqns-pert-1}, then there holds
\begin{equation}\label{tan-linear-pseudo-0}
\begin{split}
 & \frac{1}{2}\frac{d}{dt}(\|\mathcal{P}(\partial_h)v\|_{L^2(\Omega)}^2
 +\|\mathcal{P}(\partial_h)\xi^1\|_{L^2(\Sigma_0)}^2) +\frac{\nu}{2}\| \mathbb{D}(\mathcal{P}(\partial_h)v)\|_{L^2(\Omega)}^2=\sum_{j=1}^3\mathfrak{K}_j
\end{split}
\end{equation}
with
\begin{equation}\label{tan-linear-pseudo-0a}
\begin{split}
 &\mathfrak{K}_1:=\nu\int_{\Sigma_{0}}\bigg(\mathcal{P}(\partial_h)((\mathcal{B}_{9, i}^{\alpha}-2\mathcal{B}_{6, i}^{\alpha})\partial_{\alpha}v^{i})\, \mathcal{P}(\partial_h)v^1\\
&\qquad\qquad\qquad+\mathcal{P}(\partial_h)(\mathcal{B}_{7, i}^{\alpha}\partial_{\alpha}v^{i}) \,\mathcal{P}(\partial_h)v^2+\mathcal{P}(\partial_h)(\mathcal{B}_{8, i}^{\alpha}\partial_{\alpha}v^{i})\, \mathcal{P}(\partial_h)v^3\bigg) \, dS_0,\\
&\mathfrak{K}_2:=\int_{ \Omega}\mathcal{P}(\partial_h)q\,\mathcal{P}(\partial_h)(\mathcal{B}_{5, i}^{\alpha}\partial_{\gamma}v^{i}-\widetilde{a_{\alpha\,1}}\partial_1v^\alpha)\, dx,\ \mathfrak{K}_3:=\int_{ \Omega}\mathcal{P}(\partial_h) g\cdot \mathcal{P}(\partial_h) v\,dx.
\end{split}
\end{equation}
\end{lem}
\begin{proof}
Acting the operator $\mathcal{P}(\partial_h)$ to both the momentum equations and the incompressibility equation in \eqref{eqns-linear-1} implies
\begin{equation}\label{eqns-linear-pseudo-1}
\begin{cases}
 & \partial_t \mathcal{P}(\partial_h)v +\nabla\,\mathcal{P}(\partial_h)q- \nu \nabla\cdot\mathbb{D}(\mathcal{P}(\partial_h)v)=\mathcal{P}(\partial_h)g,\\
 &\nabla\cdot \mathcal{P}(\partial_h)v= \mathcal{P}(\partial_h)(-\widetilde{a_{\alpha\,1}}\partial_1v^\alpha +\mathcal{B}_{5, i}^{\alpha}\partial_{\alpha}v^{i}) \quad \text{in} \quad \Omega.
    \end{cases}
\end{equation}
We multiply the $i$-th component of the momentum equations of \eqref{eqns-linear-pseudo-1} by $\mathcal{P}(\partial_h)v^i$, sum over $i$, and integrate
over $\Omega$ to find
\begin{equation}\label{linear-pseudo-1}
\begin{split}
 & \frac{1}{2}\frac{d}{dt}\int_{ \Omega}|\mathcal{P}(\partial_h)v|^2\, dx+I+II=III
\end{split}
\end{equation}
with
$I=\int_{ \Omega} \nabla \mathcal{P}(\partial_h)\,q\,\cdot \mathcal{P}(\partial_h) v\,dx,\quad II=-\nu\int_{ \Omega}[\nabla\,\cdot \mathbb{D}(\mathcal{P}(\partial_h) v )]\cdot \mathcal{P}(\partial_h) v\,dx$,
$III= \int_{ \Omega}\mathcal{P}(\partial_h) g\cdot \mathcal{P}(\partial_h) v\,dx$.

In view of the boundary condition $v|_{\Sigma_b}=0$, integrating by parts in $I$ and $II$ shows
\begin{equation}\label{linear-pseudo-2}
\begin{split}
 &I=\int_{\Sigma_{0}}\mathcal{P}(\partial_h)v \cdot\,(\mathcal{P}(\partial_h)q\,e_1)\,dS_0-\int_{ \Omega}\mathcal{P}(\partial_h)\,q\,\grad \cdot \mathcal{P}(\partial_h)v \, dx,
\end{split}
\end{equation}
and
\begin{equation}\label{linear-pseudo-3}
\begin{split}
 &II=- \nu\int_{\Sigma_{0}}   (\mathbb{D}(\mathcal{P}(\partial_h)v)\,e_1)\cdot \mathcal{P}(\partial_h)v \, dS_0+\frac{\nu}{2}\int_{\Omega} | \mathbb{D}(\mathcal{P}(\partial_h)v)|^2 \, dx,
 \end{split}
\end{equation}
so summing up \eqref{linear-pseudo-2} and \eqref{linear-pseudo-3} provides
\begin{equation}\label{linear-pseudo-4a}
\begin{split}
 &I+II=\int_{\Sigma_{0}}\mathcal{P}(\partial_h)v\cdot\mathcal{P}(\partial_h)\bigg(q\,e_1
 -\nu\mathbb{D}(v)\,e_1\bigg)\,dS_0\\
 &\qquad\qquad\qquad-\int_{ \Omega}\mathcal{P}(\partial_h)q\,\grad \cdot \mathcal{P}(\partial_h)v \, dx+\frac{\nu}{2}\int_{\Omega} | \mathbb{D}(\mathcal{P}(\partial_h)v)|^2 \, dx.
 \end{split}
\end{equation}
For the first boundary integral of the right hand side in \eqref{linear-pseudo-4a}, applying the interface boundary condition in \eqref{eqns-linear-1} leads to
\begin{equation}\label{linear-pseudo-4}
\begin{split}
 &\int_{\Sigma_{0}}\mathcal{P}(\partial_h)v\cdot\mathcal{P}(\partial_h)\bigg(q\,e_1
 -\nu\mathbb{D}(v)\,e_1\bigg)\,dS_0\\
 &=\frac{1}{2}\frac{d}{dt}\int_{\Sigma_{0}}|\mathcal{P}(\partial_h)\xi^1|^2\,dS_0-\nu\int_{\Sigma_{0}}\bigg(\mathcal{P}(\partial_h)((\mathcal{B}_{9, i}^{\alpha}-2\mathcal{B}_{6, i}^{\alpha})\partial_{\alpha}v^{i})\, \mathcal{P}(\partial_h)v^1\\
&\qquad\qquad\qquad+\mathcal{P}(\partial_h)(\mathcal{B}_{7, i}^{\alpha}\partial_{\alpha}v^{i}) \,\mathcal{P}(\partial_h)v^2+\mathcal{P}(\partial_h)(\mathcal{B}_{8, i}^{\alpha}\partial_{\alpha}v^{i})\, \mathcal{P}(\partial_h)v^3\bigg) \, dS_0.
 \end{split}
\end{equation}
On the other hand, thanks to the second equation in \eqref{eqns-linear-pseudo-1}, one may see that
\begin{equation}\label{linear-pseudo-5}
\begin{split}
 &\int_{ \Omega}\mathcal{P}(\partial_h)q\,\grad \cdot \mathcal{P}(\partial_h)v \, dx=\int_{ \Omega}\mathcal{P}(\partial_h)q\,\mathcal{P}(\partial_h)(-\widetilde{a_{\alpha\,1}}\partial_1v^\alpha +\mathcal{B}_{5, i}^{\alpha}\partial_{\alpha}v^{i}) \, dx.
 \end{split}
\end{equation}
Hence, plugging \eqref{linear-pseudo-4} and \eqref{linear-pseudo-5} into \eqref{linear-pseudo-4a} gives rise to
\begin{equation}\label{linear-pseudo-6}
  \begin{split}
&I+II=\frac{1}{2}\frac{d}{dt}\|\mathcal{P}(\partial_h)\xi^1\|_{L^2(\Sigma_0)}^2-\nu\int_{\Sigma_{0}}
\bigg(\mathcal{P}(\partial_h)((\mathcal{B}_{9, i}^{\alpha}-2\mathcal{B}_{6, i}^{\alpha})\partial_{\alpha}v^{i})\, \mathcal{P}(\partial_h)v^1\\
&\qquad\qquad\qquad+\mathcal{P}(\partial_h)(\mathcal{B}_{7, i}^{\alpha}\partial_{\alpha}v^{i}) \,\mathcal{P}(\partial_h)v^2+\mathcal{P}(\partial_h)(\mathcal{B}_{8, i}^{\alpha}\partial_{\alpha}v^{i})\, \mathcal{P}(\partial_h)v^3\bigg) \, dS_0\\
&\qquad\qquad+\frac{\nu}{2}\| \mathbb{D}(\mathcal{P}(\partial_h)v)\|_{L^2(\Omega)}^2 +\int_{ \Omega}\mathcal{P}(\partial_h)q\,\mathcal{P}(\partial_h)(-\widetilde{a_{\alpha\,1}}\partial_1v^\alpha +\mathcal{B}_{5, i}^{\alpha}\partial_{\alpha}v^{i}) \, dx.
\end{split}
\end{equation}
Combining \eqref{linear-pseudo-6} with \eqref{linear-pseudo-1} yields \eqref{tan-linear-pseudo-0}, which completes the proof of Lemma \ref{lem-tan-pseudo-energy-1}.
\end{proof}

In order to get the energy estimates of the tangential derivatives of the velocity, we will deal with the remainder terms on the right hand side in \eqref{tan-linear-pseudo-0}. For this, let us first estimate the terms with respect to $g$ in \eqref{tan-linear-pseudo-0}.

\begin{lem}\label{lem-est-g-1}
Under the assumption in Lemma \ref{lem-tan-pseudo-energy-1}, if $(\lambda,\,\sigma_0) \in (0, 1)$ satisfies $1-\lambda< \sigma_0\leq 1-\frac{1}{2}\lambda$, and $E_3(t) \leq 1$, then
\begin{equation}\label{est-g-2}
\begin{split}
&\|\dot{\Lambda}_h^{\sigma_0} g\|_{L^2}\lesssim E_3^{\frac{1}{2}}  \dot{\mathcal{D}}_{3}^{\frac{1}{2}},\quad \|\dot{\Lambda}_h^{-\lambda}g_{jj}\|_{L^2}
  +\|\dot{\Lambda}_h^{-\lambda}\widetilde{g}_{jj}\|_{L^2}\lesssim E_3^{\frac{1}{2}} \dot{\mathcal{D}}_3^{\frac{1}{2}} \quad(\text{with}\quad j=1, 2, 3),\\
   &\|\dot{\Lambda}_h^{-\lambda} g\|_{L^2}\lesssim E_3^{\frac{1}{2}} \|\dot{\Lambda}_h^{2-\sigma_0-\lambda}\xi^1\|_{L^2(\Sigma_0)} +E_3^{\frac{1}{2}}\, \dot{\mathcal{D}}_3^{\frac{1}{2}},\\
&\|g\|_{L^2}\lesssim E_3^{\frac{1}{2}} \|\dot{\Lambda}_h^{2-\sigma_0}\xi^1\|_{L^2(\Sigma_0)} +E_3^{\frac{1}{2}}\, \dot{\mathcal{D}}_3^{\frac{1}{2}},\,\|\partial_h^{N-1} \,g\|_{L^2}\lesssim \, E_N^{\frac{1}{2}} \dot{\mathcal{D}}_3^{\frac{1}{2}}+ E_3^{\frac{1}{2}} \dot{\mathcal{D}}_N^{\frac{1}{2}} \,(\text{for} \,\,N\geq 3).
\end{split}
\end{equation}
\end{lem}

\begin{proof}
We first bound the term $g_{jj}$ with $j=1,2,3$ of $g$ in \eqref{est-g-2}. For this, we split the terms of $g_{jj}$ in \eqref{def-g-linear-1} into two types: the one is of the second-order derivative terms of $v$ and the other is of the first-order derivative terms of $v$.

 For the first one, we use the product law \eqref{product-law-1} to get
\begin{equation*}
\begin{split}
 &\|\dot{\Lambda}_h^{\sigma_0}\,(\mathcal{A}_\alpha^1\mathcal{A}_1^1\,\partial_1^2v^{\alpha})\|_{L^2}\lesssim \|\dot{\Lambda}_h^{(1+\sigma_0)/2}(\mathcal{A}_\alpha^1\mathcal{A}_1^1)\|_{L^\infty_{x_1}L^2_h}
 \|\dot{\Lambda}_h^{(1+\sigma_0)/2}\,\partial_1^2v^{\alpha}\|_{L^2}\\
 &\lesssim (\|\dot{\Lambda}_h^{(1+\sigma_0)/2}\mathcal{A}_\alpha^1\|_{L^\infty_{x_1}L^2_h}
 +\|\dot{\Lambda}_h^{(3+\sigma_0)/4}\mathcal{A}_\alpha^1\|_{L^\infty_{x_1}L^2_h}
 \|\dot{\Lambda}_h^{(3+\sigma_0)/4}(\mathcal{A}_1^1-1)\|_{L^\infty_{x_1}L^2_h})
 \|\dot{\Lambda}_h^{(1+\sigma_0)/2}\,\partial_1^2v^{\alpha}\|_{L^2},
\end{split}
\end{equation*}
which, from Lemma \ref{lem-est-aij-1}, follows
\begin{equation*}
\begin{split}
 &\|\dot{\Lambda}_h^{\sigma_0}\,(\mathcal{A}_\alpha^1\mathcal{A}_1^1\,\partial_1^2v^{\alpha})\|_{L^2} \lesssim \|\dot{\Lambda}_h^{(1+\sigma_0)/2}\,\partial_1^2v^{\alpha}\|_{L^2}
 \bigg(\|\dot{\Lambda}_h^{(1+\sigma_0)/2}\mathcal{A}_\alpha^1\|_{H^1}
 \\
 &\qquad\qquad\qquad\qquad\qquad\qquad+\|\dot{\Lambda}_h^{(3+\sigma_0)/4}\mathcal{A}_\alpha^1\|_{H^1}
 \|\dot{\Lambda}_h^{(3+\sigma_0)/4}(\mathcal{A}_1^1-1)\|_{H^1}\bigg)
\lesssim E_3^{\frac{1}{2}}\dot{\mathcal{D}}_{3}^{\frac{1}{2}}.
\end{split}
\end{equation*}
The same estimate holds for any second-order derivative terms of $v$ in $g_{jj}$ with $j=1,2,3$.

While for the first-order derivative term of $v$ in $g_{jj}$ with $j=1,2,3$, we employ the product law \eqref{product-law-1} once again to yield
\begin{equation*}
\begin{split}
 &\|\dot{\Lambda}_h^{\sigma_0}\,(\mathcal{A}_i^1\partial_1 \mathcal{A}_1^1\,\partial_1v^{i})\|_{L^2}\lesssim \|\dot{\Lambda}_h^{(1+\sigma_0)/2}(\mathcal{A}_i^1\partial_1 \mathcal{A}_1^1)\|_{L^2}\,\|\dot{\Lambda}_h^{(1+\sigma_0)/2}\partial_1 v\|_{L^\infty_{x_1}L^2_h}\\
 &\lesssim (\|\dot{\Lambda}_h^{(3+\sigma_0)/4}(\mathcal{A}_\alpha^1,\,\mathcal{A}_1^1-1)\|_{L^\infty_{x_1}L^2_h}
 (\|\dot{\Lambda}_h^{(1+\sigma_0)/2}\partial_1 \mathcal{A}_1^1\|_{L^2}+\|\dot{\Lambda}_h^{(3+\sigma_0)/4}\partial_1 \mathcal{A}_1^1\|_{L^2})\\
 &\qquad\qquad\qquad\qquad\qquad\qquad\qquad\times\,(\|\dot{\Lambda}_h^{\sigma_0}\,\partial_1 v\|_{L^2}+\|\dot{\Lambda}_h^{\sigma_0}\,\partial_1^2 v\|_{L^2}),
\end{split}
\end{equation*}
from which, applying Lemma \ref{lem-est-aij-1} implies
\begin{equation*}
\begin{split}
 &\|\dot{\Lambda}_h^{\sigma_0}\,(\mathcal{A}_i^1\partial_1 \mathcal{A}_1^1\,\partial_1v^{i})\|_{L^2}\lesssim E_3^{\frac{1}{2}} \dot{\mathcal{D}}_{3}^{\frac{1}{2}}.
\end{split}
\end{equation*}
This bound also holds true for any first-order derivative term of $v$ in $g_{jj}$ and $\widetilde{g}_{jj}$ with $j=1,2,3$.

Therefore, we conclude that, for any $j=1,2,3$,
\begin{equation}\label{est-g-2-1}
\begin{split}
 &\|\dot{\Lambda}_h^{\sigma_0}\,g_{jj}\|_{L^2}+\|\dot{\Lambda}_h^{\sigma_0}\,\widetilde{g}_{jj}\|_{L^2}\lesssim E_3^{\frac{1}{2}} \dot{\mathcal{D}}_{3}^{\frac{1}{2}}.
\end{split}
\end{equation}
For the pressure term in $g$, making use of the product law \eqref{product-law-1} again yields
\begin{equation}\label{est-g-2-2}
\begin{split}
 &\|\dot{\Lambda}_h^{\sigma_0}\,(\mathcal{A}_2^1\partial_1q)\|_{L^2}\lesssim \|\dot{\Lambda}_h^{(1+\sigma_0)/2}\, \partial_1q \|_{L^2}
 \|\dot{\Lambda}_h^{(1+\sigma_0)/2}\mathcal{A}_2^1 \|_{L^\infty_{x_1}L^2_h}\lesssim E_3^{\frac{1}{2}} \dot{\mathcal{D}}_{3}^{\frac{1}{2}},
\end{split}
\end{equation}
where we used the estimate $\|\dot{\Lambda}_h^{(1+\sigma_0)/2}\mathcal{A}_2^1 \|_{L^\infty_{x_1}L^2_h}\lesssim \|\dot{\Lambda}_h^{(1+\sigma_0)/2}\mathcal{A}_2^1 \|_{H^1}\lesssim E_3^{\frac{1}{2}}$ in Lemma \ref{lem-est-aij-1}.

In the same manner it is easy to get
\begin{equation}\label{est-g-2-3}
\begin{split}
 &\|\dot{\Lambda}_h^{\sigma_0}(\mathcal{A}_1^h\partial_hq)\|_{L^2}
 +\|\dot{\Lambda}_h^{\sigma_0}((\mathcal{A}_1^1-1)\partial_1q)\|_{L^2} +\|\dot{\Lambda}_h^{\sigma_0}((\mathcal{A}_2^2-1)\partial_2q)\|_{L^2}\\
 &+\|\dot{\Lambda}_h^{\sigma_0}( \mathcal{A}_2^3 \partial_3q)\|_{L^2}+\|\dot{\Lambda}_h^{\sigma_0}(\mathcal{A}_3^1\partial_1q)\|_{L^2}
 +\|\dot{\Lambda}_h^{\sigma_0}((\mathcal{A}_3^3-1)\partial_3q)\|_{L^2}\lesssim E_3^{\frac{1}{2}} \dot{\mathcal{D}}_{3}^{\frac{1}{2}}.
\end{split}
\end{equation}

Combing \eqref{est-g-2-1}-\eqref{est-g-2-3} yields the first inequality in \eqref{est-g-2}.

Let's now estimate $\|\dot{\Lambda}_h^{-\lambda}g\|_{L^2}$ including the $L^2$ norms of the negative horizontal derivative of $g_{jj}$, $\widetilde{g}_{jj}$ and the pressure terms.

Firstly, the product law \eqref{product-law-1} ensures
\begin{equation}\label{est-g-2-4}
\begin{split}
 &\|\dot{\Lambda}_h^{-\lambda}\,(\mathcal{A}_\alpha^1\mathcal{A}_1^1\,\partial_1^2v^{\alpha})\|_{L^2}\lesssim \|\dot{\Lambda}_h^{\sigma_0-1}(\mathcal{A}_\alpha^1\mathcal{A}_1^1)\|_{L^\infty_{x_1}L^2_h}
 \|\dot{\Lambda}_h^{2-\sigma_0-\lambda}\,\partial_1^2v^{\alpha}\|_{L^2}.
\end{split}
\end{equation}
While from Lemma \ref{lem-est-aij-1} and $E_3(t) \leq 1$, there holds
\begin{equation}\label{est-g-2-5}
\begin{split}
 & \|\dot{\Lambda}_h^{\sigma_0-1}(\mathcal{A}_\alpha^1\mathcal{A}_1^1)\|_{L^\infty_{x_1}L^2_h}\lesssim \|\dot{\Lambda}_h^{\sigma_0-1}\mathcal{A}_\alpha^1\|_{L^\infty_{x_1}L^2_h}
 +\|\dot{\Lambda}_h^{\sigma_0}\mathcal{A}_\alpha^1\|_{L^\infty_{x_1}L^2_h}
 \|\mathcal{A}_1^1-1\|_{L^\infty_{x_1}L^2_h}\\
 &\lesssim \|\dot{\Lambda}_h^{\sigma_0-1}\mathcal{A}_\alpha^1\|_{H^1}
 +\|\dot{\Lambda}_h^{\sigma_0}\mathcal{A}_\alpha^1\|_{H^1}
 \|\mathcal{A}_1^1-1\|_{H^1}\lesssim E_3^{\frac{1}{2}}.
\end{split}
\end{equation}
Substituting \eqref{est-g-2-5} into \eqref{est-g-2-4} yields
\begin{equation}\label{est-g-2-6}
\begin{split}
 &\|\dot{\Lambda}_h^{-\lambda}\,(\mathcal{A}_\alpha^1\mathcal{A}_1^1\,\partial_1^2v^{\alpha})\|_{L^2}\lesssim E_3^{\frac{1}{2}}
 \|\dot{\Lambda}_h^{2-\sigma_0-\lambda}\,\partial_1^2v^{\alpha}\|_{L^2}\lesssim  E_3^{\frac{1}{2}} \dot{\mathcal{D}}_{3}^{\frac{1}{2}},
\end{split}
\end{equation}
where the fact $2-\sigma_0-\lambda \geq \sigma_0$ (since $\sigma_0\leq 1-\frac{1}{2}\lambda$) has been used in the last inequality.

For $\mathcal{A}_2^1 \mathcal{A}_1^3\partial_3\partial_1v^{2}$, thanks to the fact $2-\sigma_0-\lambda \geq \sigma_0$ again, we have
\begin{equation}\label{est-g-2-6a}
\begin{split}
 &\|\dot{\Lambda}_h^{-\lambda}\,(\mathcal{A}_2^1 \mathcal{A}_1^3\partial_3\partial_1v^{2})\|_{L^2}\lesssim \|\dot{\Lambda}_h^{2-\sigma_0-\lambda}(\mathcal{A}_2^1 \mathcal{A}_1^3)\|_{L^2}
 \|\dot{\Lambda}_h^{\sigma_0-1}\,\partial_3\partial_1v^{2}\|_{L^\infty_{x_1}L^2_h}\lesssim  E_3^{\frac{1}{2}} \dot{\mathcal{D}}_{3}^{\frac{1}{2}}.
\end{split}
\end{equation}

This estimate also holds for any second-order derivative terms of $v$ in $g_{jj}$ and $\widetilde{g}_{jj}$ with $j=1,2,3$.

For the first-order derivative term of $v$ in $g_{jj}$ and $\widetilde{g}_{jj}$ with $j=1,2,3$, applying the product law \eqref{product-law-1}, and then using Lemma \ref{lem-est-aij-1} ensures
\begin{equation}\label{est-g-2-7}
\begin{split}
 &\|\dot{\Lambda}_h^{-\lambda}\,(\mathcal{A}_i^1\partial_1 \mathcal{A}_1^1\,\partial_1v^{i})\|_{L^2}\lesssim \|\dot{\Lambda}_h^{\sigma_0-1}(\mathcal{A}_i^1\partial_1 \mathcal{A}_1^1)\|_{L^2}\,\|\dot{\Lambda}_h^{2-\sigma_0-\lambda}\partial_1 v\|_{L^\infty_{x_1}L^2_h}\\
 &\lesssim \|\dot{\Lambda}_h^{\sigma_0-1}(\mathcal{A}_i^1\partial_1 \mathcal{A}_1^1)\|_{L^2}\,\|\dot{\Lambda}_h^{2-\sigma_0-\lambda}\partial_1 v\|_{H^1}\lesssim\,E_3^{\frac{1}{2}} \dot{\mathcal{D}}_{3}^{\frac{1}{2}}.
\end{split}
\end{equation}

This estimate remains valid for any other first-order derivative terms of $v$ in $g_{jj}$ and $\widetilde{g}_{jj}$ with $j=1,2,3$.

Therefore, we obtain that for any $j=1,2,3$
\begin{equation*}\label{est-g-2-8}
\begin{split}
 &\|\dot{\Lambda}_h^{-\lambda}\,g_{jj}\|_{L^2}+\|\dot{\Lambda}_h^{-\lambda}\,\widetilde{g}_{jj}\|_{L^2}\lesssim E_3^{\frac{1}{2}} \dot{\mathcal{D}}_{3}^{\frac{1}{2}}.
\end{split}
\end{equation*}
For the pressure terms in $g$, from \eqref{product-law-1}, it can be obtained that
\begin{equation*}\label{est-g-2-9}
\begin{split}
 &\|\dot{\Lambda}_h^{-\lambda}(\mathcal{A}_2^1\partial_1q)\|_{L^2}\lesssim \|\dot{\Lambda}_h^{\sigma_0-1}\mathcal{A}_2^1\|_{L^\infty_{x_1}L^2_h}
 \|\dot{\Lambda}_h^{2-\sigma_0-\lambda} \partial_1q \|_{L^2}\lesssim E_3^{\frac{1}{2}}\,
 \|\dot{\Lambda}_h^{2-\sigma_0-\lambda} \partial_1q \|_{L^2}\lesssim E_3^{\frac{1}{2}} \dot{\mathcal{D}}_{3}^{\frac{1}{2}}.\\
\end{split}
\end{equation*}
While for the term about $\partial_hq$, we have
\begin{equation*}\label{est-g-2-10}
\begin{split}
 &\|\dot{\Lambda}_h^{-\lambda}( \mathcal{A}_1^h \partial_hq)\|_{L^2}\lesssim \|\dot{\Lambda}_h^{\sigma_0}\mathcal{A}_1^2\|_{L^\infty_{x_1}L^2_h}
 \|\dot{\Lambda}_h^{1-\sigma_0-\lambda} \partial_hq \|_{L^2}\lesssim E_3^{\frac{1}{2}}\,
 \|\dot{\Lambda}_h^{1-\sigma_0-\lambda} \partial_hq \|_{L^2}.\\
\end{split}
\end{equation*}
 In order to control $\|\dot{\Lambda}_h^{1-\sigma_0-\lambda} \partial_hq \|_{L^2}$, we employ the Poincare's inequality and the boundary condition $q=\xi^1-2\nu\,\nabla_h\cdot v^h+\nu\, \mathcal{B}_{9, i}^{\alpha} \partial_{\alpha}v^{i}$ on $\Sigma_0$ to get
\begin{equation*}
\begin{split}
 &\|\dot{\Lambda}_h^{1-\sigma_0-\lambda} \partial_hq \|_{L^2(\Omega)}\lesssim \|\dot{\Lambda}_h^{1-\sigma_0-\lambda}\partial_h( \xi^1-2\nu\,\nabla_h\cdot v^h+\nu\, \mathcal{B}_{9, i}^{\alpha} \partial_{\alpha}v^{i})\|_{L^2(\Sigma_0)}+\|\dot{\Lambda}_h^{1-\sigma_0-\lambda} \partial_h\partial_1q \|_{L^2(\Omega)}\\
 &\lesssim \|\dot{\Lambda}_h^{2-\sigma_0-\lambda}\xi^1\|_{L^2(\Sigma_0)}+ \|\dot{\Lambda}_h^{1-\sigma_0-\lambda}\partial_h(-2\nu\,\nabla_h\cdot v^h+\nu\, \mathcal{B}_{9, i}^{\alpha} \partial_{\alpha}v^{i})\|_{H^1(\Omega)}+\|\dot{\Lambda}_h^{2-\sigma_0-\lambda}\partial_1q \|_{L^2(\Omega)}\\
 &\lesssim \|\dot{\Lambda}_h^{2-\sigma_0-\lambda}\xi^1\|_{L^2(\Sigma_0)}+ \dot{\mathcal{D}}_{3}^{\frac{1}{2}},
\end{split}
\end{equation*}
where we have used the assumption $2-\sigma_0-\lambda \geq \sigma_0$ in the last inequality.
It thus follows
\begin{equation*}\label{est-g-2-11}
\begin{split}
 &\|\dot{\Lambda}_h^{-\lambda}( \mathcal{A}_1^h \partial_hq)\|_{L^2}\lesssim E_3^{\frac{1}{2}}\,
(\|\dot{\Lambda}_h^{2-\sigma_0-\lambda}\xi^1\|_{L^2(\Sigma_0)}+ \dot{\mathcal{D}}_{3}^{\frac{1}{2}}).\\
\end{split}
\end{equation*}
Repeating the above argument to the other pressure terms in $g$ ensures that
\begin{equation*}
\begin{split}
 &\|\dot{\Lambda}_h^{-\lambda}(\mathcal{A}_2^1\partial_1q)\|_{L^2}
 +\|\dot{\Lambda}_h^{-\lambda}(\mathcal{A}_1^h\partial_hq)\|_{L^2}
 +\|\dot{\Lambda}_h^{-\lambda}((\mathcal{A}_1^1-1)\partial_1q)\|_{L^2} +\|\dot{\Lambda}_h^{-\lambda}((\mathcal{A}_2^2-1)\partial_2q)\|_{L^2}\\
 &+\|\dot{\Lambda}_h^{-\lambda}( \mathcal{A}_2^3 \partial_3q)\|_{L^2}+\|\dot{\Lambda}_h^{-\lambda}(\mathcal{A}_3^1\partial_1q)\|_{L^2}
 +\|\dot{\Lambda}_h^{-\lambda}((\mathcal{A}_3^3-1)\partial_3q)\|_{L^2}\\
 &\lesssim E_3^{\frac{1}{2}} \|\dot{\Lambda}_h^{2-\sigma_0-\lambda}\xi^1\|_{L^2(\Sigma_0)} +E_3^{\frac{1}{2}} \dot{\mathcal{D}}_3^{\frac{1}{2}}.
\end{split}
\end{equation*}
Consequently, we conclude that
\begin{equation*}\label{est-g-2-12}
\begin{split}
&\|\dot{\Lambda}_h^{-\lambda} g\|_{L^2}\lesssim E_3^{\frac{1}{2}} \|\dot{\Lambda}_h^{2-\sigma_0-\lambda}\xi^1\|_{L^2(\Sigma_0)} +E_3^{\frac{1}{2}} \dot{\mathcal{D}}_3^{\frac{1}{2}},
\end{split}
\end{equation*}
this is, the third inequality in \eqref{est-g-2} holds true.

The forth inequality can be proved by repeating the same line in the proof of the third inequality in \eqref{est-g-2}.

Finally, let us now turn to derive the last inequality in \eqref{est-g-2}.

For this, we first apply the product law \eqref{product-law-1} to get
\begin{equation*}
\begin{split}
 &\|\partial_h^{N-1} \,(\mathcal{A}_\alpha^1\mathcal{A}_1^1\,\partial_1^2v^{\alpha})\|_{L^2}\lesssim \|\partial_h^{N-1}(\mathcal{A}_\alpha^1\mathcal{A}_1^1)\|_{L^\infty_{x_1}(L^2_h)}\,
 \|\partial_1^2v^{\alpha}\|_{L^2_{x_1}L^\infty_h}+ \|\partial_h^{N-1}\partial_1^2v^{\alpha}\|_{L^2}\,
 \|\mathcal{A}_\alpha^1\mathcal{A}_1^1\|_{L^\infty}.
\end{split}
\end{equation*}
It is easy to check $\|\partial_h^{N-1} (\mathcal{A}_\alpha^1\mathcal{A}_1^1)\|_{L^\infty_{x_1}(L^2_h)}\lesssim\|\partial_h^{N-1} (\mathcal{A}_\alpha^1\mathcal{A}_1^1)\|_{H^1}\lesssim E_{N}^{\frac{1}{2}}$
and
$\|\mathcal{A}_\alpha^1\mathcal{A}_1^1\|_{L^\infty}
  \lesssim\|\dot{\Lambda}_h^{\sigma_0}\Lambda_h(\mathcal{A}_\alpha^1\mathcal{A}_1^1)\|_{H^1}\lesssim  E_3^{\frac{1}{2}}$, and by using the Sobolev embedding theorem $\|\partial_1^2v^{\alpha}\|_{L^2_{x_1}L^\infty_h} \lesssim \|\dot{\Lambda}_h^{\sigma_0}\Lambda_h\partial_1^2v^{\alpha}\|_{L^2} \lesssim \dot{\mathcal{D}}_3^{\frac{1}{2}}$, we obtain
\begin{equation*}
\begin{split}
 &\|\partial_h^{N-1} \,(\mathcal{A}_\alpha^1\mathcal{A}_1^1\,\partial_1^2v^{\alpha})\|_{L^2}\lesssim E_{N}^{\frac{1}{2}}\dot{\mathcal{D}}_3^{\frac{1}{2}}+ E_3^{\frac{1}{2}} \dot{\mathcal{D}}_{N}^{\frac{1}{2}}.
\end{split}
\end{equation*}
The same conclusion can be drawn for any other terms of $v$ in $g_{jj}$ with $j=1,2,3$.

Therefore, we deduce that for any $j=1,2,3$
\begin{equation*}
\begin{split}
 &\|\partial_h^{N-1} \,g_{jj}\|_{L^2} \lesssim  E_N^{\frac{1}{2}} \dot{\mathcal{D}}_3^{\frac{1}{2}}+ E_3^{\frac{1}{2}} \dot{\mathcal{D}}_N^{\frac{1}{2}}.
\end{split}
\end{equation*}
Similarly, thanks to the product law \eqref{product-law-1} again, it can be established that
\begin{equation*}
\begin{split}
 \|\partial_h^{N-1}\,(\mathcal{A}_2^1\partial_1q)\|_{L^2}&\lesssim \|\partial_h^{N-1}\mathcal{A}_2^1\|_{L^\infty_{x_1}L^2_h}\|\partial_1q\|_{L^2_{x_1}(L^\infty_h)}
 +\|\partial_h^{N-1}\,\partial_1q\|_{L^2}\|\mathcal{A}_2^1\|_{L^\infty}\\
 &\lesssim  E_N^{\frac{1}{2}} \dot{\mathcal{D}}_3^{\frac{1}{2}}+ E_3^{\frac{1}{2}} \dot{\mathcal{D}}_N^{\frac{1}{2}},
\end{split}
\end{equation*}
and also
\begin{equation*}
\begin{split}
 &\|\partial_h^{N-1} (\mathcal{A}_1^h\partial_hq)\|_{L^2}
 +\|\partial_h^{N-1} ((\mathcal{A}_1^1-1)\partial_1q)\|_{L^2} +\|\partial_h^{N-1} ((\mathcal{A}_2^2-1)\partial_2q)\|_{L^2}\\
 &+\|\partial_h^{N-1} ( \mathcal{A}_2^3 \partial_3q)\|_{L^2}+\|\partial_h^{N-1} (\mathcal{A}_3^1\partial_1q)\|_{L^2}
 +\|\partial_h^{N-1} ((\mathcal{A}_3^3-1)\partial_3q)\|_{L^2}\\
 &\lesssim E_N^{\frac{1}{2}} \dot{\mathcal{D}}_3^{\frac{1}{2}}+ E_3^{\frac{1}{2}} \dot{\mathcal{D}}_N^{\frac{1}{2}}.
\end{split}
\end{equation*}
We thus prove that $\|\partial_h^{N-1} \,g\|_{L^2}\lesssim
 E_N^{\frac{1}{2}} \dot{\mathcal{D}}_3^{\frac{1}{2}}+ E_3^{\frac{1}{2}} \dot{\mathcal{D}}_N^{\frac{1}{2}}$.
The proof of the lemma is therefore accomplished.
\end{proof}
With Lemmas \ref{lem-tan-pseudo-energy-1} and \ref{lem-est-g-1} in hand, we will deal with the estimates $\|\dot{\Lambda}_h^{s}v\|_{L^\infty_tL^2(\Omega)}
+\|\dot{\Lambda}_h^{s}\xi^1\|_{L^\infty_tL^2(\Sigma_0)}$ with $s=\sigma_0,\,N,\,-\lambda$.

\subsubsection{Estimate of $\|\dot{\Lambda}_h^{\sigma_0}v\|_{L^\infty_tL^2(\Omega)}+\|\dot{\Lambda}_h^{\sigma_0}\xi^1\|_{L^\infty_tL^2(\Sigma_0)}$}
\begin{lem}\label{lem-tan-decay-sigma0-1}
Under the assumption of Lemma \ref{lem-tan-pseudo-energy-1}, if $E_3(t) \leq 1$ for all the existence times $t$, then there holds
\begin{equation}\label{linear-decay-sigma0-1}
\begin{split}
\frac{d}{dt}(\|\dot{\Lambda}_h^{\sigma_0}v\|_{L^2(\Omega)}^2
+\|\dot{\Lambda}_h^{\sigma_0}\xi^1\|_{L^2(\Sigma_0)}^2) + \nu \| \mathbb{D}(\dot{\Lambda}_h^{\sigma_0}v)\|_{L^2(\Omega)}^2 \lesssim E_3^{\frac{1}{2}}\dot{\mathcal{D}}_{3}.
\end{split}
\end{equation}
\end{lem}
\begin{proof}
Taking the horizontal derivative operator $\mathcal{P}(\partial_h)=\dot{\Lambda}_h^{\sigma_0}$ in \eqref{tan-linear-pseudo-0} yields
\begin{equation}\label{linear-v-sigma0-0}
\begin{split}
 & \frac{1}{2}\frac{d}{dt}\bigg(\|\dot{\Lambda}_h^{\sigma_0}v\|_{L^2(\Omega)}^2
+\|\dot{\Lambda}_h^{\sigma_0}\xi^1\|_{L^2(\Sigma_0)}^2 \bigg) + \frac{\nu }{2}\| \mathbb{D}(\dot{\Lambda}_h^{\sigma_0}v)\|_{L^2(\Omega)}^2=\sum_{j=1}^3\mathfrak{K}_j,
\end{split}
\end{equation}
where the remainder terms $\mathfrak{K}_j$ with $j=1, 2, 3$ are defined in \eqref{tan-linear-pseudo-0a} in which we take $\mathcal{P}(\partial_h)=\dot{\Lambda}_h^{\sigma_0}$.

For the first integral $\mathfrak{K}_1$, we first observe
\begin{equation*}\label{nonlin-v-sigma0-1}
\begin{split}
|\int_{\Sigma_{0}}\dot{\Lambda}_h^{\sigma_0} (\mathcal{B}_{9, i}^{\alpha} \partial_{\alpha}v^{i})\, \dot{\Lambda}_h^{\sigma_0}v^1\, dS_0|&\lesssim \|\dot{\Lambda}_h^{\sigma_0}(\mathcal{B}_{9, i}^{\alpha} \partial_{\alpha}v^{i}\|_{L^2(\Sigma_0)}) \|\dot{\Lambda}_h^{\sigma_0}v^1\|_{L^2(\Sigma_0)}\\
  &\lesssim \|\dot{\Lambda}_h^{\sigma_0}(\mathcal{B}_{9, i}^{\alpha} \partial_{\alpha}v^{i}\|_{H^1(\Omega)}  \|\dot{\Lambda}_h^{\sigma_0}v^1\|_{H^1(\Omega)},
\end{split}
\end{equation*}
which follows from Lemma \ref{lem-est-B-Bv-1} that $
|\int_{\Sigma_{0}}\dot{\Lambda}_h^{\sigma_0} (\mathcal{B}_{6, i}^{\alpha} \partial_{\alpha}v^{i})\, \dot{\Lambda}_h^{\sigma_0}v^1\, dS_0|\lesssim E_3^{\frac{1}{2}}  \dot{\mathcal{D}}_{3}.$
Along the same line, it can be obtained that
\begin{equation*}\label{nonlin-v-sigma0-2}
\begin{split}
 & |\int_{\Sigma_{0}}\dot{\Lambda}_h^{\sigma_0}(\mathcal{B}_{9, i}^{\alpha} \partial_{\alpha}v^{i})\dot{\Lambda}_h^{\sigma_0}v^1\, dS_0|+|\int_{\Sigma_{0}}\dot{\Lambda}_h^{\sigma_0}(\mathcal{B}_{7, i}^{\alpha}\partial_{\alpha}v^{i}) \dot{\Lambda}_h^{\sigma_0}v^2\, dS_0|\\
&\qquad+|\int_{\Sigma_{0}}\dot{\Lambda}_h^{\sigma_0}(\mathcal{B}_{8, i}^{\alpha}\partial_{\alpha}v^{i})\, \dot{\Lambda}_h^{\sigma_0}v^3\, dS_0|\lesssim E_3^{\frac{1}{2}}  \dot{\mathcal{D}}_{3}.
\end{split}
\end{equation*}
We thus get
\begin{equation}\label{nonlin-v-sigma0-3}
\begin{split}
 & |\mathfrak{K}_1|\lesssim E_3^{\frac{1}{2}}  \dot{\mathcal{D}}_{3}.
\end{split}
\end{equation}
On the other hand, due to the product law \eqref{product-law-1}, one can see that
\begin{equation*}\label{nonlin-v-sigma0-4}
\begin{split}
&|\mathfrak{K}_2|\lesssim\|\dot{\Lambda}_h^{\sigma_0+1}q\|_{L^2}\,
(\|\dot{\Lambda}_h^{\sigma_0-1}(\widetilde{a_{\alpha\,1}}\partial_1v^\alpha)\|_{L^2}
+\|\dot{\Lambda}_h^{\sigma_0-1}(\mathcal{B}_{5,i}^{\alpha}\partial_{\alpha}v^{i})\|_{L^2})\\
&\lesssim\|\dot{\Lambda}_h^{\sigma_0}\partial_h\,q\|_{L^2}\,
\bigg(\|\widetilde{a_{\alpha\,1}}\|_{L^\infty_{x_1}L^2_h}
\|\dot{\Lambda}_h^{\sigma_0}\partial_1v^h\|_{L^2}
+\|\dot{\Lambda}_h^{\sigma_0}\mathcal{B}_{5, i}^{\alpha}\|_{L^2}\|\partial_{\alpha}v^{i}\|_{L^\infty_{x_1}L^2_h}\bigg),
\end{split}
\end{equation*}
which follows from \eqref{est-aij-J-1} and \eqref{est-B-1} that
\begin{equation}\label{nonlin-v-sigma0-5}
\begin{split}
&|\mathfrak{K}_2|\lesssim\|\dot{\Lambda}_h^{\sigma_0}\partial_h\,q\|_{L^2}\,
\bigg(\|\widetilde{a_{\alpha\,1}}\|_{H^1}
\|\dot{\Lambda}_h^{\sigma_0}\partial_1v^h\|_{L^2}
+\|\dot{\Lambda}_h^{\sigma_0}\mathcal{B}_{5, i}^{\alpha}\|_{L^2}\|\partial_{\alpha}v^{i}\|_{H^1}\bigg) \lesssim E_3^{\frac{1}{2}}  \dot{\mathcal{D}}_{3}.
\end{split}
\end{equation}
For the term $\mathfrak{K}_3=\int_{ \Omega}\dot{\Lambda}_h^{\sigma_0} g\cdot \dot{\Lambda}_h^{\sigma_0} v\,dx$, from \eqref{est-g-2}, it follows that
\begin{equation}\label{nonlin-v-sigma0-6}
\begin{split}
&|\mathfrak{K}_3|\lesssim \|\dot{\Lambda}_h^{\sigma_0} g\|_{L^2}\|\dot{\Lambda}_h^{\sigma_0} v\|_{L^2}\lesssim \|\dot{\Lambda}_h^{\sigma_0} g\|_{L^2}\|\dot{\Lambda}_h^{\sigma_0} v\|_{H^1}\lesssim E_3^{\frac{1}{2}}\dot{\mathcal{D}}_{3}.
\end{split}
\end{equation}
Substituting \eqref{nonlin-v-sigma0-3}-\eqref{nonlin-v-sigma0-6} into \eqref{linear-v-sigma0-0} leads to \eqref{linear-decay-sigma0-1}, which furnishes the proof of Lemma \ref{lem-tan-decay-sigma0-1}.
\end{proof}

\subsubsection{Estimate of the tangential derivatives $\|\partial_h^{N}v\|_{L^\infty_t(L^2)}$}

\begin{lem}\label{lem-tan-decay-N-1}
Let $N\geq 3$, under the assumption of Lemma \ref{lem-tan-pseudo-energy-1}, if $E_3(t) \leq 1$ for all the existence times $t$, there holds
\begin{equation}\label{linear-decay-N-1}
\begin{split}
\frac{d}{dt}\bigg(\|\partial_h^{N}v\|_{L^2(\Omega)}^2
+\|\partial_h^{N}\xi^1\|_{L^2(\Sigma_0)}^2 \bigg) + \nu \| \mathbb{D}(\partial_h^{N}v)\|_{L^2(\Omega)}^2 \lesssim \dot{\mathcal{D}}_N^{\frac{1}{2}}(E_N^{\frac{1}{2}} \dot{\mathcal{D}}_3^{\frac{1}{2}}+  E_3^{\frac{1}{2}} \dot{\mathcal{D}}_N^{\frac{1}{2}}).
\end{split}
\end{equation}
\end{lem}
\begin{proof}
Similar to the proof of Lemma \ref{lem-tan-decay-sigma0-1}, taking the operator $\mathcal{P}(\partial_h)=\partial_h^{N}$ in \eqref{tan-linear-pseudo-0}, it can be easily obtained that
\begin{equation}\label{linear-v-Ntan-0}
\begin{split}
 &\frac{1}{2}\frac{d}{dt}(\|\partial_h^{N}v\|_{L^2(\Omega)}^2+ \|\partial_h^{N}\xi^1\|_{L^2(\Sigma_0)}^2) +\frac{\nu}{2}\| \mathbb{D}(\partial_h^{N}v)\|_{L^2(\Omega)}^2=\sum_{j=1}^3\mathfrak{K}_j,
\end{split}
\end{equation}
where $\mathfrak{K}_j$ with $j=1, 2, 3$ are defined in \eqref{tan-linear-pseudo-0a} in which we take $\mathcal{P}(\partial_h)=\partial_h^{N}$.

For the boundary integral $\mathfrak{K}_1$, making use of the trace theorem ensures
\begin{equation*}\label{linear-v-Ntan-2}
\begin{split}
|\int_{\Sigma_{0}}\partial_h^{N}(\mathcal{B}_{6, i}^{\alpha} \partial_{\alpha}v^{i})\, \partial_h^{N}v^1\, dS_0|&\lesssim \|\dot{\Lambda}_h^{N-\frac{1}{2}}(\mathcal{B}_{6, i}^{\alpha} \partial_{\alpha}v^{i})\|_{L^2(\Sigma_0)}\, \|\dot{\Lambda}_h^{N+\frac{1}{2}}v^1\|_{L^2(\Sigma_0)}\\
&\lesssim \|\dot{\Lambda}_h^{N-1}(\mathcal{B}_{6, i}^{\alpha} \partial_{\alpha}v^{i})\|_{H^1(\Omega)}\, \|\dot{\Lambda}_h^{N}v^1\|_{H^1(\Omega)}.
\end{split}
\end{equation*}
We thus obtain from Lemma \ref{lem-est-B-Bv-1} that
\begin{equation}\label{linear-v-Ntan-4}
\begin{split}
&|\int_{\Sigma_{0}}\partial_h^{N}(\mathcal{B}_{6, i}^{\alpha} \partial_{\alpha}v^{i})\, \partial_h^{N}v^1\, dS_0|\lesssim \bigg( E_{N}^{\frac{1}{2}}\dot{\mathcal{D}}_3^{\frac{1}{2}}
+E_3^{\frac{1}{2}} \dot{\mathcal{D}}_N^{\frac{1}{2}}\bigg)\,\dot{\mathcal{D}}_N^{\frac{1}{2}}.
\end{split}
\end{equation}
It can be similarly proved that
\begin{equation*}\label{linear-v-Ntan-5}
\begin{split}
 & |\int_{\Sigma_{0}}\partial_h^{N}(\mathcal{B}_{9, i}^{\alpha})\partial_{\alpha}v^{i})\partial_h^{N}v^1\, dS_0|+|\int_{\Sigma_{0}}\partial_h^{N}(\mathcal{B}_{7, i}^{\alpha}\partial_{\alpha}v^{i}) \partial_h^{N}v^2\, dS_0|\\
&\qquad+|\int_{\Sigma_{0}}\partial_h^{N}(\mathcal{B}_{8, i}^{\alpha}\partial_{\alpha}v^{i})\, \partial_h^{N}v^3\, dS_0|\lesssim \bigg( E_{N}^{\frac{1}{2}}\dot{\mathcal{D}}_3^{\frac{1}{2}}
+E_3^{\frac{1}{2}} \dot{\mathcal{D}}_N^{\frac{1}{2}}\bigg)\,\dot{\mathcal{D}}_N^{\frac{1}{2}},
\end{split}
\end{equation*}
which along with \eqref{linear-v-Ntan-4} leads to
\begin{equation}\label{linear-v-Ntan-6}
\begin{split}
 & |\mathfrak{K}_1|\lesssim \bigg( E_{N}^{\frac{1}{2}}\dot{\mathcal{D}}_3^{\frac{1}{2}}
+E_3^{\frac{1}{2}} \dot{\mathcal{D}}_N^{\frac{1}{2}}\bigg)\,\dot{\mathcal{D}}_N^{\frac{1}{2}}.\\
\end{split}
\end{equation}
On the other hand, for $\mathfrak{K}_2$, it is easy to produce
\begin{equation}\label{linear-v-Ntan-7}
\begin{split}
&|\mathfrak{K}_2|\lesssim\|\partial_h^{N}q\|_{L^2}\,
(\|\partial_h^{N}(\widetilde{a_{\alpha\,1}}\partial_1v^\alpha)\|_{L^2}
+\|\partial_h^{N}(\mathcal{B}_{5, i}^{\alpha}\partial_{\gamma}v^{i})\|_{L^2}).
\end{split}
\end{equation}
Thanks to Lemma \ref{lem-est-B-Bv-1} again, one has
\begin{equation*}\label{linear-v-Ntan-8}
\begin{split}
&\|\partial_h^{N}(\widetilde{a_{\alpha\,1}}\partial_1v^\alpha)\|_{L^2}
+\|\partial_h^{N}(\mathcal{B}_{5, i}^{\alpha}\partial_{\alpha}v^{i})\|_{L^2}\lesssim  E_{N}^{\frac{1}{2}}\dot{\mathcal{D}}_3^{\frac{1}{2}}
+E_3^{\frac{1}{2}} \dot{\mathcal{D}}_N^{\frac{1}{2}},\\
\end{split}
\end{equation*}
which along with \eqref{linear-v-Ntan-7} follows
\begin{equation}\label{linear-v-Ntan-9}
\begin{split}
&|\mathfrak{K}_2|\lesssim ( E_{N}^{\frac{1}{2}}\dot{\mathcal{D}}_3^{\frac{1}{2}}
+E_3^{\frac{1}{2}} \dot{\mathcal{D}}_N^{\frac{1}{2}})\,\dot{\mathcal{D}}_N^{\frac{1}{2}}.
\end{split}
\end{equation}
Finally, for $\mathfrak{K}_3=\int_{ \Omega}\partial_h^{N} g\cdot \partial_h^{N} v\,dx$, we use \eqref{est-g-2} to get
\begin{equation}\label{linear-v-Ntan-10}
\begin{split}
&|\mathfrak{K}_3|\lesssim \|\partial_h^{N-1} g\|_{L^2}\|\partial_h^{N+1}v\|_{L^2}\lesssim ( E_{N}^{\frac{1}{2}}\dot{\mathcal{D}}_3^{\frac{1}{2}}
+E_3^{\frac{1}{2}} \dot{\mathcal{D}}_N^{\frac{1}{2}})\dot{\mathcal{D}}_{N}^{\frac{1}{2}}.
\end{split}
\end{equation}
Therefore, inserting \eqref{linear-v-Ntan-6}, \eqref{linear-v-Ntan-9}, \eqref{linear-v-Ntan-10} into \eqref{linear-v-Ntan-0} yields \eqref{linear-decay-N-1}, which gives the desired result.
\end{proof}
Thanks to Lemmas \ref{lem-tan-decay-sigma0-1}, \ref{lem-tan-decay-N-1}, and the Korn's inequality, we obtain
\begin{lem}\label{lem-tan-decay-total-1}
Let $N\geq 3$, under the assumption of Lemma \ref{lem-tan-pseudo-energy-1}, if $E_3(t) \leq 1$ for all the existence times $t$, then there holds
\begin{equation*}\label{tan-decay-total-1}
\begin{split}
&\frac{d}{dt}\bigg(\|\dot{\Lambda}_h^{\sigma_0}v\|_{L^2(\Omega)}^2
+\|\dot{\Lambda}_h^{\sigma_0}\xi^1\|_{L^2(\Sigma_0)}^2+\sum_{i=1}^N(\|\partial_h^{i}v\|_{L^2(\Omega)}^2
+\|\partial_h^{i}\xi^1\|_{L^2(\Sigma_0)}^2)\bigg) \\
&\qquad+ c_1(\|\dot{\Lambda}_h^{\sigma_0}\nabla\,v\|_{L^2(\Omega)}^2 +\sum_{i=1}^N\|\partial_h^{i}\nabla\,v\|_{L^2(\Omega)}^2 )\lesssim \dot{\mathcal{D}}_N^{\frac{1}{2}}(E_N^{\frac{1}{2}} \dot{\mathcal{D}}_3^{\frac{1}{2}}+  E_3^{\frac{1}{2}} \dot{\mathcal{D}}_N^{\frac{1}{2}}).
\end{split}
\end{equation*}
\end{lem}

\subsubsection{Estimate of $\|\dot{\Lambda}_h^{-\lambda} v\|_{L^\infty_t(L^2)}$}

\begin{lem}\label{lem-tan-bdd-lambda-1}
Under the assumption of Lemma \ref{lem-tan-pseudo-energy-1}, if$(\lambda,\,\sigma_0) \in (0, 1)$ satisfies $1-\lambda< \sigma_0\leq 1-\frac{1}{2}\lambda$, and $E_3(t) \leq 1$ for all the existence times $t$, then there holds
\begin{equation}\label{linear-bdd-lambda-1}
\begin{split}
& \frac{d}{dt}\bigg(\|\dot{\Lambda}_h^{-\lambda}v\|_{L^2(\Omega)}^2
+\|\dot{\Lambda}_h^{-\lambda}\xi^1\|_{L^2(\Sigma_0)}^2 \bigg) +c_1 \|\dot{\Lambda}_h^{-\lambda}\nabla\,v\|_{L^2(\Omega)}^2 \\
 &\lesssim \|\dot{\Lambda}_h^{-\lambda}\xi^1\|_{L^2(\Sigma_0)}E_3^{\frac{1}{2}} \dot{\mathcal{D}}_3^{\frac{1}{2}}+\|\dot{\Lambda}_h^{-\lambda}\partial_t v^{1}\|_{L^2}E_3^{\frac{1}{2}} \dot{\mathcal{D}}_3^{\frac{1}{2}} +
E_3^{\frac{1}{2}} \dot{\mathcal{D}}_3+E_3  \|\dot{\Lambda}_h^{\sigma_0}\xi^1\|_{L^2(\Sigma_0)}^2.
\end{split}
\end{equation}
\end{lem}
\begin{proof}
First, taking the operator $\mathcal{P}(\partial_h)=\dot{\Lambda}_h^{-\lambda}$ in \eqref{tan-linear-pseudo-0} shows
\begin{equation}\label{linear-bdd-lambda-2}
\begin{split}
 & \frac{1}{2}\frac{d}{dt}\bigg(\int_{ \Omega}|\dot{\Lambda}_h^{-\lambda}v|^2\, dx+ \int_{\Sigma_{0}}|\dot{\Lambda}_h^{-\lambda}\xi^1|^2\,dS_0\bigg) +\frac{\nu}{2}\int_{\Omega} | \mathbb{D}(\dot{\Lambda}_h^{-\lambda}v)|^2 \, dx=I_1+I_2+I_3
\end{split}
\end{equation}
with $I_1:=\int_{ \Omega}\dot{\Lambda}_h^{-\lambda}q\,\dot{\Lambda}_h^{-\lambda}(\mathcal{B}_{5, i}^{\alpha}\partial_{\alpha}v^{i}-\widetilde{a_{\alpha\,1}}\partial_1v^\alpha)\, dx$,
\begin{equation*}
\begin{split}
 &I_2:=\nu\int_{\Sigma_{0}}\bigg(\dot{\Lambda}_h^{-\lambda}((\mathcal{B}_{9, i}^{\alpha}-2\mathcal{B}_{6, i}^{\alpha})\partial_{\alpha}v^{i})\, \dot{\Lambda}_h^{-\lambda}v^1+\dot{\Lambda}_h^{-\lambda}(\mathcal{B}_{7, i}^{\alpha}\partial_{\alpha}v^{i}) \,\dot{\Lambda}_h^{-\lambda}v^2\\
&\qquad\qquad\qquad+\dot{\Lambda}_h^{-\lambda}(\mathcal{B}_{8, i}^{\alpha}\partial_{\alpha}v^{i})\, \dot{\Lambda}_h^{-\lambda}v^3\bigg) \, dS_0,\quad\,I_3:=\int_{ \Omega}\dot{\Lambda}_h^{-\lambda} g\cdot \dot{\Lambda}_h^{-\lambda} v\,dx.
\end{split}
\end{equation*}
Using the relation $q= \mathcal{Q}+\mathcal{H}(\xi^1)-2\nu\,\nabla_h\cdot v^h+\nu\,\mathcal{H}(\mathcal{B}_{9, i}^{\alpha})\partial_{\alpha}v^{i}$, we split $I_1$ into two parts:
\begin{equation*}\label{linear-bdd-lambda-3}
\begin{split}
 I_1&=\int_{ \Omega}\dot{\Lambda}_h^{-\lambda} \mathcal{Q}\,\dot{\Lambda}_h^{-\lambda}(\mathcal{B}_{5, i}^{\alpha}\partial_{\alpha}v^{i}-\widetilde{a_{\alpha\,1}}\partial_1v^\alpha)\, dx\\
 &\qquad+\int_{ \Omega}\dot{\Lambda}_h^{-\lambda}(\mathcal{H}(\xi^1)-2\nu\,\nabla_h\cdot v^h+\nu\,\mathcal{H}(\mathcal{B}_{9, i}^{\alpha})\partial_{\alpha}v^{i})\,\dot{\Lambda}_h^{-\lambda}(\mathcal{B}_{5, i}^{\alpha}\partial_{\alpha}v^{i}-\widetilde{a_{\alpha\,1}}\partial_1v^\alpha)\, dx\\
 &=:I_{1,1}+I_{1,2}.
 \end{split}
\end{equation*}
For $I_{1,1}$, we have
\begin{equation}\label{linear-bdd-lambda-4}
\begin{split}
 &|I_{1,1}|=|\int_{ \Omega}\dot{\Lambda}_h^{-\lambda} \mathcal{Q}\,\dot{\Lambda}_h^{-\lambda}(\mathcal{B}_{5, i}^{\alpha}\partial_{\alpha}v^{i}-\widetilde{a_{\alpha\,1}}\partial_1v^\alpha)\, dx|\\
 &\lesssim\|\dot{\Lambda}_h^{2(1-\sigma_0-\lambda)}\mathcal{Q}\|_{L^\infty_{x_1}L^2_h}
 \|\dot{\Lambda}_h^{2(\sigma_0-1)}(\mathcal{B}_{5, i}^{\alpha}\partial_{\alpha}v^{i}-\widetilde{a_{\alpha\,1}}\partial_1v^\alpha)\|_{L^1_{x_1}L^2_h}.
 \end{split}
\end{equation}
By using the product law \eqref{product-law-1}, it follows
\begin{equation}\label{linear-bdd-lambda-5}
\begin{split}
 &\|\dot{\Lambda}_h^{2(\sigma_0-1)}(\mathcal{B}_{5, i}^{\alpha}\partial_{\alpha}v^{i}-\widetilde{a_{\alpha\,1}}\partial_1v^\alpha)\|_{L^1_{x_1}L^2_h}\\
 &\lesssim \|\dot{\Lambda}_h^{\sigma_0}\mathcal{B}_{5, i}^{\alpha}\|_{L^2}\|\dot{\Lambda}_h^{\sigma_0-1}\partial_{\alpha}v^{i}\|_{L^2} +\|\dot{\Lambda}_h^{\sigma_0}\partial_1v^\alpha\|_{L^2}\|\dot{\Lambda}_h^{\sigma_0-1}\widetilde{a_{\alpha\,1}}\|_{L^2} \lesssim E_3^{\frac{1}{2}}\|\dot{\Lambda}_h^{\sigma_0}\nabla\,v\|_{L^2},
 \end{split}
\end{equation}
which along with $\mathcal{Q}|_{\Sigma_0}=0$ and \eqref{linear-bdd-lambda-4} implies
\begin{equation*}\label{linear-bdd-lambda-6}
\begin{split}
 &|I_{1, 1}|\lesssim\|\dot{\Lambda}_h^{2(1-\sigma_0-\lambda)}\partial_1\mathcal{Q}\|_{L^2}
E_3^{\frac{1}{2}}\|\dot{\Lambda}_h^{\sigma_0}\nabla\,v\|_{L^2}.
 \end{split}
\end{equation*}
The estimate of $\|\dot{\Lambda}_h^{2(1-\sigma_0-\lambda)}\partial_1\mathcal{Q}\|_{L^2}$ can be treated by
\begin{equation}\label{linear-bdd-lambda-7a}
\begin{split}
&\|\dot{\Lambda}_h^{2(1-\sigma_0-\lambda)}\partial_1\mathcal{Q}\|_{L^2}\lesssim \|\dot{\Lambda}_h^{2(1-\sigma_0-\lambda)}\partial_1q\|_{L^2} +\|\dot{\Lambda}_h^{2(1-\sigma_0-\lambda)}\partial_1\mathcal{H}(\xi^1)\|_{L^2} \\
&\qquad\qquad\qquad\qquad\qquad+\|\dot{\Lambda}_h^{2(1-\sigma_0-\lambda)}\partial_1\nabla_h\cdot v^h\|_{L^2} +\|\dot{\Lambda}_h^{2(1-\sigma_0-\lambda)}\partial_1(\mathcal{H}(\mathcal{B}_{9, i}^{\alpha})\partial_{\alpha}v^{i})\|_{L^2} \\
&\lesssim  \|\dot{\Lambda}_h^{-\lambda}\partial_1q\|_{L^2} +\|\dot{\Lambda}_h^{-\lambda}\xi^1\|_{L^2(\Sigma_0)}+\|\dot{\Lambda}_h^{\sigma_0+2}\partial_1q\|_{L^2} +\|\dot{\Lambda}_h^{\sigma_0+1}\xi^1\|_{H^1(\Sigma_0)} \\
&\qquad+\|\dot{\Lambda}_h^{-\lambda}\nabla\, v\|_{L^2} +\|\dot{\Lambda}_h^{\sigma_0+1}\nabla\, v\|_{L^2} +\|\dot{\Lambda}_h^{2(1-\sigma_0-\lambda)}\partial_1(\mathcal{H}(\mathcal{B}_{9, i}^{\alpha})\partial_{\alpha}v^{i})\|_{L^2}.
\end{split}
\end{equation}
For $\|\dot{\Lambda}_h^{-\lambda}\partial_1q\|_{L^2}$ in \eqref{linear-bdd-lambda-7a}, thanks to \eqref{prt1-q-vh-rela-2} and \eqref{est-g-2}, one can see that
\begin{equation*}\label{linear-bdd-lambda-7b}
\begin{split}
&\|\dot{\Lambda}_h^{-\lambda}\partial_1q\|_{L^2}\lesssim\|\dot{\Lambda}_h^{-\lambda}\partial_t v^{1}\|_{L^2}
 +\|\dot{\Lambda}_h^{-\lambda}\partial_h\,\nabla\,v\|_{L^2}+\|\dot{\Lambda}_h^{-\lambda}g_1\|_{L^2}
  +\|\dot{\Lambda}_h^{-\lambda}\widetilde{g}_{11}\|_{L^2}\\
  &\lesssim \|\dot{\Lambda}_h^{-\lambda}\partial_t v^{1}\|_{L^2}
 +\|\dot{\Lambda}_h^{-\lambda}\partial_h\,\nabla\,v\|_{L^2}+E_3^{\frac{1}{2}} \|\dot{\Lambda}_h^{2-\sigma_0-\lambda}\xi^1\|_{L^2(\Sigma_0)} +E_3^{\frac{1}{2}}\, \dot{\mathcal{D}}_3^{\frac{1}{2}},\\
\end{split}
\end{equation*}
which, together with \eqref{linear-bdd-lambda-7a}, $1-\lambda<\sigma_0\leq 1-\frac{\lambda}{2}$, and the estimate
\begin{equation*}
\begin{split}
&\|\dot{\Lambda}_h^{-\lambda}\partial_1(\mathcal{H}(\mathcal{B}_{9, i}^{\alpha})\partial_{\alpha}v^{i})\|_{L^2}
\lesssim\|\dot{\Lambda}_h^{-\lambda}(\partial_1\mathcal{H}(\mathcal{B}_{9, i}^{\alpha})\partial_{\alpha}v^{i})\|_{L^2} +\|\dot{\Lambda}_h^{-\lambda}(\mathcal{H}(\mathcal{B}_{9, i}^{\alpha})\partial_1\partial_{\alpha}v^{i})\|_{L^2}\\
&
\lesssim\|\dot{\Lambda}_h^{\sigma_0}\partial_1\mathcal{H}(\mathcal{B}_{9, i}^{\alpha})\|_{L^2}\|\dot{\Lambda}_h^{1-\sigma_0-\lambda}\partial_{\alpha}v^{i}\|_{L^\infty_{x_1}L^2_h} +\|\dot{\Lambda}_h^{\sigma_0}\mathcal{H}(\mathcal{B}_{9, i}^{\alpha})\|_{L^\infty_{x_1}L^2_h} \|\dot{\Lambda}_h^{1-\sigma_0-\lambda}\partial_1\partial_{\alpha}v^{i})\|_{L^2}\\
 &\lesssim E_3^{\frac{1}{2}}\, \|\dot{\Lambda}_h^{2-\sigma_0-\lambda} v \|_{H^1} \lesssim E_3^{\frac{1}{2}}\, \dot{\mathcal{D}}_3^{\frac{1}{2}}
\end{split}
\end{equation*}
by using \eqref{product-law-1}, ensures that
\begin{equation*}\label{linear-bdd-lambda-7c}
\begin{split}
&\|\dot{\Lambda}_h^{2(1-\sigma_0-\lambda)}\partial_1\mathcal{Q}\|_{L^2}\lesssim \|\dot{\Lambda}_h^{2(1-\sigma_0-\lambda)}\partial_1q\|_{L^2} +\|\dot{\Lambda}_h^{2(1-\sigma_0-\lambda)}\partial_1\mathcal{H}(\xi^1)\|_{L^2} \\
&\qquad\qquad\qquad\qquad\qquad+\|\dot{\Lambda}_h^{2(1-\sigma_0-\lambda)}\partial_1\nabla_h\cdot v^h\|_{L^2} +\|\dot{\Lambda}_h^{2(1-\sigma_0-\lambda)}\partial_1(\mathcal{H}(\mathcal{B}_{9, i}^{\alpha})\partial_{\alpha}v^{i})\|_{L^2} \\
&\lesssim   \|\dot{\Lambda}_h^{-\lambda}\xi^1\|_{L^2(\Sigma_0)}+\|\dot{\Lambda}_h^{-\lambda}\partial_t v^{1}\|_{L^2}
 +\|\dot{\Lambda}_h^{-\lambda}\nabla\,v\|_{L^2} +\dot{\mathcal{D}}_3^{\frac{1}{2}}.
\end{split}
\end{equation*}
Hence, we conclude that
\begin{equation}\label{linear-bdd-lambda-7}
\begin{split}
 &|I_{1, 1}|\lesssim  \|\dot{\Lambda}_h^{-\lambda}\xi^1\|_{L^2(\Sigma_0)}E_3^{\frac{1}{2}} \dot{\mathcal{D}}_3^{\frac{1}{2}}+(\|\dot{\Lambda}_h^{-\lambda}\partial_t v^{1}\|_{L^2}
 +\|\dot{\Lambda}_h^{-\lambda}\nabla\,v\|_{L^2})E_3^{\frac{1}{2}} \dot{\mathcal{D}}_3^{\frac{1}{2}} +
E_3^{\frac{1}{2}} \dot{\mathcal{D}}_3.
 \end{split}
\end{equation}
For $I_{1, 2}$, from \eqref{linear-bdd-lambda-5} and \eqref{linear-bdd-lambda-7}, we find
\begin{equation}\label{linear-bdd-lambda-8}
\begin{split}
 &|I_{1,2}|\lesssim \|\dot{\Lambda}_h^{2(1-\sigma_0-\lambda)}(\mathcal{H}(\xi^1)-2\nu\,\nabla_h\cdot v^h+\nu\,\mathcal{H}(\mathcal{B}_{9, i}^{\alpha})\partial_{\alpha}v^{i})\|_{L^\infty_{x_1}L^2_h}\,\\
 &\qquad\qquad\qquad\qquad\qquad\times\|\dot{\Lambda}_h^{2(1-\sigma_0)}(\mathcal{B}_{5, i}^{\alpha}\partial_{\alpha}v^{i}-\widetilde{a_{\alpha\,1}}\partial_1v^\alpha)\|_{L^1_{x_1}L^2_h}\\
 &\lesssim (\|\dot{\Lambda}_h^{-\lambda}\xi^1\|_{L^2(\Sigma_0)}
 +\|\dot{\Lambda}_h^{-\lambda}\nabla\,v\|_{L^2} +\dot{\mathcal{D}}_3^{\frac{1}{2}})  E_3^{\frac{1}{2}}\|\dot{\Lambda}_h^{\sigma_0}\nabla\,v\|_{L^2}.
 \end{split}
\end{equation}
Combining \eqref{linear-bdd-lambda-7} with \eqref{linear-bdd-lambda-8} yields
\begin{equation}\label{linear-bdd-lambda-9}
\begin{split}
 &|I_{1}|\lesssim \|\dot{\Lambda}_h^{-\lambda}\xi^1\|_{L^2(\Sigma_0)}E_3^{\frac{1}{2}} \dot{\mathcal{D}}_3^{\frac{1}{2}}+(\|\dot{\Lambda}_h^{-\lambda}\partial_t v^{1}\|_{L^2}
 +\|\dot{\Lambda}_h^{-\lambda}\nabla\,v\|_{L^2})E_3^{\frac{1}{2}} \dot{\mathcal{D}}_3^{\frac{1}{2}} +
E_3^{\frac{1}{2}} \dot{\mathcal{D}}_3.
 \end{split}
\end{equation}
For $I_2$, from \eqref{product-law-1}, it is easy to get
\begin{equation*}\label{linear-bdd-lambda-10}
\begin{split}
 &|\int_{\Sigma_{0}} \dot{\Lambda}_h^{-\lambda}(\mathcal{B}_{7, i}^{\alpha}\partial_{\alpha}v^{i}) \,\dot{\Lambda}_h^{-\lambda}v^2\, dS_0|\lesssim \|\dot{\Lambda}_h^{-\lambda}(\mathcal{B}_{7, i}^{\alpha}\partial_{\alpha}v^{i})\|_{L^2(\Sigma_0)} \|\dot{\Lambda}_h^{-\lambda}v^2\|_{L^2(\Sigma_0)} \\
 &\lesssim \|\dot{\Lambda}_h^{\sigma_0}\mathcal{B}_{7, i}^{\alpha}\|_{L^2(\Sigma_0)}\|\dot{\Lambda}_h^{1-\sigma_0-\lambda}\partial_{\alpha}v^{i}\|_{L^2(\Sigma_0)} \|\dot{\Lambda}_h^{-\lambda}v^h\|_{L^2(\Sigma_0)} \\
  &\lesssim E_3^{\frac{1}{2}}\|\dot{\Lambda}_h^{2-\sigma_0-\lambda}\nabla\,v\|_{L^2} \|\dot{\Lambda}_h^{-\lambda}\nabla\,v^h\|_{L^2}.
\end{split}
\end{equation*}
Similarly, one may see
\begin{equation*}\label{linear-bdd-lambda-11}
\begin{split}
 &|\int_{\Sigma_{0}}\bigg(\dot{\Lambda}_h^{-\lambda}((\mathcal{B}_{9, i}^{\alpha}-2\mathcal{B}_{6, i}^{\alpha})\partial_{\alpha}v^{i})\, \dot{\Lambda}_h^{-\lambda}v^1+\dot{\Lambda}_h^{-\lambda}(\mathcal{B}_{8, i}^{\alpha}\partial_{\alpha}v^{i})\, \dot{\Lambda}_h^{-\lambda}v^3\bigg) \, dS_0|\\
 &\lesssim\,E_3^{\frac{1}{2}}\|\dot{\Lambda}_h^{2-\sigma_0-\lambda}\nabla\,v\|_{L^2} \|\dot{\Lambda}_h^{-\lambda}\nabla\,v^h\|_{L^2}.
\end{split}
\end{equation*}
Hence, from the condition $1-\lambda<\sigma_0\leq 1-\frac{\lambda}{2}$, it follows
\begin{equation}\label{linear-bdd-lambda-12}
\begin{split}
 &|I_2|\lesssim E_3^{\frac{1}{2}}\|\dot{\Lambda}_h^{2-\sigma_0-\lambda}\nabla\,v\|_{L^2} \|\dot{\Lambda}_h^{-\lambda}\nabla\,v\|_{L^2}\lesssim E_3^{\frac{1}{2}}\dot{\mathcal{D}}_3^{\frac{1}{2}} \|\dot{\Lambda}_h^{-\lambda}\nabla\,v\|_{L^2}.
\end{split}
\end{equation}
Finally, to deal with $I_3$, we use \eqref{est-g-2} to deduce
\begin{equation}\label{linear-bdd-lambda-13}
\begin{split}
|I_3|&\lesssim \|\dot{\Lambda}_h^{-\lambda} g\|_{L^2} \|\dot{\Lambda}_h^{-\lambda} v\|_{L^2}\lesssim
(E_3^{\frac{1}{2}} \|\dot{\Lambda}_h^{2-\sigma_0-\lambda}\xi^1\|_{L^2(\Sigma_0)} +E_3^{\frac{1}{2}}\,
\dot{\mathcal{D}}_3^{\frac{1}{2}})\|\dot{\Lambda}_h^{-\lambda} \nabla\,v\|_{L^2}\\
&\lesssim
(E_3^{\frac{1}{2}} \|\dot{\Lambda}_h^{\sigma_0}\xi^1\|_{L^2(\Sigma_0)} +E_3^{\frac{1}{2}}\,
\dot{\mathcal{D}}_3^{\frac{1}{2}})\|\dot{\Lambda}_h^{-\lambda} \nabla\,v\|_{L^2}.
\end{split}
\end{equation}
Therefore, substituting \eqref{linear-bdd-lambda-9}, \eqref{linear-bdd-lambda-12} and \eqref{linear-bdd-lambda-13} into \eqref{linear-bdd-lambda-2} yields
\begin{equation*}\label{linear-bdd-lambda-14}
\begin{split}
 & \frac{1}{2}\frac{d}{dt}\bigg(\int_{ \Omega}|\dot{\Lambda}_h^{-\lambda}v|^2\, dx+ \int_{\Sigma_{0}}|\dot{\Lambda}_h^{-\lambda}\xi^1|^2\,dS_0\bigg) +\frac{\nu}{2}\int_{\Omega} | \mathbb{D}(\dot{\Lambda}_h^{-\lambda}v)|^2 \, dx\\
 &\lesssim \|\dot{\Lambda}_h^{-\lambda}\xi^1\|_{L^2(\Sigma_0)}E_3^{\frac{1}{2}} \dot{\mathcal{D}}_3^{\frac{1}{2}}+\|\dot{\Lambda}_h^{-\lambda}\partial_t v^{1}\|_{L^2}E_3^{\frac{1}{2}} \dot{\mathcal{D}}_3^{\frac{1}{2}} +
E_3^{\frac{1}{2}} \dot{\mathcal{D}}_3\\
 &\qquad+(E_3^{\frac{1}{2}} \|\dot{\Lambda}_h^{\sigma_0}\xi^1\|_{L^2(\Sigma_0)} +E_3^{\frac{1}{2}}\,
\dot{\mathcal{D}}_3^{\frac{1}{2}})\|\dot{\Lambda}_h^{-\lambda} \nabla\,v\|_{L^2}.
\end{split}
\end{equation*}
Applying Young's inequality and the Korn inequality implies \eqref{linear-bdd-lambda-1}, which is what we wanted to prove.
\end{proof}

Thanks to \eqref{linear-v-L2-1-0}, Lemmas \ref{lem-tan-decay-total-1} and \ref{lem-tan-bdd-lambda-1}, we have
\begin{lem}\label{lem-tan-bdd-N+1-1}
Let $N\geq 3$, under the assumption of Lemma \ref{lem-tan-pseudo-energy-1}, if $(\lambda,\,\sigma_0) \in (0, 1)$ satisfies $1-\lambda< \sigma_0\leq 1-\frac{1}{2}\lambda$, and $E_3(t) \leq 1$ for all the existence times $t$, then there holds
\begin{equation*}\label{tan-bdd-N+1-1}
\begin{split}
& \frac{d}{dt} \bigg[\|\dot{\Lambda}_h^{-\lambda}v\|_{L^2(\Omega)}^2
+\|\dot{\Lambda}_h^{-\lambda}\xi^1\|_{L^2(\Sigma_0)}^2 \\
&\qquad\qquad\qquad+\|\dot{\Lambda}_h^{\sigma_0}v\|_{L^2(\Omega)}^2
+\|\dot{\Lambda}_h^{\sigma_0}\xi^1\|_{L^2(\Sigma_0)}^2+\sum_{i=0}^{N+1}(\|\partial_h^{i}v\|_{L^2(\Omega)}^2
+\|\partial_h^{i}\xi^1\|_{L^2(\Sigma_0)}^2\bigg]\\
&+ c_1 (\|\dot{\Lambda}_h^{-\lambda}\nabla\,v\|_{L^2(\Omega)}^2+\|\nabla\,v\|_{L^2(\Omega)}^2
+\|\dot{\Lambda}_h^{\sigma_0}\nabla\,v\|_{L^2(\Omega)}^2 +\sum_{i=0}^{N+1}\|\partial_h^{i}\nabla\,v\|_{L^2(\Omega)}^2 )\\
&\lesssim\dot{\mathcal{D}}_{N+1}^{\frac{1}{2}}(E_{N+1}^{\frac{1}{2}} \dot{\mathcal{D}}_3^{\frac{1}{2}}+  E_3^{\frac{1}{2}} \dot{\mathcal{D}}_{N+1}^{\frac{1}{2}})+\|\dot{\Lambda}_h^{-\lambda}\xi^1\|_{L^2(\Sigma_0)}E_3^{\frac{1}{2}} \dot{\mathcal{D}}_3^{\frac{1}{2}}\\
 &\qquad\qquad\qquad\qquad\qquad\qquad\qquad\qquad+E_3  \|\dot{\Lambda}_h^{\sigma_0}\xi^1\|_{L^2(\Sigma_0)}^2+\|\dot{\Lambda}_h^{-\lambda}\partial_t v^{1}\|_{L^2}E_3^{\frac{1}{2}} \dot{\mathcal{D}}_3^{\frac{1}{2}}.
\end{split}
\end{equation*}
\end{lem}

\subsection{Energy estimates of the gradient of the velocity}
In this subsection, we derive energy estimates in terms of $\nabla\,v$ as well as its horizontal derivatives.
For this, we still use the linearized form \eqref{eqns-linear-1} of the system \eqref{eqns-pert-1} as in Section \ref{subsect-est-hori-1}. Set
\begin{equation*}\label{def-pseudo-energy-tv-0}
\begin{split}
 & \mathring{\mathcal{E}}(\mathcal{P}(\partial_h)\nabla\, v)\eqdefa \|\mathcal{P}(\partial_h)\nabla\, v\|_{L^2}^2+4\|\mathcal{P}(\partial_h)\nabla_h\cdot v^h\|_{L^2}^2\\
 &\qquad\qquad+2\int_{\Omega}(\mathcal{P}(\partial_h)\partial_1 v^h\,\cdot\,\mathcal{P}(\partial_h) (\nabla_h v^1) +\mathcal{P}(\partial_h)\partial_1 v^{1}\,\mathcal{P}(\partial_h)\nabla_h\cdot v^h)\,dx\\
   &\,+\frac{2}{\nu}\int_{\Omega}\mathcal{P}(\partial_h) v\cdot\mathcal{P}(\partial_h)\nabla\,\mathcal{H}(\xi^1)  \,dx-2\int_{\Omega}\sum_{j=1}^3\mathcal{P}(\partial_h)\partial_1 v^j\,\mathcal{P}(\partial_h) (\mathcal{H}(\mathcal{B}_{j+5, i}^{\alpha})\partial_{\alpha}v^{i})\,dx.
        \end{split}
\end{equation*}
The standard procedure for getting the energy estimates can be adopted to get
\begin{lem}\label{lem-pseudo-energy-tv-1}
Let $(v, \xi)$ be smooth solution to the system \eqref{eqns-pert-1}, there holds that
\begin{equation}\label{pseudo-energy-tv-1}
\begin{split}
 & \frac{\nu}{2}\frac{d}{dt}\mathring{\mathcal{E}}(\mathcal{P}(\partial_h)\nabla\, v)+\|\mathcal{P}(\partial_h)\partial_t v\|_{L^2}^2 =\sum_{j=1}^8\mathfrak{J}_j,
     \end{split}
\end{equation}
where
\begin{equation*}
\begin{split}
  &\mathfrak{J}_1:=\int_{\Omega}\mathcal{P}(\partial_h) v\cdot\mathcal{P}(\partial_h)\nabla\partial_t \mathcal{H}(\xi^1) \,dx, \,\mathfrak{J}_2:=-\int_{ \Omega}\mathcal{P}(\partial_h)\partial_1\mathcal{Q}
 \mathcal{P}(\partial_h)(\widetilde{a_{\alpha\,1}} \partial_tv^\alpha)\, dx,\\
    &\mathfrak{J}_3:=2\nu\,\int_{\Omega}\mathcal{P}(\partial_h)\partial_t v^h\,\cdot\,\mathcal{P}(\partial_h)\nabla_h\, ((\widetilde{a_{\alpha\,1}}\,\partial_1v^\alpha)
-\mathcal{B}_{5, i}^{\alpha}\,\partial_{\alpha}v^{i})\,dx,\\
     \end{split}
\end{equation*}
\begin{equation*}
\begin{split}
   &\mathfrak{J}_4:=-\nu\int_{\Omega}\mathcal{P}(\partial_h)\partial_t v\,\cdot\,\mathcal{P}(\partial_h) \nabla\,\bigg(\widetilde{a_{\alpha\,1}}\partial_1v^\alpha+(\mathcal{H}(\mathcal{B}_{9, i}^{\alpha})-\mathcal{B}_{5, i}^{\alpha})\partial_{\alpha}v^{i}\bigg) \,dx,\\
 &\mathfrak{J}_5:=-\int_{ \Omega}\mathcal{P}(\partial_h)\mathcal{Q}
 \mathcal{P}(\partial_h)\bigg(\partial_1\widetilde{a_{\alpha\,1}} \partial_tv^\alpha-\partial_t\widetilde{a_{\alpha\,1}}\partial_1v^\alpha
+\partial_t\mathcal{B}_{5, i}^{\alpha}\partial_{\alpha}v^{i}\bigg)\, dx,\\
     \end{split}
\end{equation*}
\begin{equation*}
\begin{split}
 &\mathfrak{J}_6:=\int_{ \Omega}\bigg(\mathcal{P}(\partial_h)\partial_{\alpha}\mathcal{Q} \,\mathcal{P}(\partial_h)(\mathcal{B}_{5, i}^{\alpha}\partial_tv^{i})-\mathcal{P}(\partial_h)\mathcal{Q} \,\mathcal{P}(\partial_h)(\partial_{\alpha}\mathcal{B}_{5, i}^{\alpha}\partial_tv^{i})
\bigg)\, dx,\\
&\mathfrak{J}_7:=\nu\int_{\Omega} \sum_{j=1}^3\bigg(\mathcal{P}(\partial_h)\partial_1\partial_{\alpha} v^j\,\mathcal{P}(\partial_h) \partial_t ( \mathcal{H}(\mathcal{B}_{j+5, i}^{\alpha})v^{i})
  +\mathcal{P}(\partial_h)\partial_1  v^j\,\mathcal{P}(\partial_h) \partial_t ( \partial_{\alpha} \mathcal{H}(\mathcal{B}_{j+5, i}^{\alpha})v^{i}) \\
  &\qquad+\mathcal{P}(\partial_h)\partial_t v^j\, \mathcal{P}(\partial_h)\partial_1(\mathcal{H}(\mathcal{B}_{j+5, i}^{\alpha})\partial_{\alpha}v^{i}) \bigg)\,dx,\quad \mathfrak{J}_8:=\int_{\Omega}\mathcal{P}(\partial_h)\partial_t v\cdot\mathcal{P}(\partial_h)g\,dx.
     \end{split}
\end{equation*}
\end{lem}
\begin{proof}
We multiply the $i$-th component of the momentum equations of \eqref{eqns-linear-pseudo-1} by $\mathcal{P}(\partial_h)\partial_tv^i$, sum over $i$, and integrate
over $\Omega$ to find that
\begin{equation}\label{pseudo-energy-tv-2}
\begin{split}
 & \| \mathcal{P}(\partial_h) \partial_t v\|_{L^2}^2 +I+ II=\int_{\Omega} \mathcal{P}(\partial_h) \partial_t v\,\cdot\, \mathcal{P}(\partial_h) g\,dx,
     \end{split}
\end{equation}
where
$I:=\int_{\Omega} \mathcal{P}(\partial_h) \partial_t v\,\cdot\,\, \mathcal{P}(\partial_h) \nabla\,q\,dx$, $II:=-\nu\int_{\Omega} \mathcal{P}(\partial_h) \partial_t v\,\cdot\,(\nabla\,\cdot\,\mathbb{D}( \mathcal{P}(\partial_h) v))\,dx$.

Using $q=\mathcal{Q} +\mathcal{H}(\xi^1)-2\nu\,\nabla_h\cdot v^h+\nu\,\mathcal{H}(\mathcal{B}_{9, i}^{\alpha})\partial_{\alpha}v^{i}$, the integral $I$ can be split into four parts
\begin{equation}\label{pseudo-energy-tv-3}
\begin{split}
 &I=\int_{\Omega} \mathcal{P}(\partial_h) \partial_t v \cdot  \mathcal{P}(\partial_h) \nabla \mathcal{Q} \,dx+\int_{\Omega} \mathcal{P}(\partial_h) \partial_t v\cdot \mathcal{P}(\partial_h) \nabla\,\mathcal{H}(\xi^1)  \,dx\\
 &-2\nu\int_{\Omega} \mathcal{P}(\partial_h) \partial_t v\cdot\mathcal{P}(\partial_h) \nabla\nabla_h\cdot v^h\,dx+\nu\int_{\Omega} \mathcal{P}(\partial_h) \partial_t v\cdot\mathcal{P}(\partial_h) \nabla\,(\mathcal{H}(\mathcal{B}_{9, i}^{\alpha})\partial_{\alpha}v^{i})\,dx\\
 &=:\sum_{i=1}^4I_i.
\end{split}
\end{equation}
In view of the boundary conditions $v|_{\Sigma_b}=0$ and $\mathcal{Q}|_{\Sigma_0}=0$, integrating by parts in $I_1$ shows
\begin{equation}\label{pseudo-energy-tv-4}
\begin{split}
 &I_1=-\int_{\Omega} \mathcal{P}(\partial_h) \mathcal{Q} \,\nabla\cdot \mathcal{P}(\partial_h) \partial_t v\,dx=\int_{ \Omega} \mathcal{P}(\partial_h) \mathcal{Q} \,
 \mathcal{P}(\partial_h)(\widetilde{a_{\alpha\,1}}\partial_1\partial_tv^\alpha)
 \, dx\\
 &\qquad-\int_{ \Omega} \mathcal{P}(\partial_h) \mathcal{Q} \, \mathcal{P}(\partial_h) \bigg(\mathcal{B}_{5, i}^{\alpha}\partial_{\alpha}\partial_tv^{i}
-\partial_t\widetilde{a_{\alpha\,1}}\,\partial_1v^\alpha
+\partial_t\mathcal{B}_{5, i}^{\alpha}\,\partial_{\alpha}v^{i}\bigg)\, dx,
\end{split}
\end{equation}
where we used the fact $\grad \cdot \partial_tv=-\widetilde{a_{\alpha\,1}}\partial_1\partial_tv^\alpha+(\mathcal{B}_{5, i}^{\alpha}\partial_{\alpha}\partial_tv^{i}
-\partial_t\widetilde{a_{\alpha\,1}}\,\partial_1v^\alpha
+\partial_t\mathcal{B}_{5, i}^{\alpha}\,\partial_{\alpha}v^{i})$,
which comes from the divergence-free condition \eqref{incomp-cond-fluid-3}.

Applying integration by parts ensures
\begin{equation*}
\begin{split}
 &\int_{ \Omega} \mathcal{P}(\partial_h) \mathcal{Q} \,
 \mathcal{P}(\partial_h)(\widetilde{a_{\alpha\,1}}\partial_1\partial_tv^\alpha)
 \, dx\\
 &=-\int_{ \Omega}\mathcal{P}(\partial_h)\partial_1\mathcal{Q}\,
\mathcal{P}(\partial_h)(\widetilde{a_{\alpha\,1}} \partial_tv^\alpha)\, dx-\int_{ \Omega}\mathcal{P}(\partial_h)\mathcal{Q}\,
\mathcal{P}(\partial_h)(\partial_1\widetilde{a_{\alpha\,1}} \partial_tv^\alpha)\, dx,
 \end{split}
\end{equation*}
and
\begin{equation*}
\begin{split}
 &-\int_{ \Omega} \mathcal{P}(\partial_h) \mathcal{Q} \, \mathcal{P}(\partial_h) (\mathcal{B}_{5, i}^{\alpha}\partial_{\alpha}\partial_tv^{i})\, dx\\
 &=\int_{ \Omega} \mathcal{P}(\partial_h) \partial_{\alpha}\mathcal{Q} \, \mathcal{P}(\partial_h) (\mathcal{B}_{5, i}^{\alpha}\partial_tv^{i})\, dx+\int_{ \Omega} \mathcal{P}(\partial_h) \mathcal{Q} \, \mathcal{P}(\partial_h) (\partial_{\alpha}\mathcal{B}_{5, i}^{\alpha}\partial_tv^{i})\, dx,
 \end{split}
\end{equation*}
then plugging these two equalities into \eqref{pseudo-energy-tv-4} leads to
\begin{equation}\label{pseudo-energy-tv-7}
\begin{split}
 &I_1=-\int_{ \Omega}\mathcal{P}(\partial_h)\partial_1\mathcal{Q}\,
\mathcal{P}(\partial_h)(\widetilde{a_{\alpha\,1}} \partial_tv^\alpha)\, dx+\int_{ \Omega} \mathcal{P}(\partial_h) \partial_{\alpha}\mathcal{Q} \, \mathcal{P}(\partial_h) (\mathcal{B}_{5, i}^{\alpha}\partial_tv^{i})\, dx\\
 &\qquad-\int_{ \Omega} \mathcal{P}(\partial_h) \mathcal{Q} \, \mathcal{P}(\partial_h) \bigg(\partial_1\widetilde{a_{\alpha\,1}} \partial_tv^\alpha-\partial_{\alpha}\mathcal{B}_{5, i}^{\alpha}\partial_tv^{i}
-\partial_t\widetilde{a_{\alpha\,1}}\,\partial_1v^\alpha
+\partial_t\mathcal{B}_{5, i}^{\alpha}\,\partial_{\alpha}v^{i}\bigg)\, dx.
\end{split}
\end{equation}
For $I_2$, a direct computation gives rise to
\begin{equation}\label{pseudo-energy-tv-8}
\begin{split}
 &I_2=\int_{\Omega} \mathcal{P}(\partial_h) \partial_t v\cdot\mathcal{P}(\partial_h) \nabla\,\mathcal{H}(\xi^1)  \,dx\\
 &=\frac{d}{dt}\int_{\Omega} \mathcal{P}(\partial_h) v\cdot\mathcal{P}(\partial_h) \nabla\,\mathcal{H}(\xi^1)  \,dx-\int_{\Omega} \mathcal{P}(\partial_h)  v\cdot\mathcal{P}(\partial_h) \nabla\, \mathcal{H}(v^1)  \,dx.
\end{split}
\end{equation}
While for $I_3$, it is easy to get
\begin{equation}\label{pseudo-energy-tv-8a}
\begin{split}
 &I_3=-2\nu\int_{\Omega} \mathcal{P}(\partial_h) \partial_t v^1\,\mathcal{P}(\partial_h) \partial_1\nabla_h\cdot v^h\,dx-2\nu\int_{\Omega} \mathcal{P}(\partial_h) \partial_t v^h\cdot\mathcal{P}(\partial_h) \nabla_h\nabla_h\cdot v^h\,dx\\
 &=-2\nu\int_{\Omega} \mathcal{P}(\partial_h) \partial_t v^1\,\mathcal{P}(\partial_h) \partial_1\nabla_h\cdot v^h\,dx+\nu\frac{d}{dt}\|\mathcal{P}(\partial_h)\nabla_h\cdot v^h\|_{L^2}^2.
\end{split}
\end{equation}
Thus, plugging \eqref{pseudo-energy-tv-7}-\eqref{pseudo-energy-tv-8a} into \eqref{pseudo-energy-tv-3}, one has
\begin{equation*}\label{pseudo-energy-tv-9}
\begin{split}
 &I=\frac{d}{dt}\bigg(\int_{\Omega}\mathcal{P}(\partial_h) v\cdot\mathcal{P}(\partial_h)\nabla\,\mathcal{H}(\xi^1)  \,dx+\nu\|\mathcal{P}(\partial_h)\nabla_h\cdot v^h\|_{L^2}^2\bigg)\\
 &-\int_{\Omega}\mathcal{P}(\partial_h) v\cdot\mathcal{P}(\partial_h)\nabla\,\mathcal{H}(v^1)  \,dx\\
 & \,-2\nu\int_{\Omega}\mathcal{P}(\partial_h)\partial_t v^1\,\mathcal{P}(\partial_h)\partial_1\nabla_h\cdot v^h\,dx-\int_{ \Omega}\mathcal{P}(\partial_h)\partial_1\mathcal{Q}
\mathcal{P}(\partial_h)(\widetilde{a_{\alpha\,1}} \partial_tv^\alpha)\, dx\\
 &\,-\int_{ \Omega}\mathcal{P}(\partial_h)\mathcal{Q}\,
 \mathcal{P}(\partial_h)\bigg(\partial_1\widetilde{a_{\alpha\,1}} \partial_tv^\alpha-\partial_{\alpha}\mathcal{B}_{5, i}^{\alpha}\partial_tv^{i}-\partial_t\widetilde{a_{\alpha\,1}}\,\partial_1v^\alpha
+\partial_t\mathcal{B}_{5, i}^{\alpha}\,\partial_{\alpha}v^{i}\bigg)\, dx\\
 &\,+\int_{ \Omega} \mathcal{P}(\partial_h)\partial_{\alpha}\mathcal{Q} \,\mathcal{P}(\partial_h)(\mathcal{B}_{5, i}^{\alpha}\partial_tv^{i})\, dx+\nu\,\int_{\Omega}\mathcal{P}(\partial_h)\partial_t v\,\cdot\,\mathcal{P}(\partial_h)\nabla\,( \mathcal{H}(\mathcal{B}_{9, i}^{\alpha})\partial_{\alpha}v^{i})\,dx.
\end{split}
\end{equation*}
For $II$, we split it into two parts:
\begin{equation*}\label{pseudo-energy-tv-10}
\begin{split}
   &II=-\nu\int_{\Omega}\mathcal{P}(\partial_h)\partial_t v\cdot\Delta(\mathcal{P}(\partial_h)v) \,dx-\nu\int_{\Omega}\mathcal{P}(\partial_h)\partial_t v\cdot\nabla(\nabla\cdot\mathcal{P}(\partial_h)v) \,dx=:II_1+II_2.
    \end{split}
\end{equation*}
Due to the second equation in \eqref{eqns-linear-pseudo-1}, one can see
\begin{equation}\label{pseudo-energy-tv-12}
\begin{split}
   &II_2=-\nu\int_{\Omega}\mathcal{P}(\partial_h)\partial_t v\,\cdot\,\mathcal{P}(\partial_h) (- \nabla\,(\widetilde{a_{\alpha\,1}}\partial_1v^\alpha)+\nabla\,(\mathcal{B}_{5, i}^{\alpha}\partial_{\alpha}v^{i})) \,dx.
    \end{split}
\end{equation}
The calculation of $II_1$ should be more delicate. In fact, it can be immediately verified that
\begin{equation}\label{pseudo-energy-tv-13}
\begin{split}
   II_1&=-\nu\int_{\Omega}\mathcal{P}(\partial_h)\partial_t v\cdot\,\mathcal{P}(\partial_h)\Delta_hv\,dx-\nu\int_{\Omega}\mathcal{P}(\partial_h)\partial_t v\cdot\,\mathcal{P}(\partial_h)\partial_1^2v\,dx\\
    &=\frac{\nu}{2}\frac{d}{dt}\int_{\Omega}|\mathcal{P}(\partial_h)\partial_hv|^2\,dx +\sum_{j=1}^3II_{1, j},\\
        \end{split}
\end{equation}
where $II_{1, j}:= -\nu\int_{\Omega}\mathcal{P}(\partial_h)\partial_t v^j\, \partial_1^2v^j \,dx$
with $j=1, 2, 3$.

For $II_{1, 1}$, using the equation $\partial_1v^1=\mathcal{G}^1-\nabla_h\cdot v^h+\mathcal{H}(\mathcal{B}_{6, i}^{\alpha})\partial_{\alpha}v^{i}$ in \eqref{partialv-q-1} produces that
\begin{equation}\label{pseudo-energy-tv-14}
\begin{split}
 & II_{1, 1}=- \nu\int_{\Omega}\mathcal{P}(\partial_h)\partial_t v^{1}\,\mathcal{P}(\partial_h)\partial_1\mathcal{G}^1\,dx\\
  &\qquad\qquad- \nu\int_{\Omega}\mathcal{P}(\partial_h)\partial_t v^{1}\,\mathcal{P}(\partial_h)\partial_1(-\nabla_h\cdot v^h+\mathcal{H}(\mathcal{B}_{6, i}^{\alpha})\partial_{\alpha}v^{i})\,dx.
     \end{split}
\end{equation}
For the first integral of the right hand side in \eqref{pseudo-energy-tv-14}, thanks to the boundary conditions $v|_{\Sigma_b}=0$ and $\mathcal{G}^1|_{\Sigma_0}=0$, integrating by parts yields
\begin{equation*}\label{pseudo-energy-tv-15}
\begin{split}
 & - \nu\int_{\Omega}\mathcal{P}(\partial_h)\partial_t v^{1}\,\mathcal{P}(\partial_h)\partial_1\mathcal{G}^1\,dx=\nu\int_{\Omega}\mathcal{P}(\partial_h)\partial_t\partial_1 v^{1}\,\mathcal{P}(\partial_h)\mathcal{G}^1\,dx\\
 &=\nu\int_{\Omega}\mathcal{P}(\partial_h)\partial_t\partial_1 v^{1}\,\mathcal{P}(\partial_h)(\partial_1v^1+\nabla_h\cdot v^h-\mathcal{H}(\mathcal{B}_{6, i}^{\alpha})\partial_{\alpha}v^{i})\,dx\\
  &=\frac{\nu}{2}\frac{d}{dt}\int_{\Omega}|\mathcal{P}(\partial_h)\partial_1 v^{1}|^2\,dx+\nu\frac{d}{dt}\int_{\Omega}\mathcal{P}(\partial_h)\partial_1 v^{1}\,\mathcal{P}(\partial_h)(\nabla_h\cdot v^h-\mathcal{H}(\mathcal{B}_{6, i}^{\alpha})\partial_{\alpha}v^{i})\,dx\\
   &\qquad-\nu \int_{\Omega}\mathcal{P}(\partial_h)\partial_1 v^{1}\,\mathcal{P}(\partial_h)\partial_t(\nabla_h\cdot v^h-\mathcal{H}(\mathcal{B}_{6, i}^{\alpha})\partial_{\alpha}v^{i})\,dx,\\
     \end{split}
\end{equation*}
which follows
\begin{equation}\label{pseudo-energy-tv-16}
\begin{split}
 & - \nu\int_{\Omega}\mathcal{P}(\partial_h)\partial_t v^{1}\,\mathcal{P}(\partial_h)\partial_1\mathcal{G}^1\,dx=-\nu \int_{\Omega}\mathcal{P}(\partial_h)\partial_1 \partial_{\alpha} v^{1}\,\mathcal{P}(\partial_h)\partial_t(\mathcal{H}(\mathcal{B}_{6, i}^{\alpha})v^{i})\,dx\\
 &-\nu \int_{\Omega}\mathcal{P}(\partial_h)\partial_1  v^{1}\,\mathcal{P}(\partial_h)\partial_t(\partial_{\alpha}\mathcal{H}(\mathcal{B}_{6, i}^{\alpha})v^{i})\,dx+\nu \int_{\Omega}\mathcal{P}(\partial_h)\nabla_h\partial_1 v^{1}\,\cdot\,\mathcal{P}(\partial_h)\partial_t v^h\,dx\\
 &+\frac{\nu}{2}\frac{d}{dt}\int_{\Omega}\bigg(|\mathcal{P}(\partial_h)\partial_1 v^{1}|^2+2\mathcal{P}(\partial_h)\partial_t\partial_1 v^{1}\,\mathcal{P}(\partial_h)(\nabla_h\cdot v^h-\mathcal{H}(\mathcal{B}_{6, i}^{\alpha})\partial_{\alpha}v^{i})\bigg)\,dx.
     \end{split}
\end{equation}
Plugging \eqref{pseudo-energy-tv-16} into \eqref{pseudo-energy-tv-14}, we get
\begin{equation}\label{pseudo-energy-tv-17}
\begin{split}
 & II_{1, 1}=\frac{\nu}{2}\frac{d}{dt}\bigg[\|\mathcal{P}(\partial_h)\partial_1 v^{1}\|_{L^2}^2+2\int_{\Omega}\mathcal{P}(\partial_h)\partial_1 v^{1}\,\mathcal{P}(\partial_h)(\nabla_h\cdot v^h-\mathcal{H}(\mathcal{B}_{6, i}^{\alpha})\partial_{\alpha}v^{i})\,dx\bigg]\\
   &+\nu \bigg(\int_{\Omega}\mathcal{P}(\partial_h)\nabla_h\partial_1 v^{1}\,\cdot\,\mathcal{P}(\partial_h)\partial_t v^h\,dx+ \int_{\Omega}\mathcal{P}(\partial_h)\partial_t v^{1}\,\mathcal{P}(\partial_h)\partial_1\nabla_h\cdot v^h\,dx\bigg)\\
   &-\nu \bigg(\int_{\Omega}\mathcal{P}(\partial_h)\partial_1 \partial_{\alpha} v^{1}\,\mathcal{P}(\partial_h)\partial_t(\mathcal{H}(\mathcal{B}_{6, i}^{\alpha})v^{i})\,dx+ \int_{\Omega}\mathcal{P}(\partial_h)\partial_1  v^{1}\,\mathcal{P}(\partial_h)\partial_t(\partial_{\alpha}\mathcal{H}(\mathcal{B}_{6, i}^{\alpha})v^{i})\,dx\\
   &\qquad\qquad\qquad+\int_{\Omega}\mathcal{P}(\partial_h)\partial_t v^{1}\,\mathcal{P}(\partial_h)\partial_1(\mathcal{H}(\mathcal{B}_{6, i}^{\alpha})\partial_{\alpha}v^{i})\,dx\bigg).
     \end{split}
\end{equation}
Following our analysis of the term $II_{1, 1}$, we use
$\partial_1 v^2=\mathcal{G}^2-\partial_2 v^1+\mathcal{H}(\mathcal{B}_{7, i}^{\alpha})\partial_{\alpha}v^{i}$, $\partial_1 v^3=\mathcal{G}^3-\partial_3 v^1+\mathcal{H}(\mathcal{B}_{8, i}^{\alpha})\partial_{\alpha}v^{i}$ in \eqref{partialv-q-1} to $II_{1, 2}$, $II_{1, 3}$ respectively to obtain
\begin{equation}\label{pseudo-energy-tv-18}
\begin{split}
   &II_{1, 2}+II_{1, 3}= \frac{\nu}{2}\frac{d}{dt}\bigg[\|\mathcal{P}(\partial_h)\partial_1 v^h\|_{L^2}^2+2\int_{\Omega}\mathcal{P}(\partial_h)\partial_1 v^h\,\cdot\,\mathcal{P}(\partial_h) (\nabla_h v^1) \,dx\\
   &\qquad\qquad\qquad\qquad\qquad\qquad-2\int_{\Omega} \sum_{\beta=2}^3\mathcal{P}(\partial_h)\partial_1 v^\beta\,\mathcal{P}(\partial_h) (\mathcal{H}(\mathcal{B}_{\beta+5, i}^{\alpha})\partial_{\alpha}v^{i})\,dx\bigg]\\
  &-\nu \sum_{\beta=2}^3\int_{\Omega}\bigg(\mathcal{P}(\partial_h)\partial_1\partial_{\alpha} v^\beta\,\mathcal{P}(\partial_h) \partial_t ( \mathcal{H}(\mathcal{B}_{\beta+5, i}^{\alpha})v^{i})+\mathcal{P}(\partial_h)\partial_1  v^\beta\,\mathcal{P}(\partial_h) \partial_t ( \partial_{\alpha} \mathcal{H}(\mathcal{B}_{\beta+5, i}^{\alpha})v^{i}) \\
  & \qquad\qquad\qquad\qquad\qquad\qquad\qquad\qquad+\mathcal{P}(\partial_h)\partial_t v^\beta\, \mathcal{P}(\partial_h)\partial_1(\mathcal{H}(\mathcal{B}_{\beta+5, i}^{\alpha})\partial_{\alpha}v^{i})\bigg)\,dx\\
   &+\nu\int_{\Omega} \mathcal{P}(\partial_h)\partial_1 \nabla_h\cdot v^h\,\mathcal{P}(\partial_h) \partial_t  v^1  \,dx+\nu\int_{\Omega}\mathcal{P}(\partial_h)\partial_t v^h\, \cdot\,\mathcal{P}(\partial_h)\nabla_h\partial_1 v^1 \,dx.
    \end{split}
\end{equation}
Combining \eqref{pseudo-energy-tv-17} with \eqref{pseudo-energy-tv-18} leads to
\begin{equation}\label{pseudo-energy-tv-19}
\begin{split}
   &II_{1, 1}+II_{1, 2}+II_{1, 3}= \frac{\nu}{2}\frac{d}{dt}\bigg[\|\mathcal{P}(\partial_h)\partial_1 v\|_{L^2}^2+2\int_{\Omega}\mathcal{P}(\partial_h)\partial_1 v^h\,\cdot\,\mathcal{P}(\partial_h) (\nabla_h v^1) \,dx\\
   &+2\int_{\Omega}\mathcal{P}(\partial_h)\partial_1 v^{1}\,\mathcal{P}(\partial_h)\nabla_h\cdot v^h\,dx-2\int_{\Omega}\sum_{j=1}^3\mathcal{P}(\partial_h)\partial_1 v^j\,\mathcal{P}(\partial_h) (\mathcal{H}(\mathcal{B}_{j+5, i}^{\alpha})\partial_{\alpha}v^{i})\,dx\bigg]\\
  &-\nu\int_{\Omega} \sum_{j=1}^3\bigg(\mathcal{P}(\partial_h)\partial_1\partial_{\alpha} v^j\,\mathcal{P}(\partial_h) \partial_t ( \mathcal{H}(\mathcal{B}_{j+5, i}^{\alpha})v^{i})
  +\mathcal{P}(\partial_h)\partial_1  v^j\,\mathcal{P}(\partial_h) \partial_t ( \partial_{\alpha} \mathcal{H}(\mathcal{B}_{j+5, i}^{\alpha})v^{i}) \\
  &\qquad\qquad\qquad\qquad\qquad\qquad\qquad\qquad\qquad\qquad+\mathcal{P}(\partial_h)\partial_t v^j\, \partial_1(\mathcal{H}(\mathcal{B}_{j+5, i}^{\alpha})\partial_{\alpha}v^{i}) \bigg)\,dx\\
   &+2\nu\int_{\Omega} \mathcal{P}(\partial_h)\partial_1 \nabla_h\cdot v^h\,\mathcal{P}(\partial_h) \partial_t  v^1  \,dx+2\nu\int_{\Omega}\mathcal{P}(\partial_h)\partial_t v^h\, \cdot\,\mathcal{P}(\partial_h)\nabla_h\partial_1 v^1 \,dx.
    \end{split}
\end{equation}
Using the incompressible condition $\partial_1 v^1=-\nabla_h\cdot v^h-\widetilde{a_{\alpha\,1}}\,\partial_1v^\alpha
+\mathcal{B}_{5, i}^{\alpha}\,\partial_{\alpha}v^{i}$ in \eqref{incomp-cond-fluid-3}, we infer
\begin{equation*}\label{pseudo-energy-tv-20}
\begin{split}
   &2\nu\int_{\Omega}\mathcal{P}(\partial_h)\partial_t v^h\, \cdot\,\mathcal{P}(\partial_h)\nabla_h\partial_1 v^1 \,dx=2\nu\int_{\Omega}\mathcal{P}(\partial_h)\partial_t v^h\, \cdot\,\mathcal{P}(\partial_h)\nabla_h(-\nabla_h\cdot v^h) \,dx\\
   &\qquad\qquad+2\nu\int_{\Omega}\mathcal{P}(\partial_h)\partial_t v^h\, \cdot\,\mathcal{P}(\partial_h)\nabla_h(-\widetilde{a_{\alpha\,1}}\,\partial_1v^\alpha
+\mathcal{B}_{5, i}^{\alpha}\,\partial_{\alpha}v^{i}) \,dx\\
&=\nu\,\frac{d}{dt}\|\mathcal{P}(\partial_h)\nabla_h\cdot v^h\|_{L^2}^2+2\nu\int_{\Omega}\mathcal{P}(\partial_h)\partial_t v^h\, \cdot\,\mathcal{P}(\partial_h)\nabla_h(-\widetilde{a_{\alpha\,1}}\,\partial_1v^\alpha
+\mathcal{B}_{5, i}^{\alpha}\,\partial_{\alpha}v^{i}) \,dx,
    \end{split}
\end{equation*}
which follows from \eqref{pseudo-energy-tv-19} and \eqref{pseudo-energy-tv-13} that
\begin{equation}\label{pseudo-energy-tv-21a}
\begin{split}
   &II_{1}= \frac{\nu}{2}\frac{d}{dt}\bigg[\|\mathcal{P}(\partial_h)\nabla\, v\|_{L^2}^2+\|\mathcal{P}(\partial_h)\nabla_h\cdot v^h\|_{L^2}^2+2\int_{\Omega}\mathcal{P}(\partial_h)\partial_1 v^h\,\cdot\,\mathcal{P}(\partial_h) (\nabla_h v^1) \,dx\\
   &+2\int_{\Omega}\mathcal{P}(\partial_h)\partial_1 v^{1}\,\mathcal{P}(\partial_h)\nabla_h\cdot v^h\,dx-2\int_{\Omega}\sum_{j=1}^3\mathcal{P}(\partial_h)\partial_1 v^j\,\mathcal{P}(\partial_h) (\mathcal{H}(\mathcal{B}_{j+5, i}^{\alpha})\partial_{\alpha}v^{i})\,dx\bigg]\\
  &+2\nu\int_{\Omega} \mathcal{P}(\partial_h)\partial_1 \nabla_h\cdot v^h\,\mathcal{P}(\partial_h) \partial_t  v^1  \,dx-\nu\int_{\Omega} \sum_{j=1}^3\bigg(\mathcal{P}(\partial_h)\partial_1\partial_{\alpha} v^j\,\mathcal{P}(\partial_h) \partial_t ( \mathcal{H}(\mathcal{B}_{j+5, i}^{\alpha})v^{i})
  \\
  &\qquad+\mathcal{P}(\partial_h)\partial_1  v^j\,\mathcal{P}(\partial_h) \partial_t ( \partial_{\alpha} \mathcal{H}(\mathcal{B}_{j+5, i}^{\alpha})v^{i}) +\mathcal{P}(\partial_h)\partial_t v^j\, \partial_1(\mathcal{H}(\mathcal{B}_{j+5, i}^{\alpha})\partial_{\alpha}v^{i}) \bigg)\,dx\\
   &+2\nu\int_{\Omega}\mathcal{P}(\partial_h)\partial_t v^h\, \cdot\,\mathcal{P}(\partial_h)\nabla_h(-\widetilde{a_{\alpha\,1}}\,\partial_1v^\alpha
+\mathcal{B}_{5, i}^{\alpha}\,\partial_{\alpha}v^{i}) \,dx.
    \end{split}
\end{equation}
Combining \eqref{pseudo-energy-tv-2}, \eqref{pseudo-energy-tv-7}, \eqref{pseudo-energy-tv-12}, \eqref{pseudo-energy-tv-21a} gives rise to
\begin{equation}\label{pseudo-energy-tv-21}
\begin{split}
   I+II=I+II_1+II_2=\frac{\nu}{2}\frac{d}{dt}\mathring{\mathcal{E}}(\mathcal{P}(\partial_h)\nabla\, v)-\sum_{j=1}^7\mathfrak{J}_j.
        \end{split}
\end{equation}
Inserting \eqref{pseudo-energy-tv-21} into \eqref{pseudo-energy-tv-2} yields \eqref{pseudo-energy-tv-1}, which is the desired result.
\end{proof}

\subsubsection{Estimate of $\|\dot{\Lambda}_h^{\sigma_0}\nabla\,v\|_{L^2}$}

\begin{lem}\label{lem-grad-sigma0-1}
Under the assumption of Lemma \ref{lem-pseudo-energy-tv-1}, if $E_3(t) \leq 1$ for all the existence times $t$, then there holds
\begin{equation}\label{grad-sigma0-tv-1}
\begin{split}
 & \frac{\nu}{2}\frac{d}{dt}\mathring{\mathcal{E}}(\dot{\Lambda}_h^{\sigma_0}\nabla\, v)+\|\dot{\Lambda}_h^{\sigma_0}\partial_t v\|_{L^2}^2\lesssim  \|\dot{\Lambda}_h^{\sigma_0}\nabla\,v\|_{L^2}^2+E_3^{\frac{1}{2}}\dot{\mathcal{D}}_3.
        \end{split}
\end{equation}
\end{lem}
\begin{proof}
Taking $\mathcal{P}(\partial_h)=\dot{\Lambda}_h^{\sigma_0}$ in \eqref{pseudo-energy-tv-1}, we will estimate all the integrals in the right hand side of \eqref{pseudo-energy-tv-1} one by one.

For $\mathfrak{J}_1$, we directly bound it by
\begin{equation*}\label{grad-sigma0-tv-2}
\begin{split}
 &|\mathfrak{J}_1|\lesssim\|\dot{\Lambda}_h^{\sigma_0} v\|_{L^2}\|\dot{\Lambda}_h^{\sigma_0}\nabla\,\partial_t \mathcal{H}(\xi^1)\|_{L^2}\lesssim\|\dot{\Lambda}_h^{\sigma_0} v\|_{L^2}\|\dot{\Lambda}_h^{\sigma_0}\nabla\,v\|_{L^2}\lesssim \|\dot{\Lambda}_h^{\sigma_0}\nabla\,v\|_{L^2}^2,\\
          \end{split}
\end{equation*}
While for $\mathfrak{J}_2$, the product law \eqref{product-law-1} ensures that
\begin{equation*}\label{grad-sigma0-tv-3}
\begin{split}
 &|\mathfrak{J}_2|\lesssim \|\dot{\Lambda}_h^{\sigma_0+1}\partial_1\mathcal{Q}\|_{L^2} \|\dot{\Lambda}_h^{\sigma_0-1}(\widetilde{a_{\alpha\,1}} \partial_tv^\alpha)\|_{L^2}\\
 &\lesssim \|\dot{\Lambda}_h^{\sigma_0+1}\partial_1\mathcal{Q}\|_{L^2} \|\widetilde{a_{\alpha\,1}}\|_{L^\infty_{x_1}L^2_h}
 \|\dot{\Lambda}_h^{\sigma_0}\partial_tv^h \|_{L^2}.
         \end{split}
\end{equation*}
Since $\mathcal{Q}= q-\mathcal{H}(\xi^1)  +2\nu\,\nabla_h\cdot v^h-\nu\,\mathcal{H}(\mathcal{B}_{9, i}^{\alpha})\partial_{\alpha}v^{i}$, we have
\begin{equation}\label{grad-sigma0-tv-4}
\begin{split}
& \|\dot{\Lambda}_h^{1+\sigma_0}\partial_1\mathcal{Q}\|_{L^2}\lesssim  \|\dot{\Lambda}_h^{1+\sigma_0}\partial_1q\|_{L^2}+ \|\dot{\Lambda}_h^{1+\sigma_0}\partial_1\mathcal{H}(\xi^1)\|_{L^2}  + \|\dot{\Lambda}_h^{1+\sigma_0}\partial_1\nabla_h\cdot v^h \|_{L^2}\\
 &\qquad\qquad\qquad\qquad\qquad\qquad\qquad\qquad+ \|\dot{\Lambda}_h^{1+\sigma_0}\partial_1(\mathcal{H}(\mathcal{B}_{9, i}^{\alpha})\partial_{\alpha}v^{i})\|_{L^2}\\
 &\lesssim  \|\dot{\Lambda}_h^{1+\sigma_0}\partial_1q\|_{L^2}+ \|\dot{\Lambda}_h^{1+\sigma_0} \xi^1 \|_{H^{\frac{1}{2}}(\Sigma_0)}  + \|\dot{\Lambda}_h^{1+\sigma_0}\partial_1\nabla_h\cdot v^h \|_{L^2}\\
 &\qquad\qquad\qquad\qquad\qquad\qquad\qquad\qquad+ \|\dot{\Lambda}_h^{1+\sigma_0}\partial_1(\mathcal{H}(\mathcal{B}_{9, i}^{\alpha})\partial_{\alpha}v^{i})\|_{L^2}\lesssim \dot{\mathcal{D}}_3^{\frac{1}{2}},
\end{split}
\end{equation}
which, along with $\|(\widetilde{a_{21}},\widetilde{a_{31}})\|_{L^\infty_{x_1}L^2_h}
 \lesssim E_3^{\frac{1}{2}}$, follows that
 \begin{equation*}\label{grad-sigma0-tv-5}
\begin{split}
 &|\mathfrak{J}_2|\lesssim  E_3^{\frac{1}{2}}\dot{\mathcal{D}}_3.
         \end{split}
\end{equation*}
For $\mathfrak{J}_3$, one has
 \begin{equation*}\label{grad-sigma0-tv-6}
\begin{split}
   &|\mathfrak{J}_3|\lesssim \|\dot{\Lambda}_h^{\sigma_0}\partial_t v\|_{L^2} (\|\dot{\Lambda}_h^{\sigma_0} \nabla\,(\widetilde{a_{\alpha\,1}}\partial_1v^\alpha)\|_{L^2}+\|\dot{\Lambda}_h^{\sigma_0} \nabla\,(\mathcal{B}_{5, i}^{\alpha}\partial_{\alpha}v^{i})\|_{L^2}),\\
           \end{split}
\end{equation*}
Notice that
 \begin{equation*}\label{grad-sigma0-tv-7}
\begin{split}
   &\|\dot{\Lambda}_h^{\sigma_0} \nabla\,(\widetilde{a_{\alpha\,1}}\partial_1v^\alpha)\|_{L^2} \lesssim \|\dot{\Lambda}_h^{\sigma_0}(\widetilde{a_{\alpha\,1}} \nabla\,\partial_1v^\alpha)\|_{L^2}+\|\dot{\Lambda}_h^{\sigma_0} ( \nabla\,\widetilde{a_{\alpha\,1}}\partial_1v^\alpha)\|_{L^2}\\
   & \lesssim \|\dot{\Lambda}_h^{(1+\sigma_0)/2}\widetilde{a_{\alpha\,1}}\|_{L^\infty_{x_1}(L^2_h)} \|\dot{\Lambda}_h^{(1+\sigma_0)/2}\nabla\,\partial_1v^\alpha\|_{L^2}\\
   &\qquad\qquad+\| \|\dot{\Lambda}_h^{(1+\sigma_0)/2}\nabla\,\widetilde{a_{\alpha\,1}}\|_{L^2}
   \|\dot{\Lambda}_h^{(1+\sigma_0)/2}\partial_1v^\alpha\|_{L^\infty_{x_1}L^2_h}\lesssim E_3^{\frac{1}{2}}\dot{\mathcal{D}}_3^{\frac{1}{2}}.
           \end{split}
\end{equation*}
Similarly, it can be checked that $\|\dot{\Lambda}_h^{\sigma_0} \nabla\,(\mathcal{B}_{5, i}^{\alpha}\partial_{\alpha}v^{i})\|_{L^2})\lesssim E_3^{\frac{1}{2}}\dot{\mathcal{D}}_3^{\frac{1}{2}}$,
which ensures that
 \begin{equation*}\label{grad-sigma0-tv-9}
\begin{split}
   &|\mathfrak{J}_3|\lesssim E_3^{\frac{1}{2}}\dot{\mathcal{D}}_3.
           \end{split}
\end{equation*}
The same proof remains valid for $\mathfrak{J}_4$, and one has
 \begin{equation*}\label{grad-sigma0-tv-10}
\begin{split}
   &|\mathfrak{J}_4|\lesssim E_3^{\frac{1}{2}}\dot{\mathcal{D}}_3.
           \end{split}
\end{equation*}
We control $\mathfrak{J}_5$ by
 \begin{equation*}\label{grad-sigma0-tv-11}
\begin{split}
 &|\mathfrak{J}_5|\lesssim \|\dot{\Lambda}_h^{\sigma_0+1}\mathcal{Q}\|_{L^\infty_{x_1}L^2_h} (\|\dot{\Lambda}_h^{\sigma_0-1}(\partial_1\widetilde{a_{21}} \partial_tv^2
 +\partial_1\widetilde{a_{31}} \partial_tv^3)\|_{L^1_{x_1}L^2_h}\\
 &\qquad\qquad\qquad\qquad+\|\dot{\Lambda}_h^{\sigma_0-1}(-\partial_t\widetilde{a_{21}}\,\partial_1v^2
 -\partial_t\widetilde{a_{31}}\,\partial_1v^3
+\partial_t\mathcal{B}_{5, i}^{\alpha}\,\partial_{\alpha}v^{i})\|_{L^1_{x_1}L^2_h})\\
 &\lesssim \|\dot{\Lambda}_h^{\sigma_0+1}\partial_1\mathcal{Q}\|_{L^2} \|(\partial_1\widetilde{a_{21}},\partial_1\widetilde{a_{31}}, \partial_t\widetilde{a_{21}},\,\partial_t\widetilde{a_{31}},\,\partial_t\mathcal{B}_{5, i}^{\alpha})\|_{L^2} \|\dot{\Lambda}_h^{\sigma_0}(\partial_tv^h,\,\partial_1v^h, \,\partial_{\alpha}v^{i}) \|_{L^2},
         \end{split}
\end{equation*}
which along with \eqref{grad-sigma0-tv-4} yields that
 \begin{equation*}\label{grad-sigma0-tv-12}
\begin{split}
 &|\mathfrak{J}_5|\lesssim E_3^{\frac{1}{2}}\dot{\mathcal{D}}_3.
         \end{split}
\end{equation*}
For $\mathfrak{J}_6$, in the same manner, it can be obtained that
 \begin{equation*}
\begin{split}
 &|\mathfrak{J}_6|\lesssim\|\dot{\Lambda}_h^{\sigma_0}\partial_{\alpha}\mathcal{Q}\|_{L^\infty_{x_1}L^2_h} \|\dot{\Lambda}_h^{\sigma_0}(\mathcal{B}_{5, i}^{\alpha}\partial_tv^{i})\|_{L^1_{x_1}L^2_h} +\|\dot{\Lambda}_h^{\sigma_0+1}\mathcal{Q}\|_{L^\infty_{x_1}L^2_h}  \|\dot{\Lambda}_h^{\sigma_0-1}(\partial_{\alpha}\mathcal{B}_{5, i}^{\alpha}\partial_tv^{i})\|_{L^1_{x_1}L^2_h}\\
  &\lesssim \|\dot{\Lambda}_h^{\sigma_0+1}\partial_1\mathcal{Q}\|_{L^2} (\|\dot{\Lambda}_h^{(\sigma_0+1)/2}\mathcal{B}_{5, i}^{\alpha}\|_{L^2}\|\dot{\Lambda}_h^{(\sigma_0+1)/2}\partial_tv^{i}\|_{L^2} +  \|\partial_{\alpha}\mathcal{B}_{5, i}^{\alpha}\|_{L^2}\|\dot{\Lambda}_h^{\sigma_0}\partial_tv^{i}\|_{L^2}),\\
\end{split}
\end{equation*}
which implies
 \begin{equation}\label{grad-sigma0-tv-14}
\begin{split}
 &|\mathfrak{J}_6|\lesssim  E_3^{\frac{1}{2}}\dot{\mathcal{D}}_3.
         \end{split}
\end{equation}
For $\mathfrak{J}_7$, we deal with the integral $\int_{\Omega}  \dot{\Lambda}_h^{\sigma_0}\partial_1\partial_{\alpha} v^j\,\dot{\Lambda}_h^{\sigma_0} \partial_t ( \mathcal{H}(\mathcal{B}_{j+5, i}^{\alpha})v^{i}) \,dx$ in it by
\begin{equation*}
\begin{split}
  &|\int_{\Omega}  \dot{\Lambda}_h^{\sigma_0}\partial_1\partial_{\alpha} v^j\,\dot{\Lambda}_h^{\sigma_0} \partial_t ( \mathcal{H}(\mathcal{B}_{j+5, i}^{\alpha})v^{i}) \,dx|\lesssim \| \dot{\Lambda}_h^{\sigma_0}\partial_1\partial_{\alpha} v^j\|_{L^2}\|\dot{\Lambda}_h^{\sigma_0} \partial_t ( \mathcal{H}(\mathcal{B}_{j+5, i}^{\alpha})v^{i}) \|_{L^2}\\
  &\lesssim \| \dot{\Lambda}_h^{\sigma_0}\partial_1\partial_{\alpha} v^j\|_{L^2}(\|\dot{\Lambda}_h^{\sigma_0} ( \mathcal{H}(\mathcal{B}_{j+5, i}^{\alpha})\partial_tv^{i}) \|_{L^2}+\|\dot{\Lambda}_h^{\sigma_0}  ( \partial_t \mathcal{H}(\mathcal{B}_{j+5, i}^{\alpha})v^{i}) \|_{L^2}),
              \end{split}
\end{equation*}
from this, it follows that
\begin{equation}\label{grad-sigma0-tv-15}
\begin{split}
  &|\int_{\Omega}  \dot{\Lambda}_h^{\sigma_0}\partial_1\partial_{\alpha} v^j\,\dot{\Lambda}_h^{\sigma_0} \partial_t ( \mathcal{H}(\mathcal{B}_{j+5, i}^{\alpha})v^{i}) \,dx|\\
    &\lesssim \| \dot{\Lambda}_h^{\sigma_0}\partial_1\partial_{\alpha} v^j\|_{L^2}(\|\dot{\Lambda}_h^{(\sigma_0+1)/2}\mathcal{H}(\mathcal{B}_{j+5, i}^{\alpha})\|_{L^\infty_{x_1}(L^2_h)}\|\dot{\Lambda}_h^{(\sigma_0+1)/2} \partial_tv^{i}  \|_{L^2}\\
    &\qquad +\|\dot{\Lambda}_h^{(\sigma_0+1)/2}\partial_t \mathcal{H}(\mathcal{B}_{j+5, i}^{\alpha})\|_{L^2}\|\dot{\Lambda}_h^{(\sigma_0+1)/2} v^{i}\|_{L^\infty_{x_1}L^2_h})\\
        &\lesssim \dot{\mathcal{D}}_3^{\frac{1}{2}}(E_3^{\frac{1}{2}}\dot{\mathcal{D}}_3^{\frac{1}{2}} +\dot{\mathcal{D}}_3^{\frac{1}{2}} \dot{\mathcal{E}}_3^{\frac{1}{2}})\lesssim E_3^{\frac{1}{2}}\dot{\mathcal{D}}_3,
              \end{split}
\end{equation}
where we have used that
\begin{equation*}\label{grad-sigma0-tv-16}
\begin{split}
  &\| \dot{\Lambda}_h^{\sigma_0}\partial_1\partial_{\alpha} v^j\|_{L^2}\lesssim \dot{\mathcal{D}}_3^{\frac{1}{2}}, \, \|\dot{\Lambda}_h^{(\sigma_0+1)/2}\mathcal{H}(\mathcal{B}_{j+5, i}^{\alpha})\|_{L^\infty_{x_1}(L^2_h)} \lesssim \|\dot{\Lambda}_h^{(\sigma_0+1)/2}\mathcal{H}(\mathcal{B}_{j+5, i}^{\alpha})\|_{H^1} \lesssim E_3^{\frac{1}{2}}, \\
  &\|\dot{\Lambda}_h^{(\sigma_0+1)/2}\partial_t \mathcal{H}(\mathcal{B}_{j+5, i}^{\alpha})\|_{L^2}\lesssim \dot{\mathcal{D}}_3^{\frac{1}{2}},\quad \|\dot{\Lambda}_h^{(\sigma_0+1)/2}  v^{i}\|_{L^\infty_{x_1}L^2_h} \lesssim \|\dot{\Lambda}_h^{(\sigma_0+1)/2}\nabla  v\|_{L^2} \lesssim \dot{\mathcal{E}}_3^{\frac{1}{2}}.
              \end{split}
\end{equation*}
The following results may be proved in much the same way as in the proof of \eqref{grad-sigma0-tv-15}
\begin{equation*}\label{grad-sigma0-tv-17}
\begin{split}
  &|\int_{\Omega}\dot{\Lambda}_h^{\sigma_0}\partial_1  v^j\,\dot{\Lambda}_h^{\sigma_0} \partial_t ( \partial_{\alpha} \mathcal{H}(\mathcal{B}_{j+5, i}^{\alpha})v^{i}) \,dx|\lesssim E_3^{\frac{1}{2}}\dot{\mathcal{D}}_3,\\
  &|\int_{\Omega}\dot{\Lambda}_h^{\sigma_0}\partial_t v^j\, \dot{\Lambda}_h^{\sigma_0}\partial_1(\mathcal{H}(\mathcal{B}_{j+5, i}^{\alpha})\partial_{\alpha}v^{i})\,dx|\lesssim E_3^{\frac{1}{2}}\dot{\mathcal{D}}_3.
          \end{split}
\end{equation*}
Hence, it follows
\begin{equation*}\label{grad-sigma0-tv-18}
\begin{split}
  &|\mathfrak{J}_7|\lesssim E_3^{\frac{1}{2}}\dot{\mathcal{D}}_3.
          \end{split}
\end{equation*}
Finally, making directly use of \eqref{est-g-2} implies
\begin{equation*}\label{grad-sigma0-tv-19}
\begin{split}
 &|\mathfrak{J}_8|\lesssim\|\dot{\Lambda}_h^{\sigma_0}\partial_t v\|_{L^2}\|\dot{\Lambda}_h^{\sigma_0}g\|_{L^2}\lesssim  E_3^{\frac{1}{2}}\dot{\mathcal{D}}_3.
 \end{split}
\end{equation*}
Therefore, we conclude that
\begin{equation*}\label{grad-sigma0-tv-20}
\begin{split}
 &\sum_{j=1}^8|\mathfrak{J}_j|\lesssim  \|\dot{\Lambda}_h^{\sigma_0}\nabla\,v\|_{L^2}^2+E_3^{\frac{1}{2}}\dot{\mathcal{D}}_3,
 \end{split}
\end{equation*}
which leads to \eqref{grad-sigma0-tv-1}, and we complete the proof of Lemma \ref{lem-grad-sigma0-1}.
\end{proof}

\subsubsection{Estimate of $\|\partial_h^{N-1}\nabla\,v\|_{L^\infty_t(L^2)}$}

\begin{lem}\label{lem-tan-grad-N-1}
Let $N\geq 3$, under the assumption of Lemma \ref{lem-pseudo-energy-tv-1}, if $E_3(t) \leq 1$ for all the existence times $t$, then there holds
\begin{equation}\label{tan-grad-N-tv-1}
\begin{split}
 & \frac{\nu}{2}\frac{d}{dt}\mathring{\mathcal{E}}(\partial_h^{N-1}\nabla\, v)+\|\partial_h^{N-1}\partial_t v\|_{L^2}^2 \lesssim \|\partial_h^{N-1}\nabla\,v\|_{L^2}^2+\dot{\mathcal{D}}_N^{\frac{1}{2}} (E_3^{\frac{1}{2}}\dot{\mathcal{D}}_{N}^{\frac{1}{2}}
   +E_{N}^{\frac{1}{2}}\,\dot{\mathcal{D}}_3^{\frac{1}{2}})
   +E_{N}\,\dot{\mathcal{D}}_3.
        \end{split}
\end{equation}
\end{lem}

\begin{proof}
We estimate all the integrals in the right hand side of \eqref{pseudo-energy-tv-1}, where we take $\mathcal{P}(\partial_h)=\partial_h^{N-1}$.
We first estimate $\mathfrak{J}_1$ to get
\begin{equation*}\label{tan-grad-N-tv-2}
\begin{split}
 &|\mathfrak{J}_1|\lesssim\|\partial_h^{N-1} v\|_{L^2}\|\partial_h^{N-1}\nabla\,\partial_t \mathcal{H}(\xi^1)\|_{L^2}\lesssim\|\partial_h^{N-1} v\|_{L^2}\|\partial_h^{N-1}\nabla\,v\|_{L^2}\lesssim \|\partial_h^{N-1}\nabla\,v\|_{L^2}^2.
          \end{split}
\end{equation*}
Notice that
\begin{equation*}\label{tan-grad-N-tv-3}
\begin{split}
 &|\mathfrak{J}_2|\lesssim \|\partial_h^{N-1}\partial_1\mathcal{Q}\|_{L^2} \|\partial_h^{N-1}(\widetilde{a_{\alpha\,1}} \partial_tv^\alpha)\|_{L^2}\\
 &\lesssim \|\partial_h^{N-1}\partial_1\mathcal{Q}\|_{L^2} (\|\partial_h^{N-1}\widetilde{a_{\alpha\,1}}\|_{L^\infty_{x_1}L^2_h} \|\partial_tv^h \|_{L^2_{x_1}L^\infty_h}+\|\widetilde{a_{\alpha\,1}}\|_{L^\infty} \|\partial_h^{N-1}\partial_tv^h \|_{L^2})\\
 &\lesssim \|\partial_h^{N-1}\partial_1\mathcal{Q}\|_{L^2} (E_3^{\frac{1}{2}}\dot{\mathcal{D}}_{N}^{\frac{1}{2}}
   +E_{N}^{\frac{1}{2}}\,\dot{\mathcal{D}}_3^{\frac{1}{2}}),\\
         \end{split}
\end{equation*}
Since $\mathcal{Q}= q-\mathcal{H}(\xi^1)  +2\nu\,\nabla_h\cdot v^h-\nu\,\mathcal{H}(\mathcal{B}_{9, i}^{\alpha})\partial_{\alpha}v^{i}$, we have
\begin{equation}\label{tan-grad-N-tv-4}
\begin{split}
& \|\partial_h^{N-1}\partial_1\mathcal{Q}\|_{L^2}\lesssim  \|\partial_h^{N-1}\partial_1q\|_{L^2}+ \|\partial_h^{N-1}\partial_1\mathcal{H}(\xi^1)\|_{L^2}  + \|\partial_h^{N-1}\partial_1\nabla_h\cdot v^h \|_{L^2}\\
 &\qquad\qquad\qquad\qquad\qquad\qquad\qquad\qquad+ \|\partial_h^{N-1}\partial_1(\mathcal{H}(\mathcal{B}_{9, i}^{\alpha})\partial_{\alpha}v^{i})\|_{L^2}\\
 &\lesssim  \|\partial_h^{N-1}\partial_1q\|_{L^2}+ \|\partial_h^{N-1} \xi^1 \|_{H^{\frac{1}{2}}(\Sigma_0)}  + \|\partial_h^{N-1}\partial_1\nabla_h\cdot v^h \|_{L^2}+ \|\partial_h^{N-1}\partial_1(\mathcal{H}(\mathcal{B}_{9, i}^{\alpha})\partial_{\alpha}v^{i})\|_{L^2}\\
 &\lesssim \dot{\mathcal{D}}_N^{\frac{1}{2}}
   +E_{N}^{\frac{1}{2}}\,\dot{\mathcal{D}}_3^{\frac{1}{2}},
\end{split}
\end{equation}
where we used the estimate \eqref{est-Bv-1} $ \|\partial_h^{N-1}(\mathcal{B}_{9, i}^{\alpha}\partial_{\alpha}v^{i})\|_{H^1}\lesssim E_3^{\frac{1}{2}}\dot{\mathcal{D}}_{N}^{\frac{1}{2}}
   +E_{N}^{\frac{1}{2}}\,\dot{\mathcal{D}}_3^{\frac{1}{2}}$
and
\begin{equation*}\label{tan-grad-N-tv-6}
\begin{split}
   \|\partial_h^{N-1}\xi^1\|_{H^{\frac{1}{2}}(\Sigma_0)}&\lesssim \|\partial_h^{N-1}q\|_{H^{\frac{1}{2}}(\Sigma_0)} +\|\,\partial_h^{N-1}\nabla_h\cdot v^h\|_{H^{\frac{1}{2}}(\Sigma_0)} +\|\partial_h^{N-1}(\mathcal{B}_{9, i}^{\alpha}\partial_{\alpha}v^{i})\|_{H^{\frac{1}{2}}(\Sigma_0)}\\
   &\lesssim \dot{\mathcal{D}}_N^{\frac{1}{2}}
   +E_{N}^{\frac{1}{2}}\,\dot{\mathcal{D}}_3^{\frac{1}{2}}.
         \end{split}
\end{equation*}
Hence, one has
\begin{equation*}\label{tan-grad-N-tv-7}
\begin{split}
 &|\mathfrak{J}_2|\lesssim (\dot{\mathcal{D}}_N^{\frac{1}{2}}
   +E_{N}^{\frac{1}{2}}\,\dot{\mathcal{D}}_3^{\frac{1}{2}})(E_3^{\frac{1}{2}}\dot{\mathcal{D}}_{N}^{\frac{1}{2}}
   +E_{N}^{\frac{1}{2}}\,\dot{\mathcal{D}}_3^{\frac{1}{2}})\lesssim \dot{\mathcal{D}}_N^{\frac{1}{2}} (E_3^{\frac{1}{2}}\dot{\mathcal{D}}_{N}^{\frac{1}{2}}
   +E_{N}^{\frac{1}{2}}\,\dot{\mathcal{D}}_3^{\frac{1}{2}})
   +E_{N}\,\dot{\mathcal{D}}_3.
         \end{split}
\end{equation*}
For $\mathfrak{J}_3$, we get
\begin{equation*}\label{tan-grad-N-tv-8}
\begin{split}
   &|\mathfrak{J}_3|\lesssim \|\partial_h^{N-1}\partial_t v\|_{L^2} (\|\partial_h^{N-1} \nabla\,(\widetilde{a_{\alpha\,1}}\partial_1v^\alpha)\|_{L^2}+\|\partial_h^{N-1} \nabla\,(\mathcal{B}_{5, i}^{\alpha}\partial_{\alpha}v^{i})\|_{L^2}),
           \end{split}
\end{equation*}
which along with Lemma \ref{lem-est-B-Bv-1} leads to
\begin{equation*}\label{tan-grad-N-tv-12}
\begin{split}
   &|\mathfrak{J}_3|\lesssim \dot{\mathcal{D}}_N^{\frac{1}{2}}(E_3^{\frac{1}{2}}\dot{\mathcal{D}}_N^{\frac{1}{2}}
    +E_N^{\frac{1}{2}}\dot{\mathcal{D}}_3^{\frac{1}{2}}).
           \end{split}
\end{equation*}
This estimate also holds true for $\mathfrak{J}_4$:
\begin{equation*}\label{tan-grad-N-tv-13}
\begin{split}
   &|\mathfrak{J}_4|\lesssim \dot{\mathcal{D}}_N^{\frac{1}{2}}(E_3^{\frac{1}{2}}\dot{\mathcal{D}}_N^{\frac{1}{2}}
    +E_N^{\frac{1}{2}}\dot{\mathcal{D}}_3^{\frac{1}{2}}).
           \end{split}
\end{equation*}
For $\mathfrak{J}_5$, we first get
\begin{equation*}\label{tan-grad-N-tv-14}
\begin{split}
 &|\int_{ \Omega}\partial_h^{N-1}\mathcal{Q}\,
 \partial_h^{N-1}(\partial_1\widetilde{a_{\alpha\,1}} \partial_tv^\alpha)\, dx|\lesssim \|\partial_h^{N-1}\mathcal{Q}\|_{L^\infty_{x_1}L^2_h} \|\partial_h^{N-1}(\partial_1\widetilde{a_{\alpha\,1}} \partial_tv^\alpha)\|_{L^1_{x_1}L^2_h} \\
 &\lesssim \|\partial_h^{N-1}\partial_1\mathcal{Q}\|_{L^2} ( \|\partial_h^{N-1}\partial_1\widetilde{a_{\alpha\,1}}\|_{L^2} \|\partial_tv^\alpha\|_{L^2_{x_1}L^\infty_h}+ \|\partial_h^{N-1}\partial_tv^\alpha\|_{L^2}
\|\partial_1\widetilde{a_{\alpha\,1}}\|_{L^2_{x_1}L^\infty_h}),
         \end{split}
\end{equation*}
which along with the estimate \eqref{tan-grad-N-tv-4} yields that
\begin{equation}\label{tan-grad-N-tv-15}
\begin{split}
 &|\int_{ \Omega}\partial_h^{N-1}\mathcal{Q}\,
 \partial_h^{N-1}(\partial_1\widetilde{a_{\alpha\,1}} \partial_tv^\alpha)\, dx|\\
&\lesssim (\dot{\mathcal{D}}_N^{\frac{1}{2}}+E_N^{\frac{1}{2}}\dot{\mathcal{D}}_3^{\frac{1}{2}})(E_3^{\frac{1}{2}}\dot{\mathcal{D}}_N^{\frac{1}{2}}
    +E_N^{\frac{1}{2}}\dot{\mathcal{D}}_3^{\frac{1}{2}})\lesssim \dot{\mathcal{D}}_N^{\frac{1}{2}} (E_3^{\frac{1}{2}}\dot{\mathcal{D}}_{N}^{\frac{1}{2}}
   +E_{N}^{\frac{1}{2}}\,\dot{\mathcal{D}}_3^{\frac{1}{2}})
   +E_{N}\,\dot{\mathcal{D}}_3,
         \end{split}
\end{equation}
Repeating the argument in the proof of \eqref{tan-grad-N-tv-15}, one can immediately obtain
\begin{equation*}\label{tan-grad-N-tv-16}
\begin{split}
 &|\int_{ \Omega}\partial_h^{N-1}\mathcal{Q} \,\,\partial_h^{N-1}(-\partial_t\widetilde{a_{\alpha\,1}}\,\partial_1v^\alpha
+\partial_t\mathcal{B}_{5, i}^{\alpha}\,\partial_{\alpha}v^{i})\, dx|\\
&\lesssim\|\partial_h^{N}\mathcal{Q} \|_{L^2} \|\partial_h^{N-2}(-\partial_t\widetilde{a_{\alpha\,1}}\,\partial_1v^\alpha
+\partial_t\mathcal{B}_{5, i}^{\alpha}\,\partial_{\alpha}v^{i})\|_{L^2}\\
&\lesssim\|\partial_h^{N}\mathcal{Q} \|_{L^2} (\|\partial_h^{N-2}(\partial_t\widetilde{a_{\alpha\,1}}, \partial_t\mathcal{B}_{5, i}^{\alpha})\|_{L^2} \|(\partial_1v^h, \partial_{\alpha}v^{i})\|_{L^\infty}\\
&\qquad\qquad\qquad\qquad+\|\partial_h^{N-2}(\partial_1v^h, \partial_{\alpha}v^{i})\|_{L^2} \|(\partial_t\widetilde{a_{21}}, \partial_t\widetilde{a_{31}}, \partial_t\mathcal{B}_{5, i}^{\alpha})\|_{L^\infty})\\
&\lesssim  (\dot{\mathcal{D}}_N^{\frac{1}{2}}+E_N^{\frac{1}{2}}\dot{\mathcal{D}}_3^{\frac{1}{2}})(E_3^{\frac{1}{2}}\dot{\mathcal{D}}_N^{\frac{1}{2}}
    +E_N^{\frac{1}{2}}\dot{\mathcal{D}}_3^{\frac{1}{2}})\lesssim  \dot{\mathcal{D}}_N^{\frac{1}{2}}(E_3^{\frac{1}{2}}\dot{\mathcal{D}}_N^{\frac{1}{2}}
    +E_N^{\frac{1}{2}}\dot{\mathcal{D}}_3^{\frac{1}{2}})+E_N\,\dot{\mathcal{D}}_3.
     \end{split}
\end{equation*}
Therefore, we have
\begin{equation*}\label{tan-grad-N-tv-17}
\begin{split}
 &|\mathfrak{J}_5| \lesssim \dot{\mathcal{D}}_N^{\frac{1}{2}} (E_3^{\frac{1}{2}}\dot{\mathcal{D}}_{N}^{\frac{1}{2}}
   +E_{N}^{\frac{1}{2}}\,\dot{\mathcal{D}}_3^{\frac{1}{2}})
   +E_{N}\,\dot{\mathcal{D}}_3.
         \end{split}
\end{equation*}
The same conclusion can be drawn for $\mathfrak{J}_6$, $\mathfrak{J}_7$, and $\mathfrak{J}_8$
\begin{equation*}\label{tan-grad-N-tv-18}
\begin{split}
 &|\mathfrak{J}_6|\lesssim \dot{\mathcal{D}}_N^{\frac{1}{2}} (E_3^{\frac{1}{2}}\dot{\mathcal{D}}_{N}^{\frac{1}{2}}
   +E_{N}^{\frac{1}{2}}\,\dot{\mathcal{D}}_3^{\frac{1}{2}})
   +E_{N}\,\dot{\mathcal{D}}_3, \quad|\mathfrak{J}_7|+|\mathfrak{J}_8|\lesssim \dot{\mathcal{D}}_N^{\frac{1}{2}} (E_3^{\frac{1}{2}}\dot{\mathcal{D}}_{N}^{\frac{1}{2}}
   +E_{N}^{\frac{1}{2}}\,\dot{\mathcal{D}}_3^{\frac{1}{2}}).
         \end{split}
\end{equation*}
Therefore, we obtain
\begin{equation*}\label{tan-grad-N-tv-25}
\begin{split}
 &\sum_{j=1}^8|\mathfrak{J}_j|\lesssim \|\partial_h^{N-1}\nabla\,v\|_{L^2}^2+\dot{\mathcal{D}}_N^{\frac{1}{2}} (E_3^{\frac{1}{2}}\dot{\mathcal{D}}_{N}^{\frac{1}{2}}
   +E_{N}^{\frac{1}{2}}\,\dot{\mathcal{D}}_3^{\frac{1}{2}})
   +E_{N}\,\dot{\mathcal{D}}_3,
          \end{split}
\end{equation*}
and then reach \eqref{tan-grad-N-tv-1}, which completes the proof of Lemma \ref{lem-tan-grad-N-1}.
\end{proof}
With Lemmas \ref{lem-grad-sigma0-1} and \ref{lem-tan-grad-N-1} in hand, we have
\begin{lem}\label{lem-tan-grad-total-1}
Let $N\geq 3$, under the assumption of Lemma \ref{lem-tan-pseudo-energy-1}, if $E_3(t) \leq 1$ for all the existence times $t$, then there holds
\begin{equation*}\label{tan-grad-total-1}
\begin{split}
&\frac{d}{dt}\bigg(\mathring{\mathcal{E}}(\dot{\Lambda}_h^{\sigma_0}\nabla\, v) +\sum_{k=1}^{N-1}\mathring{\mathcal{E}}(\partial_h^{k}\nabla\, v)\bigg)+ c_2(\|\dot{\Lambda}_h^{\sigma_0}\partial_t v\|_{L^2}^2 +\sum_{k=1}^{N-1}\|\partial_h^{k}\partial_t v\|_{L^2}^2 )\\
&\leq C_1 (\|\dot{\Lambda}_h^{\sigma_0}\nabla\,v\|_{L^2}^2 +\sum_{k=1}^{N-1}\|\partial_h^{k}\nabla\,v\|_{L^2}^2 + E_N^{\frac{1}{2}} \dot{\mathcal{D}}_3^{\frac{1}{2}}\dot{\mathcal{D}}_N^{\frac{1}{2}}+  E_3^{\frac{1}{2}} \dot{\mathcal{D}}_N+E_N\,\dot{\mathcal{D}}_3).
\end{split}
\end{equation*}
\end{lem}
Thanks to Lemmas \ref{lem-tan-decay-total-1} and \ref{lem-tan-grad-total-1}, for any small positive constant $\delta \leq \min\{1, \frac{c_1}{2C_1}\}$ (which will be determined later), we have
\begin{lem}\label{lem-decay-total-1}
Let $N\geq 3$, under the assumption of Lemma \ref{lem-tan-pseudo-energy-1}, if $E_3(t) \leq 1$ for all the existence times $t$, then there holds
\begin{equation*}\label{tan-decay-total-1}
\begin{split}
&\frac{d}{dt}\bigg[\|\dot{\Lambda}_h^{\sigma_0}v\|_{L^2(\Omega)}^2
+\|\dot{\Lambda}_h^{\sigma_0}\xi^1\|_{L^2(\Sigma_0)}^2+\sum_{i=1}^N(\|\partial_h^{i}v\|_{L^2(\Omega)}^2
+\|\partial_h^{i}\xi^1\|_{L^2(\Sigma_0)}^2\\
&\quad+\delta\bigg(\mathring{\mathcal{E}}(\dot{\Lambda}_h^{\sigma_0}\nabla\, v) +\sum_{k=1}^{N-1}\mathring{\mathcal{E}}(\partial_h^{k}\nabla\, v)\bigg)\bigg] \\
&\qquad+ \frac{c_1}{2}(\|\dot{\Lambda}_h^{\sigma_0}\nabla\,v\|_{L^2(\Omega)}^2 +\sum_{i=1}^N\|\partial_h^{i}\nabla\,v\|_{L^2(\Omega)}^2 )+\delta\,c_2(\|\dot{\Lambda}_h^{\sigma_0}\partial_t v\|_{L^2}^2 +\sum_{k=1}^{N-1}\|\partial_h^{k}\partial_t v\|_{L^2}^2 )\\
&\lesssim E_N^{\frac{1}{2}} \dot{\mathcal{D}}_3^{\frac{1}{2}}\dot{\mathcal{D}}_N^{\frac{1}{2}}+  E_3^{\frac{1}{2}} \dot{\mathcal{D}}_N+E_N\,\dot{\mathcal{D}}_3.
\end{split}
\end{equation*}
\end{lem}

\subsubsection{Estimate of $\|\dot{\Lambda}_h^{-\lambda}\nabla\,v\|_{L^\infty_t(L^2)}$}

\begin{lem}\label{lem-grad-bdd-lambda-1}
Under the assumption of Lemma \ref{lem-tan-pseudo-energy-1}, if$(\lambda,\,\sigma_0) \in (0, 1)$ satisfies $1-\lambda< \sigma_0\leq 1-\frac{1}{2}\lambda$, and $E_3(t) \leq 1$ for all the existence times $t$, then there holds
\begin{equation}\label{grad-lambda-tv-1}
\begin{split}
 & \frac{d}{dt}\mathring{\mathcal{E}}(\dot{\Lambda}_h^{-\lambda}\nabla\, v)+c_1\|\dot{\Lambda}_h^{-\lambda}\partial_t v\|_{L^2}^2 \\
 &\leq C_1\bigg(\|\dot{\Lambda}_h^{-\lambda}\nabla\,v\|_{L^2}^2+E_3^{\frac{1}{2}} \dot{\mathcal{D}}_3^{\frac{1}{2}}\|\dot{\Lambda}_h^{-\lambda}\xi^1 \|_{H^{\frac{1}{2}}(\Sigma_0)}+E_3\|\dot{\Lambda}_h^{\sigma_0}\xi^1\|_{L^2(\Sigma_0)}^2+E_3^{\frac{1}{2}}\, \dot{\mathcal{D}}_3\bigg).
     \end{split}
\end{equation}
\end{lem}
\begin{proof}
Taking $\mathcal{P}(\partial_h)=\dot{\Lambda}_h^{-\lambda}$ in  \eqref{pseudo-energy-tv-1},  we will estimate all the integrals in the right hand side of \eqref{pseudo-energy-tv-1}.

For $\mathfrak{J}_1:=\int_{\Omega}\dot{\Lambda}_h^{-\lambda} v\,\cdot\,\nabla\,\dot{\Lambda}_h^{-\lambda}\partial_t \mathcal{H}(\xi^1)  \,dx$, it is easy to see
\begin{equation*}\label{grad-lambda-tv-2}
\begin{split}
 &|\mathfrak{J}_1|\lesssim\|\dot{\Lambda}_h^{-\lambda}v\|_{L^2}\|\dot{\Lambda}_h^{-\lambda}\nabla\,\partial_t \mathcal{H}(\xi^1)\|_{L^2}\lesssim\| \dot{\Lambda}_h^{-\lambda}v\|_{L^2}\|\dot{\Lambda}_h^{-\lambda}\nabla\,v\|_{L^2}\lesssim \|\dot{\Lambda}_h^{-\lambda}\nabla\,v\|_{L^2}^2.
          \end{split}
\end{equation*}
While the product law \eqref{product-law-1} ensures that
\begin{equation*}\label{grad-lambda-tv-3}
\begin{split}
 |\mathfrak{J}_2|&\lesssim \|\dot{\Lambda}_h^{-\lambda}\partial_1\mathcal{Q}\|_{L^2} \|\dot{\Lambda}_h^{-\lambda}(\widetilde{a_{\alpha\,1}} \partial_tv^\alpha)\|_{L^2} \\
 &\lesssim \|\dot{\Lambda}_h^{-\lambda}\partial_1\mathcal{Q}\|_{L^2}  \|\dot{\Lambda}_h^{1-\lambda-\sigma_0}(\widetilde{a_{\alpha\,1}})\|_{L^\infty_{x_1}L^2_h}
 \|\dot{\Lambda}_h^{\sigma_0}\partial_tv^h \|_{L^2}.
         \end{split}
\end{equation*}
Thanks to $\mathcal{Q}= q-\mathcal{H}(\xi^1)  +2\nu\,\nabla_h\cdot v^h-\nu\,\mathcal{H}(\mathcal{B}_{9, i}^{\alpha})\partial_{\alpha}v^{i}$,
we have
\begin{equation*}\label{grad-lambda-tv-4}
\begin{split}
& \|\dot{\Lambda}_h^{-\lambda}\partial_1\mathcal{Q}\|_{L^2}\lesssim  \|\dot{\Lambda}_h^{-\lambda}\partial_1q\|_{L^2}+ \|\dot{\Lambda}_h^{-\lambda}\partial_1\mathcal{H}(\xi^1)\|_{L^2}  + \|\dot{\Lambda}_h^{-\lambda}\partial_1\nabla_h\cdot v^h \|_{L^2}\\
 &\qquad\qquad\qquad\qquad\qquad\qquad\qquad\qquad+ \|\dot{\Lambda}_h^{-\lambda}\partial_1(\mathcal{H}(\mathcal{B}_{9, i}^{\alpha})\partial_{\alpha}v^{i})\|_{L^2}.
\end{split}
\end{equation*}
According to the equation of $\partial_1q$, we get from \eqref{est-g-2} that
\begin{equation*}\label{grad-lambda-tv-3a}
\begin{split}
 &\|\dot{\Lambda}_h^{-\lambda}\partial_1q\|_{L^2}\lesssim\|\dot{\Lambda}_h^{-\lambda}\partial_t v^{1}\|_{L^2}
 +\|\dot{\Lambda}_h^{-\lambda}\partial_h\,\nabla\,v\|_{L^2}+\|\dot{\Lambda}_h^{-\lambda}g_1\|_{L^2}
  +\|\dot{\Lambda}_h^{-\lambda}\widetilde{g}_{11}\|_{L^2}\\
  &\lesssim\|\dot{\Lambda}_h^{-\lambda}\partial_t v^{1}\|_{L^2}
 +\|\dot{\Lambda}_h^{-\lambda}\partial_h\,\nabla\,v\|_{L^2}+E_3^{\frac{1}{2}} \|\dot{\Lambda}_h^{2-\sigma_0-\lambda}\xi^1\|_{L^2(\Sigma_0)} +E_3^{\frac{1}{2}}\, \dot{\mathcal{D}}_3^{\frac{1}{2}},\\
       \end{split}
\end{equation*}
which, together with
\begin{equation*}\label{grad-lambda-tv-3b}
\begin{split}
 &\|\dot{\Lambda}_h^{-\lambda}\partial_1(\mathcal{H}(\mathcal{B}_{9, i}^{\alpha})\partial_{\alpha}v^{i})\|_{L^2}\lesssim\|\dot{\Lambda}_h^{\sigma_0} v\|_{H^1}\|\dot{\Lambda}_h^{2-\sigma_0-\lambda}\mathcal{B}_{9, i}^{\alpha}\|_{H^1} \lesssim\,E_3^{\frac{1}{2}}\, \dot{\mathcal{D}}_3^{\frac{1}{2}}
       \end{split}
\end{equation*}
(where we used the fact $ 2-\lambda-\sigma_0\geq \sigma_0$ since $ \sigma_0\leq 1-\frac{\lambda}{2}$), follows that
\begin{equation}\label{grad-lambda-tv-3c}
\begin{split}
\|\dot{\Lambda}_h^{-\lambda}\partial_1\mathcal{Q}\|_{L^2}\lesssim &\|\dot{\Lambda}_h^{-\lambda}\xi^1 \|_{H^{\frac{1}{2}}(\Sigma_0)}+E_3^{\frac{1}{2}} \|\dot{\Lambda}_h^{2-\sigma_0-\lambda}\xi^1\|_{L^2(\Sigma_0)} \\
 &+ \|\dot{\Lambda}_h^{-\lambda}\partial_t v^{1}\|_{L^2}
 +\|\dot{\Lambda}_h^{-\lambda}\partial_h\,\nabla\,v\|_{L^2}+E_3^{\frac{1}{2}}\, \dot{\mathcal{D}}_3^{\frac{1}{2}}\\
 \lesssim &\|\dot{\Lambda}_h^{-\lambda}\xi^1 \|_{H^{\frac{1}{2}}(\Sigma_0)}+ \|\dot{\Lambda}_h^{-\lambda}\partial_t v^{1}\|_{L^2}
 +\|\dot{\Lambda}_h^{-\lambda}\partial_h\,\nabla\,v\|_{L^2}+E_3^{\frac{1}{2}}\, \dot{\mathcal{D}}_3^{\frac{1}{2}},
\end{split}
\end{equation}
where the interpolation inequality $\|\dot{\Lambda}_h^{2-\sigma_0-\lambda}\xi^1\|_{L^2(\Sigma_0)}\lesssim \|\dot{\Lambda}_h^{-\lambda}\xi^1\|_{L^2(\Sigma_0)}+\|\dot{\Lambda}_h^{1+\sigma_0}\xi^1\|_{L^2(\Sigma_0)}$ is used in the last inequality.
Therefore, we obtain, from
\begin{equation*}\label{grad-lambda-tv-3d}
\begin{split}
&\|\dot{\Lambda}_h^{1-\lambda-\sigma_0}(\widetilde{a_{21}},\widetilde{a_{31}})\|_{L^\infty_{x_1}L^2_h}
 \lesssim \|\dot{\Lambda}_h^{2-\lambda-2\sigma_0}\dot{\Lambda}_h^{\sigma_0-1}(\widetilde{a_{21}},\widetilde{a_{31}})\|_{H^1} \lesssim E_3^{\frac{1}{2}}
 \end{split}
\end{equation*}
(where we used the fact $ 2-\lambda-2\sigma_0\geq 0$ since $ \sigma_0\leq 1-\frac{\lambda}{2}$),  that
 \begin{equation*}\label{grad-lambda-tv-5}
\begin{split}
&|\mathfrak{J}_2|\lesssim  E_3^{\frac{1}{2}} \dot{\mathcal{D}}_3^{\frac{1}{2}}\bigg(\|\dot{\Lambda}_h^{-\lambda}\xi^1 \|_{H^{\frac{1}{2}}(\Sigma_0)}+ \|\dot{\Lambda}_h^{-\lambda}\partial_t v^{1}\|_{L^2}
 +\|\dot{\Lambda}_h^{-\lambda}\partial_h\,\nabla\,v\|_{L^2}+E_3^{\frac{1}{2}}\, \dot{\mathcal{D}}_3^{\frac{1}{2}}\bigg)\\
 &\lesssim  E_3^{\frac{1}{2}} \dot{\mathcal{D}}_3^{\frac{1}{2}}\|\dot{\Lambda}_h^{-\lambda}\xi^1 \|_{H^{\frac{1}{2}}(\Sigma_0)}+ E_3^{\frac{1}{2}} \dot{\mathcal{D}}_3^{\frac{1}{2}}\|\dot{\Lambda}_h^{-\lambda}\partial_t v^{1}\|_{L^2}
 +E_3^{\frac{1}{2}} \dot{\mathcal{D}}_3^{\frac{1}{2}}\|\dot{\Lambda}_h^{-\lambda}\partial_h\,\nabla\,v\|_{L^2}+E_3^{\frac{1}{2}}\, \dot{\mathcal{D}}_3.
         \end{split}
\end{equation*}
For $\mathfrak{J}_3$, one has
 \begin{equation*}\label{grad-lambda-tv-6}
\begin{split}
   &|\mathfrak{J}_3|\lesssim \|\dot{\Lambda}_h^{-\lambda}\partial_t v\|_{L^2} (\| \dot{\Lambda}_h^{-\lambda}\nabla\,(\widetilde{a_{\alpha\,1}}\partial_1v^\alpha)\|_{L^2}+\| \dot{\Lambda}_h^{-\lambda}\nabla\,(\mathcal{B}_{5, i}^{\alpha}\partial_{\alpha}v^{i})\|_{L^2}),\\
           \end{split}
\end{equation*}
Notice that
 \begin{equation*}\label{grad-lambda-tv-7}
\begin{split}
   &\|\dot{\Lambda}_h^{-\lambda}\nabla\,(\widetilde{a_{\alpha\,1}}\partial_1v^\alpha)\|_{L^2} \lesssim \|\dot{\Lambda}_h^{-\lambda}(\widetilde{a_{\alpha\,1}} \nabla\,\partial_1v^\alpha)\|_{L^2}+\|\dot{\Lambda}_h^{-\lambda}( \nabla\,\widetilde{a_{\alpha\,1}}\partial_1v^\alpha)\|_{L^2}\\
   & \lesssim \|\dot{\Lambda}_h^{\sigma_0-1}\widetilde{a_{\alpha\,1}}\|_{L^\infty_{x_1}L^2_h} \|\dot{\Lambda}_h^{2-\sigma_0-\lambda}\nabla\,\partial_1v^\alpha\|_{L^2}+\| \|\dot{\Lambda}_h^{\sigma_0-1}\nabla\,\widetilde{a_{\alpha\,1}}\|_{L^2}
   \|\dot{\Lambda}_h^{2-\sigma_0-\lambda}\partial_1v^h\|_{L^\infty_{x-1}L^2_h}\\
   &\lesssim E_3^{\frac{1}{2}}\dot{\mathcal{D}}_3^{\frac{1}{2}}.
           \end{split}
\end{equation*}
Similarly, we have
 \begin{equation*}\label{grad-lambda-tv-8}
\begin{split}
   &\|\dot{\Lambda}_h^{-\lambda}\nabla\,(\mathcal{B}_{5, i}^{\alpha}\partial_{\alpha}v^{i})\|_{L^2})\lesssim E_3^{\frac{1}{2}}\dot{\mathcal{D}}_3^{\frac{1}{2}},\\
           \end{split}
\end{equation*}
which ensures that
 \begin{equation}\label{grad-lambda-tv-9}
\begin{split}
   &|\mathfrak{J}_3|\lesssim \|\dot{\Lambda}_h^{-\lambda}\partial_t v\|_{L^2} E_3^{\frac{1}{2}}\dot{\mathcal{D}}_3^{\frac{1}{2}}.
           \end{split}
\end{equation}
Similar to the proof of \eqref{grad-lambda-tv-9}, one can obtain
 \begin{equation*}\label{grad-lambda-tv-10}
\begin{split}
   &|\mathfrak{J}_4|\lesssim \|\dot{\Lambda}_h^{-\lambda}\partial_t v\|_{L^2} E_3^{\frac{1}{2}}\dot{\mathcal{D}}_3^{\frac{1}{2}}.
           \end{split}
\end{equation*}
We control $\mathfrak{J}_5$ by
 \begin{equation*}\label{grad-lambda-tv-11}
\begin{split}
 &|\mathfrak{J}_5|\lesssim \|\dot{\Lambda}_h^{-\lambda}\mathcal{Q}\|_{L^\infty_{x_1}L^2_h} (\|\dot{\Lambda}_h^{-\lambda}(\partial_1\widetilde{a_{\alpha\,1}} \partial_tv^\alpha-\partial_t\widetilde{a_{\alpha\,1}}\,\partial_1v^\alpha
+\partial_t\mathcal{B}_{5, i}^{\alpha}\,\partial_{\alpha}v^{i})\|_{L^1_{x_1}L^2_h})\\
 &\lesssim \|\dot{\Lambda}_h^{-\lambda}\partial_1\mathcal{Q}\|_{L^2} \bigg( \|\dot{\Lambda}_h^{1-\sigma_0-\lambda}\partial_1\widetilde{a_{\alpha\,1}}\|_{L^\infty_{x_1}L^2_h} \|\dot{\Lambda}_h^{\sigma_0}\partial_tv^h\|_{L^2}\\
 &+ \|\dot{\Lambda}_h^{1-\sigma_0-\lambda}\partial_t\widetilde{a_{\alpha\,1}}\|_{L^2} \|\dot{\Lambda}_h^{\sigma_0}\partial_1v^h\|_{L^\infty_{x_1}L^2_h}+ \|\dot{\Lambda}_h^{2-\sigma_0-\lambda}\partial_t\mathcal{B}_{5, i}^{\alpha}\|_{L^2} \|\dot{\Lambda}_h^{\sigma_0-1}\partial_{h}v\|_{L^\infty_{x_1}L^2_h}\bigg)\\
 &\lesssim \|\dot{\Lambda}_h^{-\lambda}\partial_1\mathcal{Q}\|_{L^2}    E_3^{\frac{1}{2}}\dot{\mathcal{D}}_3^{\frac{1}{2}},
         \end{split}
\end{equation*}
which along with \eqref{grad-lambda-tv-3c} follows that
 \begin{equation*}\label{grad-lambda-tv-12}
\begin{split}
 &|\mathfrak{J}_5| \lesssim  E_3^{\frac{1}{2}} \dot{\mathcal{D}}_3^{\frac{1}{2}}\|\dot{\Lambda}_h^{-\lambda}\xi^1 \|_{H^{\frac{1}{2}}(\Sigma_0)}+ E_3^{\frac{1}{2}} \dot{\mathcal{D}}_3^{\frac{1}{2}}\|\dot{\Lambda}_h^{-\lambda}\partial_t v^{1}\|_{L^2}
 +E_3^{\frac{1}{2}} \dot{\mathcal{D}}_3^{\frac{1}{2}}(\|\dot{\Lambda}_h^{-\lambda}\partial_h\,\nabla\,v\|_{L^2}+ \dot{\mathcal{D}}_3^{\frac{1}{2}}).
         \end{split}
\end{equation*}
Similarly, for $\mathfrak{J}_6$, we have
 \begin{equation*}\label{grad-lambda-tv-13}
\begin{split}
 &|\mathfrak{J}_6|\lesssim\,E_3^{\frac{1}{2}} \dot{\mathcal{D}}_3^{\frac{1}{2}}\|\dot{\Lambda}_h^{-\lambda}\xi^1 \|_{H^{\frac{1}{2}}(\Sigma_0)}+ E_3^{\frac{1}{2}} \dot{\mathcal{D}}_3^{\frac{1}{2}}\|\dot{\Lambda}_h^{-\lambda}\partial_t v^{1}\|_{L^2}
 +E_3^{\frac{1}{2}} \dot{\mathcal{D}}_3^{\frac{1}{2}}(\|\dot{\Lambda}_h^{-\lambda}\partial_h\,\nabla\,v\|_{L^2}+ \dot{\mathcal{D}}_3^{\frac{1}{2}}).
\end{split}
\end{equation*}
For $\mathfrak{J}_7$, we bound the integral $\int_{\Omega}\dot{\Lambda}_h^{-\lambda}\partial_1\partial_{\alpha} v^j\,\dot{\Lambda}_h^{-\lambda}\partial_t ( \mathcal{H}(\mathcal{B}_{j+5, i}^{\alpha})v^{i}) \,dx$ in it by
\begin{equation*}
\begin{split}
  &|\int_{\Omega} \dot{\Lambda}_h^{-\lambda}\partial_1\partial_{\alpha} v^j\,\dot{\Lambda}_h^{-\lambda}\partial_t ( \mathcal{H}(\mathcal{B}_{j+5, i}^{\alpha})v^{i}) \,dx|\\
  &\lesssim \|\dot{\Lambda}_h^{\sigma_0-1}\partial_1\partial_{\alpha} v^j\|_{L^\infty_{x_1}L^2_h}\|\dot{\Lambda}_h^{1-\sigma_0-2\lambda}\partial_t ( \mathcal{H}(\mathcal{B}_{j+5, i}^{\alpha})v^{i}) \|_{L^1_{x_1}L^2_h}\\
  &\lesssim \|\dot{\Lambda}_h^{\sigma_0}\partial_1 v\|_{H^1}(\|\dot{\Lambda}_h^{1-\sigma_0-2\lambda}(\mathcal{H}(\mathcal{B}_{j+5, i}^{\alpha})\partial_tv^{i})\|_{L^1_{x_1}L^2_h}+\|\dot{\Lambda}_h^{1-\sigma_0-2\lambda}(\partial_t \mathcal{H}(\mathcal{B}_{j+5, i}^{\alpha})v^{i})\|_{L^1_{x_1}L^2_h}),
              \end{split}
\end{equation*}
which, together with $2-2\sigma_0-2\lambda \geq -\lambda$ and
\begin{equation*}
\begin{split}
  &\|\dot{\Lambda}_h^{1-\sigma_0-2\lambda}(\mathcal{H}(\mathcal{B}_{j+5, i}^{\alpha})\partial_tv^{i})\|_{L^1_{x_1}L^2_h}\lesssim\|\dot{\Lambda}_h^{2-\sigma_0-\lambda}\mathcal{H}(\mathcal{B}_{j+5, i}^{\alpha})\|_{L^2}\|\dot{\Lambda}_h^{-\lambda}\partial_tv\|_{L^2}\lesssim  E_3^{\frac{1}{2}}\|\dot{\Lambda}_h^{-\lambda}\partial_tv\|_{L^2},\\
  &\|\dot{\Lambda}_h^{1-\sigma_0-2\lambda}(\partial_t \mathcal{H}(\mathcal{B}_{j+5, i}^{\alpha})v^{i})\|_{L^1_{x_1}L^2_h}\lesssim \|\dot{\Lambda}_h^{\sigma_0}\partial_t \mathcal{H}(\mathcal{B}_{j+5, i}^{\alpha}) \|_{L^2}\|\dot{\Lambda}_h^{2-2\sigma_0-2\lambda} v\|_{L^2}\lesssim \dot{\mathcal{D}}_3^{\frac{1}{2}}E_3^{\frac{1}{2}},
              \end{split}
\end{equation*}
follows that
\begin{equation*}\label{grad-lambda-tv-15}
\begin{split}
  &|\int_{\Omega}\partial_1\partial_{\alpha} v^j\,\partial_t ( \mathcal{H}(\mathcal{B}_{j+5, i}^{\alpha})v^{i}) \,dx|\lesssim \dot{\mathcal{D}}_3^{\frac{1}{2}}  E_3^{\frac{1}{2}}\|\dot{\Lambda}_h^{-\lambda}\partial_tv\|_{L^2}
    +\dot{\mathcal{D}}_3E_3^{\frac{1}{2}}.
              \end{split}
\end{equation*}
Similarly, we have
\begin{equation*}\label{grad-lambda-tv-17}
\begin{split}
  &|\int_{\Omega}\dot{\Lambda}_h^{\sigma_0}\partial_1  v^j\,\dot{\Lambda}_h^{\sigma_0} \partial_t ( \partial_{\alpha} \mathcal{H}(\mathcal{B}_{j+5, i}^{\alpha})v^{i}) \,dx|\lesssim \dot{\mathcal{D}}_3^{\frac{1}{2}}  E_3^{\frac{1}{2}}\|\dot{\Lambda}_h^{-\lambda}\partial_tv\|_{L^2}
    +\dot{\mathcal{D}}_3E_3^{\frac{1}{2}},\\
  &|\int_{\Omega}\dot{\Lambda}_h^{\sigma_0}\partial_t v^j\, \dot{\Lambda}_h^{\sigma_0}\partial_1(\mathcal{H}(\mathcal{B}_{j+5, i}^{\alpha})\partial_{\alpha}v^{i})\,dx|\lesssim \dot{\mathcal{D}}_3^{\frac{1}{2}}  E_3^{\frac{1}{2}}\|\dot{\Lambda}_h^{-\lambda}\partial_tv\|_{L^2}
    +\dot{\mathcal{D}}_3E_3^{\frac{1}{2}}.
          \end{split}
\end{equation*}
Hence, we have
\begin{equation*}\label{grad-lambda-tv-18}
\begin{split}
  &|\mathfrak{J}_7|\lesssim \dot{\mathcal{D}}_3^{\frac{1}{2}}  E_3^{\frac{1}{2}}\|\dot{\Lambda}_h^{-\lambda}\partial_tv\|_{L^2}
    +\dot{\mathcal{D}}_3E_3^{\frac{1}{2}}.
          \end{split}
\end{equation*}
Finally, due to \eqref{est-g-2}, we find
\begin{equation*}\label{grad-lambda-tv-19}
\begin{split}
 &|\mathfrak{J}_8|\lesssim\|\dot{\Lambda}_h^{-\lambda}\partial_t v\|_{L^2}\|\dot{\Lambda}_h^{-\lambda}g\|_{L^2}\lesssim  E_3^{\frac{1}{2}}\|\dot{\Lambda}_h^{-\lambda}\partial_t v\|_{L^2}(\|\dot{\Lambda}_h^{2-\sigma_0-\lambda}\xi^1\|_{L^2(\Sigma_0)} + \dot{\mathcal{D}}_3^{\frac{1}{2}}).
 \end{split}
\end{equation*}
Therefore, we obtain that
\begin{equation*}\label{grad-lambda-tv-20}
\begin{split}
 \sum_{j=1}^8|\mathfrak{J}_j|\lesssim  &\|\dot{\Lambda}_h^{-\lambda}\nabla\,v\|_{L^2}^2+E_3^{\frac{1}{2}} \dot{\mathcal{D}}_3^{\frac{1}{2}}\|\dot{\Lambda}_h^{-\lambda}\xi^1 \|_{H^{\frac{1}{2}}(\Sigma_0)}+ E_3^{\frac{1}{2}} \dot{\mathcal{D}}_3^{\frac{1}{2}}\|\dot{\Lambda}_h^{-\lambda}\partial_t v^{1}\|_{L^2}\\
 &
 +E_3^{\frac{1}{2}} \dot{\mathcal{D}}_3^{\frac{1}{2}}\|\dot{\Lambda}_h^{-\lambda}\partial_h\,\nabla\,v\|_{L^2}+E_3^{\frac{1}{2}}\, \dot{\mathcal{D}}_3+E_3^{\frac{1}{2}}\|\dot{\Lambda}_h^{-\lambda}\partial_t v\|_{L^2}\|\dot{\Lambda}_h^{2-\sigma_0-\lambda}\xi^1\|_{L^2(\Sigma_0)},
 \end{split}
\end{equation*}
which together with Young's inequality applied leads to \eqref{grad-lambda-tv-1}, and the desired result is proved.
\end{proof}

\subsubsection{Estimate of $\|\nabla\,v\|_{L^\infty_t(L^2)}$}

\begin{lem}\label{lem-grad0-bdd-0-1}
Under the assumption of Lemma \ref{lem-pseudo-energy-tv-1}, if $E_3(t) \leq 1$ for all the existence times $t$, then there holds
\begin{equation}\label{grad-0-tv-1}
\begin{split}
 & \frac{d}{dt}\mathring{\mathcal{E}}(\nabla\, v)+c_1\|\partial_t v\|_{L^2}^2 \\
 &\lesssim \|\nabla\,v\|_{L^2}^2+ E_3^{\frac{1}{2}} \dot{\mathcal{D}}_3^{\frac{1}{2}} \|\dot{\Lambda}_h^{2(1-\sigma_0)}\xi^1 \|_{H^{\frac{1}{2}}(\Sigma_0)}
 + E_3^{\frac{1}{2}} \dot{\mathcal{D}}_3+E_3 \|\dot{\Lambda}_h^{\sigma_0}\xi^1 \|_{H^{\frac{3}{2}}(\Sigma_0)}^2.
     \end{split}
\end{equation}
\end{lem}

With Lemmas \ref{lem-tan-grad-total-1}, \ref{lem-grad-bdd-lambda-1} and \ref{lem-grad0-bdd-0-1} in hand, we have
\begin{lem}\label{lem-tan-grad-N+1-1}
Let $N\geq 3$, under the assumption of Lemma \ref{lem-tan-pseudo-energy-1}, if$(\lambda,\,\sigma_0) \in (0, 1)$ satisfies $1-\lambda< \sigma_0\leq 1-\frac{1}{2}\lambda$, and $E_3(t) \leq 1$ for all the existence times $t$, then there holds
\begin{equation}\label{tan-grad-N+1-1}
\begin{split}
&\frac{d}{dt}\bigg(\mathring{\mathcal{E}}(\dot{\Lambda}_h^{-\lambda}\nabla\, v)+\mathring{\mathcal{E}}(\dot{\Lambda}_h^{\sigma_0}\nabla\, v) +\sum_{k=1}^{N}\mathring{\mathcal{E}}(\partial_h^{k}\nabla\, v)\bigg)\\
&\quad+ c_2\bigg(\|\dot{\Lambda}_h^{-\lambda}\partial_t v\|_{L^2}^2+\|\dot{\Lambda}_h^{\sigma_0}\partial_t v\|_{L^2}^2 +\sum_{k=0}^{N}\|\partial_h^{k}\partial_t v\|_{L^2}^2 \bigg)\\
&\leq C_1 \bigg(\|\dot{\Lambda}_h^{-\lambda}\nabla\,v\|_{L^2}^2+\|\dot{\Lambda}_h^{\sigma_0}\nabla\,v\|_{L^2}^2 +\sum_{k=0}^{N}\|\partial_h^{k}\nabla\,v\|_{L^2}^2 + E_{N+1}^{\frac{1}{2}} \dot{\mathcal{D}}_3^{\frac{1}{2}}\dot{\mathcal{D}}_{N+1}^{\frac{1}{2}}+  E_3^{\frac{1}{2}} \dot{\mathcal{D}}_{N+1}\\
&\qquad\qquad+E_{N+1}\,\dot{\mathcal{D}}_3+ E_3^{\frac{1}{2}} \dot{\mathcal{D}}_3^{\frac{1}{2}} \|\dot{\Lambda}_h^{-\lambda}\xi^1 \|_{H^{\frac{1}{2}}(\Sigma_0)}
+E_3 \|\dot{\Lambda}_h^{\sigma_0}\xi^1 \|_{H^{\frac{3}{2}}(\Sigma_0)}^2\bigg).
\end{split}
\end{equation}
\end{lem}

Thanks to Lemmas \ref{lem-tan-bdd-N+1-1} and \ref{lem-tan-grad-N+1-1}, for any small positive constant $\delta \leq \min\{1, \frac{c_1}{2C_1}\}$ (which will be determined later on), we have
\begin{lem}\label{lem-bdd-total-N+1-1}
Let $N\geq 3$, under the assumption of Lemma \ref{lem-tan-pseudo-energy-1}, if$(\lambda,\,\sigma_0) \in (0, 1)$ satisfies $1-\lambda< \sigma_0\leq 1-\frac{1}{2}\lambda$, and $E_3(t) \leq 1$ for all the existence times $t$, then there holds
\begin{equation}\label{bdd-total-N+1-1}
\begin{split}
&\frac{d}{dt}\widehat{\mathcal{E}}_{N+1, tan, \delta} +\widehat{\mathcal{D}}_{N+1, tan, \delta}\leq C_3 \bigg( E_{N+1}^{\frac{1}{2}} \dot{\mathcal{D}}_3^{\frac{1}{2}}\dot{\mathcal{D}}_{N+1}^{\frac{1}{2}}+  E_3^{\frac{1}{2}} \dot{\mathcal{D}}_{N+1}+E_{N+1}\,\dot{\mathcal{D}}_3\\
&\qquad\qquad\qquad+ E_3^{\frac{1}{2}} \dot{\mathcal{D}}_3^{\frac{1}{2}} \|\dot{\Lambda}_h^{-\lambda}\xi^1 \|_{H^{\frac{1}{2}}(\Sigma_0)}
+E_3 \|\dot{\Lambda}_h^{\sigma_0}\xi^1 \|_{H^{\frac{3}{2}}(\Sigma_0)}^2\bigg)+C_3\delta^{-1}E_3 \dot{\mathcal{D}}_3.
\end{split}
\end{equation}
with
\begin{equation}\label{def-bdd-total-N+1-1}
\begin{split}
&\widehat{\mathcal{E}}_{N+1, tan, \delta}:=\|\dot{\Lambda}_h^{-\lambda}v\|_{L^2(\Omega)}^2
+\|\dot{\Lambda}_h^{-\lambda}\xi^1\|_{L^2(\Sigma_0)}^2+\|\dot{\Lambda}_h^{\sigma_0}v\|_{L^2(\Omega)}^2
+\|\dot{\Lambda}_h^{\sigma_0}\xi^1\|_{L^2(\Sigma_0)}^2\\
&\,+\sum_{i=0}^{N+1}(\|\partial_h^{i}v\|_{L^2(\Omega)}^2
+\|\partial_h^{i}\xi^1\|_{L^2(\Sigma_0)}^2+\delta\bigg(\mathring{\mathcal{E}}(\dot{\Lambda}_h^{-\lambda}\nabla\, v)+\mathring{\mathcal{E}}(\dot{\Lambda}_h^{\sigma_0}\nabla\, v) +\sum_{k=0}^{N}\mathring{\mathcal{E}}(\partial_h^{k}\nabla\, v)\bigg)\\
\end{split}
\end{equation}
and
\begin{equation}\label{def-bdd-total-N+1-1}
\begin{split}
\widehat{\mathcal{D}}_{N+1, tan, \delta}:=&\frac{c_1}{2} (\|\dot{\Lambda}_h^{-\lambda}\nabla\,v\|_{L^2(\Omega)}^2+\|\dot{\Lambda}_h^{\sigma_0}\nabla\,v\|_{L^2(\Omega)}^2 +\sum_{i=0}^{N+1}\|\partial_h^{i}\nabla\,v\|_{L^2(\Omega)}^2 )\\
&+\delta\,c_2\bigg(\|\dot{\Lambda}_h^{-\lambda}\partial_t v\|_{L^2}^2+\|\dot{\Lambda}_h^{\sigma_0}\partial_t v\|_{L^2}^2 +\sum_{k=0}^{N}\|\partial_h^{k}\partial_t v\|_{L^2}^2 \bigg).
\end{split}
\end{equation}

\end{lem}

\renewcommand{\theequation}{\thesection.\arabic{equation}}
\setcounter{equation}{0}
\section{Stokes estimates}\label{sect-stokes}

In this section, we will investigate the dissipative estimates in terms of $\nabla^2v$ and $\nabla q$ as well as their horizontal derivatives in the $L^2$ framework, which is based on the Stokes estimates.

We first recall the classical regularity theory for the
Stokes problem with mixed boundary conditions on the boundary stated in \cite{Agmon-D-N-1964}.

\begin{lem}[\cite{Agmon-D-N-1964}]\label{lem-Stokes-1}
Suppose $v$, $q$ solve
\begin{equation*}\label{eqns-Stokes-1}
 \begin{cases}
  & -\nabla\,\cdot \mathbb{D}(v )+\grad\, q=\phi \in H^{r-2}(\Omega),\\
   &\dive\,v=\psi \in H^{r-1}(\Omega),\\
   &(q\,\mathbb{I}-\mathbb{D}(v))e_1= k \in H^{r-\frac{3}{2}}(\Sigma_0),\\
   &v|_{\Sigma_b}=0.
 \end{cases}
\end{equation*}
Then, for $r\geq 2$,
\begin{equation*}\label{est-Stokes-1}
 \begin{split}
  & \|v\|_{H^r(\Omega)}^2+ \|q\|_{H^{r-1}(\Omega)}^2\lesssim \|\phi\|_{H^{r-2}(\Omega)}^2+\|\psi\|_{H^{r-1}(\Omega)}^2
  +\| k\|_{H^{r-\frac{3}{2}}(\Sigma_0)}^2.
 \end{split}
\end{equation*}
\end{lem}

Denote the cut-off function near the bottom by the smooth function
\begin{equation*}
\chi_{in}(x)=
\begin{cases}
&1, \quad x \in [-\underline{b}, -\frac{2}{3}\underline{b});\\
& \mbox{smooth} \quad \in [0, 1], \quad r \in (-\frac{2}{3}\underline{b}, -\frac{1}{3}\underline{b}];\\
&0, \quad x \in (-\frac{1}{3}\underline{b}, 0],
\end{cases}
\end{equation*}
with $|\frac{d^n}{d r^n}\chi_{in}(r)| \lesssim \underline{b}^{-n}$ for any $n \in \mathbb{N}$, and the cut-off function near the free surface by the smooth function
\begin{equation*}
\chi_{f}(x)=
\begin{cases}
&1, \quad x \in (-\frac{2}{3}\underline{b}, 0];\\
& \mbox{smooth} \quad \in [0, 1], \quad r \in (-\frac{5}{6}\underline{b}, -\frac{2}{3}\underline{b}];\\
&0, \quad x \in [-\underline{b}, -\frac{5}{6}\underline{b}),
\end{cases}
\end{equation*}
with $|\frac{d^n}{d r^n}\chi_{f}(r)| \lesssim \underline{b}^{-n}$ for any $n \in \mathbb{N}$.
Notice that $\chi_{f}(\text{Supp}{\chi'_{in}})\equiv 1$. We also denote $\Omega_{in}:=\Omega\,\chi_{in}$ with the measure $\chi_{in}\,dx$ and $\Omega_{f}:=\Omega\,\chi_{f}$ with the measure $\chi_{f}\,dx$.

Since there are different boundary conditions on the top boundary $\Sigma_0$ and the bottom boundary $\Sigma_b$ in \eqref{eqns-linear-1}, we need to deal with these two different situations separately.

\subsection{Estimates near the bottom}

Define $v_{in}\eqdefa v\,\chi_{in}$, $q_{in}\eqdefa q\,\chi_{in}$, then, according to \eqref{eqns-linear-1}, $(v_{in}, q_{in})$ solves
\begin{equation*}\label{eqns-stokes-bottom-1}
 \begin{cases}
 &  - \nu \nabla\,\cdot \mathbb{D}(v_{in} ) +\nabla\,q_{in}=-\partial_t v\,\chi_{in}+g_{in},\\
   &\dive\,v_{in}=\psi_{in},\\
    &(q_{in}\,\mathbb{I}-\mathbb{D}(v_{in}))e_1|_{\Sigma_b}=0,\\
   &v_{in}|_{\Sigma_0}=0,
 \end{cases}
\end{equation*}
where
\begin{equation*}\label{stokes-bottom-2}
 \begin{split}
 &g_{in} \thicksim g\chi_{in}+\chi_{in}'\grad v+\chi_{in}''\,v+\chi_{in}'\,q,\,  \psi_{in} \thicksim \psi\chi_{in}+\chi_{in}' v^1, \,\psi:= -\widetilde{a_{\alpha\,1}}\partial_1v^\alpha+\mathcal{B}_{5, i}^{\alpha}\partial_{\alpha}v^{i}.
 \end{split}
\end{equation*}
Thanks to Lemma \ref{lem-Stokes-1}, we obtain
\begin{equation*}\label{pseudo-stokes-bottom-1}
 \begin{split}
  & \|\mathcal{P}(\partial_h)v_{in}\|_{H^2(\Omega)}^2+ \|\mathcal{P}(\partial_h)q_{in}\|_{H^{1}(\Omega)}^2\\
  &\lesssim   \|\mathcal{P}(\partial_h)\partial_t v\,\chi_{in}\|_{L^2(\Omega)}^2+ \|\mathcal{P}(\partial_h)g_{in}\|_{L^2(\Omega)}^2
  +\|\mathcal{P}(\partial_h)\psi_{in}\|_{H^{1}(\Omega)}^2,
 \end{split}
\end{equation*}
which implies
\begin{equation}\label{pseudo-stokes-bottom-2}
 \begin{split}
  & \|\mathcal{P}(\partial_h)v\|_{H^2(\Omega_{in})}^2+ \|\mathcal{P}(\partial_h)q\|_{H^{1}(\Omega_{in})}^2\lesssim  \|\mathcal{P}(\partial_h)v \|_{H^1(\Omega_{f})}^2+ \|\mathcal{P}(\partial_h)q\|_{L^2(\Omega_{f})}^2\\
  &\qquad\qquad\qquad\qquad+ \|\mathcal{P}(\partial_h)\partial_t v\|_{L^2(\Omega_{in})}^2+ \|\mathcal{P}(\partial_h)g\|_{L^2(\Omega)}^2
  +\|\mathcal{P}(\partial_h)\psi\|_{H^{1}(\Omega)}^2,
 \end{split}
\end{equation}
where we used the fact $\chi_{f}(\text{Supp}{\chi'_{in}})\equiv 1$.

\subsection{Estimates near the free boundary}
Consider the equation of $\partial_1^2\,v^{\beta}$  (with $\beta=2, 3$)
\begin{equation}\label{eqns-vh-sec-1}
\begin{split}
   &-\nu\partial_1^2\,v^{\beta}+\partial_\beta\,q=-\partial_t v^{\beta}+\nu\Delta_h\,v^{\beta}+g_{\beta}+\nu\widetilde{g}_{\beta\beta}.
     \end{split}
\end{equation}
\begin{lem}\label{lem-pseudo-near-interface-1}
Let $(v, \xi)$ be smooth solution to the system \eqref{eqns-pert-1}, then there holds that
  \begin{equation}\label{pseudo-near-interface-2}
\begin{split}
&\nu\|\mathcal{P}(\partial_h)\partial_1^2v^{h}\|_{L^2(\Omega_f)}^2
+\frac{1}{2}\frac{d}{dt}\|\mathcal{P}(\partial_h)\partial_{h}\xi^1\|_{L^2(\Sigma_0)}^2
=\sum_{j=1}^7\mathfrak{I}_j,\\
\end{split}
\end{equation}
where
\begin{equation*}
\begin{split}
&\mathfrak{I}_1:=2\nu\int_{\Sigma_0} \mathcal{P}(\partial_h)\partial_{\beta} v^1\, \mathcal{P}(\partial_h)\partial_{\beta}\nabla_h\cdot v^h\,dS_0,\,\mathfrak{I}_2:=-\int_{\Omega} \mathcal{P}(\partial_h)\partial_1v^{\beta}\, \mathcal{P}(\partial_h)\partial_{\beta}\partial_1q\,\chi_f\,dx,\\
&\quad \mathfrak{I}_3:=-\int_{\Omega} \mathcal{P}(\partial_h)\partial_1v^{\beta}\, \mathcal{P}(\partial_h)\partial_{\beta}q\,\chi_f'\,dx,\,\mathfrak{I}_4:=\int_{\Omega}\mathcal{P}(\partial_h)\partial_1^2v^{\beta}\,\mathcal{P}(\partial_h)(- \partial_t v^{\beta}+\nu\Delta_h\,v^{\beta})\,\chi_f\,dx,\\
&\mathfrak{I}_5:=-\nu\,\int_{\Sigma_0} \mathcal{P}(\partial_h)\partial_{\beta} v^1\, \mathcal{P}(\partial_h)\partial_{\beta}(\mathcal{B}_{9, i}^{\alpha}\partial_{\alpha}v^{i}  )\,dS_0,\\
&\mathfrak{I}_6:=\int_{\Sigma_0} \mathcal{P}(\partial_h)(\mathcal{B}_{\beta+5, i}^{\alpha}\partial_{\alpha}v^{i})\, \mathcal{P}(\partial_h)\partial_{\beta}q\,dS_0,\,\mathfrak{I}_7:=\int_{\Omega}\mathcal{P}(\partial_h)\partial_1^2v^{\beta}\,
\mathcal{P}(\partial_h)(g_{\beta}+\nu\widetilde{g}_{\beta\beta})\,\chi_f\,dx.
\end{split}
\end{equation*}
\end{lem}
\begin{proof}
Applying the operator $\mathcal{P}(\partial_h)$ to \eqref{eqns-vh-sec-1} yields
\begin{equation}\label{eqns-vh-sec-2}
\begin{split}
   &-\nu\mathcal{P}(\partial_h)\partial_1^2\,v^{\beta}+\partial_\beta\,\mathcal{P}(\partial_h)q=-\partial_t \mathcal{P}(\partial_h)v^{\beta}+\nu\Delta_h\,\mathcal{P}(\partial_h)v^{\beta}
   +\mathcal{P}(\partial_h)(g_{\beta}+\nu\widetilde{g}_{\beta\beta}).
     \end{split}
\end{equation}
Multiplying \eqref{eqns-vh-sec-2} by $-\mathcal{P}(\partial_h)\partial_1^2\,v^{\beta}\chi_f$ and integrating it in $\Omega$, we get
\begin{equation}\label{eqns-vh-sec-3}
\begin{split}
\nu\int_{\Omega} |\mathcal{P}(\partial_h)\partial_1^2v^{\beta}|^2\chi_f\,dx+I=&\int_{\Omega}\mathcal{P}(\partial_h)\partial_1^2v^{\beta}\,\mathcal{P}(\partial_h)(- \partial_t v^{\beta}+\nu\Delta_h\,v^{\beta})\,\chi_f\,dx\\
 &+\int_{\Omega}\mathcal{P}(\partial_h)\partial_1^2v^{\beta}\,
\mathcal{P}(\partial_h)(g_{\beta}+\nu\widetilde{g}_{\beta\beta})\,\chi_f\,dx
\end{split}
\end{equation}
with $I=-\int_{\Omega} \mathcal{P}(\partial_h)\partial_1^2v^{\beta}\,\mathcal{P}(\partial_h)\partial_{\beta}q\,\chi_f\,dx$.

Integrating by parts in $I$ gives rise to
\begin{equation*}\label{eqns-vh-sec-4}
\begin{split}
&I=-\int_{\Sigma_0} \mathcal{P}(\partial_h)\partial_1v^{\beta}\, \mathcal{P}(\partial_h)\partial_{\beta}q\,dS_0\\
&\qquad\qquad+\int_{\Omega} \mathcal{P}(\partial_h)\partial_1v^{\beta}\, \mathcal{P}(\partial_h)\partial_{\beta}\partial_1q\,\chi_f\,dx+\int_{\Omega} \mathcal{P}(\partial_h)\partial_1v^{\beta}\, \mathcal{P}(\partial_h)\partial_{\beta}q\,\chi_f'\,dx.
\end{split}
\end{equation*}
Thanks to the boundary conditions of $\partial_1v^{\beta}$  and $q$ on $\Sigma_0$, we have
\begin{equation*}\label{eqns-vh-sec-5}
\begin{split}
&-\int_{\Sigma_0} \mathcal{P}(\partial_h)\partial_1v^{\beta}\, \mathcal{P}(\partial_h)\partial_{\beta}q\,dS_0=-\int_{\Sigma_0} \mathcal{P}(\partial_h)\bigg(-\partial_{\beta} v^1+\mathcal{B}_{\beta+5, i}^{\alpha}\partial_{\alpha}v^{i}\bigg)\, \mathcal{P}(\partial_h)\partial_{\beta}q\,dS_0\\
&=\int_{\Sigma_0} \mathcal{P}(\partial_h)\partial_{\beta} v^1\,\mathcal{P}(\partial_h) \partial_{\beta}q\,dS_0-\int_{\Sigma_0} \mathcal{P}(\partial_h)(\mathcal{B}_{\beta+5, i}^{\alpha}\partial_{\alpha}v^{i})\, \mathcal{P}(\partial_h)\partial_{\beta}q\,dS_0\\
&= \int_{\Sigma_0} \mathcal{P}(\partial_h)\partial_{\beta} v^1\, \mathcal{P}(\partial_h)\bigg(\partial_{\beta}\xi^1-2\nu\partial_{\beta}\nabla_h\cdot v^h  +\nu\partial_{\beta}(\mathcal{B}_{9, i}^{\alpha}\partial_{\alpha}v^{i}  )\bigg)\,dS_0\\
&\qquad-\int_{\Sigma_0} \mathcal{P}(\partial_h)(\mathcal{B}_{\beta+5, i}^{\alpha}\partial_{\alpha}v^{i})\, \mathcal{P}(\partial_h)\partial_{\beta}q\,dS_0,
\end{split}
\end{equation*}
which follows that
\begin{equation*}\label{eqns-vh-sec-6}
\begin{split}
&-\int_{\Sigma_0} \mathcal{P}(\partial_h)\partial_1v^{\beta}\, \mathcal{P}(\partial_h)\partial_{\beta}q\,dS_0\\
&=\frac{1}{2}\frac{d}{dt}\int_{\Sigma_0}|\mathcal{P}(\partial_h)\partial_{\beta}\xi^1|^2\,dS_0-2\nu\int_{\Sigma_0} \mathcal{P}(\partial_h)\partial_{\beta} v^1\, \mathcal{P}(\partial_h)\partial_{\beta}\nabla_h\cdot v^h\,dS_0\\
&\,+\nu\,\int_{\Sigma_0} \mathcal{P}(\partial_h)\partial_{\beta} v^1\, \mathcal{P}(\partial_h)\partial_{\beta}(\mathcal{B}_{9, i}^{\alpha}\partial_{\alpha}v^{i}  )\,dS_0-\int_{\Sigma_0} \mathcal{P}(\partial_h)(\mathcal{B}_{\beta+5, i}^{\alpha}\partial_{\alpha}v^{i})\, \mathcal{P}(\partial_h)\partial_{\beta}q\,dS_0.
\end{split}
\end{equation*}
Therefore, we obtain
\begin{equation}\label{eqns-vh-sec-7}
\begin{split}
&I=\frac{1}{2}\frac{d}{dt}\int_{\Sigma_0}|\mathcal{P}(\partial_h)\partial_{\beta}\xi^1|^2\,dS_0-2\nu\int_{\Sigma_0} \mathcal{P}(\partial_h)\partial_{\beta} v^1\, \mathcal{P}(\partial_h)\partial_{\beta}\nabla_h\cdot v^h\,dS_0\\
&\,+\nu\,\int_{\Sigma_0} \mathcal{P}(\partial_h)\partial_{\beta} v^1\, \mathcal{P}(\partial_h)\partial_{\beta}(\mathcal{B}_{9, i}^{\alpha}\partial_{\alpha}v^{i}  )\,dS_0-\int_{\Sigma_0} \mathcal{P}(\partial_h)(\mathcal{B}_{0, \beta+5, \alpha, i}\partial_{\alpha}v^{i})\, \mathcal{P}(\partial_h)\partial_{\beta}q\,dS_0\\
&\qquad\qquad+\int_{\Omega} \mathcal{P}(\partial_h)\partial_1v^{\beta}\, \mathcal{P}(\partial_h)\partial_{\beta}\partial_1q\,\chi_f\,dx+\int_{\Omega} \mathcal{P}(\partial_h)\partial_1v^{\beta}\, \mathcal{P}(\partial_h)\partial_{\beta}q\,\chi_f'\,dx.
\end{split}
\end{equation}
Plugging \eqref{eqns-vh-sec-7} into \eqref{eqns-vh-sec-3}, we get \eqref{pseudo-near-interface-2}.
\end{proof}

\subsection{Estimate of $\|(\dot{\Lambda}_h^{\sigma_0}\nabla^2v, \,\dot{\Lambda}_h^{\sigma_0}\nabla\,q)\|_{L^2_t(L^2)}$}

\begin{lem}\label{lem-laplace-sigma0-1}
Under the assumption of Lemma \ref{lem-pseudo-energy-tv-1}, if $E_3(t) \leq 1$ for all the existence times $t$, then there holds
  \begin{equation}\label{laplace-sigma0-0}
\begin{split}
&\frac{d}{dt}\|\dot{\Lambda}_h^{\sigma_0}\partial_{h}\xi^1\|_{L^2(\Sigma_0)}^2
+c_3(\|\dot{\Lambda}_h^{\sigma_0}\nabla^2v\|_{L^2(\Omega)}^2 +\|\dot{\Lambda}_h^{\sigma_0}\nabla\,q\|_{L^2(\Omega)}^2)\\
  &\lesssim \|\dot{\Lambda}_h^{\sigma_0}\nabla\, v\|_{L^2}^2+\|\partial_{h}^2\nabla\, v\|_{L^2}^2+(\|\dot{\Lambda}_h^{\sigma_0}\partial_h\partial_t v\|_{L^2}^2+\|\dot{\Lambda}_h^{\sigma_0}\partial_t v\|_{L^2}^2)+ E_3^{\frac{1}{2}} \dot{\mathcal{D}}_3.
\end{split}
\end{equation}
\end{lem}
\begin{proof}
First, taking $\mathcal{P}(\partial_h)=\dot{\Lambda}_h^{\sigma_0}$ in \eqref{prt1-q-vh-rela-2}, we get
  \begin{equation}\label{laplace-sigma0-2}
\begin{split}
  &\|\dot{\Lambda}_h^{\sigma_0}\partial_1q\|_{L^2}\lesssim \|\dot{\Lambda}_h^{\sigma_0}\partial_t v^{1}\|_{L^2}+\|\dot{\Lambda}_h^{\sigma_0}\partial_h\,v \|_{H^1}
  + E_3^{\frac{1}{2}} \dot{\mathcal{D}}_{3}^{\frac{1}{2}},\\
 &\|\dot{\Lambda}_h^{\sigma_0}\partial_1^2v^1\|_{L^2}\lesssim \|\dot{\Lambda}_h^{\sigma_0}\partial_h\,v \|_{H^1} +E_3^{\frac{1}{2}} \dot{\mathcal{D}}_{3}^{\frac{1}{2}},\quad \|\dot{\Lambda}_h^{\sigma_0-1}\partial_1^2v^1\|_{L^2}\lesssim \|\dot{\Lambda}_h^{\sigma_0}\,v \|_{H^1} +E_3^{\frac{1}{2}} \dot{\mathcal{D}}_{3}^{\frac{1}{2}},\\
   &\|-\nu\dot{\Lambda}_h^{\sigma_0}\partial_1^2\,v^{h}+\dot{\Lambda}_h^{\sigma_0}\partial_hq\|_{L^2}
  \lesssim \|\dot{\Lambda}_h^{\sigma_0}\partial_t v^{h}\|_{L^2}+\|\dot{\Lambda}_h^{\sigma_0}\partial_h\,v \|_{H^1} +E_3^{\frac{1}{2}} \dot{\mathcal{D}}_{3}^{\frac{1}{2}},\\
  &\|-\nu\dot{\Lambda}_h^{\sigma_0}\partial_1^2\,v^{h}+\dot{\Lambda}_h^{\sigma_0}\partial_hq\|_{L^2(\Omega_f)}
  \lesssim \|\dot{\Lambda}_h^{\sigma_0}\partial_t v^{h}\|_{L^2}+\|\dot{\Lambda}_h^{\sigma_0}\partial_h\,v \|_{H^1} +E_3^{\frac{1}{2}} \dot{\mathcal{D}}_{3}^{\frac{1}{2}},
     \end{split}
\end{equation}
and in particular,
\begin{equation}\label{laplace-sigma0-15}
\begin{split}
  &\|\dot{\Lambda}_h^{\sigma_0}\partial_hq\|_{L^2(\Omega_f)}
  \lesssim\|\dot{\Lambda}_h^{\sigma_0}\partial_1^2\,v^{h}\|_{L^2(\Omega_f)}+ \|\dot{\Lambda}_h^{\sigma_0}\partial_t v \|_{L^2}+\|\dot{\Lambda}_h^{\sigma_0}\partial_h\,v \|_{H^1} +E_3^{\frac{1}{2}} \dot{\mathcal{D}}_{3}^{\frac{1}{2}}.
     \end{split}
\end{equation}
While taking $\mathcal{P}(\partial_h)=\dot{\Lambda}_h^{\sigma_0}\partial_h$ in \eqref{pseudo-stokes-bottom-2}, we have
\begin{equation*}\label{laplace-sigma0-18}
 \begin{split}
  & \|\dot{\Lambda}_h^{\sigma_0}\partial_hv\|_{H^2(\Omega_{in})}^2+ \|\dot{\Lambda}_h^{\sigma_0}\partial_hq\|_{H^{1}(\Omega_{in})}^2\lesssim  \|\dot{\Lambda}_h^{\sigma_0}\partial_hv \|_{H^1(\Omega_{f})}^2+ \|\dot{\Lambda}_h^{\sigma_0}\partial_hq\|_{L^2(\Omega_{f})}^2\\
  &\qquad\qquad\qquad\qquad+ \|\dot{\Lambda}_h^{\sigma_0}\partial_h\partial_t v\|_{L^2(\Omega_{in})}^2+ \|\dot{\Lambda}_h^{\sigma_0}\partial_hg\|_{L^2(\Omega)}^2
  +\|\mathcal{P}(\partial_h)\psi\|_{H^{1}(\Omega)}^2
 \end{split}
\end{equation*}
with
$\psi= -\widetilde{a_{\alpha\,1}}\partial_1v^\alpha+\mathcal{B}_{5, i}^{\alpha}\partial_{\alpha}v^{i}$.

Thanks to \eqref{est-g-2}, we find
\begin{equation*}\label{laplace-sigma0-19}
 \begin{split}
  & \|\dot{\Lambda}_h^{\sigma_0}\partial_hg\|_{L^2(\Omega)}^2
  +\|\dot{\Lambda}_h^{\sigma_0}\partial_h\psi\|_{H^{1}(\Omega)}^2 \lesssim E_3\dot{\mathcal{D}}_3,
 \end{split}
\end{equation*}
so we get from \eqref{laplace-sigma0-15} that
\begin{equation*}\label{laplace-sigma0-20}
 \begin{split}
  & \|\dot{\Lambda}_h^{\sigma_0}\partial_hv\|_{H^2(\Omega_{in})}^2+ \|\dot{\Lambda}_h^{\sigma_0}\partial_hq\|_{H^{1}(\Omega_{in})}^2\lesssim  \|\dot{\Lambda}_h^{\sigma_0}\partial_hv \|_{H^1}^2+ \|\dot{\Lambda}_h^{\sigma_0}\partial_h\partial_t v\|_{L^2}^2+ E_3\dot{\mathcal{D}}_3\\
  &\qquad\qquad\qquad+ \bigg(\|\dot{\Lambda}_h^{\sigma_0}\partial_1^2\,v^{h}\|_{L^2(\Omega_f)}+ \|\dot{\Lambda}_h^{\sigma_0}\partial_t v^{h}\|_{L^2}+\|\dot{\Lambda}_h^{\sigma_0}\partial_h\,v \|_{H^1} +E_3^{\frac{1}{2}} \dot{\mathcal{D}}_{3}^{\frac{1}{2}}\bigg)^2.
 \end{split}
\end{equation*}
Hence, we get
\begin{equation*}\label{laplace-sigma0-21a}
 \begin{split}
   \|\dot{\Lambda}_h^{\sigma_0}\partial_hq\|_{L^2(\Omega_{in})}\lesssim  &\|\dot{\Lambda}_h^{\sigma_0}\partial_hv \|_{H^1}+ \|\dot{\Lambda}_h^{\sigma_0}\partial_h\partial_t v\|_{L^2}+ \|\dot{\Lambda}_h^{\sigma_0}\partial_t v^{h}\|_{L^2} \\
   &+E_3^{\frac{1}{2}} \dot{\mathcal{D}}_{3}^{\frac{1}{2}}+\|\dot{\Lambda}_h^{\sigma_0}\partial_1^2\,v^{h}\|_{L^2(\Omega_f)},
 \end{split}
\end{equation*}
which along with \eqref{laplace-sigma0-15} and the first inequality in \eqref{laplace-sigma0-2} ensures
\begin{equation}\label{laplace-sigma0-21}
\begin{split}
  &\|\dot{\Lambda}_h^{\sigma_0}\nabla\,q\|_{L^2(\Omega)}\lesssim \|\dot{\Lambda}_h^{\sigma_0}\partial_1q\|_{L^2(\Omega)}+ \|\dot{\Lambda}_h^{\sigma_0}\partial_hq\|_{L^2(\Omega_f)}+\|\dot{\Lambda}_h^{\sigma_0}\partial_hq\|_{L^2(\Omega_{in})}\\
  &
  \lesssim\|\dot{\Lambda}_h^{\sigma_0}\partial_1^2\,v^{h}\|_{L^2(\Omega_f)}+ \|\dot{\Lambda}_h^{\sigma_0}\partial_h\partial_t v\|_{L^2}+ \|\dot{\Lambda}_h^{\sigma_0}\partial_t v^{h}\|_{L^2} +\|\dot{\Lambda}_h^{\sigma_0}\partial_h\,v \|_{H^1} +E_3^{\frac{1}{2}} \dot{\mathcal{D}}_{3}^{\frac{1}{2}},
     \end{split}
\end{equation}
and then, due to \eqref{laplace-sigma0-2}, we have
  \begin{equation}\label{laplace-sigma0-21a}
\begin{split}
   &\|\dot{\Lambda}_h^{\sigma_0}\partial_1^2\,v^{h}\|_{L^2}\lesssim\|\dot{\Lambda}_h^{\sigma_0}\partial_hq\|_{L^2}
  +\|\dot{\Lambda}_h^{\sigma_0}\partial_t v^{h}\|_{L^2}+\|\dot{\Lambda}_h^{\sigma_0}\partial_h\,v \|_{H^1} +E_3^{\frac{1}{2}} \dot{\mathcal{D}}_{3}^{\frac{1}{2}}\\
  &\lesssim\|\dot{\Lambda}_h^{\sigma_0}\partial_1^2\,v^{h}\|_{L^2(\Omega_f)}+ \|\dot{\Lambda}_h^{\sigma_0}\partial_h\partial_t v\|_{L^2}+ \|\dot{\Lambda}_h^{\sigma_0}\partial_t v^{h}\|_{L^2} +\|\dot{\Lambda}_h^{\sigma_0}\partial_h\,v \|_{H^1} +E_3^{\frac{1}{2}} \dot{\mathcal{D}}_{3}^{\frac{1}{2}}.
     \end{split}
\end{equation}
On the other hand, taking $\mathcal{P}(\partial_h)=\dot{\Lambda}_h^{\sigma_0}$ in \eqref{pseudo-near-interface-2} gives
 \begin{equation*}\label{laplace-sigma0-3}
\begin{split}
&\nu\|\dot{\Lambda}_h^{\sigma_0}\partial_1^2v^{h}\|_{L^2(\Omega_f)}^2
+\frac{1}{2}\frac{d}{dt}\|\dot{\Lambda}_h^{\sigma_0}\partial_{h}\xi^1\|_{L^2(\Sigma_0)}^2
=\sum_{j=1}^7\mathfrak{I}_j.
\end{split}
\end{equation*}
For the remainder terms $\mathfrak{I}_j$ with $j=1, 2, 3, 4$, from the linear terms, we have
 \begin{equation*}\label{laplace-sigma0-4}
\begin{split}
&|\mathfrak{I}_1|\lesssim \|\dot{\Lambda}_h^{\sigma_0}\partial_{\beta} v^1\|_{H^{\frac{1}{2}}(\Sigma_0)}\, \|\dot{\Lambda}_h^{\sigma_0}\nabla_h\cdot v^h\|_{H^{\frac{1}{2}}(\Sigma_0)}\lesssim \|\dot{\Lambda}_h^{\sigma_0}\partial_{h} v\|_{H^1(\Omega)}^2,\\
\end{split}
\end{equation*}
 \begin{equation*}\label{laplace-sigma0-5}
\begin{split}
&|\mathfrak{I}_2|\lesssim \|\dot{\Lambda}_h^{\sigma_0+1}\partial_1v^{\beta}\|_{L^2}\, \|\dot{\Lambda}_h^{\sigma_0}\partial_1q\|_{L^2(\Omega_f)}\lesssim \|\dot{\Lambda}_h^{\sigma_0+1}\nabla\,v\|_{L^2}\, \|\dot{\Lambda}_h^{\sigma_0}\partial_1q\|_{L^2},
\end{split}
\end{equation*}
 \begin{equation*}\label{laplace-sigma0-6}
\begin{split}
|\mathfrak{I}_3|\lesssim\|\dot{\Lambda}_h^{\sigma_0}\partial_1v^{\beta}\|_{L^2} \|\dot{\Lambda}_h^{\sigma_0+1}q\|_{L^2},\\
\end{split}
\end{equation*}
and
 \begin{equation*}\label{laplace-sigma0-7}
\begin{split}
&|\mathfrak{I}_4|\lesssim \|\dot{\Lambda}_h^{\sigma_0}\partial_1^2v^{h}\|_{L^2(\Omega_f)}(\|\dot{\Lambda}_h^{\sigma_0}\partial_t v^{h}\|_{L^2}+\|\dot{\Lambda}_h^{\sigma_0}\partial_h^2\,v^{h}\|_{L^2}).
\end{split}
\end{equation*}
For the nonlinear remainder terms $\mathfrak{I}_5$, $\mathfrak{I}_6$, and $\mathfrak{I}_7$, there hold
 \begin{equation*}\label{laplace-sigma0-8}
\begin{split}
&|\mathfrak{I}_5|\lesssim\|\dot{\Lambda}_h^{\sigma_0}\partial_{\beta} v^1\|_{L^2(\Sigma_0)} \|\dot{\Lambda}_h^{\sigma_0}\partial_{\beta}(\mathcal{B}_{9, i}^{\alpha}\partial_{\alpha}v^{i}  )\|_{L^2(\Sigma_0)}\\
&\lesssim \|\dot{\Lambda}_h^{\sigma_0}\partial_{\beta} v^1\|_{H^1(\Omega)} \|\dot{\Lambda}_h^{\sigma_0}\partial_{\beta}(\mathcal{B}_{9, i}^{\alpha}\partial_{\alpha}v^{i}  )\|_{H^1(\Omega)}\lesssim E_3^{\frac{1}{2}}\dot{\mathcal{D}}_3.
\end{split}
\end{equation*}
 \begin{equation*}\label{laplace-sigma0-9}
\begin{split}
&|\mathfrak{I}_6|\lesssim \|\dot{\Lambda}_h^{\sigma_0}(\mathcal{B}_{\beta+5, i}^{\alpha}\partial_{\alpha}v^{i})\|_{L^2(\Sigma_0)} \|\dot{\Lambda}_h^{\sigma_0}\partial_{\beta}q\|_{L^2(\Sigma_0)}\lesssim E_3^{\frac{1}{2}}\dot{\mathcal{D}}_3,
\end{split}
\end{equation*}
and
 \begin{equation*}\label{laplace-sigma0-10}
\begin{split}
&|\mathfrak{I}_7|\lesssim \|\dot{\Lambda}_h^{\sigma_0}\partial_1^2v^{\beta}\|_{L^2(\Omega_f)}(
\|\dot{\Lambda}_h^{\sigma_0}g_{\beta}\|_{L^2}+\|\dot{\Lambda}_h^{\sigma_0}\widetilde{g}_{\beta\beta}\|_{L^2})
\lesssim E_3^{\frac{1}{2}}\dot{\mathcal{D}}_3,
\end{split}
\end{equation*}
Hence, we get
 \begin{equation*}\label{laplace-sigma0-11}
\begin{split}
&\nu\|\dot{\Lambda}_h^{\sigma_0}\partial_1^2v^{h}\|_{L^2(\Omega_f)}^2
+\frac{1}{2}\frac{d}{dt}\|\dot{\Lambda}_h^{\sigma_0}\partial_{h}\xi^1\|_{L^2(\Sigma_0)}^2\\
&\lesssim\|\dot{\Lambda}_h^{\sigma_0}\partial_{h} v\|_{H^1(\Omega)}^2+
\|\dot{\Lambda}_h^{\sigma_0+1}\nabla\,v\|_{L^2}\, \|\dot{\Lambda}_h^{\sigma_0}\partial_1q\|_{L^2}+\|\dot{\Lambda}_h^{\sigma_0}\partial_1v^{h}\|_{L^2} \|\dot{\Lambda}_h^{\sigma_0}\partial_hq\|_{L^2}\\
&\qquad+\|\dot{\Lambda}_h^{\sigma_0}\partial_1^2v^{h}\|_{L^2(\Omega_f)}(\|\dot{\Lambda}_h^{\sigma_0}\partial_t v\|_{L^2}+\|\dot{\Lambda}_h^{\sigma_0}\partial_h^2\,v\|_{L^2})+E_3^{\frac{1}{2}}\dot{\mathcal{D}}_3,
\end{split}
\end{equation*}
which follows from \eqref{laplace-sigma0-21} that
 \begin{equation*}\label{laplace-sigma0-12}
\begin{split}
&\nu\|\dot{\Lambda}_h^{\sigma_0}\partial_1^2v^{h}\|_{L^2(\Omega_f)}^2
+\frac{d}{dt}\|\dot{\Lambda}_h^{\sigma_0}\partial_{h}\xi^1\|_{L^2(\Sigma_0)}^2\\
&\lesssim\|\dot{\Lambda}_h^{\sigma_0}\partial_{h} \nabla\,v\|_{L^2}^2+\|\dot{\Lambda}_h^{\sigma_0}\partial_t v\|_{L^2}^2+E_3^{\frac{1}{2}}\dot{\mathcal{D}}_3+
\|\dot{\Lambda}_h^{\sigma_0}(\nabla\,v,\,\partial_h\nabla\,v)\|_{L^2}
\|\dot{\Lambda}_h^{\sigma_0}\partial_1^2\,v^{h}\|_{L^2(\Omega_f)}\\
&+\|\dot{\Lambda}_h^{\sigma_0}(\nabla\,v,\,\partial_h\nabla\,v)\|_{L^2}
( \|\dot{\Lambda}_h^{\sigma_0}\partial_h\partial_t v\|_{L^2}+ \|\dot{\Lambda}_h^{\sigma_0}\partial_t v^{h}\|_{L^2} +\|\dot{\Lambda}_h^{\sigma_0}\partial_h\,v \|_{H^1} +E_3^{\frac{1}{2}} \dot{\mathcal{D}}_{3}^{\frac{1}{2}}).
\end{split}
\end{equation*}
Hence, one has
 \begin{equation}\label{laplace-sigma0-14}
\begin{split}
&c_1\|\dot{\Lambda}_h^{\sigma_0}\partial_1^2v^{h}\|_{L^2(\Omega_f)}^2
+\frac{d}{dt}\|\dot{\Lambda}_h^{\sigma_0}\partial_{h}\xi^1\|_{L^2(\Sigma_0)}^2\\
&\lesssim\,E_3^{\frac{1}{2}}\dot{\mathcal{D}}_3+
\|\dot{\Lambda}_h^{\sigma_0}(\nabla\,v,\,\partial_h\nabla\,v)\|_{L^2}^2+
\|\dot{\Lambda}_h^{\sigma_0}(\partial_h\partial_t v,\,\partial_t v^{h})\|_{L^2}^2.
\end{split}
\end{equation}
Combining \eqref{laplace-sigma0-14},\eqref{laplace-sigma0-21},\eqref{laplace-sigma0-21a}, we obtain
 \begin{equation*}\label{laplace-sigma0-16}
\begin{split}
&\frac{d}{dt}\|\dot{\Lambda}_h^{\sigma_0}\partial_{h}\xi^1\|_{L^2(\Sigma_0)}^2
+c_2(\|\dot{\Lambda}_h^{\sigma_0}\partial_1^2v^{h}\|_{L^2(\Omega)}^2 +\|\dot{\Lambda}_h^{\sigma_0}\nabla\,q\|_{L^2(\Omega)}^2)\\
&\lesssim \,E_3^{\frac{1}{2}}\dot{\mathcal{D}}_3+
\|\dot{\Lambda}_h^{\sigma_0}(\nabla\,v,\,\partial_h\nabla\,v)\|_{L^2}^2+
\|\dot{\Lambda}_h^{\sigma_0}(\partial_h\partial_t v,\,\partial_t v^{h})\|_{L^2}^2,
     \end{split}
\end{equation*}
which along with the second inequality in \eqref{laplace-sigma0-2} follows \eqref{laplace-sigma0-0}.
The proof of the lemma is accomplished.
\end{proof}

\subsection{Estimate of $\|\partial_h^{N-1}\partial_1^2v^h\|_{L^2_tL^2}$}

\begin{lem}\label{lem-laplace-N-1}
Under the assumption of Lemma \ref{lem-pseudo-energy-tv-1}, if $E_3(t) \leq 1$ for all the existence times $t$, then there holds
\begin{equation}\label{laplace-N-0}
\begin{split}
&\frac{d}{dt}\|\partial_h^{N}\xi^1\|_{L^2(\Sigma_0)}^2
+c_3(\|\partial_h^{N-1}(\nabla^2\,v,\,\nabla\,q)\|_{L^2})
\\
&\lesssim \|\partial_h^{N-2}(\partial_t v,\,\partial_h\partial_t v)\|_{L^2}^2+\|\partial_h^{N-1}(\partial_h\,v,\,v) \|_{H^1}^2  +\|\dot{\Lambda}_h^{\sigma_0}\partial_1^2\,v^{h}\|_{L^2(\Omega)}^2\\
&\qquad+\dot{\mathcal{D}}_{N}^{\frac{1}{2}} (E_3^{\frac{1}{2}} \dot{\mathcal{D}}_{N}^{\frac{1}{2}}  + E_N^{\frac{1}{2}} \dot{\mathcal{D}}_{3}^{\frac{1}{2}})  + E_N \dot{\mathcal{D}}_{3}
 \end{split}
\end{equation}
\end{lem}
\begin{proof}
First, taking $\mathcal{P}(\partial_h)=\partial_h^{N-1}$ in \eqref{prt1-q-vh-rela-2}, we get
  \begin{equation}\label{laplace-N-1}
\begin{split}
  &\|\partial_h^{N-1}\partial_1q\|_{L^2}\lesssim \|\partial_h^{N-1}\partial_t v^{1}\|_{L^2}+\|\partial_h^{N-1}\partial_h\,v \|_{H^1}
  + (E_3^{\frac{1}{2}} \dot{\mathcal{D}}_{N}^{\frac{1}{2}}  + E_N^{\frac{1}{2}} \dot{\mathcal{D}}_{3}^{\frac{1}{2}}),\\
  &\|\partial_h^{N-1}\partial_1^2v^1\|_{L^2}\lesssim \|\partial_h^{N-1}\partial_h\,v \|_{H^1} + (E_3^{\frac{1}{2}} \dot{\mathcal{D}}_{N}^{\frac{1}{2}}  + E_N^{\frac{1}{2}} \dot{\mathcal{D}}_{3}^{\frac{1}{2}}),\\
  &\|\partial_h^{N-1}(\partial_hq-\nu\partial_1^2\,v^{h})\|_{L^2}\lesssim \|\partial_h^{N-1}\partial_t v\|_{L^2}+\|\partial_h^{N-1}\partial_h\,v \|_{H^1}  +(E_3^{\frac{1}{2}} \dot{\mathcal{D}}_{N}^{\frac{1}{2}}  + E_N^{\frac{1}{2}} \dot{\mathcal{D}}_{3}^{\frac{1}{2}}),\\
  &\|\partial_h^{N-1}(\partial_hq-\nu\partial_1^2\,v^{h})\|_{L^2(\Omega_f)}\lesssim \|\partial_h^{N-1}\partial_t v\|_{L^2}+\|\partial_h^{N-1}\partial_h\,v \|_{H^1}
  +(E_3^{\frac{1}{2}} \dot{\mathcal{D}}_{N}^{\frac{1}{2}}  + E_N^{\frac{1}{2}} \dot{\mathcal{D}}_{3}^{\frac{1}{2}}),\\
     \end{split}
\end{equation}
and in particular,
  \begin{equation}\label{laplace-N-2}
\begin{split}
 \|\partial_h^{k-1}\partial_hq\|_{L^2(\Omega_f)}\lesssim &\|\partial_h^{k-1}\partial_1^2\,v^{h}\|_{L^2(\Omega_f)}\\
 &+\|\partial_h^{k-1}\partial_t v\|_{L^2}+\|\partial_h^{k-1}\partial_h\,v \|_{H^1}
  +(E_3^{\frac{1}{2}} \dot{\mathcal{D}}_{k}^{\frac{1}{2}}  + E_k^{\frac{1}{2}} \dot{\mathcal{D}}_{3}^{\frac{1}{2}})
     \end{split}
\end{equation}
for any $k \geq 3$.
While taking $\mathcal{P}(\partial_h)=\partial_h^{N-1}$ in \eqref{pseudo-stokes-bottom-2}, we have
  \begin{equation*}\label{laplace-N-3}
 \begin{split}
  & \|\partial_h^{N-1}v\|_{H^2(\Omega_{in})}^2+ \|\partial_h^{N-1}q\|_{H^{1}(\Omega_{in})}^2\lesssim  \|\partial_h^{N-1}v \|_{H^1(\Omega_{f})}^2+ \|\partial_h^{N-1}q\|_{L^2(\Omega_{f})}^2\\
  &\qquad\qquad\qquad\qquad+ \|\partial_h^{N-1}\partial_t v\|_{L^2(\Omega_{in})}^2+ \|\partial_h^{N-1}g\|_{L^2(\Omega)}^2
  +\|\partial_h^{N-1}\psi\|_{H^{1}(\Omega)}^2
 \end{split}
\end{equation*}
with
$\psi= -(\widetilde{a_{21}}\partial_1v^2+\widetilde{a_{31}}\partial_1v^3)+\mathcal{B}_{5, i}^{\alpha}\partial_{\alpha}v^{i}$.

Thanks to \eqref{est-g-2}, we find
\begin{equation*}\label{laplace-N-4}
 \begin{split}
&  \|\partial_h^{N-1}g\|_{L^2(\Omega)}^2
  +\|\partial_h^{N-1}\psi\|_{H^{1}(\Omega_{in})}^2\lesssim E_3 \dot{\mathcal{D}}_{N} + E_{N}\,\dot{\mathcal{D}}_{3}.
 \end{split}
\end{equation*}
so we get from \eqref{laplace-N-2} that
\begin{equation*}\label{laplace-N-5}
 \begin{split}
 \|\partial_h^{N-1}v\|_{H^2(\Omega_{in})}^2+ \|\partial_h^{N-1}q\|_{H^{1}(\Omega_{in})}^2\lesssim  & \|\partial_h^{N-1}v \|_{H^1}^2+ \|\partial_h^{N-1}\partial_t v\|_{L^2}^2+\|\partial_h^{N-2}\partial_t v\|_{L^2}^2 \\
&+ E_3 \dot{\mathcal{D}}_{N}+ E_{N}\,\dot{\mathcal{D}}_{3}+ \|\partial_h^{N-2}\partial_1^2\,v^{h}\|_{L^2(\Omega_f)}^2.
 \end{split}
\end{equation*}
In particular, we have
\begin{equation*}\label{laplace-N-6}
 \begin{split}
\|\partial_h^{N-1}\nabla\,q\|_{L^2(\Omega_{in})}^2\lesssim  & \|\partial_h^{N-1}v \|_{H^1}+ \|\partial_h^{N-1}\partial_t v\|_{L^2}+\|\partial_h^{N-2}\partial_t v\|_{L^2} \\
&+ E_3^{\frac{1}{2}} \dot{\mathcal{D}}_{N}^{\frac{1}{2}}+ E_{N}^{\frac{1}{2}}\,\dot{\mathcal{D}}_{3}^{\frac{1}{2}}+ \|\partial_h^{N-2}\partial_1^2\,v^{h}\|_{L^2(\Omega_f)}.
 \end{split}
\end{equation*}
Thanks to \eqref{laplace-N-1}, we get
    \begin{equation*}\label{laplace-N-7}
 \begin{split}
  &\|\partial_h^{N-1}\nabla\,q\|_{L^2(\Omega_f)}\lesssim \|\partial_h^{N-1}\partial_1\,q\|_{L^2(\Omega_f)}+  \|\partial_h^{N}\,q\|_{L^2(\Omega_f)}\\
    &\lesssim \|\partial_h^{N-1}\partial_t v\|_{L^2}+\|\partial_h^{N-1}\partial_h\,v \|_{H^1}
  + (E_3^{\frac{1}{2}} \dot{\mathcal{D}}_{N}^{\frac{1}{2}}  + E_N^{\frac{1}{2}} \dot{\mathcal{D}}_{3}^{\frac{1}{2}})+ \|\partial_h^{N-1}\partial_1^2\,v^{h}\|_{L^2(\Omega_f)}.
 \end{split}
\end{equation*}
Hence, one has
   \begin{equation*}\label{laplace-N-7}
 \begin{split}
  \|\partial_h^{N-1}\nabla\,q\|_{L^2(\Omega)}&\lesssim\|\partial_h^{N-1}\nabla\,q\|_{L^2(\Omega_f)}
  +\|\partial_h^{N-1}\nabla\,q\|_{L^2(\Omega_{in})}\\
    &\lesssim \|\partial_h^{N-2}(\partial_t v,\,\partial_h\partial_t v)\|_{L^2}+\|\partial_h^{N-1}(\partial_h\,v,\,v) \|_{H^1}
  + (E_3^{\frac{1}{2}} \dot{\mathcal{D}}_{N}^{\frac{1}{2}}  + E_N^{\frac{1}{2}} \dot{\mathcal{D}}_{3}^{\frac{1}{2}})\\
  &\quad+\|\dot{\Lambda}_h^{\sigma_0}\partial_1^2\,v^{h}\|_{L^2(\Omega_f)}+ \|\partial_h^{N-1}\partial_1^2\,v^{h}\|_{L^2(\Omega_f)},
 \end{split}
\end{equation*}
and also there holds by using \eqref{laplace-N-1} that
   \begin{equation}\label{laplace-N-8}
 \begin{split}
\|\partial_h^{N-1}(\partial_1^2\,v^{h},\,&\nabla\,q)\|_{L^2(\Omega)}\lesssim \|\partial_h^{N-2}(\partial_t v,\,\partial_h\partial_t v)\|_{L^2}+\|\partial_h^{N-1}(\partial_h\,v,\,v) \|_{H^1}\\
&  + (E_3^{\frac{1}{2}} \dot{\mathcal{D}}_{N}^{\frac{1}{2}}  + E_N^{\frac{1}{2}} \dot{\mathcal{D}}_{3}^{\frac{1}{2}})+\|\dot{\Lambda}_h^{\sigma_0}\partial_1^2\,v^{h}\|_{L^2(\Omega_f)}+ \|\partial_h^{N-1}\partial_1^2\,v^{h}\|_{L^2(\Omega_f)}.
 \end{split}
\end{equation}
On the other hand, taking $\mathcal{P}(\partial_h)=\partial_h^{N-1}$ in \eqref{pseudo-near-interface-2} gives
 \begin{equation*}\label{laplace-N-9}
\begin{split}
&\nu\|\partial_h^{N-1}\partial_1^2v^{h}\|_{L^2(\Omega_f)}^2
+\frac{1}{2}\frac{d}{dt}\|\partial_h^{N-1}\partial_{h}\xi^1\|_{L^2(\Sigma_0)}^2
=\sum_{j=1}^7\mathfrak{I}_j.
\end{split}
\end{equation*}
For the remainder terms $\mathfrak{I}_j$ with $j=1, 2, 3, 4$, from the linear terms, we have
 \begin{equation*}\label{laplace-N-10}
\begin{split}
|\mathfrak{I}_1|&\lesssim \|\partial_h^{N-1}\partial_{\beta} v^1\|_{H^{\frac{1}{2}}(\Sigma_0)}\, \|\partial_h^{N-1}\nabla_h\cdot v^h\|_{H^{\frac{1}{2}}(\Sigma_0)}\\
&\lesssim \|\partial_h^{N-1}\partial_{\beta} v^1\|_{H^1}\, \|\partial_h^{N-1}\nabla_h\cdot v^h\|_{H^1}\lesssim \|\partial_h^{N}\nabla\,v\|_{L^2}^2,\\
\end{split}
\end{equation*}
 \begin{equation*}\label{laplace-N-11}
\begin{split}
&|\mathfrak{I}_2|+|\mathfrak{I}_3|\lesssim (\|\partial_h^{N}\partial_1v^{h}\|_{L^2}+ \|\partial_h^{N-1}\partial_1v^{h}\|_{L^2} )\|\partial_h^{N-1}\nabla\,q\|_{L^2},
\end{split}
\end{equation*}
and
 \begin{equation*}\label{laplace-N-13}
\begin{split}
&|\mathfrak{I}_4|\lesssim\|\partial_h^{N-1}\partial_1^2v^{h}\|_{L^2(\Omega_f)}(\|\partial_h^{N-1}\partial_t v^{h}\|_{L^2}+\|\partial_h^{N+1}v^{h}\|_{L^2}).
\end{split}
\end{equation*}
For the nonlinear remainder terms $\mathfrak{I}_5$, $\mathfrak{I}_6$, and $\mathfrak{I}_7$, there hold
 \begin{equation*}\label{laplace-N-14}
\begin{split}
&|\mathfrak{I}_5|\lesssim \|\partial_h^{N-1}\partial_{\beta} v^1\|_{H^{\frac{1}{2}}(\Sigma_0)} \|\partial_h^{N-1}(\mathcal{B}_{9, i}^{\alpha}\partial_{\alpha}v^{i}  )\|_{H^{\frac{1}{2}}(\Sigma_0)}\\
&\lesssim \|\partial_h^{N-1}\partial_{\beta} v^1\|_{H^1} \|\partial_h^{N-1}(\mathcal{B}_{9, i}^{\alpha}\partial_{\alpha}v^{i}  )\|_{H^1}\lesssim \dot{\mathcal{D}}_{N}^{\frac{1}{2}} (E_3^{\frac{1}{2}} \dot{\mathcal{D}}_{N}^{\frac{1}{2}}  + E_N^{\frac{1}{2}} \dot{\mathcal{D}}_{3}^{\frac{1}{2}}),
\end{split}
\end{equation*}
 \begin{equation*}\label{laplace-N-16}
\begin{split}
&|\mathfrak{I}_6|\lesssim\|\partial_h^{N-1}(\mathcal{B}_{\beta+5, i}^{\alpha}\partial_{\alpha}v^{i})\|_{H^{\frac{1}{2}}(\Sigma_0)}\| \partial_h^{N-1}q\|_{H^{\frac{1}{2}}(\Sigma_0)}\\
&\lesssim\|\partial_h^{N-1}(\mathcal{B}_{\beta+5, i}^{\alpha}\partial_{\alpha}v^{i})\|_{H^1}\| \partial_h^{N-1}q\|_{H^1}\lesssim \dot{\mathcal{D}}_{N}^{\frac{1}{2}} (E_3^{\frac{1}{2}} \dot{\mathcal{D}}_{N}^{\frac{1}{2}}  + E_N^{\frac{1}{2}} \dot{\mathcal{D}}_{3}^{\frac{1}{2}}),
\end{split}
\end{equation*}
and
 \begin{equation*}\label{laplace-N-16}
\begin{split}
&|\mathfrak{I}_7|\lesssim \|\partial_h^{N-1}\partial_1^2v^{h}\|_{L^2(\Omega_f)}(\|\partial_h^{N-1}g_{\beta}\|_{L^2}
+\|\partial_h^{N-1}\widetilde{g}_{\beta\beta}\|_{L^2})\\
&\lesssim \|\partial_h^{N-1}\partial_1^2v^{h}\|_{L^2(\Omega_f)} (E_3^{\frac{1}{2}} \dot{\mathcal{D}}_{N}^{\frac{1}{2}}  + E_N^{\frac{1}{2}} \dot{\mathcal{D}}_{3}^{\frac{1}{2}}),
\end{split}
\end{equation*}
where we used the inequality in \eqref{est-g-2} in the last inequality.

Hence, we have
 \begin{equation*}\label{laplace-N-17}
\begin{split}
&\nu\|\partial_h^{N-1}\partial_1^2v^{\beta}\|_{L^2(\Omega_f)}^2
+\frac{1}{2}\frac{d}{dt}\|\partial_h^{N-1}\partial_{\beta}\xi^1\|_{L^2(\Sigma_0)}^2\\
&\lesssim \|\partial_h^{N}\nabla\,v\|_{L^2}^2+\dot{\mathcal{D}}_{N}^{\frac{1}{2}} (E_3^{\frac{1}{2}} \dot{\mathcal{D}}_{N}^{\frac{1}{2}}  + E_N^{\frac{1}{2}} \dot{\mathcal{D}}_{3}^{\frac{1}{2}})+(\|\partial_h^{N}\partial_1v^{h}\|_{L^2}+ \|\partial_h^{N-1}\partial_1v^{h}\|_{L^2} )\|\partial_h^{N-1}\nabla\,q\|_{L^2}\\
&+\|\partial_h^{N-1}\partial_1^2v^{h}\|_{L^2(\Omega_f)}(\|\partial_h^{N-1}\partial_t v^{h}\|_{L^2}+\|\partial_h^{N+1}\,v^{h}\|_{L^2}+E_3^{\frac{1}{2}} \dot{\mathcal{D}}_{N}^{\frac{1}{2}}  + E_N^{\frac{1}{2}} \dot{\mathcal{D}}_{3}^{\frac{1}{2}}),
\end{split}
\end{equation*}
which follows that
 \begin{equation}\label{laplace-N-18}
\begin{split}
&\nu\|\partial_h^{N-1}\partial_1^2v^{\beta}\|_{L^2(\Omega_f)}^2
+\frac{d}{dt}\|\partial_h^{N-1}\partial_{\beta}\xi^1\|_{L^2(\Sigma_0)}^2\\
&\lesssim \|\partial_h^{N}\nabla\,v\|_{L^2}^2+\|\partial_h^{N-1}\partial_t v^{h}\|_{L^2}^2+(\|\partial_h^{N}\partial_1v^{h}\|_{L^2}+ \|\partial_h^{N-1}\partial_1v^{h}\|_{L^2} )\|\partial_h^{N-1}\nabla\,q\|_{L^2}\\
&+\dot{\mathcal{D}}_{N}^{\frac{1}{2}} (E_3^{\frac{1}{2}} \dot{\mathcal{D}}_{N}^{\frac{1}{2}}  + E_N^{\frac{1}{2}} \dot{\mathcal{D}}_{3}^{\frac{1}{2}})  + E_N \dot{\mathcal{D}}_{3}.
\end{split}
\end{equation}
 Combining \eqref{laplace-N-18} with \eqref{laplace-N-8} yields
   \begin{equation*}\label{laplace-N-19}
 \begin{split}
&\frac{d}{dt}\|\partial_h^{N-1}\partial_{\beta}\xi^1\|_{L^2(\Sigma_0)}^2
+c_2\|\partial_h^{N-1}(\partial_1^2\,v^{h},\,\nabla\,q)\|_{L^2(\Omega)}^2\\
&\lesssim \|\partial_h^{N-2}(\partial_t v,\,\partial_h\partial_t v)\|_{L^2}^2+\|\partial_h^{N-1}(\partial_h\,v,\,v) \|_{H^1}^2  +\|\dot{\Lambda}_h^{\sigma_0}\partial_1^2\,v^{h}\|_{L^2(\Omega)}^2\\
&\qquad+\dot{\mathcal{D}}_{N}^{\frac{1}{2}} (E_3^{\frac{1}{2}} \dot{\mathcal{D}}_{N}^{\frac{1}{2}}  + E_N^{\frac{1}{2}} \dot{\mathcal{D}}_{3}^{\frac{1}{2}})  + E_N \dot{\mathcal{D}}_{3},
 \end{split}
\end{equation*}
which proves \eqref{laplace-N-0}, and completes the proof of Lemma \ref{lem-laplace-N-1}.
\end{proof}

With Lemmas \ref{lem-laplace-sigma0-1}, \ref{lem-laplace-N-1} in hand, we get
\begin{lem}\label{lem-laplace-xi-decay-1}
Under the assumption of Lemma \ref{lem-pseudo-energy-tv-1}, if $E_3(t) \leq 1$ for all the existence times $t$, then there holds
\begin{equation}\label{laplace-xi-decay-0}
\begin{split}
&\frac{d}{dt}(\|\dot{\Lambda}_h^{\sigma_0}\partial_{h}\xi^1\|_{L^2(\Sigma_0)}^2
+\sum_{k=1}^{N-1}\|\partial_h^{k+1}\xi^1\|_{L^2(\Sigma_0)}^2)
+c_4\bigg(\|\dot{\Lambda}_h^{\sigma_0}\,(\nabla^2\,v,\,\nabla\,q)\|_{L^2}^2\\
&\qquad\qquad
+\|\dot{\Lambda}_h^{\sigma_0+1}\xi^1\|_{H^\frac{1}{2}(\Sigma_0)}
+\sum_{k=1}^{N-1}\|\partial_h^{k}(\nabla^2\,v,\,\nabla\,q)\|_{L^2}^2
+\sum_{k=2}^{N-1}\|\partial_h^{k}\xi^1\|_{H^\frac{1}{2}(\Sigma_0)}\bigg)
\\
&\leq  C_4(\|\dot{\Lambda}_h^{\sigma_0}\nabla\,v\|_{L^2(\Omega)}^2 +\sum_{i=1}^N\|\partial_h^{i}\nabla\,v\|_{L^2(\Omega)}^2+\|\dot{\Lambda}_h^{\sigma_0}\partial_t v\|_{L^2}^2 +\sum_{k=1}^{N-1}\|\partial_h^{k}\partial_t v\|_{L^2}^2 )\\
&\qquad+C_4(\dot{\mathcal{D}}_{N}^{\frac{1}{2}} (E_3^{\frac{1}{2}} \dot{\mathcal{D}}_{N}^{\frac{1}{2}}  + E_N^{\frac{1}{2}} \dot{\mathcal{D}}_{3}^{\frac{1}{2}})+ E_N  \dot{\mathcal{D}}_{3}).
 \end{split}
\end{equation}
\end{lem}
\begin{proof}
Thanks to Lemmas \ref{lem-laplace-sigma0-1} and \ref{lem-laplace-N-1}, we get
\begin{equation}\label{laplace-xi-decay-1}
\begin{split}
&\frac{d}{dt}(\|\dot{\Lambda}_h^{\sigma_0}\partial_{h}\xi^1\|_{L^2(\Sigma_0)}^2
+\sum_{k=1}^{N-1}\|\partial_h^{k+1}\xi^1\|_{L^2(\Sigma_0)}^2)\\
&\qquad\qquad
+c_5\bigg(\|\dot{\Lambda}_h^{\sigma_0}\,(\nabla^2\,v,\,\nabla\,q)\|_{L^2}^2+\sum_{k=1}^{N-1}\|\partial_h^{k}(\nabla^2\,v,\,\nabla\,q)\|_{L^2}^2\bigg)
\\
&\leq  C_5(\|\dot{\Lambda}_h^{\sigma_0}\nabla\,v\|_{L^2(\Omega)}^2 +\sum_{i=1}^N\|\partial_h^{i}\nabla\,v\|_{L^2(\Omega)}^2+\|\dot{\Lambda}_h^{\sigma_0}\partial_t v\|_{L^2}^2 +\sum_{k=1}^{N-1}\|\partial_h^{k}\partial_t v\|_{L^2}^2 )\\
&\qquad+C_5(\dot{\mathcal{D}}_{N}^{\frac{1}{2}} (E_3^{\frac{1}{2}} \dot{\mathcal{D}}_{N}^{\frac{1}{2}}  + E_N^{\frac{1}{2}} \dot{\mathcal{D}}_{3}^{\frac{1}{2}})+ E_N  \dot{\mathcal{D}}_{3}).
 \end{split}
\end{equation}
On the other hand, due to \eqref{q-interface-1}, one has
\begin{equation*}
\begin{split}
& \|\dot{\Lambda}_h^{\sigma_0+1}\xi^1\|_{H^\frac{1}{2}(\Sigma_0)}
\lesssim\|\dot{\Lambda}_h^{\sigma_0+1}q\|_{H^\frac{1}{2}(\Sigma_0)}
+\|\|\dot{\Lambda}_h^{\sigma_0+1}\nabla_h\cdot v^h\|_{H^\frac{1}{2}(\Sigma_0)}+\|\dot{\Lambda}_h^{\sigma_0+1}(\mathcal{B}_{9, i}^{\alpha}\partial_{\alpha}v^{i}) \|_{H^\frac{1}{2}(\Sigma_0)}\\
&\lesssim\|\dot{\Lambda}_h^{\sigma_0}\partial_hq\|_{H^1(\Omega)}
+\|\dot{\Lambda}_h^{\sigma_0}\partial_h^2v\|_{H^1(\Omega)}+\|\dot{\Lambda}_h^{\sigma_0}\partial_h(\mathcal{B}_{9, i}^{\alpha}\partial_{\alpha}v^{i}) \|_{H^1(\Omega)},\\
& \|\partial_h^{k}\xi^1\|_{H^\frac{1}{2}(\Sigma_0)}
\lesssim\|\partial_h^{k}q\|_{H^\frac{1}{2}(\Sigma_0)}+\|\partial_h^{k}\nabla_h\cdot v^h\|_{H^\frac{1}{2}(\Sigma_0)}+\|\partial_h^{k}(\mathcal{B}_{9, i}^{\alpha}\partial_{\alpha}v^{i}) \|_{H^\frac{1}{2}(\Sigma_0)}\\
&\lesssim\|\partial_h^{k}q\|_{H^1(\Omega)}
+\|\partial_h^{k+1}v\|_{H^1(\Omega)}+\|\partial_h^{k}(\mathcal{B}_{9, i}^{\alpha}\partial_{\alpha}v^{i}) \|_{H^1(\Omega)} \quad (\text{for} \quad k=2, ..., N-1),\\
\end{split}
\end{equation*}
which follows that
\begin{equation*}\label{laplace-xi-decay-2}
\begin{split}
& \|\dot{\Lambda}_h^{\sigma_0+1}\xi^1\|_{H^\frac{1}{2}(\Sigma_0)}
\lesssim\|\dot{\Lambda}_h^{\sigma_0}\partial_hq\|_{H^1(\Omega)}
+\|\dot{\Lambda}_h^{\sigma_0}\partial_hv\|_{H^1(\Omega)}+\|\partial_h^3v\|_{H^1(\Omega)},\\
& \|\partial_h^{k}\xi^1\|_{H^\frac{1}{2}(\Sigma_0)}
\lesssim \|\partial_h^{k}q\|_{H^1(\Omega)}
+\|\partial_h^{k}\partial_hv\|_{H^1(\Omega)}+E_3^{\frac{1}{2}} \dot{\mathcal{D}}_{k+1}^{\frac{1}{2}}  + E_{k+1}^{\frac{1}{2}} \dot{\mathcal{D}}_{3}^{\frac{1}{2}}
\end{split}
\end{equation*}
with $k=2, ..., N-1$.

Hence, we have
\begin{equation*}\label{laplace-xi-decay-3}
\begin{split}
&\|\dot{\Lambda}_h^{\sigma_0+1}\xi^1\|_{H^\frac{1}{2}(\Sigma_0)}^2
+\sum_{k=2}^{N-1}\|\partial_h^{k}\xi^1\|_{H^\frac{1}{2}(\Sigma_0)}^2
\\
&\leq  C_6\bigg(\|\dot{\Lambda}_h^{\sigma_0}\partial_hq\|_{H^1(\Omega)}^2
+\|\dot{\Lambda}_h^{\sigma_0}\partial_hv\|_{H^1(\Omega)}+\|\partial_h^3v\|_{H^1(\Omega)}^2\\
&\qquad\qquad\qquad
+\sum_{k=2}^{N-1}(\|\partial_h^{k}q\|_{H^1(\Omega)}^2
+\|\partial_h^{k}\partial_hv\|_{H^1(\Omega)}^2)+E_3^{\frac{1}{2}} \dot{\mathcal{D}}_{N} + E_{N}  \dot{\mathcal{D}}_{3}\bigg),
 \end{split}
\end{equation*}
which along with \eqref{laplace-xi-decay-1} ensures \eqref{laplace-xi-decay-0}.
\end{proof}

\renewcommand{\theequation}{\thesection.\arabic{equation}}
\setcounter{equation}{0}
\section{Total energy estimates}\label{sect-total}

Combining Lemma \ref{lem-laplace-xi-decay-1} with Lemma \ref{lem-decay-total-1}, and taking the positive $\delta$ in Lemma \ref{lem-decay-total-1} so small that $\delta \leq \min\{1, \frac{c_1}{2C_1},\,\sqrt{\frac{c_1}{4C_4}},\,\frac{c_2}{4C_4}\}$, we immediately get the following energy estimate.

\begin{lem}\label{lem-tangrad-decay-total-1}
Let $N\geq 3$, under the assumption of Lemma \ref{lem-tan-pseudo-energy-1}, if $E_3(t) \leq 1$ for all the existence times $t$, then there holds
\begin{equation*}\label{decay-total-1}
\begin{split}
&\frac{d}{dt}\widehat{\dot{\mathcal{E}}}_{N, \delta}+\widehat{\dot{\mathcal{D}}}_{N, \delta}\leq C_5 ( E_N^{\frac{1}{2}} \dot{\mathcal{D}}_3^{\frac{1}{2}}\dot{\mathcal{D}}_N^{\frac{1}{2}}+  E_3^{\frac{1}{2}} \dot{\mathcal{D}}_N+E_N\,\dot{\mathcal{D}}_3)
\end{split}
\end{equation*}
with
\begin{equation*}\label{def-decay-total-1}
\begin{split}
&\widehat{\dot{\mathcal{E}}}_{N, \delta}:=\|\dot{\Lambda}_h^{\sigma_0}v\|_{L^2(\Omega)}^2
+\|\dot{\Lambda}_h^{\sigma}\xi^1\|_{L^2(\Sigma_0)}^2+\sum_{i=1}^N(\|\partial_h^{i}v\|_{L^2(\Omega)}^2
+\|\partial_h^{i}\xi^1\|_{L^2(\Sigma_0)}^2)\\
&\quad+\delta\bigg(\mathring{\mathcal{E}}(\dot{\Lambda}_h^{\sigma_0}\nabla\, v) +\sum_{k=1}^{N-1}\mathring{\mathcal{E}}(\partial_h^{k}\nabla\, v)\bigg)+\delta^2(\|\dot{\Lambda}_h^{\sigma_0}\partial_{h}\xi^1\|_{L^2(\Sigma_0)}^2
+\sum_{k=1}^{N-1}\|\partial_h^{k+1}\xi^1\|_{L^2(\Sigma_0)}^2)
\end{split}
\end{equation*}
and
\begin{equation*}\label{def-decay-total-2}
\begin{split}
&\widehat{\dot{\mathcal{D}}}_{N, \delta}:=\frac{c_1}{4}(\|\dot{\Lambda}_h^{\sigma_0}\nabla\,v\|_{L^2(\Omega)}^2 +\sum_{i=1}^N\|\partial_h^{i}\nabla\,v\|_{L^2(\Omega)}^2 )+\frac{c_2}{2}\delta(\|\dot{\Lambda}_h^{\sigma_0}\partial_t v\|_{L^2}^2 +\sum_{k=1}^{N-1}\|\partial_h^{k}\partial_t v\|_{L^2}^2 )\\
&\quad+\delta^2c_4\bigg(\|\dot{\Lambda}_h^{\sigma_0}\,(\nabla^2\,v,\,\nabla\,q)\|_{L^2}^2
+\sum_{k=1}^{N-1}\|\partial_h^{k}(\nabla^2\,v,\,\nabla\,q)\|_{L^2}^2\\
&\qquad\qquad\qquad\qquad\qquad\qquad
+\|\dot{\Lambda}_h^{\sigma_0+1}\xi^1\|_{H^\frac{1}{2}(\Sigma_0)}
+\sum_{k=2}^{N-1}\|\partial_h^{k}\xi^1\|_{H^\frac{1}{2}(\Sigma_0)}\bigg).
\end{split}
\end{equation*}
\end{lem}

From this, together with Lemma \ref{lem-bdd-total-N+1-1}, It is convenient to derive

\begin{lem}\label{lem-bdddecay-total-N+1-1}
Let $N\geq 3$, under the assumption of Lemma \ref{lem-tan-pseudo-energy-1}, if $(\lambda,\,\sigma_0) \in (0, 1)$ satisfies $1-\lambda< \sigma_0\leq 1-\frac{1}{2}\lambda$, and $E_3(t) \leq 1$ for all the existence times $t$, then there holds
\begin{equation*}\label{bdddecay-total-N+1-1}
\begin{split}
&\frac{d}{dt}(\widehat{\mathcal{E}}_{N+1, tan, \delta}+\widehat{\dot{\mathcal{E}}}_{N+1, \delta}) +(\widehat{\mathcal{D}}_{N+1, tan, \delta}+\widehat{\dot{\mathcal{D}}}_{N+1, \delta})\\
&\leq C_6 \bigg( E_{N+1}^{\frac{1}{2}} \dot{\mathcal{D}}_3^{\frac{1}{2}}\dot{\mathcal{D}}_{N+1}^{\frac{1}{2}}+  E_3^{\frac{1}{2}} \dot{\mathcal{D}}_{N+1}+E_{N+1}\,\dot{\mathcal{D}}_3\\
&\qquad\qquad\qquad+ E_3^{\frac{1}{2}} \dot{\mathcal{D}}_3^{\frac{1}{2}} \|\dot{\Lambda}_h^{-\lambda}\xi^1 \|_{H^{\frac{1}{2}}(\Sigma_0)}
+E_3 \|\dot{\Lambda}_h^{\sigma_0}\xi^1 \|_{H^{\frac{3}{2}}(\Sigma_0)}^2\bigg)+C_6\delta^{-1}E_3 \dot{\mathcal{D}}_3.
\end{split}
\end{equation*}
\end{lem}

The following comparison lemma discovers the equivalence of two types of the energy.
\begin{lem}[Comparison lemma]\label{lem-equiv-decay-total-N+1-1}
Let $N\geq 4$, under the assumption of Lemma \ref{lem-tan-pseudo-energy-1}, if$(\lambda,\,\sigma_0) \in (0, 1)$ satisfies $1-\lambda< \sigma_0\leq 1-\frac{1}{2}\lambda$, and $E_N(t) \leq 1$ for all the existence times $t$, then there exists a genius positive constant $\delta_0$  such that
\begin{equation}\label{equiv-decay-total-N+1-1}
\begin{split}
& \dot{\mathcal{E}}_N\thicksim \widehat{\dot{\mathcal{E}}}_{N, \delta_0},\quad   \dot{\mathcal{D}}_{N} \thicksim \widehat{\dot{\mathcal{D}}}_{N, \delta_0},\\
&\mathcal{E}_{N+1}\thicksim\widehat{\mathcal{E}}_{N+1, tan, \delta_0}+\widehat{\dot{\mathcal{E}}}_{N+1, \delta_0},\quad  \mathcal{D}_{N+1}\thicksim \widehat{\mathcal{D}}_{N+1, tan, \delta_0}+\widehat{\dot{\mathcal{D}}}_{N+1, \delta_0}.
\end{split}
\end{equation}
\end{lem}
\begin{proof}
In order to prove \eqref{equiv-decay-total-N+1-1}, we need only to estimate the nonlinear energy $\mathring{\mathcal{E}}(\mathcal{P}(\partial_h)\nabla\, v)$ from upper and low ones.
In fact, due to the definition of $\mathring{\mathcal{E}}(\mathcal{P}(\partial_h)\nabla\, v)$, we obtain
\begin{equation*}
\begin{split}
 & \mathring{\mathcal{E}}(\dot{\Lambda}_h^{\sigma_0}\nabla\, v)\leq \|\dot{\Lambda}_h^{\sigma_0}\nabla\, v\|_{L^2}^2+4\|\dot{\Lambda}_h^{\sigma_0}\nabla_h\cdot v^h\|_{L^2}^2+\|\dot{\Lambda}_h^{\sigma_0}\partial_1 v\|_{L^2}^2+\|\dot{\Lambda}_h^{\sigma_0}\nabla_h v\|_{L^2}^2 \\
   &\,+\frac{2}{\nu}\|\dot{\Lambda}_h^{\sigma_0} v\|_{L^2}\|\dot{\Lambda}_h^{\sigma_0}\nabla\,\mathcal{H}(\xi^1)\|_{L^2} +4\sum_{j=1}^3\|\dot{\Lambda}_h^{\sigma_0}\partial_1 v\|_{L^2}\|\dot{\Lambda}_h^{\sigma_0} (\mathcal{H}(\mathcal{B}_{j+5, i}^{\alpha})\partial_{\alpha}v^{i})\|_{L^2}\\
  &\leq 3\|\dot{\Lambda}_h^{\sigma_0}\nabla\, v\|_{L^2}^2+C_0(\|\dot{\Lambda}_h^{\sigma_0}\partial_hv\|_{L^2}^2+\|\dot{\Lambda}_h^{\sigma_0} v\|_{L^2}^2+\|\dot{\Lambda}_h^{\sigma_0}\xi^1\|_{H^{\frac{1}{2}}(\Sigma_0)}^2 ) +C_0E_3^{\frac{1}{2}}\|\dot{\Lambda}_h^{\sigma_0}\Lambda_h^2\nabla\,v\|_{L^2}^2,
        \end{split}
\end{equation*}
\begin{equation*}
\begin{split}
 & \sum_{k=1}^{N-1}\mathring{\mathcal{E}}(\partial_h^{k}\nabla\, v)\leq \sum_{k=1}^{N-1}\bigg(\|\partial_h^{k}\nabla\, v\|_{L^2}^2+4\|\partial_h^{k}\nabla_h\cdot v^h\|_{L^2}^2+2\|\partial_h^{k}\partial_1 v\|_{L^2}\,\|\partial_h^{k}\nabla_h v\|_{L^2} \\
   &\,+C_0\|\partial_h^{k}v\|_{L^2}\|\partial_h^k\xi^1 \|_{H^{\frac{1}{2}}(\Sigma_0)} +2\|\partial_h^{k}\partial_1 v\|_{L^2}\,\sum_{j=1}^3\|\partial_h^{k} (\mathcal{H}(\mathcal{B}_{j+5, i}^{\alpha})\partial_{\alpha}v^{i})\|_{L^2}\\
   &\leq \sum_{k=1}^{N-1}\bigg(2\|\partial_h^{k}\nabla\, v\|_{L^2}^2+C_0(\|\partial_h^{k}v\|_{L^2}^2+\|\partial_h^k\xi^1 \|_{H^{\frac{1}{2}}(\Sigma_0)}^2 +E_N^{\frac{1}{2}}\|\partial_h^{k}\nabla\, v\|_{L^2}^2)\bigg),
        \end{split}
\end{equation*}
and
\begin{equation*}\label{def-pseudo-energy-tv-0}
\begin{split}
 & \mathring{\mathcal{E}}(\dot{\Lambda}_h^{\sigma_0}\nabla\, v)\geq \frac{1}{2}\|\dot{\Lambda}_h^{\sigma_0}\nabla\, v\|_{L^2}^2\\
 &\qquad-C_0(\|\dot{\Lambda}_h^{\sigma_0}\partial_hv\|_{L^2}^2+\|\dot{\Lambda}_h^{\sigma_0} v\|_{L^2}^2+\|\dot{\Lambda}_h^{\sigma_0}\xi^1\|_{H^{\frac{1}{2}}(\Sigma_0)}^2+
 E_3^{\frac{1}{2}}\|\dot{\Lambda}_h^{\sigma_0}\Lambda_h^2\nabla\,v\|_{L^2}^2),
        \end{split}
\end{equation*}
\begin{equation*}\label{def-pseudo-energy-tv-0}
\begin{split}
\sum_{k=1}^{N-1}\mathring{\mathcal{E}}(\partial_h^{k}\nabla\, v)\geq \sum_{k=1}^N\bigg(&\frac{1}{2}\|\partial_h^{k}\nabla\, v\|_{L^2}^2-C_0(\|\partial_h^{k}\Lambda_hv\|_{L^2}^2+\|\partial_h^k\xi^1 \|_{H^{\frac{1}{2}}(\Sigma_0)}^2)\\
 &-C_0(E_{k+1}\|\partial_h\Lambda_h^2\,v\|_{L^2}^2+E_{3} \|\partial_h\Lambda_h^{k}\,v\|_{L^2}^2)\bigg).
        \end{split}
\end{equation*}
From which, we may immediately obtain \eqref{equiv-decay-total-N+1-1} if we take a suitable positive constant $\delta_0$. This finishes the proof of Lemma \ref{lem-equiv-decay-total-N+1-1}.
\end{proof}

With Lemma \ref{lem-equiv-decay-total-N+1-1} in hand, for simplicity, we denote $\widehat{\dot{\mathcal{E}}}_{N, \delta_0}$, $\widehat{\dot{\mathcal{D}}}_{N, \delta_0}$,
$\widehat{\mathcal{E}}_{N+1, tan, \delta_0}+\widehat{\dot{\mathcal{E}}}_{N+1, \delta_0}$, $\widehat{\mathcal{D}}_{N+1, tan, \delta_0}+\widehat{\dot{\mathcal{D}}}_{N+1, \delta_0}$ by $\widehat{\dot{\mathcal{E}}}_{N}$, $\widehat{\dot{\mathcal{D}}}_{N}$,
$\widehat{\mathcal{E}}_{N+1, tan}+\widehat{\dot{\mathcal{E}}}_{N+1}$, $\widehat{\mathcal{D}}_{N+1, tan}+\widehat{\dot{\mathcal{D}}}_{N+1}$ respectively, and $\widehat{\mathcal{E}}_{N+1}:=\widehat{\mathcal{E}}_{N+1, tan}+\widehat{\dot{\mathcal{E}}}_{N+1}$, $\widehat{\mathcal{D}}_{N+1}:=\widehat{\mathcal{D}}_{N+1, tan}+\widehat{\dot{\mathcal{D}}}_{N+1}$. And we also denote the positive constant $\mathfrak{C}_0\geq 1$ such that
\begin{equation}\label{equiv-total-cont-1}
\begin{split}
&\mathfrak{C}_0^{-1}\widehat{\dot{\mathcal{E}}}_{N}\leq  \dot{\mathcal{E}}_N\leq \mathfrak{C}_0 \widehat{\dot{\mathcal{E}}}_{N},\quad   \mathfrak{C}_0^{-1} \widehat{\dot{\mathcal{D}}}_{N}\leq\dot{\mathcal{D}}_{N} \leq \mathfrak{C}_0 \widehat{\dot{\mathcal{D}}}_{N},\\
&\mathfrak{C}_0^{-1}\widehat{\mathcal{E}}_{N+1}\leq \mathcal{E}_{N+1}\leq \mathfrak{C}_0\widehat{\mathcal{E}}_{N+1},\quad  \mathfrak{C}_0^{-1} \widehat{\mathcal{D}}_{N+1}\leq \mathcal{D}_{N+1}\leq \mathfrak{C}_0 \widehat{\mathcal{D}}_{N+1}.
\end{split}
\end{equation}

Therefore, we restate Lemmas \ref{lem-tangrad-decay-total-1} and \ref{lem-bdddecay-total-N+1-1} to get the total energy estimates.
\begin{lem}\label{lem-restate-decay-total-N+1-1}
Let $N\geq 4$, under the assumption of Lemma \ref{lem-tan-pseudo-energy-1}, if $(\lambda,\,\sigma_0) \in (0, 1)$ satisfies $1-\lambda< \sigma_0\leq 1-\frac{1}{2}\lambda$, and $E_N(t) \leq 1$ for all the existence times $t$, then there hold
\begin{equation}\label{restate-decay-total-N+1-1}
\begin{split}
&\frac{d}{dt}\widehat{\dot{\mathcal{E}}}_{N}+2\mathfrak{c}_1\dot{\mathcal{D}}_{N}\leq \mathfrak{C}_1 E_N^{\frac{1}{2}} \,\dot{\mathcal{D}}_N,
\end{split}
\end{equation}
and
\begin{equation}\label{restate-bdddecay-total-N+1-1}
\begin{split}
&\frac{d}{dt}\widehat{\mathcal{E}}_{N+1} +2\mathfrak{c}_1\mathcal{D}_{N+1}\leq \mathfrak{C}_1 \bigg( E_3^{\frac{1}{2}} \dot{\mathcal{D}}_{N+1}+E_{N+1}\,\dot{\mathcal{D}}_3+ E_3 \dot{\mathcal{D}}_3^{\frac{1}{2}} +E_3 \dot{\mathcal{E}}_3\bigg).
\end{split}
\end{equation}
\end{lem}

\renewcommand{\theequation}{\thesection.\arabic{equation}}
\setcounter{equation}{0}

\section{Proof of Theorem \ref{thm-main}}\label{sect-proof-mainthm}

Thanks to \eqref{restate-decay-total-N+1-1} and \eqref{restate-bdddecay-total-N+1-1} in Lemma \ref{lem-restate-decay-total-N+1-1}, we will show that the low order energy $E_{N}$ is uniformly bounded, while the low order energy $E_{N+1}$ grows algebraicly as the time $t$ goes to infinity by an inductive argument. Hence, inspired by the two-tiered energy method \cite{Guo-Tice-2}, we may find the decay estimate of the low order energy $\dot{\mathcal{E}}_N$. For this, we first prove the following decay estimate.

\begin{lem}[Decay estimate]\label{lem-diff-ineq-1}
If the non-negative function $f$ satisfies the differential inequality
\begin{equation}\label{diff-ineq-1}
\begin{cases}
&\frac{d}{dt}f+c_1\,f^{1+s}
 \leq 0, \quad \forall\,\, t>0, \\
 &f|_{t=0}=f_0
\end{cases}
\end{equation}
for two positive constants $s>0$ and $c_1>0$, then there holds for any $t>0$
\begin{equation}\label{diff-ineq-2}
\begin{split}
&f(t)\leq (c_1s)^{-\frac{1}{s}}((c_1s)^{-1}f^{-s}_0+t)^{-\frac{1}{s}}.
\end{split}
\end{equation}
\end{lem}
\begin{proof}
Due to the differential inequality in \eqref{diff-ineq-1}, we have
\begin{equation*}\label{diff-ineq-3}
\begin{split}
 &f^{-(1+s)}\frac{d}{dt}f+c_1 \leq 0.
\end{split}
\end{equation*}
Hence, one has
\begin{equation}\label{diff-ineq-4}
\begin{split}
\frac{d}{dt}f^{-s}\geq c_1\,s.
\end{split}
\end{equation}
Integrating \eqref{diff-ineq-4} on $[0, t]$ implies
\begin{equation*}\label{diff-ineq-5}
\begin{split}
f^{-s}(t)\geq f^{-s}_0+c_1s\,t,
\end{split}
\end{equation*}
which follows \eqref{diff-ineq-2}.
\end{proof}

We are now in a position to complete the proof of Theorem \ref{thm-main}.

\begin{proof}[Proof of Theorem \ref{thm-main}]
Thanks to the local well-posedness theorem (Theorem \ref{thm-local}), there
exists a positive time $T^{\ast}$ such that the system \eqref{eqns-pert-1} with initial data $(\xi_0,\,v_0)$  has a
unique solution  $(\xi, \, v)$ with
\begin{equation*}\label{linear-v-sigma0-6}
\begin{split}
&(\xi,\,v) \in \mathcal{C}([0, T^{\ast}); \mathfrak{F}_{N+1}).
 \end{split}
\end{equation*}

Without loss of generality, we may assume that $T^{\ast}$
maximal time of the existence to this solution.
The aim of what follows is to prove that $T^{\ast}=+\infty$.

Suppose, by way of contradiction, that  $T^{\ast}<+\infty$. We will show that solutions can actually be
extended past $T^{\ast}$ and satisfy $E_{N+1}(T^{\ast\ast})<+\infty$ for some $T^{\ast\ast}>T^{\ast}$, which contradicts the definition of $T^{\ast}$. It suffices to show that there is a positive constant $C_0$ such that
\begin{equation}\label{bdd-soln-extend-1}
\begin{split}
&\lim_{t\nearrow T^{\ast}}E_{N}(t) \leq C_0\,E_{N}(0),\quad \lim_{t\nearrow T^{\ast}}(\langle\,t\,\rangle^{-1}E_{N+1}(t)) \leq C_0\,E_{N+1}(0).
 \end{split}
\end{equation}
Indeed, if \eqref{bdd-soln-extend-1} holds, then $\|J-1\|_{L^\infty} \leq C_1 E_{3}(t)\leq C_0C_1\,E_{3}(0)\leq  C_0C_1\, \epsilon_0$, so we take $\epsilon_0 \leq\frac{1}{2C_1C_0}$ so that $\|J-1\|_{L^\infty} \leq \frac{1}{2}$, which ensures that $J(t) \geq \frac{1}{2}$ for any existence time $t$. While the second inequality implies $\lim_{t\nearrow T^{\ast}}E_{N+1}(t)\leq C_0\langle\,T^{\ast}\,\rangle\,E_{N+1}(0)$, which contradicts the fact that $T^{\ast}<+\infty$ is the maximal existence time by applying Theorem \ref{thm-local}, and we can
extend the solution past $T^{\ast}$ to some $T^{\ast\ast}>T^{\ast}$.

Let's now focus on the proof of \eqref{bdd-soln-extend-1}, which is reduced to claim that
\begin{equation*}\label{equal-time-1a}
\begin{split}
T^{\ast}=\sup\bigg\{&T\in [0, T^{\ast}]|(\xi,\,v) \in \mathcal{C}([0, T); \mathfrak{F}_{N+1}), \\
&E_{N}(t) \leq C_0(E_{N}(0)+\mathcal{E}_{N+1}(0)), \,E_{N+1}(t) \leq C_0(1+t)E_{N+1}(0)\bigg\},
 \end{split}
\end{equation*}
where the constant $C_0$ will be determined later.

Indeed, setting
\begin{equation}\label{equal-time-1}
\begin{split}
T_1\eqdefa&\sup\bigg\{T\in [0, T^{\ast}]|(\xi,\,v) \in \mathcal{C}([0, T); \mathfrak{F}_{N+1}), \quad\,\mathcal{E}_{N+1}(t) \leq 2\mathfrak{C}_0^2\mathcal{E}_{N+1}(0),\\
&E_{N}(t) \leq C_0(E_{N}(0)+\mathcal{E}_{N+1}(0)), \,E_{N+1}(t) \leq C_0(1+t)E_{N+1}(0)\quad \forall \,t \in [0, T)\bigg\},
 \end{split}
\end{equation}
we first get that $T_1\leq T^{\ast}$. On the other hand, we will verify that $T_1\geq T^{\ast}$. Otherwise, we suppose that $T_1< T^{\ast}$ by a contradiction argument.

From this, take $\epsilon_0 \in (0, \frac{\mathfrak{c}_1^2}{4\mathfrak{C}_1^2C_0}]$ small enough, then
\begin{equation*}\label{c-0-1}
\mathfrak{C}_1 E_N^{\frac{1}{2}} \leq  \mathfrak{C}_1C_0^{\frac{1}{2}}E_{N}(0)^{\frac{1}{2}}\leq  \mathfrak{C}_1C_0^{\frac{1}{2}}\epsilon_0^{\frac{1}{2}}\leq \frac{1}{2}\mathfrak{c}_1,
 \end{equation*}
 we get from \eqref{restate-decay-total-N+1-1} that
\begin{equation}\label{decay-est-N-1}
\begin{split}
&\frac{d}{dt}\widehat{\dot{\mathcal{E}}}_{N}+\mathfrak{c}_1\dot{\mathcal{D}}_{N} \leq 0.\\
\end{split}
\end{equation}
Hence, thanks to the definitions of $\dot{\mathcal{E}}_{N}$, ${\mathcal{E}}_{N}$, and $\dot{\mathcal{D}}_{N}$, we get from the interpolation inequality
\begin{equation*}
\begin{split}
& \|\dot{\Lambda}_h^{\sigma_0}\xi^1\|_{L^2(\Sigma_0)}
\lesssim\|\dot{\Lambda}_h^{1+\sigma_0}\xi^1\|_{L^2(\Sigma_0)}^{\frac{\lambda+\sigma_0}{1+\lambda+\sigma_0}}
\|\dot{\Lambda}_h^{-\lambda}\xi^1\|_{L^2(\Sigma_0)}^{\frac{1}{1+\lambda+\sigma_0}},\\
& \|\dot{\Lambda}_h^{N}\xi^1\|_{L^2(\Sigma_0)}\lesssim\|\dot{\Lambda}_h^{N-\frac{1}{2}}\xi^1\|_{L^2(\Sigma_0)}^{\frac{2}{3}}
\|\dot{\Lambda}_h^{N+1}\xi^1\|_{L^2(\Sigma_0)}^{\frac{1}{3}}
\end{split}
\end{equation*}
that
\begin{equation*}
\begin{split}
&\dot{\mathcal{E}}_{N} \leq \mathfrak{C}_2 (\dot{\mathcal{D}}_{N})^{\frac{\lambda+\sigma_0}{1+\lambda+\sigma_0}} (\mathcal{E}_{N+1})^{\frac{1}{1+\lambda+\sigma_0}},
\end{split}
\end{equation*}
where $\mathfrak{C}_2$ is the interpolation constant independent of $C_0$.

Hence, one has
\begin{equation*}
\begin{split}
&\widehat{\dot{\mathcal{E}}}_{N} \leq \mathfrak{C}_0 \dot{\mathcal{E}}_{N} \leq  \mathfrak{C}_0 \mathfrak{C}_2 (\dot{\mathcal{D}}_{N})^{\frac{\lambda+\sigma_0}{1+\lambda+\sigma_0}} (\mathcal{E}_{N+1})^{\frac{1}{1+\lambda+\sigma_0}}\leq  \mathfrak{C}_0 \mathfrak{C}_2  (2\mathfrak{C}_0^2)^{\frac{1}{1+\lambda+\sigma_0}} (\dot{\mathcal{D}}_{N})^{\frac{\lambda+\sigma_0}{1+\lambda+\sigma_0}} \mathcal{E}_{N+1}(0) ^{\frac{1}{1+\lambda+\sigma_0}}.
\end{split}
\end{equation*}
Denote $\widetilde{\mathfrak{C}}_2:=\mathfrak{C}_0 \mathfrak{C}_2  (2\mathfrak{C}_0^2)^{\frac{1}{1+\lambda+\sigma_0}}$, then there holds
\begin{equation*}
\begin{split}
&\widehat{\dot{\mathcal{E}}}_{N} \leq \widetilde{\mathfrak{C}}_2 (\dot{\mathcal{D}}_{N})^{\frac{\lambda+\sigma_0}{1+\lambda+\sigma_0}} \mathcal{E}_{N+1}(0)^{\frac{1}{1+\lambda+\sigma_0}},
\end{split}
\end{equation*}
which implies
\begin{equation*}
\begin{split}
&[\widetilde{\mathfrak{C}}_2^{-1}\widehat{\dot{\mathcal{E}}}_{N}]^{\frac{1+\lambda+\sigma_0}{\lambda+\sigma_0}}
\mathcal{E}_{N+1}(0)^{-\frac{1}{\lambda+\sigma_0}} \leq  \dot{\mathcal{D}}_{N}.
\end{split}
\end{equation*}
Therefore, from \eqref{decay-est-N-1}, we get
\begin{equation*}
\begin{split}
&\frac{d}{dt}\widehat{\dot{\mathcal{E}}}_{N}+\mathfrak{c}_1[\widetilde{\mathfrak{C}}_2^{-1}
\widehat{\dot{\mathcal{E}}}_{N}]^{\frac{1+\lambda+\sigma_0}{\lambda+\sigma_0}}
\mathcal{E}_{N+1}(0)^{-\frac{1}{\lambda+\sigma_0}} \leq 0,
\end{split}
\end{equation*}
that is,
\begin{equation*}
\begin{split}
&\frac{d}{dt}\widehat{\dot{\mathcal{E}}}_{N}+\mathfrak{c}_2\,\mathcal{E}_{N+1}(0)^{-\frac{1}{\lambda+\sigma_0}}\,
\widehat{\dot{\mathcal{E}}}_{N}^{\frac{1+\lambda+\sigma_0}{\lambda+\sigma_0}}
 \leq 0
\end{split}
\end{equation*}
with $\mathfrak{c}_2:=\mathfrak{c}_1[\widetilde{\mathfrak{C}}_2^{-1}
]^{\frac{1+\lambda+\sigma_0}{\lambda+\sigma_0}}=\mathfrak{c}_1[\mathfrak{C}_0\mathfrak{C}_2
]^{-\frac{1+\lambda+\sigma_0}{\lambda+\sigma_0}}(2\mathfrak{C}_0^2)^{-\frac{1}{\lambda+\sigma_0}}$.

Applying Lemma \ref{lem-restate-decay-total-N+1-1}, where we take
\begin{equation*}
\begin{split}
&c_1:=\mathfrak{c}_2\,\mathcal{E}_{N+1}(0)^{-\frac{1}{\lambda+\sigma_0}},\,s:=\frac{1}{\lambda+\sigma_0},\quad f(t):=\widehat{\dot{\mathcal{E}}}_{N}(t),
\end{split}
\end{equation*}
we infer that
\begin{equation*}
\begin{split}
&f(t)\leq (\frac{\mathfrak{c}_2}{\lambda+\sigma_0})^{-(\lambda+\sigma_0)}\mathcal{E}_{N+1}(0)\bigg(\mathfrak{c}_2^{-1}(\lambda+\sigma_0)
(\mathcal{E}_{N+1}(0)[\widehat{\dot{\mathcal{E}}}_{N}(0)]^{-1})^{\frac{1}{\lambda+\sigma_0}}
+t\bigg)^{-(\lambda+\sigma_0)}.
\end{split}
\end{equation*}
According to the definitions of $\widehat{\dot{\mathcal{E}}}_{N}(0)$ and $\mathcal{E}_{N+1}(0)$,  we denote that $\widehat{\dot{\mathcal{E}}}_{N}(0) \leq C_1(N) \mathcal{E}_{N+1}(0)$ with the positive constant $C_1(N)$ depending only on $N$, so one can see that
\begin{equation*}
\begin{split}
 \mathfrak{c}_2^{-1}(\lambda+\sigma_0)
(\mathcal{E}_{N+1}(0)[\widehat{\dot{\mathcal{E}}}_{N}(0)]^{-1})^{\frac{1}{\lambda+\sigma_0}} \geq \mathfrak{c}_2^{-1}(\lambda+\sigma_0)
C_1(N)^{-\frac{1}{\lambda+\sigma_0}},
\end{split}
\end{equation*}
which leads to
\begin{equation*}
\begin{split}
&f(t)\leq (\frac{\mathfrak{c}_2}{\lambda+\sigma_0})^{-(\lambda+\sigma_0)}\mathcal{E}_{N+1}(0)\bigg(\mathfrak{c}_2^{-1}(\lambda+\sigma_0)
C_1(N)^{-\frac{1}{\lambda+\sigma_0}}+t\bigg)^{-(\lambda+\sigma_0)}.
\end{split}
\end{equation*}

Therefore, we obtain
\begin{equation*}
\begin{split}
&\widehat{\dot{\mathcal{E}}}_{N}(t)  \leq \mathfrak{C}_3 \mathcal{E}_{N+1}(0)\langle\,t\,\rangle^{-(\lambda+\sigma_0)}
\end{split}
\end{equation*}
with $\mathfrak{C}_3 =\mathfrak{C}_3(\lambda+\sigma_0, C_1(N), \mathfrak{c}_2)$,
and then
\begin{equation*}\label{decay-EN-1}
\begin{split}
&\dot{\mathcal{E}}_{N}(t) \leq \mathfrak{C}_0\mathfrak{C}_3 \mathcal{E}_{N+1}(0)\langle\,t\,\rangle^{-(\lambda+\sigma_0)},\quad \int_0^t\dot{\mathcal{E}}_{N}\,d\tau \leq \widetilde{\mathfrak{C}}_3\mathcal{E}_{N+1}(0)
\end{split}
\end{equation*}
with $\widetilde{\mathfrak{C}}_3=\mathfrak{C}_0\mathfrak{C}_3 \int_0^t\langle\,\tau\,\rangle^{-(\lambda+\sigma_0)}\,d\tau$.

Since
\begin{equation*}
\begin{split}
&\frac{d}{dt}(\langle\,t\,\rangle^{(1+\lambda+\sigma_0)/2}\widehat{\dot{\mathcal{E}}}_{N})
+\mathfrak{c}_1\langle\,t\,\rangle^{(1+\lambda+\sigma_0)/2}\dot{\mathcal{D}}_{N} \\
&=\langle\,t\,\rangle^{(1+\lambda+\sigma_0)/2}\frac{d}{dt}(\widehat{\dot{\mathcal{E}}}_{N})
+\mathfrak{c}_1\langle\,t\,\rangle^{(1+\lambda+\sigma_0)/2}\dot{\mathcal{D}}_{N}+ \widehat{\dot{\mathcal{E}}}_{N}\langle\,t\,\rangle^{(1+\lambda+\sigma_0)/2-1}(1+\lambda+\sigma_0)/2,\\
\end{split}
\end{equation*}
from \eqref{decay-est-N-1}, we get
\begin{equation*}
\begin{split}
&\frac{d}{dt}(\langle\,t\,\rangle^{(1+\lambda+\sigma_0)/2}\widehat{\dot{\mathcal{E}}}_{N})
+\mathfrak{c}_1\langle\,t\,\rangle^{(1+\lambda+\sigma_0)/2}\dot{\mathcal{D}}_{N} \leq \widehat{\dot{\mathcal{E}}}_{N}\langle\,t\,\rangle^{(1+\lambda+\sigma_0)/2-1}(1+\lambda+\sigma_0)/2,
\end{split}
\end{equation*}
and then
\begin{equation*}
\begin{split}
&\frac{d}{dt}(\langle\,t\,\rangle^{(1+\lambda+\sigma_0)/2}\widehat{\dot{\mathcal{E}}}_{N})
+\mathfrak{c}_1\langle\,t\,\rangle^{(1+\lambda+\sigma_0)/2}\dot{\mathcal{D}}_{N} \leq \mathfrak{C}_3\mathcal{E}_{N+1}(0)
\langle\,t\,\rangle^{-(1+\lambda+\sigma_0)/2}(1+\lambda+\sigma_0)/2.
\end{split}
\end{equation*}
It follows that
\begin{equation}\label{decay-DN-1}
\begin{split}
&\int_0^t\langle\,\tau\,\rangle^{(1+\lambda+\sigma_0)/2}\dot{\mathcal{D}}_{N}(\tau)\,d\tau \leq \mathfrak{C}_4 \mathcal{E}_{N+1}(0),\quad \int_0^t\dot{\mathcal{D}}_{N}(\tau)^{\frac{1}{2}}\,d\tau \leq \widetilde{\mathfrak{C}}_4 \mathcal{E}_{N+1}(0)^{\frac{1}{2}}.
\end{split}
\end{equation}

With this in hand, we are ready to estimate $E_N$.

Since
\begin{equation*}
\begin{split}
&\|\dot{\Lambda}_h^{\sigma_0-1}\partial_1\xi^1(t)\|_{H^1}
 +\|\dot{\Lambda}_h^{\sigma_0}\partial_1\xi^h(t)\|_{H^1} +\|\dot{\Lambda}_h^{\sigma_0}\xi(t)\|_{L^2}+ \sum_{i=1}^{N-1}\|\partial_h^i\xi(t)\|_{H^2}\\
 & \leq \|\dot{\Lambda}_h^{\sigma_0-1}\partial_1\xi^1(0)\|_{H^1}
 +\|\dot{\Lambda}_h^{\sigma_0}\partial_1\xi^h(0)\|_{H^1} +\|\dot{\Lambda}_h^{\sigma_0}\xi(0)\|_{L^2}+ \sum_{i=1}^{N-1}\|\partial_h^i\xi(0)\|_{H^2}\\
 &\quad+\int_0^t(\|\dot{\Lambda}_h^{\sigma_0-1}\partial_1v^1\|_{H^1}+
 \|\dot{\Lambda}_h^{\sigma_0}\partial_1v^h\|_{H^1}+\|\dot{\Lambda}_h^{\sigma_0}v\|_{L^2}+ \sum_{i=1}^{N-1}\|\partial_h^iv\|_{H^2})\,d\tau,
\end{split}
\end{equation*}
one can see that
\begin{equation*}\label{def-energy-2}
\begin{split}
&E_{\text{h}, N}(t)\leq C_2(N)E_{\text{h}, N}(0)+C_2(N)(\int_0^t\dot{\mathcal{D}}_{N}^{\frac{1}{2}}\,d\tau)^2.
\end{split}
\end{equation*}
Due to \eqref{decay-DN-1}, we find that
\begin{equation*}
\begin{split}
&E_N(t)=E_{\text{h}, N}(t)+\mathcal{E}_{N}(t)\leq  C_2(N)E_{\text{h}, N}(0)+C_2(N)(\int_0^t\dot{\mathcal{D}}_{N}^{\frac{1}{2}}\,d\tau)^2+\mathcal{E}_{N}(t)\\
 &\leq  C_2(N)E_{\text{h}, N}(0)+(C_2(N)\widetilde{\mathfrak{C}}_4^2+2\mathfrak{C}_0^2) \mathcal{E}_{N+1}(0),
\end{split}
\end{equation*}
and then
\begin{equation}\label{bdd-EN-decay-1}
\begin{split}
&E_N(t)\leq  \mathfrak{C}_5(E_N(0)+\mathcal{E}_{N+1}(0)) \quad \text{with}\quad \mathfrak{C}_5:=C_2(N)(\widetilde{\mathfrak{C}}_4^2+1)+2\mathfrak{C}_0^2.
\end{split}
\end{equation}

Notice that
\begin{equation*}
\begin{split}
&\|\partial_h^{N}\nabla^2\xi\|_{L^2}(t)\leq \|\partial_h^{N}\nabla^2\xi(0)\|_{L^2}+\int_0^t \|\partial_h^{N}\nabla^2v\|_{L^2}(\tau)\,d\tau \leq \|\partial_h^{N}\nabla^2\xi\|_{L^2}(0)+\int_0^t \mathcal{D}^{\frac{1}{2}}_{N+1}\,d\tau,\\
\end{split}
\end{equation*}
which along with \eqref{bdd-EN-decay-1} and the definition of $T_1$ leads to
\begin{equation}\label{est-E-N+1-bdd-1}
\begin{split}
&{E}_{N+1}(t) \leq\|\partial_h^{N}\nabla^2\xi\|_{L^2}^2+2(\mathcal{E}_{N+1}+E_{N})\\
&\leq 2\|\partial_h^{N}\nabla^2\xi(0)\|_{L^2}^2+2\,t\,\int_0^t \mathcal{D}_{N+1}\,d\tau+4\mathfrak{C}_0^2 \mathcal{E}_{N+1}(0)+2\mathfrak{C}_5(E_N(0)+\mathcal{E}_{N+1}(0))\\
&\leq 2\,t\,\int_0^t \mathcal{D}_{N+1}\,d\tau+\widetilde{\mathfrak{C}}_5 E_{N+1}(0)
\end{split}
\end{equation}
with $\widetilde{\mathfrak{C}}_5:=2+4\mathfrak{C}_0^2 +2\mathfrak{C}_5$.

On the other hand, thanks to \eqref{restate-bdddecay-total-N+1-1}, we get
\begin{equation*}\label{diff-ineq-EN+1-1}
\begin{split}
\frac{d}{dt}{\widehat{\mathcal{E}}}_{N+1}+\mathfrak{c}_1{\mathcal{D}}_{N+1} \leq \mathfrak{C}_1 \bigg(E_{N+1}\,\dot{\mathcal{D}}_3+ E_3 (\dot{\mathcal{D}}_3^{\frac{1}{2}} +\dot{\mathcal{E}}_3)\bigg).
\end{split}
\end{equation*}
Hence, from \eqref{est-E-N+1-bdd-1}, one has
\begin{equation*}
\begin{split}
\frac{d}{dt}({\widehat{\mathcal{E}}}_{N+1}+\mathfrak{c}_1\int_0^t \mathcal{D}_{N+1}\,d\tau) &\leq  2\mathfrak{C}_1\,t\,\dot{\mathcal{D}}_3\,\int_0^t \mathcal{D}_{N+1}\,d\tau\\
&+ 2\mathfrak{C}_1 \mathfrak{C}_0^2\mathcal{E}_{N+1}(0)(\dot{\mathcal{D}}_3^{\frac{1}{2}} +\dot{\mathcal{E}}_3)+ \mathfrak{C}_1  \widetilde{\mathfrak{C}}_5 E_{N+1}(0)\,\dot{\mathcal{D}}_3.
\end{split}
\end{equation*}
Thanks to the Gronwall inequality, we get
\begin{equation*}
\begin{split}
& {\widehat{\mathcal{E}}}_{N+1}(t)+\mathfrak{c}_1\int_0^t \mathcal{D}_{N+1}\,d\tau  \leq \exp\{2\mathfrak{C}_1\mathfrak{c}_1^{-1}\,\int_0^t\tau\,\dot{\mathcal{D}}_3\,d\tau\} \\
&\times \bigg[{\widehat{\mathcal{E}}}_{N+1}(0)+  2\mathfrak{C}_1 \mathfrak{C}_0^2\mathcal{E}_{N+1}(0)\int_0^t (\dot{\mathcal{D}}_3^{\frac{1}{2}} +\dot{\mathcal{E}}_3)\,d\tau+ \mathfrak{C}_1\widetilde{\mathfrak{C}}_5 E_{N+1}(0)\,\int_0^t \dot{\mathcal{D}}_3\,d\tau\bigg],
\end{split}
\end{equation*}
which follows that
\begin{equation*}
\begin{split}
& {\widehat{\mathcal{E}}}_{N+1}(t)+\mathfrak{c}_1\int_0^t \mathcal{D}_{N+1}\,d\tau  \leq \exp\{2\mathfrak{C}_1\mathfrak{c}_1^{-1}\,\mathfrak{C}_4 \mathcal{E}_{N+1}(0)\} \\
&\times \bigg[{\widehat{\mathcal{E}}}_{N+1}(0)+  2\mathfrak{C}_1 \mathfrak{C}_0^2\mathcal{E}_{N+1}(0)\, (\widetilde{\mathfrak{C}}_4 \mathcal{E}_{N+1}(0)^{\frac{1}{2}}+\widetilde{\mathfrak{C}}_3 \mathcal{E}_{N+1}(0))+ \mathfrak{C}_1  \widetilde{\mathfrak{C}}_5 E_{N+1}(0)\,\mathfrak{C}_4\mathcal{E}_{N+1}(0)\bigg].
\end{split}
\end{equation*}
Therefore, we get
\begin{equation*}\label{E-N+1-est-bdd-1}
\begin{split}
& {\widehat{\mathcal{E}}}_{N+1}(t)+\mathfrak{c}_1\int_0^t \mathcal{D}_{N+1}\,d\tau\leq \exp\{2\mathfrak{C}_1\mathfrak{c}_1^{-1}\,\mathfrak{C}_4 \epsilon_0\}  \widehat{\mathcal{E}}_{N+1}(0)(1+ \mathfrak{C}_6\epsilon_0^{\frac{1}{2}})
\end{split}
\end{equation*}
with $\mathfrak{C}_6=2\mathfrak{C}_1 \mathfrak{C}_0^3 (\widetilde{\mathfrak{C}}_4 +\widetilde{\mathfrak{C}}_3)+ \mathfrak{C}_1  \widetilde{\mathfrak{C}}_5 \mathfrak{C}_4\mathfrak{C}_0 $.

Taking $\epsilon_0 \leq \min\{\frac{\mathfrak{c}_1}{2\mathfrak{C}_1\,\mathfrak{C}_4 }\log(\frac{5}{4}),\,\frac{1}{25\mathfrak{C}_6^2}\}$, we obtain
\begin{equation*}\label{E-N+1-est-bdd-2}
\begin{split}
\sup_{\tau\in [0, t]}\widehat{\mathcal{E}}_{N+1}(\tau) +\mathfrak{c}_1\int_0^t\mathcal{D}_{N+1}(\tau)\,d\tau \leq \frac{3}{2}\widehat{\mathcal{E}}_{N+1}(0),
\end{split}
\end{equation*}
and then
\begin{equation*}\label{E-N+1-est-bdd-3}
\begin{split}
&\mathcal{E}_{N+1}(t) \leq \frac{3}{2}\mathfrak{C}_0^2 \mathcal{E}_{N+1}(0).
\end{split}
\end{equation*}
Therefore, we get
\begin{equation*}\label{E-N+1-est-bdd-4}
\begin{split}
{E}_{N+1}(t) &\leq 2\,t\,\int_0^t \mathcal{D}_{N+1}\,d\tau+\mathfrak{C}_5 E_{N+1}(0)\\
&\leq  t\,\frac{3}{\mathfrak{c}_1}\mathfrak{C}_0^2 \mathcal{E}_{N+1}(0)+\mathfrak{C}_5 E_{N+1}(0)\leq  \widetilde{\mathfrak{C}}_6\,(t+1)\,E_{N+1}(0))
\end{split}
\end{equation*}
with $\widetilde{\mathfrak{C}}_6=\frac{3}{\mathfrak{c}_1}\mathfrak{C}_0^2 +\mathfrak{C}_5$.

Notice that all the constants $\mathfrak{C}_i,\, \widetilde{\mathfrak{C}}_j$ (with $i=0, 1, ..., 6$, $j=1,...,6$) above are independent of $C_0$.

We take $\epsilon_0$ small enough (satisfying above conditions) and the constant $C_0=2(\mathfrak{C}_5+\widetilde{\mathfrak{C}}_6)$ in \eqref{equal-time-1}, then, according to the above estimates, for any $t \in [0, T_1)$, there hold that
\begin{equation*}\label{contradict-1-1}
\begin{split}
&\mathcal{E}_{N+1}(t) \leq \frac{3}{2}\mathfrak{C}_0^2\mathcal{E}_{N+1}(0),\,
E_{N}(t) \leq \frac{1}{2}C_0(E_{N}(0)+\mathcal{E}_{N+1}(0)), \\
&E_{N+1}(t) \leq \frac{1}{2}C_0(1+t)E_{N+1}(0),
 \end{split}
\end{equation*}
which contradicts the fact that $T_1$ is maximal by the definition. This ends the proof of Theorem \ref{thm-main}.
\end{proof}

\renewcommand{\theequation}{\thesection.\arabic{equation}}
\setcounter{equation}{0}

\appendix

\renewcommand{\theequation}{\thesection.\arabic{equation}}
\setcounter{equation}{0}

\section{Proof of Theorem \ref{thm-nondecay}}\label{appendix-1}

\begin{proof}[Proof of Theorem \ref{thm-nondecay}]

Motivated by the method in \cite{Beale-1981}, we shall suppose that $v^{(1)}$ is known and show that a contradiction arises in solving for $v^{(2)}$.

We first assume that $\nabla\,\cdot\,\theta=0$. Because of the assumption
about $(v(\varepsilon),\,q(\varepsilon))$, the flow map $\xi(\varepsilon)$ satisfies
\begin{equation*}\label{expansion-xi-1a}
  \begin{split}
   \xi(t)= \xi(0)+\int_0^t v\,d\tau=\varepsilon\,(\theta+\int_0^t v^{(1)}\,d\tau)+\varepsilon^2\,\int_0^t v^{(2)}\,d\tau+\varepsilon^3\,\int_0^t v^{(3)}\,d\tau,
  \end{split}
\end{equation*}
and then $\xi(\varepsilon)$ has
 an expansion
 \begin{equation}\label{expansion-xi-1}
  \begin{split}
   \xi(\varepsilon)=\varepsilon\,\xi^{(1)}+\varepsilon^2\,\xi^{(2)}+ \varepsilon^3\,\xi^{(3)},
  \end{split}
\end{equation}
and also
\begin{equation}\label{determine-eta-1}
  \begin{split}
  &J(\varepsilon)=1+\varepsilon \nabla\cdot\xi^{(1)}+O(\varepsilon^2),\quad  \mathcal{A}_i^j(\varepsilon)=\delta_{i}^j-\varepsilon\,\partial_i\xi^{(1)}_j
+O(\varepsilon^2).
  \end{split}
\end{equation}
Since $\nabla_{J\mathcal{A}}\cdot v=0$, vanishing the order $\varepsilon$ of it, we get from \eqref{expansion-v-1} and \eqref{determine-eta-1} that
 \begin{equation}\label{div-cond-v-1}
  \begin{split}
   \nabla\,\cdot\, v^{(1)}=0.
  \end{split}
\end{equation}
While from \eqref{expansion-xi-1} and the fact $\nabla\cdot\theta=0$, it follows
\begin{equation}\label{div-cond-v-2}
  \begin{split}
&\varepsilon\,\nabla\cdot\xi^{(1)}+\varepsilon^2\,(\nabla\cdot\xi^{(2)}+ \varepsilon\,\nabla\cdot\xi^{(3)})\\
&=   \varepsilon(\int_0^t \nabla\cdot v^{(1)}\,d\tau)+\varepsilon^2\int_0^t (   \nabla\cdot\,v^{(2)}+\varepsilon    \nabla\cdot\,v^{(3)})\,d\tau,
  \end{split}
\end{equation}
then from \eqref{div-cond-v-1}, vanishing the order $\varepsilon$ in \eqref{div-cond-v-2} implies
 \begin{equation}\label{div-cond-xi-1}
  \begin{split}
   \nabla\cdot\xi^{(1)}=0,
  \end{split}
\end{equation}

The next divergence term $v^{(2)}$ is determined by $v^{(1)}$ from $\nabla_{\mathcal{A}}\cdot v=0$. Indeed, thanks to \eqref{div-cond-v-1} and \eqref{determine-eta-1}, one has
\begin{equation}\label{div-exp-1}
  \begin{split}
  & 0=\nabla_{\mathcal{A}}\cdot v=\mathcal{A}_i^j\partial_jv_i=(\delta_{i}^j-\varepsilon\,\partial_i\xi^{(1)}_j
+O(\varepsilon^2))(\varepsilon\,\partial_jv^{(1)}_i+\varepsilon^2\,\partial_jv^{(2)}_i+O(\varepsilon^3))\\
&=\varepsilon^2(\nabla\,\cdot\,v^{(2)}-\partial_i\xi^{(1)}_j\partial_jv^{(1)}_i)+O(\varepsilon^3).
  \end{split}
\end{equation}
Vanishing the order $\varepsilon^2$ in \eqref{div-exp-1} leads to
\begin{equation}\label{div-exp-2}
  \begin{split}
\nabla\,\cdot\,v^{(2)}=\partial_i\xi^{(1)}_j\partial_jv^{(1)}_i,
  \end{split}
\end{equation}
and then
\begin{equation}\label{div-exp-2a}
  \begin{split}
\nabla\,\cdot\,v^{(2)}=\frac{1}{2}\partial_j(\partial_i\xi^{(1)}_jv^{(1)}_i+\xi^{(1)}_i\partial_iv^{(1)}_j)
  =\frac{1}{2}\partial_t\nabla\cdot(\xi^{(1)}\cdot\nabla\xi^{(1)}),
  \end{split}
\end{equation}
where we used \eqref{div-cond-v-1} and \eqref{div-cond-xi-1} in the second equality.

Set
\begin{equation*}\label{div-exp-3}
  \begin{split}
  &\hbar(t, x)\eqdefa\int_{0}^t\nabla\,\cdot\,v^{(2)}\,d\tau, \quad \mathfrak{l}(t)\eqdefa \int_{\Omega}\hbar(t, x)\,dx,\\
    \end{split}
\end{equation*}
we obtain from \eqref{div-exp-2a} that
\begin{equation*}\label{div-exp-3a}
  \begin{split}
  &\hbar(t, x)=\frac{1}{2}\nabla\cdot(\xi^{(1)}_1\partial_1\xi^{(1)}+\xi^{(1)}_h\cdot\nabla_h\xi^{(1)})(t)
  -\frac{1}{2}\nabla\cdot(\theta_1\partial_1\theta+\theta_h\cdot\nabla_h\theta),\\
    \end{split}
\end{equation*}
and
\begin{equation*}\label{div-exp-4}
  \begin{split}
  &\mathfrak{l}(t)=\frac{1}{2}\int_{\Sigma_0}(\xi^{(1)}_1\partial_1\xi^{(1)}_1+\xi^{(1)}_h\cdot\nabla_h\xi^{(1)}_1)(t)\,dS
  -\frac{1}{2}\int_{\Sigma_0}(\theta_1\partial_1\theta_1+\theta_h\cdot\nabla_h\theta_1)\,dS\\
  &=\frac{1}{2}\int_{\Sigma_0}(-\xi^{(1)}_1\nabla_h\cdot\xi^{(1)}_h+\xi^{(1)}_h\cdot\nabla_h\xi^{(1)}_1)(t)\,dS
  -\frac{1}{2}\int_{\Sigma_0}(\theta_1\partial_1\theta_1-\nabla_h\cdot\theta_h\theta_1)\,dS\\
   &=\int_{\Sigma_0}(\xi^{(1)}_h\cdot\nabla_h\xi^{(1)}_1)(t)\,dS
   -\int_{\Sigma_0}\theta_1\partial_1\theta_1 \,dS.
  \end{split}
\end{equation*}
Due to \eqref{cond-infty-xi-1},
\begin{equation*}\label{div-exp-5}
  \begin{split}
  &|\int_{\Sigma_0}(\xi^{(1)}_h\cdot\nabla_h\xi^{(1)}_1)(t)\,dS|\lesssim \|\xi^{(1)}_h(t)\|_{L^2(\Sigma_0)} \|\nabla_h\xi^{(1)}_1(t)\|_{L^2(\Sigma_0)}\rightarrow 0 \quad (\text{as} \quad t\rightarrow+\infty),
  \end{split}
\end{equation*}
applying Lebesgue Dominated Convergence Theorem yields
\begin{equation*}\label{div-exp-6}
  \begin{split}
  \int_{0}^\infty\int_{\Omega}\nabla\,\cdot\,v^{(2)}\,dxd\tau&=\mathfrak{l}(\infty)=\int_{\Sigma_0}(\xi^{(1)}_h\cdot\nabla_h\xi^{(1)}_1)(\infty)\,dS
   -\int_{\Sigma_0}\theta_1\partial_1\theta_1 \,dS\\
  &
  = -\int_{\Sigma_0}\theta_1\partial_1\theta_1 \,dS.
  \end{split}
\end{equation*}
On the other hand, set $\mathfrak{a}(t, x_h)\eqdefa \int_0^t\int_{-\underline{b}}^0\nabla\,\cdot\,v^{(2)}\,dx_1\,d\tau $, one can see
\begin{equation*}\label{div-exp-7}
  \begin{split}
  \mathfrak{a}(t, x_h)&=\int_0^t v^{(2)}_1|_{\Sigma_0}\,d\tau
      +\int_0^t\int_{-\underline{b}}^0\nabla_h\,\cdot\,v^{(2)}_h\,dx_1\,d\tau\\
           &  =\eta^{(2)}_1|_{\Sigma_0}(t)-\eta^{(2)}_1|_{\Sigma_0}(0)
      +\int_0^t\int_{-\underline{b}}^0\nabla_h\,\cdot\,v^{(2)}_h\,dx_1\,d\tau,\\
        \end{split}
\end{equation*}
then
\begin{equation}\label{div-exp-8}
  \begin{split}
\mathfrak{a}(\infty, x_h)&=\eta^{(2)}_1|_{\Sigma_0}(\infty)-\eta^{(2)}_1|_{\Sigma_0}(0)
      +\nabla_h\,\cdot\int_0^\infty\int_{-\underline{b}}^0\,v^{(2)}_h\,dx_1\,d\tau\\
      &=\int_0^\infty\int_{-\underline{b}}^0\nabla_h\cdot v^{(2)}_h\,dx_1\,d\tau.
  \end{split}
\end{equation}

While from \eqref{div-exp-2}, a direct computation implies
\begin{equation*}\label{div-exp-10}
  \begin{split}
  \mathfrak{a}(t, x_h) =&\int_0^t\int_{-\underline{b}}^0\partial_i\xi^{(1)}_j\partial_jv^{(1)}_i\,dx_1\,d\tau
  =\int_0^t\int_{-\underline{b}}^0\partial_\gamma\xi^{(1)}_1\partial_1v^{(1)}_\gamma\,dx_1\,d\tau\\
  &\quad+\int_0^t\int_{-\underline{b}}^0(\partial_1\xi^{(1)}_1\partial_1v^{(1)}_1
  +\partial_1\xi^{(1)}_\alpha\partial_\alpha v^{(1)}_1+\partial_\gamma\xi^{(1)}_\alpha\partial_\alpha v^{(1)}_\gamma)\,dx_1\,d\tau,
        \end{split}
\end{equation*}
where
\begin{equation*}\label{div-exp-12}
  \begin{split}
  &\int_0^t\int_{-\underline{b}}^0\partial_\gamma\xi^{(1)}_1\partial_1v^{(1)}_\gamma\,dx_1\,d\tau=
  \int_0^t\int_{-\underline{b}}^0\partial_\gamma\xi^{(1)}_1\partial_1\partial_t\xi^{(1)}_\gamma\,dx_1\,d\tau\\
  &=
\int_{-\underline{b}}^0(\partial_\gamma\xi^{(1)}_1\partial_1 \xi^{(1)}_\gamma)|_{\tau=0}^t\,dx_1-
  \int_0^t\int_{-\underline{b}}^0\partial_\gamma v^{(1)}_1\partial_1\xi^{(1)}_\gamma\,dx_1\,d\tau\\
    &=
\int_{-\underline{b}}^0\partial_\gamma\xi^{(1)}_1(t)\partial_1\xi^{(1)}_\gamma(t))\,dx_1
-\int_{-\underline{b}}^0(\partial_\gamma\theta^{(1)}_1\partial_1\theta^{(1)}_\gamma)\,dx_1-
  \int_0^t\int_{-\underline{b}}^0\partial_\gamma v^{(1)}_1\partial_1\xi^{(1)}_\gamma\,dx_1\,d\tau,
        \end{split}
\end{equation*}
so it can be deduced that
\begin{equation*}\label{div-exp-13}
  \begin{split}
  &\mathfrak{a}(t, x_h)=\int_{-\underline{b}}^0\partial_\gamma\xi^{(1)}_1(t)\partial_1\xi^{(1)}_\gamma(t))\,dx_1
-\int_{-\underline{b}}^0(\partial_\gamma\theta^{(1)}_1\partial_1\theta^{(1)}_\gamma)\,dx_1\\
&+\int_0^t\int_{-\underline{b}}^0(\nabla_h\cdot\xi^{(1)}_h\nabla_h\cdot v^{(1)}_h
  +\partial_1\xi^{(1)}_\alpha\partial_\alpha v^{(1)}_1+\partial_\gamma\xi^{(1)}_\alpha\partial_\alpha v^{(1)}_\gamma-\partial_\gamma v^{(1)}_1\partial_1\xi^{(1)}_\gamma)\,dx_1\,d\tau.
        \end{split}
\end{equation*}
We thus prove that
\begin{equation*}\label{div-exp-14}
  \begin{split}
  \|\mathfrak{a}(t, x_h)\|_{L^1(\mathbb{R}^2_h)}\lesssim &\|\partial_\gamma\xi^{(1)}_1(t)\|_{L^2(\Omega)}
  \|\partial_1\xi^{(1)}_\gamma(t)\|_{L^2(\Omega)}+\|\partial_\gamma\theta^{(1)}_1\|_{L^2(\Omega)}
  \|\partial_1\theta^{(1)}_\gamma\|_{L^2(\Omega)}\\
  &+\int_0^t\|\nabla\xi^{(1)}\|_{L^2(\Omega)}\|\nabla_h\cdot v^{(1)}\|_{L^2(\Omega)}\,d\tau,
        \end{split}
\end{equation*}
which follows
\begin{equation*}\label{div-exp-15}
  \begin{split}
  &\|\mathfrak{a}(\infty, x_h)\|_{L^1(\mathbb{R}^2_h)}\lesssim \|\partial_h\xi^{(1)}_1\|_{L^\infty([0, {\infty}); L^2(\Omega))} \|\partial_1\xi^{(1)}_h\|_{L^\infty([0, {\infty}); L^2(\Omega))}\\
  &
+\|\partial_h\theta^{(1)}_1\|_{L^2(\Omega)}\|\partial_1\theta^{(1)}_h\|_{L^2(\Omega)}
+\|\nabla\xi^{(1)}\|_{L^\infty([0, {\infty}); L^2(\Omega))} \int_0^\infty\|\partial_h v^{(1)}\|_{L^2(\Omega)}\,d\tau<+\infty,
        \end{split}
\end{equation*}
namely, $\mathfrak{a}(\infty, x_h)\in L^1(\mathbb{R}^2_h)$.
From this, integrating $\mathfrak{a}(\infty, x_h)$ of \eqref{div-exp-8} with respect to $x_h\in\mathbb{R}^2_h$ yields
\begin{equation*}\label{div-exp-16}
  \begin{split}
\int_{\mathbb{R}^2_h}\mathfrak{a}(\infty, x_h)\,dx_h=\int_0^\infty\int_{-\underline{b}}^0\int_{\mathbb{R}^2_h}\nabla_h\cdot v^{(2)}_h\,dx_hdx_1\,d\tau=0,
  \end{split}
\end{equation*}
that is
\begin{equation}\label{div-exp-21}
  \begin{split}
&0=\int_{\Sigma_0}\mathfrak{a}(\infty, x_h)\,dS=\mathfrak{l}(\infty).
  \end{split}
\end{equation}

Aiming for a contradiction by choosing $\theta$ so that $\mathfrak{l}(\infty)$ is nonzero, suppose that
$a \in C_{c}^{\infty}(\Sigma_0; \mathbb{R})$ is nonzero. Let $w \in \mathcal{H}_{0, \text{tan}, 3}^2$ be some vector field such that $w_3=a$,
$\partial_1w_3=a$ on $\Sigma_0$, $w_2=\partial_1w_2=0$ on $\Sigma_0$. If $\theta=\nabla\times w$, we have $\nabla\cdot \theta=0$ and $\theta_1=\partial_2w_3-\partial_3w_2=\partial_2w_3=\partial_2a$,  $\partial_1\theta_1=\partial_1\partial_2w_3-\partial_1\partial_3w_2=\partial_2\partial_1w_3=\partial_2a$.
Hence, one has
\begin{equation*}
  \begin{split}
&\mathfrak{l}(\infty)=-\int_{\Sigma_0}\theta_1\partial_1\theta_1 \,dS=-\int_{\Sigma_0}(\partial_2a)^2\,dS<0,
  \end{split}
\end{equation*}
which contradicts \eqref{div-exp-21}. This ends the proof of Theorem \ref{thm-nondecay}.
\end{proof}

\bigbreak \noindent {\bf Acknowledgments.}
The author would like to thank Professor J. Thomas Beale for his valuable suggestions and constructive comments.
This work is supported in part by the National Natural Science Foundation of China under the Grant 11571279.

\end{document}